%% file: moe-master.tex
\documentclass[12pt,a4paper]{amsart}  
\usepackage{graphicx,amssymb,stmaryrd,xspace}
\RequirePackage[raiselinks=false,colorlinks=true,citecolor=blue,urlcolor=blue,linkcolor=blue,bookmarksopen=true]{hyperref}

\def\inputfile{moe-master.tex}
\usepackage{fancyheadings}
\makeatletter

\renewenvironment{abstract}{%
  \ifx\maketitle\relax
    \ClassWarning{\@classname}{Abstract should precede
      \protect\maketitle\space in AMS document classes; reported}%
  \fi
  \global\setbox\abstractbox=\vtop \bgroup
    \normalfont\Small
    \list{}{\labelwidth\z@
      \leftmargin2pc \rightmargin\leftmargin
      \listparindent\normalparindent \itemindent\z@
      \parsep\z@ \@plus\p@
      
    }%
    \item[\hskip\labelsep\scshape\abstractname.]%
}{%
  \endlist\egroup
  \ifx\@setabstract\relax \@setabstracta \fi
}
\def\@setabstract{\@setabstracta \global\let\@setabstract\relax}
\def\@setabstracta{%
  \ifvoid\abstractbox
  \else
    \skip@20\p@ \advance\skip@-\lastskip
    \advance\skip@-\baselineskip \vskip\skip@
    \box\abstractbox
    \prevdepth\z@ 
  \fi
}

\makeatother

\usepackage{texdraw}
\input{txdtools}

\everytexdraw{%
	\drawdim pt \linewd 0.5 \textref h:C v:C
	\setunitscale 1
}
\newcommand{\onedot}{
  \bsegment
    \move (0 0) \fcir f:0 r:2
  \esegment
}
\newcommand{\Onedot}{
  \bsegment
    \move (0 0) \fcir f:0 r:2.5
  \esegment
}

\usepackage[v2,all]{xy}
\newcommand{\ulpullback}[1][ul]{\save*!/#1-4ex/#1:(-1,1)@^{|-}\restore}
\newcommand{\dlpullback}[1][dl]{\save*!/#1-4ex/#1:(-1,1)@^{|-}\restore}
\newcommand{\urpullback}[1][ur]{\save*!/#1-4ex/#1:(-1,1)@^{|-}\restore}
\newcommand{\drpullback}[1][dr]{\save*!/#1-4ex/#1:(-1,1)@^{|-}\restore}

\usepackage{ stmaryrd }
\newcommand{\genmap}{\rightarrow\Mapsfromchar}

\def\eq{{\mathrm{eq}}}
\def\iso{{\mathrm{iso}}}
\def\Map{{\mathrm{Map}}}
\newcommand{\Eq}{\operatorname{Eq}}

\def\M{M\"obius\xspace}
\newcommand{\FILT}{tight\xspace}
\newdir{ >}{{}*!/-7pt/@{>}}
  
\def\overarrow#1{{\vec{#1}}}
\def\nondeg{\overarrow}

\newcommand{\Phieven}{\Phi_{\text{\rm even}}}
\newcommand{\Phiodd}{\Phi_{\text{\rm odd}}}
\newcommand{\Beven}{\B_{\text{\rm even}}}
\newcommand{\Bodd}{\B_{\text{\rm odd}}}

\newcommand{\wtil}{\widetilde}

\newcommand{\ov}{\overline}
\newcommand{\un}{\underline}
\newcommand{\Deltagen}{\Delta_{\text{\rm gen}}}
\newcommand{\Deltainj}{\Delta_{\text{\rm inj}}}

\newcommand{\culf}{\mathrm{cULF}}

\newcommand{\bigcat}{\widehat{\kat{Cat}}}
\renewcommand{\Im}{\operatorname{Im}}
\newcommand{\Bij}{\operatorname{Bij}}
\newcommand{\aInt}{\kat{aInt}}
\newcommand{\Int}{\kat{Int}}
\newcommand{\Decomp}{\kat{Dcmp}}
\newcommand{\cDecomp}{\kat{cDcmp}}
\newcommand{\Vect}{\kat{Vect}}
\newcommand{\vect}{\kat{vect}}
\newcommand{\FD}{\kat{FD}}
\newcommand{\cFD}{\kat{cFD}}
\newcommand{\Dec}{\operatorname{Dec}}

\newcommand{\Nat}{\operatorname{Nat}}
\newcommand{\Cons}{\operatorname{Cons}}

\def\Id{\text{Id}}

\newcommand{\fdb}{Fa{\`a} di Bruno\xspace}

\providecommand{\norm}[1]{\left| {#1}\right|}
\providecommand{\normnorm}[1]{\left|\!\left| {#1}\right|\!\right|}

\def\onto{\twoheadrightarrow}
\def\into{\hookrightarrow}
\newcommand{\shortsetminus}{\,\raisebox{1pt}{\ensuremath{\mathbb r}\,}}

\newcommand{\pro}[1]{\underleftarrow{#1}}
\newcommand{\ind}[1]{\underrightarrow{#1}}
\providecommand{\kat}[1]{\text{\textbf{\textsl{#1}}}}

\newcommand{\LIN}{\kat{LIN}}
\newcommand{\lin}{\kat{lin}}
\newcommand{\Lin}{\kat{Lin}}
\newcommand{\Inc}{\operatorname{Inc}}
\newcommand{\MI}{\mathbf{MI}}

\newcommand{\rat}{\rightarrowtail}
\newcommand{\lat}{\leftarrowtail}

\newcommand{\upperstar}{^{\raisebox{-0.25ex}[0ex][0ex]{\(\ast\)}}}

\newcommand{\lowershriek}{_!}
\newcommand{\uppershriek}{^!}

\newcommand{\isopil}{\stackrel{\raisebox{0.1ex}[0ex][0ex]{\(\sim\)}}%
			{\raisebox{-0.15ex}[0.28ex]{\(\rightarrow\)}}}
\newcommand{\isleftadjointto}{\dashv}
\newcommand{\tensor}{\otimes}
\newcommand{\op}{^{\text{{\rm{op}}}}}

\renewcommand{\epsilon}{\varepsilon}

\newcommand{\Set}{\kat{Set}}
\newcommand{\Grpd}{\kat{Grpd}}

\newcommand{\grpd}{\kat{grpd}}

\newcommand{\N}{\mathbb{N}}
\newcommand{\Z}{\mathbb{Z}}
\newcommand{\F}{\mathbb{F}}
\newcommand{\A}{\mathbb{A}}
\newcommand{\B}{\mathbb{B}}
\newcommand{\Q}{\mathbb{Q}}

\newcommand{\C}{\mathbb{C}}
\newcommand{\D}{\mathbb{D}}
\newcommand{\I}{\mathbb{I}}

\newcommand{\R}{\mathbb{R}}
\newcommand{\ground}{\Bbbk}

\usepackage{mathrsfs}

\newcommand{\MM}{\mathscr{M}}
\newcommand{\CC}{\mathscr{C}}
\newcommand{\DD}{\mathscr{D}}
\newcommand{\EE}{\mathscr{E}}

\def\lTo{\longleftarrow}
\def\rTo{\longrightarrow}

\newcommand{\comma}{\raisebox{1pt}{$\downarrow$}}
\newcommand{\name}[1]{\ulcorner #1\urcorner}

\newcommand{\Hom}{\operatorname{Hom}}
\newcommand{\Ext}{\operatorname{Ext}}

\newcommand{\Fun}{\operatorname{Fun}}
\newcommand{\PsFun}{\kat{PsFun}}
\newcommand{\Aut}{\operatorname{Aut}}
\newcommand{\id}{\operatorname{id}}
\newcommand{\colim}{\operatornamewithlimits{colim}}
\newcommand{\Ar}{\operatorname{Ar}}

\makeatletter
\def\subsection{\@startsection{subsection}{2}%
  \z@{.5\linespacing\@plus.7\linespacing}{.5\linespacing}%
  {\normalfont\bfseries}}
\makeatother

\newtheorem{lemma}{Lemma.}[subsection]
\newtheorem{prop}[lemma]{Proposition}
\newtheorem{thm}[lemma]{Theorem}
\newtheorem{theorem}[lemma]{Theorem}
\newtheorem{cor}[lemma]{Corollary}
\theoremstyle{definition}
\newtheorem{eks}[lemma]{Example}
\newtheorem{BM}[lemma]{Remark}

\newtheorem{taller}[lemma]{$\!\!$}

\newenvironment{blanko}[1]%
{\begin{taller}{\normalfont\bfseries  #1}\normalfont}%
{\end{taller}}

\newenvironment{blanko*}[1]{\begin{list}{\bf {#1} }%
{\setlength{\labelsep}{0mm}\setlength{\leftmargin}{0mm}%
\setlength{\labelwidth}{0mm}\setlength{\listparindent}{\parindent}%
\setlength{\parsep}{\parskip}\setlength{\partopsep}{0mm}}%
\item%
}{\end{list}}

\newenvironment{deff}%
{\begin{list}{\em Definition. }%
{\setlength{\labelsep}{0mm}\setlength{\leftmargin}{0mm}%
\setlength{\labelwidth}{0mm}\setlength{\listparindent}{\parindent}%
\setlength{\parsep}{\parskip}\setlength{\partopsep}{0mm}}%
\item}{\end{list}}

\newenvironment{proof*}[1]{\begin{list}{\em #1 }%
{\setlength{\labelsep}{0mm}\setlength{\leftmargin}{0mm}%
\setlength{\labelwidth}{0mm}\setlength{\listparindent}{\parindent}%
\setlength{\parsep}{\parskip}\setlength{\partopsep}{0mm}}%
\item}{\qed\end{list}}

\thanks{The first author 
  was partially supported by grants 
  MTM2010-15831, 
  MTM2012-38122-C03-01, 
  2014-SGR-634          
  and
  MTM2013-42178-P, 
  the second author 
  by 
MTM2009-10359, 
MTM2010-20692, 
and SGR1092-2009, 
and the third author 
by
MTM2010-15831 and MTM2013-42178-P}

\author{Imma G\'alvez-Carrillo}
\address{Departament de Matem\`atica Aplicada III
      \\Universitat Polit\`ecnica de Catalunya
      \\Escola d'Enginyeria de Terrassa 
      \\Carrer Colom 1\\08222 Terrassa (Barcelona)\\Spain}
\email{m.immaculada.galvez@upc.edu}

\author{Joachim Kock}
\address{Departament de Matem\`atiques
       \\Universitat Aut\`onoma de Barcelona
       \\08193 Bellaterra (Barcelona), Spain}
\email{kock@mat.uab.cat}

\author{Andrew Tonks}
\address{Department of Mathematics\\ 
University of Leicester\\ 
University Road\\ 
Leicester LE1 7RH, UK}
\email{apt12@le.ac.uk}

\title[Decomposition spaces]{Decomposition spaces, incidence algebras and M\"obius 
inversion}
\date{6 July, 2015}                                
\date{}                                
\begin{document}
\begin{abstract}
  We introduce the notion of decomposition space as a general framework for
  incidence algebras and \M inversion.  A decomposition space is a simplicial
  $\infty$-groupoid satisfying an exactness condition weaker than the Segal
  condition, expressed in terms of generic and free maps in Delta.  Just as the
  Segal condition expresses up-to-homotopy composition, the new condition
  expresses decomposition.  We work as much as possible on the objective level
  of linear algebra with coefficients in $\infty$-groupoids, and develop the
  necessary homotopy linear algebra along the way.  Independently of finiteness
  conditions, to any decomposition space there is associated an incidence
  (co)algebra (with coefficients in $\infty$-groupoids), and under a 
  completeness
  condition (weaker than the Rezk condition) this incidence algebra is shown to
  satisfy an objective \M inversion principle \`a la Lawvere--Menni.  Examples of
  decomposition spaces beyond Segal spaces are given by the Waldhausen
  $S$-construction of an abelian (or stable infinity) category.  Their incidence
  algebras are various kinds of Hall algebras.  Another class of examples are
  Schmitt restriction species.  Imposing certain homotopy finiteness conditions
  yields the notion of \M decomposition space, covering the notion of \M
  category of Leroux (itself a common generalisation of locally finite posets
  (Rota et al.)  and finite decomposition monoids (Cartier--Foata)), as well as
  many constructions of D\"ur, including the Fa\`a di Bruno and Connes-Kreimer
  bialgebras.  We take a functorial viewpoint throughout, emphasising
  conservative ULF functors, and show that most reduction procedures in the
  classical theory of incidence coalgebras are examples of this notion, and in
  particular that many are examples of decalage of decomposition spaces.  Our
  main theorem concerns the Lawvere-Menni Hopf algebra of \M intervals, which
  contains the universal \M function (but does not come from a \M category): we
  establish that \M intervals (in the $\infty$-setting) form a \M decomposition
  space that is in a sense universal for \M decomposition spaces and conservative ULF functors.

  NOTE: The notion of decomposition space was arrived at
  independently by Dyckerhoff and Kapranov (arXiv:1212.3563) who
  call them unital $2$-Segal spaces.  Our theory is quite
  orthogonal to theirs: the definitions are different in spirit
  and appearance, and the theories differ in terms of motivation,
  examples and directions.  For the few overlapping results
  (`decalage of decomposition is Segal' and `Waldhausen's $S$ is
  decomposition'), our approach seems generally simpler.
\end{abstract}



\setlength{\textheight}{645pt}

\vspace*{-24pt}

\maketitle


\vspace*{-12pt}

\noindent
NOTE: This manuscript is no longer intended for publication.
Instead it has been split into six papers: \cite{GKT:HLA}, \cite{GKT:DSIAMI-1},
\cite{GKT:DSIAMI-2}, \cite{GKT:MI}, \cite{GKT:ex} and \cite{GKT:restriction}.


\small

\tableofcontents

\normalsize

\setlength{\textheight}{632pt}

\setcounter{section}{-2}

\addtocontents{toc}{\protect\setcounter{tocdepth}{1}}

\input{motiv}
\input{everything}

\subsection*{Acknowledgments}

We are indebted first of all to Andr\'e Joyal, not only for his influence
through the theory of species and the theory of quasi-categories, essential
frameworks for our contribution, but also for his interest in our work, for many
enlightening discussions and advice that helped shape it.  We have also learned
a lot from collaboration and discussions with David Gepner.  Finally we thank
Mathieu Anel, Kurusch Ebrahimi-Fard, Tom Leinster, and Fei Xu, for
their interest and feedback.

\addtocontents{toc}{\protect\setcounter{tocdepth}{2}}

\section{Preliminaries on $\infty$-groupoids and $\infty$-categories}

\input{groupoids}


\section{Decomposition spaces}
\input{simplicial}
\input{segal}
\input{new-decomp-section}

\input{ulf}
\input{dec}
\input{monoids-a}


\section{Incidence (co)algebras}
\label{sec:COALG}

\input{coass}

\input{unDelta}

\input{prop-proof}
\input{coalg}
\input{monoids-b}

\input{complete}
\input{convolution}

\input{phi}
\input{half}

\input{split}
\input{length}

\input{finite-decomp-and-sectioncoeff}

\input{M-condition}
\input{M-inversion-algebraic-level}

\section{Examples}
\label{sec:ex}

\input{ex}
\input{trees}
\input{hall}

\input{cancellation}



\input{restriction}

\section{The decomposition space of \M intervals}
\label{sec:master}

\input{complete-fish}


\input{appendix}
%
%
%
%
%


\nocite{Lurie:HA}
\nocite{Rota:Moebius}
\nocite{Joyal:CRM}
\nocite{Joyal:qCat+Kan}
\nocite{Schmitt:1994}
\nocite{JoyalMR633783}
\nocite{JoniRotaMR544721}
\nocite{Stanley:MR513004}
\nocite{Doubilet-Rota-Stanley}


\end{document}

%% file: motiv.tex
\def\inputfile{motiv.tex}

\section{Introduction}

\subsection*{Background and motivation}

Leroux's notion of \M category \cite{Leroux:1975} generalises at the same 
time locally finite posets (Rota~\cite{Rota:Moebius}) and Cartier--Foata
finite-decomposition monoids~\cite{Cartier-Foata}, the two classical settings 
for incidence algebras and \M inversion.
An important
advantage of having these classical theories on the same footing is that the
appropriate class of functors, the conservative ULF functors (unique lifting of 
factorisations) (\ref{sec:cULF}),
connect different
examples, and in particular give nice explanations of the process of reduction
which is important in getting the most interesting algebras out of posets, a
process that was sometimes rather ad hoc.  As the most classical example of 
this process,
the divisibility poset $(\N^\times,\mid)$ (considered as a category) admits a conservative ULF
functor to the multiplicative monoid $(\N^\times,\times)$ (considered as a category with
only one object).  This functor induces a homomorphism of incidence coalgebras
which is precisely the reduction map from the `raw' incidence coalgebra of the
divisibility poset to its reduced incidence coalgebra, which is isomorphic to the
Cartier--Foata incidence coalgebra of the multiplicative monoid.

Shortly after Leroux's work, D\"ur~\cite{Dur:1986} studied more involved categorical
structures to extract further examples of incidence 
algebras and study their \M functions.
In particular he realised what was later called the Connes--Kreimer 
Hopf algebra as the reduced incidence coalgebra of a certain category of
root-preserving forest embeddings, modulo the equivalence 
relation that identifies two root-preserving forest embeddings if their 
complement crowns are isomorphic forests.
Another prominent example fitting D\"ur's formalism 
is the \fdb bialgebra, previously obtained in \cite{JoyalMR633783}
from the category of surjections, which is however not a \M category.


Our work on Fa\`a di Bruno formulae in bialgebras of
trees~\cite{GalvezCarrillo-Kock-Tonks:1207.6404} prompted us to look for a more
general version of Leroux's theory, which would naturally realise the \fdb and
Connes--Kreimer bialgebras as incidence coalgebras.  A sequence of
generalisations and simplifications of the theory led to the notion of
decomposition space which is a main notion in the present work.

The first abstraction step is to follow the objective method, pioneered in this context by
Lawvere and Menni~\cite{LawvereMenniMR2720184}, working directly with the
combinatorial objects, using linear algebra with coefficients
in $\Set$ rather than working with numbers and functions on the vector
spaces spanned by the objects.

To illustrate this, observe that a vector in the free vector 
space on a set $S$ is just a collection of scalars indexed
by (a finite subset of) $S$.   The objective counterpart is 
a family of sets indexed by $S$, i.e.~an object in the slice
category $\Set_{/S}$, and linear maps at this level are given
by spans $S \leftarrow M \to T$. The \M inversion principle
states an equality between certain linear maps (elements in
the incidence algebra).  At the objective level, such an
equality can be expressed as a bijection between sets
in the spans representing those linear functors.
In this way, algebraic identities are revealed to be just the 
cardinality of bijections of sets, which carry much more information.

In the present work, the coefficients are $\infty$-groupoids,
meaning that the role of vector spaces is played by slices of the
$\infty$-category of $\infty$-groupoids.
In~\cite{GKT:HLA} we have developed the necessary `homotopy linear algebra'
and homotopy cardinality, extending many results of 
Baez-Hoffnung-Walker~\cite{Baez-Hoffnung-Walker:0908.4305} who worked with
$1$-groupoids.
At the objective level, where all results and proofs are
naturally bijective, finiteness conditions do not play an essential role, since
it is just as easy to handle infinite sets as finite ones.  The price to pay
for working at the objective level is
the absence of additive inverses: in particular, \M functions cannot exist in
the usual form of an alternating sum.  However, an explicit equivalence expressing the \M
inversion principle can be obtained by splitting into even- and odd-length chains, and 
under the appropriate finiteness assumptions one can pass from the
objective level to the numerical level by taking cardinality; the
even-odd split version of \M inversion then yields the usual form of an
alternating sum. 

There are two levels of finiteness conditions needed in order
to take cardinality and arrive at algebraic (numerical) results: namely, just in
order to obtain a numerical coalgebra, for each arrow $f$ and for each $n\in\N$,
there should be only finitely many decompositions of $f$ into a chain of $n$
arrows.  Second, in order to obtain also \M inversion, the following additional
finiteness condition is needed: for each arrow $f$, there is an upper bound on the 
number of non-identity arrows in a chain of arrows composing to $f$.  The latter condition 
is important in its own right, as it is the condition for the existence of a 
length filtration, useful in many applications.

The importance of chains of arrows naturally suggests a simplicial viewpoint, 
regarding a category as a simplicial set via its nerve. 
Leroux's
theory can be formulated in terms of simplicial sets, and many of the arguments 
then rely on certain simple pullback conditions, the first being
the Segal condition which characterises categories among simplicial sets.

The fact that combinatorial objects typically have symmetries prompted the
upgrade from sets to groupoids, in fact a substantial conceptual 
simplification~\cite{GalvezCarrillo-Kock-Tonks:1207.6404}.
This upgrade is essentially straightforward, as long as the notions involved
are taken in a correct homotopy sense: bijections of sets are replaced
by equivalences of groupoids; the slices playing the role of vector spaces are
homotopy slices, the pullbacks and fibres involved in the functors are homotopy
pullbacks and homotopy fibres, and the sums are homotopy sums (i.e.~colimits
indexed by groupoids, just as classical sums are colimits indexed by sets).  The
passage to numbers and incidence algebras in the classical sense now goes via
groupoid cardinality.  In this setting
  one may as well abandon also the 
strict notion of simplicial object in favour of a pseudo-functorial analogue.
%
%
For example,
%
%
%
%
%
%
the classifying space of $(\B,+,0)$, the monoidal groupoid of finite sets and bijections
under disjoint union,
%
is
actually only a pseudofunctor $\mathbf B:\Delta\op\to\Grpd$.
%
This level of abstraction allows us to state for example that the incidence
algebra of $\mathbf B$ is the category of species with the Cauchy product
(suggested as an exercise by Lawvere and Menni~\cite{LawvereMenniMR2720184}).


While it is doable to handle the $2$-category theory involved to
deal with groupoids, pseudo-functors, pseudo-natural isomorphisms, and so on,
much conceptual clarity is obtained by passing immediately to
$\infty$-groupoids: thanks to the monumental effort of 
Joyal~\cite{Joyal:qCat+Kan}, \cite{Joyal:CRM}, Lurie~\cite{Lurie:HTT} and others,
$\infty$-groupoids can now be handled efficiently. At least at the
elementary level we work on, where all that is needed is some basic knowledge
about (homotopy) pullbacks and (homotopy) sums, everything looks very much like
the category of sets.  So we work throughout with certain simplicial
$\infty$-groupoids.  Weak categories in $\infty$-groupoids are precisely Rezk
complete Segal spaces~\cite{Rezk:MR1804411}.  Our theory at this level says that
for any Rezk complete Segal space there is a natural incidence coalgebra
defined with coefficients in $\infty$-groupoids, and that the objective
sign-free \M inversion principle holds.  To extract numerical coalgebras from
this, some homotopy finiteness conditions must be imposed, and
the passage to numbers is then via homotopy cardinality.

The final abstraction step, which becomes the starting point for the paper,
is to notice that in fact neither the Segal condition
nor the Rezk condition is needed in full in order to get a (co)associative
(co)algebra and a \M inversion principle.  Coassociativity follows from (in fact
is essentially equivalent to) the {\em decomposition space axiom} (see
\ref{sec:decomp} for the axiom, and the discussion at the beginning of 
Section~\ref{sec:COALG} for its derivation from
coassociativity): a decomposition space
is a simplicial $\infty$-groupoid sending generic-free pushout
squares in $\Delta$ to pullbacks. 
Whereas the Segal
condition is the expression of the ability to compose morphisms, the new
condition is about the ability to decompose, which of course in general is
easier to achieve than composability.
In order to get the \M inversion principle
(with coefficients in $\infty$-groupoids), a completeness condition is needed,
but it is weaker than the Rezk axiom: it is enough that
$s_0:X_0 \to X_1$ is a monomorphism.  Such simplicial
$\infty$-groupoids we call {\em complete decomposition spaces}.  Every
Rezk complete Segal space is a complete decomposition space. 

It is likely that all incidence (co)algebras can be realised
directly (without imposing a reduction)
as incidence (co)algebras of decomposition spaces.
The decomposition space is found by analysing the reduction step.
For example, D\"ur realises the $q$-binomial
coalgebra as the reduced incidence coalgebra of the category of
finite-dimensional vector spaces over a finite field and linear injections, by
imposing the equivalence relation identifying two linear injections if their
quotients are isomorphic.  Trying to realise the reduced incidence coalgebra
directly as a decomposition space immediately leads to
the Waldhausen $S$-construction, which is a general class of examples:
we show that for any abelian category or stable $\infty$-category, 
the Waldhausen $S$-construction is a decomposition space (which is not Segal).
Under the appropriate finiteness conditions, the resulting incidence algebras
include the (derived) Hall algebras.

As another example
we show that 
the Butcher--Connes--Kreimer bialgebra is directly the incidence coalgebra of a
decomposition space of combinatorial forests, without the need of reductions.
This decomposition space is not 
a Segal space.
In fact we fit this example into a general class of examples of decomposition 
spaces, which includes
also all Schmitt coalgebras of restriction species~\cite{Schmitt:hacs}.
We introduce the notion of \emph{directed restriction 
species}, a class of decomposition spaces that includes the Butcher--Connes-Kreimer
bialgebra of trees
as well as related constructions with directed graphs.

\smallskip

The appropriate notion of 
morphism between decomposition spaces 
is that of conservative ULF functor.  These induce coalgebra homomorphisms.
Many relationships between
incidence coalgebras, and in particular most of the reductions that play a 
central role in the classical theory (from Rota~\cite{Rota:Moebius} and D\"ur~\cite{Dur:1986}
to Schmitt~\cite{Schmitt:1994}),
are induced from conservative ULF functors. 
The simplicial viewpoint taken in this work reveals furthermore that many of
these conservative ULF functors are actually instances of the
notion of decalage, which goes back to Lawvere~\cite{Lawvere:ordinal} and 
Illusie~\cite{Illusie1}. Decalage is in fact 
an important ingredient in the theory to relate decomposition spaces 
to Segal spaces: we observe that the decalage of a decomposition 
space is a Segal space.

\smallskip

Our final example of a decomposition space constitutes our main theorem.
Lawvere showed in the 1980s that there is a Hopf algebra of \M intervals which contains
the universal \M function.  The first published account is by
Lawvere--Menni~\cite{LawvereMenniMR2720184},
where also the objective method is explored.  More precisely, this Hopf algebra 
is obtained from the collection of all iso-classes of \M intervals, 
and features a canonical coalgebra homomorphism from any incidence
coalgebra of a \M category $X$, defined by sending an arrow in $X$ to its 
factorisation interval.  Although this Hopf algebra is universal for incidence
coalgebras of \M categories, it is not itself the incidence coalgebra of a \M 
category.

We show that it {\em is} a decomposition space.  
In fact, in order for this to work smoothly (and obtain the correct universal 
properties), we are forced now to work in $\infty$-groupoids --- this is an 
important motivation for this abstraction step.
We construct the decomposition space
of all intervals, and
establish that it is universal for decomposition spaces.  This involves
constructing homotopy-meaningful intervals from any given simplex in any
given decomposition space.
The main tools here are the universal property of pullbacks
and certain factorisation systems on various $\infty$-categories related
to decomposition spaces.  The main factorisation system, the wide-cULF
factorisation system on the $\infty$-category of intervals,
generalises the generic-free factorisation system on
$\Delta$ which was the cornerstone for our theory of decomposition spaces.

\smallskip

Throughout we have strived for deriving all results from
elementary principles, such as pullbacks, factorisation systems
and other universal constructions.  It is also characteristic
for our approach that we are able to reduce many technical
arguments to simplicial combinatorics.  The main notions are
formulated in terms of the generic-free factorisation system in
$\Delta$.  To establish coassociativity we explore also
$\un\Delta$ (the
algebraist's Delta, including the empty ordinal) and
establish and exploit a universal property of its twisted arrow
category.  As a general method for establishing functoriality in
free maps, we study a certain category $\nabla$ of convex
correspondences in $\un\Delta$.  Finally, in order to construct
the universal decomposition space of intervals, we study the
category $\Xi$ of finite strict intervals, yet another variation of
the simplex category, related to it by an adjunction.  These 
`simplicial preliminaries' are likely to have applications also
outside the theory of decomposition spaces.

\subsection*{Related work: $2$-Segal spaces of Dyckerhoff and Kapranov}

The notion of decomposition space was arrived at independently
by Dyckerhoff and Kapranov~\cite{Dyckerhoff-Kapranov:1212.3563}:
a decomposition space is essentially the same thing as what they call
a unital $2$-Segal space.  We hasten to give them full credit
for having arrived at the notion first.  Unaware of their work,
we arrived at the same notion from a very different path, and
the theory we have developed for it is mostly orthogonal to
theirs.  

The definitions are different in appearance: the definition of
decomposition space refers to preservation of certain pullbacks,
whereas the definition of $2$-Segal space (reproduced in
\ref{DK} below) refers to triangulations of convex polygons.
The coincidence of the notions was noticed by Mathieu Anel
because two of the basic results are the same: specifically, the
characterisation in terms of decalage and Segal spaces (our
Theorem~\ref{thm:decomp-dec-segal}) and the result that the
Waldhausen $S$-construction of a stable $\infty$-category is a
decomposition space (our Theorem~\ref{thm:WaldhausenS}) were
obtained independently (and first)
in~\cite{Dyckerhoff-Kapranov:1212.3563}.

We were motivated by rather elementary aspects of combinatorics
and quantum field theory, and our examples are all drawn from
incidence algebras and \M inversion, whereas Dyckerhoff and Kapranov
were motivated
by representation theory, geometry, and homological algebra, and
develop a theory with a much vaster range of examples in mind:
in addition to Hall algebras and Hecke algebras they find cyclic
bar construction, mapping class groups and surface geometry (see
also \cite{Dyckerhoff-Kapranov:1306.2545} and
\cite{Dyckerhoff-Kapranov:1403.5799}), construct a Quillen model
structure and relate to topics of interest in higher category
theory such as $\infty$-$2$-categories and operads.

In the end we think our contribution is just a little corner of
a vast theory, but an important little corner, and we hope that
our viewpoints and insights will prove useful also for
the rest of the theory.

\subsection*{Related work on \M categories}

  Where incidence algebras and \M inversion are concerned, our
  work descends from Leroux et al.~\cite{Leroux:1975}, \cite{Content-Lemay-Leroux}, 
  \cite{Leroux:IJM}, D\"ur~\cite{Dur:1986} and
  Lawvere-Menni~\cite{LawvereMenniMR2720184}.  There is a different
  notion of \M category, due to Haigh~\cite{Haigh}.  The two
  notions have been compared, and to some extent unified, by
  Leinster~\cite{Leinster:1201.0413}, who calls Leroux's \M inversion {\em
  fine} and Haigh's {\em coarse} (as it only depends on the
  underlying graph of the category).  We should mention also
  the $K$-theoretic \M inversion for quasi-finite EI categories 
  of L\"uck and collaborators \cite{Luck:1989}, \cite{Fiore-Luck-Sauer:0908.3417}.

%% file: everything.tex
\def\inputfile{everything.tex}

\subsection*{Summary by section}

We proceed to summarise our results, section by section.

\medskip

We begin in Section 0 with a review of some elementary
notions from the theory of $\infty$-categories.
It is our contention that this should
be enough to render the paper accessible also to readers without
prior experience with $\infty$-categories.

\medskip

In Section 1, after a few preliminaries on simplicial objects
and Segal spaces, we introduce the main notion of this work,
decomposition spaces:

\begin{deff}
  A simplicial space $X:\Delta\op\to\Grpd$ is called a {\em decomposition space}
  when it takes generic-free pushouts in $\Delta$ to pullbacks.
\end{deff}
We give a few equivalent pullback characterisations, and observe that
every Segal space is a decomposition space.  The relevant notion of 
morphism is that of conservative ULF functor (unique lifting of 
factorisations):
\begin{deff}
  A simplicial map is called {\em ULF} if it is cartesian on generic face maps,
  and it is called {\em conservative} if cartesian on degeneracy maps.
  We write {\em cULF} for conservative and ULF.
\end{deff}
After some variations, we come to decalage, and establish the following
important relationship between Segal spaces and decalage:

\medskip

\noindent
{\bf Theorem~\ref{thm:decomp-dec-segal}.} {\em A simplicial space
$X$ is a decomposition space if and only if both
$\Dec_\top(X)$ and $\Dec_\bot(X)$ are Segal spaces, and the two 
comparison maps back to $X$ are cULF.}
\medskip

\noindent We also introduce the notion of {\em monoidal decomposition space}, as a monoid object in the monoidal
$\infty$-category of decomposition spaces and cULF maps.

\medskip

In Section \ref{sec:COALG} we establish that decomposition spaces induce 
coalgebras (with coefficients in $\infty$-groupoids), and that cULF maps induce
coalgebra homomorphisms. The coalgebra associated to a decomposition
space $X$ is the slice $\infty$-category $\Grpd_{/X_1}$, and its comultiplication map
is given by the span
$$
X_1 \stackrel{d_1}\leftarrow X_2 \stackrel{(d_2,d_0)}\to X_1\times X_1.
$$
We first explain how the decomposition space axioms
are distilled from a naive notion of coassociativity.
To establish coassociativity formally, vindicating the idea outlined, we first need some
more simplicial preliminaries.  In particular we introduce the twisted arrow
category $\DD$ of the category of finite ordinals, which is monoidal under 
{\em external sum}.  We show that simplicial objects in a cartesian monoidal
category can be characterised as monoidal functors on $\DD$, and characterise
decomposition spaces as those simplicial spaces whose extension to $\DD$
preserves certain pullback squares. 

The homotopy coassociativity of the incidence coalgebra is established in terms of
the monoidal structure on $\DD$.
The incidence
algebra of a monoidal decomposition space 
is naturally a {\em bialgebra}.

\bigskip

In Section~\ref{sec:complete} we come to the important notion of complete 
decomposition space:

\begin{deff}
  A decomposition space $X$ is  {\em complete} when $s_0: X_0 \to X_1$
  is a monomorphism (or equivalently, all degeneracy maps are monomorphisms 
  (\ref{all-s-mono})).
\end{deff}
This condition ensures that
the notion of nondegenerate simplex is well-behaved and
can be measured on the principal edges.
Let $\nondeg X_r \subset X_r$ denote the subspace of {\em effective}
$r$-simplices, i.e.~simplices all
of whose principal edges are nondegenerate.  For a complete decomposition
space, $\nondeg X_r$ can also be characterised as the subspace of
nondegenerate simplices (\ref{effective=nondegen}).  

After establishing the basic results about these notions, we come to \M
inversion in Section~\ref{sec:Minv}.  For any decomposition space, the linear
dual of the comultiplication yields a convolution product, defining the {\em
incidence algebra}.  This contains, in particular, the {\em zeta functor}
$\zeta$, given by the span $X_1 \stackrel=\leftarrow X_1 \to 1$, and the counit
$\epsilon$ (a neutral element for convolution) given by $X_1 \leftarrow X_0 \to
1$.  Consider the linear functor
$\Phi_r$ given by the span $X_1 \leftarrow \nondeg X_r \to 1$.  We can now establish the
\M inversion 
principle (with hopefully self-explanatory notation):

\medskip

\noindent
{\bf Theorem~\ref{thm:zetaPhi}.} {\em
For a complete decomposition space,
\begin{align*}
\zeta * \Phieven
 &\;\;=\;\; \epsilon\;\; +\;\; \zeta * \Phiodd,\\
=\;\;\Phieven *\zeta &\;\;=\;\; \epsilon \;\;+ \;\; \Phiodd*\zeta.
\end{align*}}

\medskip


Having established the general \M inversion principle on the objective
level, we want to analyse the finiteness conditions needed for this
principle to descend to the numerical level of $\Q$-algebras.
There are two conditions: $X$ should be of locally finite length (which we 
also call {\em tight}) (Section~\ref{sec:length}), and $X$ should be locally finite
(Section~\ref{sec:findec}).  Complete 
decomposition spaces satisfying both condition are called {\em \M decomposition 
spaces} (Section~\ref{sec:M}).  We analyse the two conditions separately.

The first condition is equivalent to the existence of a certain
{\em length filtration}, which is useful in applications.  Although many
examples coming from combinatorics do satisfy this condition, it is actually a 
rather strong condition, as witnessed by the following result:

\medskip
\noindent {\em Every decomposition space with length filtration
is the left Kan extension of a semi-simplicial space.}

\medskip

We can actually prove this theorem for more general simplicial spaces, and
digress to establish this.  It is useful first to analyse further the property
that nondegeneracy can be measured on the principal edges.  We show
(Proposition~\ref{halfdecomp}) that this happens precisely for complete
simplicial spaces for which degeneracy maps form pullbacks with free maps,
i.e.~`half' of the decomposition space axioms; we call such simplicial spaces
{\em stiff}, and establish their basic properties in Section~\ref{sec:stiff},
giving in particular a characterisation of conservative maps in terms of
nondegeneracy preservation~(\ref{stiff-cons}).

A complete simplicial space is called {\em split} if all face maps preserve
nondegenerate simplices.  This condition, which is the subject of 
Section~\ref{sec:split}, is the 
analogue of the condition for categories that identities are indecomposable,
enjoyed in particular by \M categories in the sense of Leroux.
Split implies stiff.
We prove that a simplicial space is split if and only if it is the left Kan
extension along $\Deltainj \subset \Delta$ of a semi-simplicial space
$\Deltainj\op\to\Grpd$, and in fact we establish more precisely:

\medskip

\noindent
{\bf Theorem~\ref{thm:semisimpl=splitcons}.} {\em
Left Kan extension induces an equivalence of $\infty$-categories
$$
\Fun(\Deltainj\op,\Grpd) \simeq \kat{Split}^{\mathrm cons},
$$
the $\infty$-category of split simplicial spaces and conservative maps.}


\begin{deff}
  The {\em length} of an arrow $f$ is the greatest dimension of an effective
  simplex with long edge $f$.  We say that a complete decomposition space is
  {\em of locally finite length} --- we also say {\em tight} --- when every
  arrow has finite length.
\end{deff}

Tight decomposition spaces are split (Corollary~\ref{prop:tight=>split}).
If a tight decomposition space is a Segal space, then it is Rezk complete.
(Proposition~\ref{prop:FILTSegal=Rezk}).

\medskip

We show that a tight decomposition space $X$ has a filtration
$$
X_\bullet^{(0)} \into X_\bullet^{(1)} \into \cdots \into X
$$
of cULF monomorphisms, the so-called {\em length filtration}.
This is precisely the structure needed to get a filtration of the associated
coalgebra (\ref{coalgebrafiltGrpd}).

\bigskip

In Section~\ref{sec:findec} we impose the finiteness condition needed to be able
to take homotopy cardinality and obtain coalgebras and algebras at the numerical level 
of $\Q$-vector spaces (and pro-vector spaces).

\begin{deff}
  A decomposition space $X$ is called {\em locally finite} (\ref{finitary}) when
  $X_1$ is locally finite and $s_0:X_0\to X_1$ and $d_1: X_2 \to X_1$ are finite
  maps. 
\end{deff}

The condition `locally finite'
extends the notion of locally finite for posets.
The condition
ensures that the coalgebra structure descends to 
finite-groupoid coefficients, and hence, via homotopy cardinality, to
$\Q$-algebras.  In Section~\ref{sec:tioncoeff} we calculate the section 
coefficients (structure constants for the (co)multiplication) in some
easy cases. 

Finally we introduce the \M condition: 
\begin{deff}
  A complete decomposition space is called {\em \M} when it is locally finite
  and of locally finite length (i.e.~is tight).
\end{deff}
This is the condition needed for the general \M inversion formula
to descend to finite-groupoid coefficients and $\Q$-coefficients, giving the
following formula
for the \M function (convolution inverse to the zeta function):
$$
\norm{\mu} = \norm{\Phieven}-\norm{\Phiodd} .
$$

\bigskip

In Section~\ref{sec:ex} we give examples.  The first batch of examples, similar 
to the binomial posets of Doubilet-Rota-Stanley~\cite{Doubilet-Rota-Stanley},
are straightforward but serve to illustrate two key points: (1) the incidence algebra
in question is realised directly from a decomposition space, without a
reduction step, and reductions are typically given by cULF functors;
(2) at the objective level, the convolution algebra is a monoidal
structure of species (specifically: the usual Cauchy product of species, the 
shuffle product of $\mathbb L$-species, the Dirichlet product of arithmetic species,
the Joyal-Street external product of $q$-species, and the Morrison `Cauchy' 
product of $q$-species).  In each of these cases, a power series representation
results from taking cardinality.

The next class of examples includes the \fdb bialgebra, the
Butcher-Connes-Kreimer bialgebra of trees, with several variations, and similar
structures on directed graphs (cf.~Manchon~\cite{Manchon:MR2921530} and
Manin~\cite{Manin:MR2562767}).  In Subsection~\ref{sec:Wald} we come to an
important class of examples, showing that the Waldhausen $S$-construction on an
abelian category, or a stable $\infty$-category, is a decomposition space.
While this example plays only a minor role in our work, it is a cornerstone
in the work of Dyckerhoff and Kapranov~\cite{Dyckerhoff-Kapranov:1212.3563}, 
\cite{Dyckerhoff-Kapranov:1306.2545}, \cite{Dyckerhoff-Kapranov:1403.5799},
\cite{Dyckerhoff:CRM}; we refer to their work for the remarkable richness of this 
class of examples.
We finish the section by computing the \M function in a few cases, and commenting
on certain cancellations that occur in the process of taking cardinality,
substantiating that these cancellations are not possible at the objective level
(this is related to the distinction between bijections and natural bijections).

\medskip

In Section~\ref{sec:restriction} we show that Schmitt coalgebras of
restriction species \cite{Schmitt:hacs} (such as graphs, matroids, posets, etc.) come from 
decomposition spaces.  We also introduce a new notion of {\em directed restriction species}.  Whereas ordinary restriction species are presheaves of the category of
finite sets and injections, directed restriction species are presheaves on the 
category of finite posets and convex inclusions.  Examples covered by this 
notion are the Butcher-Connes-Kreimer bialgebra and the Manchon-Manin bialgebra
of directed graphs.  Both ordinary and directed restriction species are shown
to be examples of a construction of decomposition spaces from what we call
sesquicartesian fibrations, certain cocartesian fibrations over the category of 
finite ordinals that are also cartesian over convex maps.

\medskip

In Section~\ref{sec:master} we come to our main theorem, constructing a
`universal decomposition space', the (large) decomposition space of intervals.  The idea
(due to Lawvere) is that to an arrow there is associated its category of
factorisations, which is an interval.  To set this up, we exploit factorisation 
systems and adjunctions derived from them, and start out in 
Subsection~\ref{sec:fact} with some general results about factorisation systems.
Specifically we describe a situation in which a factorisation system lifts
across an adjunction to produce a new factorisation system, and hence a new
adjunction.  Before coming to intervals in \ref{sec:fact-int}, we 
need flanked decomposition spaces (\ref{sec:flanked}):
these are certain presheaves on the category $\Xi$ of nonempty finite linear orders 
with a top and a bottom element.  The $\infty$-category of flanked decomposition spaces
features the important {\em wide-cartesian} factorisation system, where `wide'
is to be thought of as endpoint-preserving, and cartesian is like 
`distance-preserving'.  There is also the basic adjunction between decomposition
spaces and flanked decomposition spaces, which in fact is the double dec 
construction.  Intervals are first defined as certain flanked decomposition spaces
which are contractible in degree $-1$ (this condition encodes an initial and a 
terminal object) (\ref{aInt}), and via the basic adjunction we obtain the definitive $\infty$-category
of intervals as a full subcategory of the $\infty$-category of complete decomposition 
spaces (\ref{towards-Int}); it features the wide-cULF  factorisation system 
(\ref{fact-Int}), which extends the 
generic-free factorisation system on $\Delta$ (\ref{IntDelta}).  The factorisation-interval
construction can now finally be described (Theorem~\ref{Thm:I})
as a coreflection from complete decomposition 
spaces to intervals (or more precisely, on certain coslice categories).
We show that every interval is a Segal space (\ref{prop:i*flanked=Segal}).
The universal decomposition 
space $U$ of intervals (which lives in a bigger universe) can finally (\ref{U}) be defined very formally as a natural right 
fibration over $\Delta$ whose total space has objects wide interval maps from 
an ordinal.  In plain words, $U$ consists of subdivided 
intervals.

 \medskip
 
 \noindent
{\bf Theorem~\ref{UcompleteDecomp}.} {\em $U$ is a complete decomposition space.}

\medskip

\noindent
The factorisation-interval construction yields a canonical functor $X \to U$,
called the {\em classifying map}.

 \medskip
 
 \noindent
{\bf Theorem~\ref{thm:IU=cULF}.} {\em The classifying map is cULF.}

\medskip

We 
conjecture
that $U$ is universal for complete 
decomposition spaces and cULF maps, and prove the following partial result:

 \medskip
 
 \noindent
{\bf Theorem~\ref{thm:connected}.} {\em 
For each complete decomposition space $X$, the space 
  $\Map_{\cDecomp^{\culf}}(X,U)$ is connected.}
  
  \medskip

We finish in Subsection~\ref{sec:MI} by imposing the \M condition, obtaining
the corresponding finite results.
A {\em \M interval} is an interval which is \M as a decomposition space.
We show that every \M interval is a Rezk complete Segal space 
(\ref{prop:Mint=Rezk}).
There is a decomposition space of {\em all}\/ \M intervals, and it is shown to be small.

Our main theorem in this section is now:

 \medskip
 
 \noindent
{\bf Theorem~\ref{thm:MI=M}.} {\em The decomposition space of all \M intervals 
is \M.}

\medskip

\noindent
It follows that it admits a \M inversion formula with coefficients in finite 
$\infty$-groupoids
or in $\Q$, and since every \M decomposition space admits a 
canonical cULF functor to it, we find that \M inversion in every incidence 
algebra (of a \M decomposition space) is induced from this master formula.

\medskip

In the Appendix we develop what we need about homotopy linear algebra and 
homotopy cardinality.  For the sake of flexibility (regarding what notions
of finiteness the future will bring) we first work out the notions without
finiteness conditions.  The role of vector spaces is played by groupoid slices
$$
\Grpd_{/S} ,
$$
shown to be the homotopy-sum completion of $S$,
and the role of linear maps is played by
linear functors, i.e.~given by pullback and lowershriek
along spans.
We explain how to interpret scalar multiplication and sums (together: linear 
combinations), and how to expand these operations in coordinates.  The canonical
basis
is given by the `names', functors $\name x: 1 \to S$.
Groupoid slices and linear functors assemble into an $\infty$-category,
which is monoidal closed.  The tensor product is given by
$$
\Grpd_{/S} \tensor \Grpd_{/T} = \Grpd_{/S\times T} .
$$

In Subsection~\ref{sec:finite} we get into the subtleties of finiteness conditions.
An $\infty$-groupoid $B$ is {\em locally finite} if at each base point $b$ the
homotopy groups $\pi_i (B,b)$ are finite for $i\geq1$ and are trivial for $i$
sufficiently large.  It is called {\em finite} if furthermore it has only
finitely many components.  The cardinality of a finite $\infty$-groupoid is the sum (over
the connected components) of the alternating product of the homotopy groups.  We
work out the basic properties of this notion.

For the $\infty$-groupoid version of linear algebra, we are strict about duality
issues, developed in the setting of vector spaces and 
profinite-dimensional vector spaces (a brief review is in \ref{vect-rappels}).
The role of vector spaces is played by finite-groupoid slices $\grpd_{/S}$
(where $S$ is a locally finite $\infty$-groupoid), while the role of 
profinite-dimensional vector spaces is played by finite-presheaf categories
$\grpd^S$.  Linear maps are given by spans of {\em finite type}, meaning $S \stackrel 
p\leftarrow M \stackrel q\to T$ in which $p$ is a finite map.  Prolinear maps are given
by spans of {\em profinite type}, where $q$ is a finite map.  
In the end we have two $\infty$-categories: $\ind\lin$ whose objects are the 
finite-groupoid slices $\grpd_{/S}$ and whose mapping spaces are $\infty$-groupoids of 
finite-type spans, and the $\infty$-category $\pro\lin$ whose objects are 
finite-presheaf categories $\grpd^S$, and whose mapping spaces are $\infty$-groupoids of 
profinite-type spans.

Finally we follow Baez-Hoffnung-Walker~\cite{Baez-Hoffnung-Walker:0908.4305}
in defining cardinality in terms of a 
`meta cardinality' functor, which induces cardinality notions in all slices.
In our setting, this amounts to a functor
\begin{eqnarray*}
  \normnorm{ \ } : \ind\lin & \longrightarrow & \Vect  \\
  \grpd_{/S} & \longmapsto & \Q_{\pi_0 S}
\end{eqnarray*}
and a dual functor
\begin{eqnarray*}
  \normnorm{ \ } : \pro\lin  & \longrightarrow & \pro\vect  \\
  \grpd^S & \longmapsto & \Q^{\pi_0 S} .
\end{eqnarray*}
For each fixed $\infty$-groupoid $S$, this gives an individual notion of
cardinality $\norm{ \ } : \grpd_{/S} \to \Q_{\pi_0 S}$ (and dually
$\norm { \ } : \grpd^S \to \Q^{\pi_0 S}$), since vectors are just
linear maps from the ground field.

The vector space $\Q_{\pi_0 S}$ is spanned by the elements $\delta_s :=
\norm{\name s}$.  Dually, the profinite-dimensional vector space $\Q^{\pi_0 S}$
is spanned by the characteristic functions $\delta^t
=
\frac{\norm{h^t}}{\norm{\Omega(S,t)}}$
(the cardinality of the representable functors divided by the
cardinality of the loop space).

%% file: groupoids.tex
\def\inputfile{groupoids.tex}

\label{sec:groupoids}

\begin{blanko}{Groupoids and $\infty$-groupoids.}
  Although most of our motivating examples can be naturally cast in the setting
  of $1$-groupoids, we have chosen to work in the setting of $\infty$-groupoids.
  This is on one hand the natural generality of the theory, and on the other
  hand a considerable conceptual simplification: thanks to the monumental effort
  of Joyal~\cite{Joyal:qCat+Kan}, \cite{Joyal:CRM} and Lurie~\cite{Lurie:HTT},
  the theory of $\infty$-categories has now reached a stage where it is just as 
  workable as the theory of $1$-groupoids --- if not more!  The philosophy is that, modulo a
  few homotopy {\em caveats}, one is allowed to think as if working in the
  category of sets.  A recent forceful vindication of this philosophy is
  Homotopy Type Theory~\cite{HoTT}, in which a syntax that resembles set theory
  is shown to be a powerful language for general homotopy types.
  
  A recurrent theme in the present work is to upgrade combinatorial
  constructions from sets to $\infty$-groupoids.  To this end the first step
  consists in understanding the construction in abstract terms, often in terms
  of pullbacks and sums, and then the second step consists in copying over the
  construction to the $\infty$-setting.  The $\infty$-category theory needed
  will be accordingly elementary, and it is our contention that it should be
  feasible to read this work without prior experience with $\infty$-groupoids or
  $\infty$-categories, simply by substituting the word `set' for the word
  `$\infty$-groupoid'.  Even at the $0$-level, our theory contributes
  interesting insight, revealing many constructions in the classical theory to
  be governed by very general principles proven useful also in
  other areas of mathematics.
  
  The following short review of some basic aspects of $\infty$-categories should
  suffice for reading this paper, except the final Section~\ref{sec:master},
  where some slightly more advanced machinery is used.

\end{blanko}

\begin{blanko}{From posets to Rezk categories.}
  A few remarks may be in order to relate these viewpoints with classical
  combinatorics.  A $1$-groupoid is the same thing as an ordinary groupoid, and
  a $0$-groupoid is the same thing as a set.  A $(-1)$-groupoid is the same
  thing as a truth value: up to equivalence there exist only two
  $(-1)$-groupoids, namely the contractible groupoid (a point) and the empty
  groupoid.  A poset is essentially the same thing as a category in which all
  the mapping spaces are $(-1)$-groupoids.  An ordinary category is a category
  in which all the mapping spaces are $0$-groupoids.  Hence the theory of
  incidence algebras of posets of Rota and collaborators can be seen as the
  $(-1)$-level of the theory.  Cartier--Foata theory and Leroux theory take
  place at the $0$-level.  We shall see that in a sense the natural setting for
  combinatorics is the $1$-level, since this level naturally takes into account
  that combinatorial structures can have symmetries.  (From this viewpoint, it
  looks as if the classical theory compensates for working one level below the
  natural one by introducing reductions.)  
  It is convenient to follow this ladder to infinity: the good
  notion of category with $\infty$-groupoids as mapping spaces is
  that of Rezk complete Segal space, also called Rezk category;
  this is the level of generality of the present work
\end{blanko}
  
%
%

\begin{blanko}{$\infty$-categories and $\infty$-groupoids.}
  By $\infty$-category we mean quasi-category~\cite{Joyal:qCat+Kan}.  These are
  simplicial sets satisfying the weak Kan condition: inner horns admit a filler.
  (An ordinary category is a simplicial set in which every inner horn admits a {\em 
  unique} filler.)
  We refer to Joyal~\cite{Joyal:qCat+Kan}, \cite{Joyal:CRM} and
  Lurie~\cite{Lurie:HTT}.  The definition does not actually matter much in this
  work.  The main point, Joyal's great insight,
  is that category theory can be generalised to
  quasi-categories, and that the results look the same, although to bootstrap
  the theory very different techniques are required.  There are other
  implementations of $\infty$-categories, such as complete Segal spaces, see 
  Bergner~\cite{Bergner:0610239} for a survey.  We
  will only use results that hold in all implementations, and for this reason we
  say $\infty$-category instead of referring explicitly to quasi-categories.
  Put another way, we shall only ever distinguish quasi-categories up to
  (categorial) equivalence, and most of the constructions rely on universal 
  properties such as pullback, which in any case only determine the objects
  up to equivalence.
  
  An $\infty$-groupoid is an $\infty$-category in
  which all morphisms are invertible.  We often say {\em space} instead of
  $\infty$-groupoid, as they are a combinatorial substitute for topological 
  spaces up to homotopy; for example, to each object $x$ in an $\infty$-groupoid 
  $X$,
  there are associated homotopy groups $\pi_n(X,x)$ for $n>0$.
  In terms of quasi-categories,
  $\infty$-groupoids are precisely Kan complexes, i.e.~simplicial sets in which
  every horn, not just the inner ones, admits a filler.
  
  $\infty$-groupoids
  play the role analogous to sets in classical category theory.  In particular,
  for any two objects $x,y$ in an $\infty$-category $\CC$ there is (instead of a hom set)
  a mapping space $\Map_\CC(x,y)$ which is an $\infty$-groupoid.  
  $\infty$-categories form a (large) 
  $\infty$-category denoted $\kat{Cat}$. $\infty$-groupoids form a (large)
  $\infty$-category denoted $\Grpd$; it can be described explicitly as the 
  coherent nerve of the (simplicially enriched) category of Kan complexes. 
  Given two $\infty$-categories $\DD$, 
  $\CC$, there is
  a functor $\infty$-category $\Fun(\DD,\CC)$.  Since $\DD$ and $\CC$ are 
  objects in the $\infty$-category  $\kat{Cat}$ we also have the 
  $\infty$-groupoid $\Map_{\kat{Cat}}(\DD,\CC)$, which can also be described as 
  the
  maximal sub-$\infty$-groupoid inside $\Fun(\DD,\CC)$.
\end{blanko}

\begin{blanko}{Defining $\infty$-categories and sub-$\infty$-categories.}
  While in ordinary category theory one can define a category by saying what the
  objects and the arrows are (and how they compose), this from-scratch approach
  is more difficult for $\infty$-categories, as one would have to specify the
  simplices in all dimensions and verify the filler condition (that is,
  describe the $\infty$-category as a quasi-category).  In practice,
  $\infty$-categories are constructed from existing ones by general
  constructions that automatically guarantee that the result is again an
  $\infty$-category, although the construction typically uses universal
  properties in such a way that the resulting $\infty$-category is only defined
  up to equivalence.  To specify a sub-$\infty$-category of an $\infty$-category
  $\CC$, it suffices to specify a subcategory of the homotopy category of $\CC$
  (i.e.~the category whose hom sets are $\pi_0$ of the mapping spaces of $\CC$),
  and then pull back along the components functor.  What this amounts to in
  practice is to specify the objects (closed under equivalences) and specifying
  for each pair of objects $x,y$ a full sub-$\infty$-groupoid of the
  mapping space $\Map_\CC(x,y)$, also closed under equivalences, and closed 
  under
  composition.
%
%
\end{blanko}

\begin{blanko}{Monomorphisms.}\label{def:mono}
  A map of $\infty$-groupoids $f:X\to Y$ is a {\em monomorphism} when its
  fibres are $(-1)$-groupoids (i.e.~are either empty or contractible).
 In other words, it is fully faithful as a functor: 
$\Map_X(a,b) \to \Map_Y(fa,fb)$ is an equivalence.
  In some respects, this notion behaves like for sets: for example, if $f$
  is a monomorphism, then there is a complement $Z:=Y\shortsetminus X$ such that
  $Y \simeq X + Z$.  Hence a monomorphism is essentially an equivalence
  from $X$ onto some connected components of $Y$.  On the other hand, 
  a crucial difference from sets to $\infty$-groupoids is that diagonal
  maps of $\infty$-groupoids are not in general monomorphisms.  In fact $X \to X \times 
  X$ is a monomorphism if and only if $X$ is discrete (i.e.~equivalent to a set).
\end{blanko}

\begin{blanko}{Diagram categories and presheaves.}
  Every $1$-category is also a quasi-category via
  its nerve.  In particular we have the $\infty$-category $\Delta$ of non-empty
  finite ordinals, and for each $n\geq 0$ the $\infty$-category
  $\Delta[n]$ which is the nerve of
  the linearly ordered set $\{0 \leq 1 \leq \cdots \leq n\}$. As an important
  example of a functor $\infty$-category, for a given $\infty$-category $I$, we
  have the $\infty$-category of presheaves $\Fun(I\op,\Grpd)$, and there is a 
  Yoneda lemma that works as in the case of ordinary categories.
  In particular
  we have the $\infty$-category $\Fun(\Delta\op,\Grpd)$ of simplicial
  $\infty$-groupoids, which will be one of our main objects of study.
  
  Since arrows in an $\infty$-category do not compose on the nose (one can talk 
  about `a' composite, not `the' composite), the $1$-categorical notion of
  commutative diagram does not make sense.  Commutative
  triangle in an $\infty$-category $\CC$
  means instead `object in the functor $\infty$-category 
  $\Fun(\Delta[2],\CC)$': the $2$-dimensional face of $\Delta[2]$ is mapped to
  a $2$-cell in $\CC$ mediating between the composite of the $01$ and $12$
  edges and the long edge $02$.  Similarly, `commutative square' means
  object in the functor $\infty$-category $\Fun(\Delta[1]\times \Delta[1], 
  \CC)$.  In general, `commutative diagram of shape $I$' means object in
  $\Fun(I,\CC)$, so when we say for example `simplicial $\infty$-groupoid' it is
  not implied that the usual simplicial identities hold on the nose. 
\end{blanko}

\begin{blanko}{Adjoints, limits and colimits.}
  There are notions of adjoint functors, limits and colimits, which behave in the
  same way as these notions in ordinary category theory, and are characterised 
  by
  universal properties up to equivalence.  For example, the 
  singleton set $*$ (also denoted $1$), or any contractible $\infty$-groupoid,
  is a terminal object in $\Grpd$.
\end{blanko}

\begin{blanko}{Pullbacks and fibres.}
  Central to this work is the notion of pullback: given
  two morphisms of $\infty$-groupoids $X \to S \leftarrow Y$, there is a square
  $$\xymatrix{
  X\times_S Y \drpullback \ar[r] \ar[d] & Y \ar[d] \\
  X \ar[r] & S
  }$$
  called the pullback, an example of a limit.
  It is defined via a universal property, as
  a terminal object in a certain auxiliary $\infty$-category consisting
  of squares with sides $X \to S \leftarrow Y$.
  All formal properties of pullbacks of sets carry over to $\infty$-groupoids.
  
  Given a morphism of $\infty$-groupoids, $p:X \to S$, and an object $s\in S$
  (which in terms of quasi-categories can be thought of as a zero-simplex of $S$,
  but which more abstractly is encoded as a map $* \stackrel s \to S$ from
  the terminal $\infty$-groupoid $* = \Delta[0]$), the fibre of $p$ over $s$
  is simply the pullback
  $$\xymatrix{
  X_s \drpullback \ar[r] \ar[d] & X \ar[d]^p \\
  {*} \ar[r]_s & S
  }$$
\end{blanko}

\begin{blanko}{Working in the $\infty$-category of $\infty$-groupoids, versus 
    working in the model category of simplicial sets.}
  When working with $\infty$-categories in terms of quasi-categories, one often
  works in the Joyal model structure on simplicial sets (whose fibrant objects
  are precisely the quasi-categories).  This is a very powerful technique,
  exploited masterfully by Joyal~\cite{Joyal:CRM} and Lurie~\cite{Lurie:HTT}, and
  essential to bootstrap the whole theory.  In the present work, we can
  benefit from their work, and since our constructions are generally elementary,
  we do not need to invoke model structure arguments, but can get away with
  synthetic arguments.  To illustrate the difference, consider the following 
  version of the Segal 
  condition (see~\ref{segalpqr} for details): we shall formulate it and use it by 
  simply saying {\em the natural square
  $$\xymatrix{
     X_2 \ar[r]\ar[d] & X_1 \ar[d] \\
     X_1 \ar[r] & X_0
  }$$
  is a pullback}.  This is a statement taking place in the $\infty$-category of
  $\infty$-groupoids.  A Joyal--Lurie style formulation would rather take place
  in the category of simplicial sets with the Joyal model structure and say
  something like {\em the natural map $X_2 \to X_1 \times_{X_0} X_1$ is an 
  equivalence}.  Here $X_1\times_{X_0} X_1$ refers to the actual $1$-categorical
  pullback in the category of simplicial sets, which does not coincide with
  $X_2$ on the nose, but is only naturally equivalent to it.
  \end{blanko}

The following Lemma is used many times in our work. It is a straightforward
  extension of a familiar result in $1$-category theory:

\begin{lemma}\label{pbk}
  If in a prism diagram of $\infty$-groupoids
  $$
  \vcenter{\xymatrix{
   \cdot\dto \rto &  \cdot\dto \rto &  \cdot\dto \\
  \cdot\rto & \cdot \rto & \cdot 
  }}
  $$
  the outer rectangle and the right-hand square are pullbacks,
  then the left-hand square is a pullback.
\end{lemma}
A few remarks are in order.  Note that we talk about a prism, i.e.~a 
$\Delta[1]\times\Delta[2]$-diagram:  although we have only drawn two of the 
squares of the prism, there is a third, whose horizontal sides are composites
of the two indicated arrows.  The triangles of the prism are not drawn either,
because they are the fillers that exist by the axioms of quasi-categories.
The proof follows the proof in the classical case, except that instead of saying
`given two arrows such and such, there exists a unique arrow making the diagram
commute, etc.', one has to argue with equivalences of mapping spaces (or slice
$\infty$-categories).  See for example Lurie~\cite{Lurie:HTT}, Lemma~4.4.2.1
(for the dual case of pushouts).

\begin{blanko}{Homotopy sums.}
  In ordinary category theory, a colimit indexed by a discrete category (that
  is, a set) is the same thing as a sum (coproduct).  For $\infty$-categories,
  the role of sets is played by $\infty$-groupoids.  A colimit indexed by an
  $\infty$-groupoid is called a {\em homotopy sum}.  In the case of
  $1$-groupoids, these sums are ordinary sums weighted by inverses of symmetry
  factors.  Their importance was stressed in
  \cite{GalvezCarrillo-Kock-Tonks:1207.6404}: by dealing with homotopy sums
  instead of ordinary sums, the formulae start to look very much like in the case
  of sets.  For example, given a map of $\infty$-groupoids $X \to S$, we have
  that $X$ is the homotopy sum of its fibres.
\end{blanko}

\begin{blanko}{Slice categories and polynomial functors.}
  Maps of $\infty$-groupoids with codomain $S$ form the objects of
  a slice $\infty$-category
  $\Grpd_{/S}$, which behaves very much like a slice category in ordinary
  category theory.  For example, for the terminal object $*$ we have
  $\Grpd_{/*} \simeq \Grpd$. 
  Again a word of warning is due: when we refer to the 
  $\infty$-category $\Grpd_{/S}$ we only refer to an object determined up to 
  equivalence of $\infty$-categories by a certain universal property (Joyal's 
  insight of defining slice categories as adjoint to a join 
  operation~\cite{Joyal:qCat+Kan}).
  In the Joyal model structure for quasi-categories, this category is 
  represented by an explicit simplicial set.  However, there is more than
  one possibility, depending on which explicit version of the join operator
  is employed (and of course these are canonically equivalent). In the works
  of Joyal and Lurie, these different versions are distinguished, and each has
  some technical advantages.  In the present work we shall only need properties
  that hold for both, and we shall not distinguish them.
  
  Pullback along a morphism $f: T \to S$
  defines an $\infty$-functor $f\upperstar :\Grpd_{/S} \to \Grpd_{/T}$.  This functor
  is right adjoint to the functor $f\lowershriek:\Grpd_{/T} \to \Grpd_{/S}$ given by
  post-composing with $f$.  (This construction requires some care: as
  composition is not canonically defined, one has to choose composites.
  One can check that different choices yield equivalent
  functors.)  The following Beck-Chevalley rule (push-pull formula)
  \cite{Gepner-Kock}
  holds for $\infty$-groupoids: given a pullback square
  $$\xymatrix{
  \cdot \drpullback \ar[r]^f \ar[d]_p & \cdot \ar[d]^q \\
  \cdot \ar[r]_g & \cdot}$$
  there is a canonical equivalence of functors 
  \begin{equation}\label{BC}
  p\lowershriek \circ f\upperstar \simeq g\upperstar \circ q\lowershriek .
 \end{equation}
\end{blanko}

\begin{blanko}{Families.}
  A map of $\infty$-groupoids $X \to S$ can be interpreted as a family of 
  $\infty$-groupoids parametrised by $S$, namely the fibres $X_s$.  Just as 
  for sets, the same family
  can also be interpreted as a presheaf $S \to \Grpd$.  Precisely,
  for each $\infty$-groupoid $S$, we have the fundamental equivalence
  $$
  \Grpd_{/S} \isopil \Fun(S,\Grpd) ,
  $$
  which takes a family $X \to S$ to the functor sending $s \mapsto X_s$. 
  In the other direction, given a functor $F:S\to\Grpd$,
  its colimit is the total space of a family $X \to S$.
\end{blanko}

\begin{blanko}{Symmetric monoidal $\infty$-categories.}
  There is a notion of symmetric monoidal $\infty$-category, but it is
  technically more involved than the $1$-category case, since in general higher
  coherence data has to be specified beyond the $1$-categorical associator and
  Mac Lane pentagon condition.  This theory has been developed in detail by
  Lurie \cite[Ch.2]{Lurie:HA}, subsumed in the
  general theory of $\infty$-operads.  In the present work, a few monoidal
  structures play an important role, but since they are directly induced by
  cartesian product, we have preferred to deal with
  them in an informal (and possibly not completely rigorous) way, 
  with the same freedom as one deals with cartesian products in
  ordinary category theory.  In these specific cases the formal treatment should
  not present any technical difficulties.

\end{blanko}

%% file: simplicial.tex
\def\inputfile{simplicial.tex}

\subsection{Simplicial preliminaries}

Our work relies heavily on simplicial machinery.  We briefly review the
notions needed, to establish conventions and notation. 




\begin{blanko}{The simplex category (the topologist's Delta).}
  Recall that the `simplex category' $\Delta$ is the category whose objects are the nonempty
  finite ordinals
  $$
  [k] := \{ 0,1,2,\ldots, k\} ,
  $$
  and whose morphisms are the monotone maps.  These are generated by
  the coface maps $d^i : [n-1]\to [n]$, which are the monotone
  injective functions for which $i \in [n]$ is not in
  the image, and codegeneracy maps $s^i:[n+1] \to [n]$, which are
  monotone surjective functions for which $i \in [n]$ has a double preimage.
  We write $d^\bot:=d^0$ and
  $d^\top:=d^n$ for the outer coface maps.
\end{blanko}

\begin{blanko}{Generic and free maps.}\label{generic-and-free}
  The category $\Delta$ has a generic-free factorisation system.
  A morphism of $\Delta$ is termed \emph{generic}, and written
  $g:[m]\genmap [n]$, if it preserves end-points, $g(0)=0$ and
  $g(m)=n$.  A morphism is termed \emph{free}, and written $f: [m]\rat
  [n]$, if it is distance preserving, $f(i+1)=f(i)+1$ for $0\leq i\leq
  m-1$.  The generic maps are generated by the codegeneracy maps and
  the inner coface maps, and the free maps are generated by the outer
  coface maps.  Every morphism in $\Delta$ factors uniquely as a generic
  map followed by a free map, as detailed below.
\end{blanko}

\begin{blanko}{Background remarks.}
  The notions of generic and free maps are general notions in category theory,
  introduced by Weber \cite{Weber:TAC13,Weber:TAC18}, who extracted the
  notion from earlier work of Joyal~\cite{Joyal:foncteurs-analytiques}; a
  recommended entry point to the theory is
  Berger--Melli\`es--Weber~\cite{Berger-Mellies-Weber:1101.3064}.
  The notion makes sense for example whenever there is a cartesian monad on a
  presheaf category $\CC$: in the Kleisli category, the free maps are those from
  $\CC$, and the generic maps are those generated by the monad. In practice,
  this is restricted to a suitable subcategory of combinatorial nature.
  In the case at
  hand the monad is the free-category monad on the category of directed
  graphs, and $\Delta$ arises as the restriction of the Kleisli category
  to the subcategory of non-empty linear graphs.
  Other important instances of generic-free factorisation systems are found in
  the category of rooted trees \cite{Kock:0807} (where the monad is the
  free-operad monad), the category of Feynman graphs~\cite{Joyal-Kock:0908.2675}
  (where the monad is the free-modular-operad monad), the category of directed
  graphs~\cite{Kock:1407.3744} (where the monad is the free-properad monad), and
  Joyal's cellular category $\Theta$ \cite{Berger:Adv} (where the monad is the
  free-omega-category monad).
\end{blanko}

\begin{blanko}{Amalgamated ordinal sum.}
   The {\em amalgamated ordinal sum over $[0]$} of two
   objects $[m]$ and $[n]$, denoted $[m]\pm [n]$,
   is given by the pushout of free maps
   \begin{equation}\label{pm}\vcenter{\xymatrix@C+1.2pt{
   [0] \ar@{ >->}[r]^{(d^\top)^n} \ar@{ >->}[d]_{(d^\bot)^m} & [n]
   \ar@{ >->}[d]^{(d^\bot)^m}\\
   [m] \ar@{ >->}[r]_(0.4){(d^\top)^n} & [m]\pm[n] \makebox[0em][l]{${}=[m+n]$}\ulpullback
   }}\end{equation}
   This operation is not functorial on all maps in $\Delta$, 
   but on the subcategory $\Deltagen$ of generic maps it is functorial
   and defines a monoidal structure on $\Deltagen$
   (dual to ordinal sum (cf.~Lemma~\ref{lem:Delta-duality})).

   The free maps $f:[n]\rat [m]$ are precisely the maps
   that can be written
    $$
    f:[n]\rat [a]\pm[n]\pm[b].
    $$
    Every generic map with source $[a]\pm[n]\pm[b]$ splits as
    $$
    (\xymatrix@!C{{}[a]\ar@{ ->|}[r]^{g_1}&[a']})
    \;\pm\;
    (\xymatrix@!C{{}[n]\ar@{ ->|}[r]^{g}&[k]})
    \;\pm\;
    (\xymatrix@!C{{}[b]\ar@{ ->|}[r]^{g_2}&[b']})
    $$

    With these observations we can be explicit about the
    generic-free factorisation:
\end{blanko}

\begin{lemma}\label{lem:genfactsplit}
  With notation as above, the generic-free factorisation of a free map $f$
  followed by a generic map $g_1\pm g\pm g_2$ is given by
    \begin{equation}\label{gffg}\vcenter{\xymatrix{
     [n] \ar@{ >->}[r]^-f\ar@{->|}[d]_g & [a]\pm[n]\pm[b] \ar@{->|}[d]^{g_1\pm 
     g\pm g_2} \\
     [k]\ar@{ >->}[r] & [a']\pm[k]\pm[b']
  }}\end{equation}
\end{lemma}
\begin{blanko}{Identity-extension squares.}
   A square \eqref{gffg} in which $g_1$ and $g_2$ are identity maps is called an  {\em identity-extension square}.
\end{blanko}
\begin{lemma}\label{genfreepushout}
  Generic and free maps in $\Delta$ admit pushouts along each other, and the 
  resulting maps are again generic and free. In fact, generic-free pushouts are precisely the identity extension squares.
\begin{equation*}
\vcenter{\xymatrix{
     [n] \ar@{ >->}[r]\ar@{->|}[d] & [a]\pm[n]\pm[b] \ar@{->|}[d] \\
     [k]\ar@{ >->}[r] & [a]\pm[k]\pm[b]
}}\end{equation*}
%
\end{lemma}

These pushouts are fundamental to this work.  We will define
decomposition spaces to be simplicial spaces $X:\Delta\op\to\Grpd$
that send these pushouts to pullbacks.

The previous lemma has the following easy corollary.
\begin{cor}\label{d1s0Delta}
  Every codegeneracy map is a pushout (along a free map) of $s^0:[0]\to[1]$,
  and every generic coface maps is a pushout (along a free map) of 
  $d^1:[2]\to[1]$.
%
\end{cor}

%% file: segal.tex
\def\inputfile{segal.tex}

\subsection{Segal spaces}

\begin{blanko}{Simplicial $\infty$-groupoids.}\label{simpl-Grpd}
  Our main object of study will be simplicial $\infty$-groupoids subject to
  various exactness conditions, all formulated in terms of pullbacks.
  More precisely we work in the functor $\infty$-category
  $$
  \Fun(\Delta\op,\Grpd) ,
  $$
  whose vertices are functors from the $\infty$-category $\Delta\op$ to the
  $\infty$-category $\Grpd$. 
  In particular, the simplicial
  identities for $X:\Delta\op\to\Grpd$ are not strictly commutative squares;
  rather they are $\Delta[1]\times \Delta[1]$-diagrams in $\Grpd$, hence come
  equipped with a homotopy between the two ways around in the square.  But this
  is precisely the setting for pullbacks.
\end{blanko}

Consider a simplicial $\infty$-groupoid $X: \Delta\op \to \Grpd$.
We recall the {\em Segal maps}
$$
(\partial_{0,1},\dots,\partial_{r-1,r}):X_r \longrightarrow X_1 \times_{X_0} 
\cdots \times_{X_0} X_1 \qquad r\geq 0.
$$
where $\partial_{k-1,k}:X_r\to X_1$ is induced by the map 
$[1]\rat[r]$ 
sending 0,1  to $k-1,k$.

A {\em Segal space} is a simplicial $\infty$-groupoid satisfying the Segal 
condition,
namely that the Segal maps are equivalences.

%
%

\begin{lemma}\label{segalpqr}
  The following conditions are equivalent, for any simplicial $\infty$-groupoid $X$:
\begin{enumerate}
\item $X$ satisfies the Segal condition,
$$X_r \stackrel\simeq\longrightarrow X_1 \times_{X_0} \cdots \times_{X_0} X_1 \qquad r\geq 0.$$
\item The following square is a pullback for all $p,q\geq r$
$$
\vcenter{\xymatrix{
   X_{p-r+q}\dto_{{{d_{p+1}}^{q-r}}}
\drpullback 
\rto^-{{d_{0}}^{p-r}}
 &  X_q\dto^{{{d_{r+1}}^{q-r}}} \\
   X_{p}\rto_-{{d_{0}}^{p-r}}  &  X_{r}
  }}
$$
\item The following square is a pullback for all $n>0$
$$
\vcenter{\xymatrix{
   X_{n+1}\dto_{d_\top}\drpullback 
\rto^-{d_\bot}
 &  X_n\dto^{d_\top} \\
   X_{n}\rto_-{d_\bot}  &  X_{n-1}
  }}
$$
\item The following square is a pullback for all $p,q\geq 0$
$$
\vcenter{\xymatrix{
   X_{p+q}\dto_{{{d_{p+1}}^{q}}}
\drpullback 
\rto^-{{d_{0}}^{p}}
 &  X_q\dto^{{{d_{1}}^{q}}} \\
   X_{p}\rto_-{{d_{0}}^{p}}  &  X_{0}
  }}
$$
\end{enumerate}
\end{lemma}
\begin{proof}
It is straightforward to show that the Segal condition implies (2). 
Now (3) and (4) are special cases of (2). Also (3) implies (2):
the pullback in (2) is a composite of pullbacks of the type given in (3). 
Finally one shows inductively that (4) implies the Segal condition (1).
\end{proof}  

A map $f:Y\to X$ of simplicial spaces is {\em cartesian} on an arrow $[n]\to[k]$ in $\Delta$ if the naturality square for $f$ with respect to this arrow is a pullback.
\begin{lemma}\label{cart/Segal=Segal}
    If a simplicial map $f: Y \to X$ is 
    cartesian on outer face maps, and if $X$ 
    is a Segal space, then $Y$ is a Segal space too.
\end{lemma}


\begin{blanko}{Rezk completeness.}
  Let $J$ denote the (ordinary) nerve of the 
  groupoid generated by one isomorphism $0 \to 1$.  
  A Segal space $X$ is {\em Rezk complete} when 
  the natural map
  $$
  \Map(*, X) \to \Map(J,X)
  $$
  (obtained by precomposing with $J \to *$)
  is an equivalence of $\infty$-groupoids.
  It means that the space of identity arrows is
  equivalent to the space of equivalences.  
  (See \cite[Thm.6.2]{Rezk:MR1804411}, \cite{Bergner:0610239} 
  and \cite{Joyal-Tierney:0607820}.) 
  A Rezk complete Segal space is also called a {\em Rezk category}.
\end{blanko}

\begin{blanko}{Ordinary nerve.}
    Let $\CC$ be a small $1$-category.  The {\em nerve} of $\CC$ is the simplicial set
  \begin{eqnarray*}
 N\CC:   \Delta\op & \longrightarrow & \Set  \\
    {}[n]  & \longmapsto & \Fun([n],\CC) ,
  \end{eqnarray*}
  where $\Fun([n],\CC)$ is the {\em set} of strings of $n$ composable arrows.
  Sub\-examples of this are given by any poset or any monoid.  The simplicial
  sets that arise like this are precisely those satisfying the Segal condition
  (which is strict in this context).  If each set is regarded as a discrete
  $\infty$-groupoid, $N\CC$ is thus a Segal space.  In general it is not
  Rezk complete, since some object may have a nontrivial automorphism.
As an example, if $\CC$ is a
  one-object groupoid (i.e.~a group), then inside $(N\CC)_1$ the space of
  equivalences is the whole set $(N\CC)_1$, but the degeneracy map $s_0 :
  (N\CC)_0 \to (N\CC)_1$ is not an equivalence (unless the group is trivial).
\end{blanko}

\begin{blanko}{The fat nerve of an essentially small $1$-category.}\label{fatnerve}
  In most cases it is more interesting to consider the
  {\em fat nerve}, defined as the simplicial {\em groupoid} 
  \begin{eqnarray*}
    X :\Delta\op & \longrightarrow & \Grpd  \\
    {}[k] & \longmapsto & \Map(\Delta[k],\CC),
  \end{eqnarray*}
  where $\Map(\Delta[k],\CC)$ is the mapping space, defined as the
  maximal subgroupoid of the functor category $\Fun(\Delta[k],\CC)$.
  In other words, $(\mathbf N\CC)_n$ is the groupoid
  whose objects are strings of $n$ composable arrows in $\CC$ and
  whose morphisms are connecting isos between such strings:
   $$
   \xymatrix{
   \cdot \ar[r] \ar[d]^*-[@]=0+!L{\scriptstyle \sim} & \cdot \ar[d]^*-[@]=0+!L{\scriptstyle \sim}
   \ar[r]  & \cdot \ar[r] \ar[d]^*-[@]=0+!L{\scriptstyle \sim} &\cdots\ar[r]&
    \cdot \ar[d]^*-[@]=0+!L{\scriptstyle \sim} \\
   \cdot \ar[r] & \cdot\ar[r] & \cdot\ar[r] &\cdots\ar[r]& \cdot
   }$$
   It is straightforward to check the Segal condition, remembering that the
   pullbacks involved are homotopy pullbacks.  For instance, the pullback
   $X_1\times_{X_0} X_1$ has as objects strings of `weakly composable'
   arrows, in the sense that the target of the first arrow is isomorphic
   to the source of the second, and a comparison isomorphism is specified.  The
   Segal map $X_2 \to X_1 \times_{X_0} X_1$ is the inclusion of the subgroupoid
   consisting of strictly composable pairs.  But any
   weakly composable pair is isomorphic to a strictly composable pair, and
   the comparison isomorphism is unique, hence the inclusion $X_2 \into
   X_1\times_{X_0} X_1$ is an equivalence.  Furthermore, the fat nerve is Rezk 
   complete.  Indeed, it is easy to see that inside $X_1$, the equivalences
   are the invertible arrows of $\CC$. But any invertible arrow is
   equivalent to an identity arrow.


  Note that if $\CC$ is a category with no non-trivial
  isomorphisms (e.g.~any \M category in the sense of Leroux)
  then the fat nerve coincides with the ordinary nerve, and if
  $\CC$ is just equivalent to such a category
  then the fat nerve is level-wise equivalent to the ordinary nerve of any
  skeleton of $\CC$.
\end{blanko}

\begin{blanko}{Joyal--Tierney $t\uppershriek$ --- the fat nerve of an 
  $\infty$-category.}
  The fat nerve construction is just a special case of the general
  construction $t\uppershriek$ of Joyal and Tierney~\cite{Joyal-Tierney:0607820},
  which is a functor from quasi-categories to complete Segal spaces, meaning 
  specifically certain simplicial objects in the category of Kan complexes:
  given a quasi-category $\CC$, the complete Segal space $t\uppershriek \CC$
  is given by
  \begin{eqnarray*}
    \Delta\op & \longrightarrow & \kat{Kan}  \\
    {}[n] & \longmapsto & \big[ [k] \mapsto \kat{sSet}(\Delta[n] \times \Delta'[k], 
    \CC) \big]
  \end{eqnarray*}
  where $\Delta'[k]$ denotes the groupoid freely generated by a string of $k$
  invertible arrows.  They show that $t\uppershriek$ constitutes in fact a (right)
  Quillen equivalence between the simplicial sets with the Joyal model
  structure, and bisimplicial sets with the Rezk model structure. 
  
  Taking a more invariant viewpoint, talking about $\infty$-groupoids
  abstractly, the Joyal--Tierney $t\uppershriek$ functor associates to an
  $\infty$-category $\CC$ the Rezk complete Segal space
  \begin{eqnarray*}
    \Delta\op & \longrightarrow & \Grpd  \\
    {}[n] & \longmapsto & \Map(\Delta[n], \CC) .
  \end{eqnarray*}
\end{blanko}

\begin{blanko}{Fat nerve of bicategories with only invertible $2$-cells.}
  From a bicategory $\CC$ with only invertible $2$-cells one can get a complete
  Segal bigroupoid by a construction analogous to the fat nerve.  (In fact, this
  can be viewed as the $t\uppershriek$ construction applied to the so-called
  Duskin nerve of $\CC$.)
%
%
%
%
The {\em fat nerve} of a bicategory $\CC$ is the simplicial bigroupoid
\begin{eqnarray*}
\Delta\op & \longrightarrow & \kat{2}\Grpd \\
{}[n] & \longmapsto & \PsFun(\Delta[n],\CC) ,
\end{eqnarray*}
the $2$-groupoid of normalised pseudofunctors.
\end{blanko}

\begin{blanko}{Monoidal groupoids.}
  Important examples of the previous situation come from monoidal groupoids
  $(\MM,\tensor,I)$. 
  The fat nerve construction applied to the classifying space $B\MM$ yields in
  this case a complete Segal bigroupoid, with zeroth space $B\MM^\eq$, the
  classifying space of the full subcategory $\MM^{\eq}$ spanned by the
  tensor-invertible objects.
  
  The fat nerve construction can be simplified considerably in the case that $\MM^{\eq}$
  is contractible.  This happens precisely when every tensor-invertible object is
  isomorphic to the unit object $I$ and $I$ admits no
  non-trivial automorphisms.
\end{blanko}

\begin{prop}\label{prop:BM}
  If $(\MM,\tensor,I)$ is a monoidal groupoid such that $\MM^\eq$ is 
  contractible, then the simplicial bigroupoid given by the classifying space is
  equivalent to the simplicial $1$-groupoid
  \begin{eqnarray*}
    \Delta\op & \longrightarrow & \kat{1-\Grpd}  \\
    {}[n] & \longmapsto & \MM\times\MM\times\dots\times\MM =: \MM^n.
  \end{eqnarray*}
  where the outer face maps project away an outer factor, the inner face maps
  tensor together two adjacent factors, and the degeneracy maps insert a neutral 
  object. 
\end{prop} 
\noindent We have omitted the proof, to avoid going into $2$-category theory.
(Note that the simplicial $1$-groupoid that we obtain is not {\em strictly}
simplicial, unless the monoidal structure is strict.)

  Examples of monoidal groupoids satisfying the conditions of the Proposition
  are the monoidal groupoid $(\kat{FinSet}, +, 0)$ of finite sets and bijections
  or the monoidal groupoid $(\kat{Vect}, \oplus, \mathbf 0)$ of vector spaces
  and linear isomorphisms under direct sum.  In contrast, the monoidal groupoid
  $(\Vect, \tensor, \ground)$ of vector spaces and linear isomorphisms under
  tensor product is not of this kind, as the unit object has many automorphisms.
  The assignment $[n] \mapsto \Vect^{\otimes n}$ does constitute a Segal $1$-groupoid,
  but it is not Rezk complete.

%% file: new-decomp-section.tex
\subsection{Decomposition spaces}
\label{sec:decomp}

\def\inputfile{new-decomp-section.tex}

Recall from Lemma~\ref{genfreepushout} that generic and free maps in $\Delta$ admit pushouts along each other.
\begin{deff}
  A \emph{decomposition space} is a simplicial $\infty$-groupoid 
  $$
  X:\Delta\op\to\Grpd
  $$ 
such that the image of any pushout diagram in $\Delta$ of a generic map $g$ along 
a free map $f$ is a pullback of $\infty$-groupoids,
$$ X\!\left(\!\! 
\vcenter{\xymatrix{{}
   [p] \drpullback{} &  [m]\ar[l]_{g'}  \\
   [q]\ar[u]^-{f'}   & \ar[l]^{g} [n] \ar[u]_-{f}
  }}
\right)
\qquad=\qquad  \vcenter{\xymatrix{
   X_{p}\ar[d]_{{f'}^*} \ar[r]^-{{g'}^*} \drpullback&  X_m\ar[d]^{{f}^*}  \\
   X_q\ar[r]_-{{g}^*}   &  X_n .
  }}
$$
\end{deff}

\begin{BM}\label{DK}
   The notion of decomposition space can be seen as an 
   abstraction of coalgebra, cf.~Section~\ref{sec:COALG} below:
   it is precisely the condition required to obtain a counital
   coassociative comultiplication on $\Grpd_{/X_1}$.

  The notion 
  is equivalent to the notion of
  unital (combinatorial) $2$-Segal space introduced by Dyckerhoff and 
  Kapranov~\cite{Dyckerhoff-Kapranov:1212.3563} (their Definition~2.3.1, 
  Definition~2.5.2, Definition 5.2.2, Remark~5.2.4).  Briefly, their
  definition goes as follows.  For any triangulation $T$ of a convex polygon
  with $n$ vertices, there is induced a simplicial subset $\Delta^T \subset 
  \Delta[n]$.  A simplicial space $X$ is called $2$-Segal if, for every
  triangulation $T$ of every convex $n$-gon, the induced map $\Map(\Delta[n],X) \to 
  \Map(\Delta^T,X)$ is a weak homotopy equivalence.
  Unitality is defined in terms of pullback conditions involving 
  degeneracy maps, similar to our \eqref{unital-cond} below.  The equivalence between
  decomposition spaces and unital $2$-Segal spaces follows from 
  Proposition~2.3.2 of \cite{Dyckerhoff-Kapranov:1212.3563}
  which gives a pullback criterion for the $2$-Segal condition.
\end{BM}

\begin{blanko}{Alternative formulations of the pullback condition.}
  To verify the conditions of the definition, it will in fact be 
sufficient to check a smaller collection of squares.  On the other hand, the definition will imply that many other squares of interest are pullbacks too.  
The formulation in terms of generic and free maps is preferred
  both for practical reasons and for its conceptual simplicity
  compared to the smaller or larger collections of squares.


Recall from Lemma \ref{genfreepushout} that the generic-free pushouts used in the definition are just the identity extension squares,
$$\xymatrix{
     [n] \ar@{ >->}[d]\ar@{->|}[rr]^-g &&[k] \ar@{ >->}[d] \\
[a]\pm[n]\pm[b]     \ar@{->|}[rr]_-{\id\pm g\pm\id} && [a]\pm[k]\pm[b]}$$
Such a square can be written as a vertical composite of squares in which either $a=1$ and $b=0$, or vice-versa. In turn, since
the generic map $g$ 
is a composite of inner face maps $d^i:[m-1]\to[m]$ ($0<i<m$) and degeneracy maps $s^j:[m+1]\to[m]$, these squares are horizontal composites of pushouts of a single generic $d^i$ or $s^j$ along $d^\bot$ or $d^\top$. 
Thus, to check that $X$ is a decomposition space, it is sufficient to check the following special cases
are pullbacks, for $0<i<n$ and  $0\leq j\leq n$:
$$  
\xymatrix{
   X_{1+n}\ar[d]_{d_\bot}\drpullback 
\ar[r]^-{d_{1+i}}
 &  X_n\ar[d]^{d_\bot} \\
   X_{n}\ar[r]_-{d_i}  &  X_{n-1},
  }
\qquad
\xymatrix{
   X_{n+1}\ar[d]_{d_\top}\drpullback 
\ar[r]^-{d_{i}}
 &  X_n\ar[d]^{d_\top} \\
   X_{n}\ar[r]_-{d_i}  &  X_{n-1},
  }
$$
\begin{equation}\label{unital-cond}
\vcenter{\xymatrix{
   X_{1+n}\ar[r]^{s_{1+j}}\drpullback 
\ar[d]_-{d_\bot}
 & X_{1+n+1}\ar[d]^{d_\bot} \\
  X_n  \ar[r]_-{s_j}  &  X_{n+1},
}}\qquad
\vcenter{\xymatrix{
   X_{n+1}\ar[d]_{d_\top}\drpullback 
\ar[r]^-{s_{j}}
 &  X_{n+1+1} \ar[d]^{d_\top} \\
X_n  \ar[r]_-{s_j}  &  X_{n+1}.
  }}
\end{equation}


The following proposition shows we can be more economic: instead of checking all $0<i<n$ it is enough 
to
check all $n\geq 2$ and {\em some} $0<i<n$, and instead of checking all
$0\leq j \leq n$ it is enough to check the case $j=n=0$.
\end{blanko}

\begin{prop}\label{onlyfourdiags}
A simplicial $\infty$-groupoid $X$ is a decomposition space if and only if 
the following diagrams are pullbacks
$$  
\xymatrix{
   X_1\ar[r]^{s_1}\drpullback 
\ar[d]_-{d_\bot}
 & X_2\ar[d]^{d_\bot} \\
  X_0  \ar[r]_-{s_0}  &  X_1,
}\qquad
\xymatrix{
   X_1\ar[d]_{d_\top}\drpullback 
\ar[r]^-{s_0}
 &  X_2 \ar[d]^{d_\top} \\
X_0  \ar[r]_-{s_0}  &  X_1,
  }
$$
and the following diagrams are pullbacks for some choice of $i=i_n$, 
$0<i<n$, for each $n\geq 2$:
$$  
\xymatrix{
   X_{1+n}\ar[d]_{d_\bot}\drpullback 
\ar[r]^-{d_{1+i}}
 &  X_n\ar[d]^{d_\bot} \\
   X_{n}\ar[r]_-{d_i}  &  X_{n-1},
  }
\qquad
\xymatrix{
   X_{n+1}\ar[d]_{d_\top}\drpullback 
\ar[r]^-{d_{i}}
 &  X_n\ar[d]^{d_\top} \\
   X_{n}\ar[r]_-{d_i}  &  X_{n-1}.
  }
$$
\end{prop}

\begin{proof}
  To see the non-necessity of the other degeneracy cases, observe
  that for $n>0$, every degeneracy map $s_j: X_n \to X_{n+1}$
  is the section of an {\em inner} face map
  $d_i$ (where $i=j$ or $i=j+1$).  Now in the diagram
  $$
  \xymatrix{
   X_{1+n}\ar[r]^{s_{1+j}} \ar[d]_-{d_\bot} & X_{1+n+1}\ar[d]^{d_\bot} \ar[r]^{d_{1+i}} & 
   X_{1+n} \ar[d]^{d_\bot}\\
  X_n  \ar[r]_-{s_j}  &  X_{n+1} \ar[r]_{d_i} & X_n,
  }$$
  the horizontal composites are identities, so the outer rectangle is a
  pullback, and the right-hand square is a pullback since it is one of cases
  outer face with inner face.  Hence the left-hand square, by Lemma~\ref{pbk}, is a pullback too.
  The case $s_0: X_0 \to X_1$ is the only degeneracy map that is not the section
  of an inner face map, so we cannot eliminate the two cases involving this map.
  The non-necessity of the other inner-face-map cases is the content of the
  following lemma.
\end{proof}

\begin{lemma}\label{lem:fewerdiagrams}
The following are equivalent for a simplicial $\infty$-groupoid $X$. 
\begin{enumerate}
  \item  For each $n\geq2$, the following diagram is a pullback for all $0<i<n$:
$$  
\vcenter{\xymatrix{
   X_{1+n}\ar[d]_{d_\bot}\drpullback 
\ar[r]^-{d_{1+i}}
 &  X_n\ar[d]^{d_\bot} \\
   X_{n}\ar[r]_-{d_i}  &  X_{n-1},
  }}
\qquad\left(\text{resp. }\vcenter{
\xymatrix{
   X_{n+1}\ar[d]_{d_\top}\drpullback 
\ar[r]^-{d_{i}}
 &  X_n\ar[d]^{d_\top} \\
   X_{n}\ar[r]_-{d_i}  &  X_{n-1},
  }}\right)
$$

 \item 
For each $n\geq2$, the above diagram is a pullback for some $0<i<n$.
 \item 
For each $n\geq2$, the following diagram is a pullback:
$$
\vcenter{\xymatrix{
   X_{1+n}\ar[d]_{d_\bot}\drpullback 
\ar[r]^-{{d_2}^{n-1}}
 &  X_2\ar[d]^{d_\bot} \\
   X_n\ar[r]_-{{d_1}^{n-1}}  &  X_{1}
  }}
\qquad\left(\text{resp. }\vcenter{\xymatrix{
   X_{n+1}\ar[d]_{d_\top}\drpullback 
\ar[r]^-{{d_1}^{n-1}}
 &  X_2\ar[d]^{d_\top} \\
   X_n\ar[r]_-{{d_1}^{n-1}}  &  X_{1}
  }}\right)
$$
\end{enumerate}
\end{lemma}

\begin{proof} The hypothesised pullback in (2) is a special case of that in (1),
  and that in (3) is a horizontal composite of those in (2), since there is a unique generic map $[1]\to[n]$ in $\Delta$ for each $n$. 
The implication (3) $\Rightarrow$ (1) follows by Lemma~\ref{pbk} and the commutativity for 
$0<i<n$ of the diagram
$$
\vcenter{\xymatrix{
   X_{1+n} \drpullback \ar[r]^{d_{1+i}} \ar[d]_{d_\bot} & 
   X_n \drpullback \ar[r]^{{d_2}^{n-1}} \ar[d]_{d_\bot} &
   X_2 \ar[d]^{d_\bot} \\
   X_n \ar[r]_{d_i} &
   X_{n-1} \ar[r]_{{d_1}^{n-1}} &
   X_1
   }}
   $$
Similarly for the `resp.' case.
\end{proof}

\begin{prop}\label{prop:segal=>decomp1}
  Any Segal space is a decomposition space.
\end{prop}
\begin{proof}
  Let $X$ be Segal space.
%
In the diagram ($n\geq 2$)
  $$
\vcenter{\xymatrix{
   X_{n+1}\ar[d]_{d_\bot} 
\ar[r]^-{d_n}
 &  X_n\ar[d]_{d_\bot}\drpullback 
\ar[r]^-{d_\top}
 &  X_{n-1}\ar[d]^{d_\bot} \\
   X_{n}\ar[r]_-{d_{n-1}}  &  X_{n-1} \ar[r]_-{d_\top}  &  X_{n-2} ,
  }}
$$
since the horizontal composites are equal to $d_\top\circ d_\top$, both the outer
rectangle and the right-hand square are pullbacks by the Segal condition 
(\ref{segalpqr}~(3)).
Hence the left-hand square is a pullback.  This establishes the third
pullback condition in Proposition~\ref{onlyfourdiags}. 
In the diagram
  $$
\vcenter{\xymatrix{
   X_{1}\ar[d]_{d_\bot} 
\ar[r]^-{s_1}
 &  X_2\ar[d]_{d_\bot}\drpullback 
\ar[r]^-{d_\top}
 &  X_1\ar[d]^{d_\bot} \\
   X_0\ar[r]_-{s_0}  &  X_1 \ar[r]_-{d_\top}  &  X_0 ,
  }}
$$
since the horizontal composites are identities, the outer rectangle is a 
pullback, and the right-hand square is a pullback by the Segal condition.
Hence the left-hand square is a pullback, establishing the
first of the
pullback conditions in Proposition~\ref{onlyfourdiags}.
The remaining two conditions of Proposition~\ref{onlyfourdiags}, those
involving $d_\top$ instead of $d_\bot$, are obtained
similarly by interchanging the roles of $\bot$ and $\top$.
\end{proof}




\begin{BM}
  This result was also obtained by 
  Dyckerhoff and Kapranov~\cite{Dyckerhoff-Kapranov:1212.3563}
  (Propositions~2.3.3, 2.5.3, and 5.2.6).
\end{BM}

Corollary \ref{d1s0Delta} implies the following important property of decomposition spaces.

\begin{lemma}\label{lem:s0d1}
    In a decomposition space $X$, every generic face map is a pullback of
    $d_1: X_2 \to X_1$, and every degeneracy map is a pullback of $s_0 :X_0 \to X_1$.
\end{lemma}
Thus, even though the spaces in degree $\geq 2$
are not fibre products of $X_1$ as in a Segal space,
the higher generic face maps and degeneracies 
are determined by `unit' and `composition',
$$
\xymatrix{
X_0 \ar[r]^{s_0} & X_1 & \ar[l]_{d_1} X_2 .
}$$

In $\Delta\op$ there are more pullbacks than those between generic and free.
Diagram \eqref{pm} in \ref{generic-and-free} is a pullback in
$\Delta\op$ that is not preserved by all decomposition spaces, though it is
preserved by all Segal spaces.  On the other hand, certain other pullbacks in
$\Delta\op$ are preserved by general decomposition spaces.  We call them
colloquially `bonus pullbacks':

\begin{lemma}\label{bonus-pullbacks}
  For a decomposition space $X$, the following squares are pullbacks: 
    $$\vcenter{
  \xymatrix{
  X_{n+1}  \drpullback\ar[r]^{d_{j}} \ar[d]_{s_i} & X_n \ar[d]^{s_i} \\
  X_{n+2} \ar[r]_{d_{j+1}} & X_{n+1}
  }} \; \text{for all $i<j$,\; and }\;
  \vcenter{\xymatrix{
  X_{n+1}  \drpullback\ar[r]^{d_{j}} \ar[d]_{s_{i+1}} & X_n \ar[d]^{s_i} \\
  X_{n+2} \ar[r]_{d_{j}} & X_{n+1}
  }}\; \text{for all $j\leq i$}.
  $$
\end{lemma}

%
\begin{proof}
  We treat the case $i<j$; for the other case, interchange the roles of $\top$ and 
  $\bot$.  Postcompose horizontally with sufficiently many ${d_\top}$ to make the 
  total composite  free:
  $$\xymatrix{
  X_{n+1}  \ar[r]^{d_{j}} \ar[d]_{s_i} & X_n \ar[d]_{s_i} 
  \ar[r]^{{d_\top}^{n+1-j}} 
  \drpullback & X_{j-1} \ar[d]^{s_i} \\
  X_{n+2} \ar[r]_{d_{j+1}} & X_{n+1} \ar[r]_{{d_\top}^{n+1-j}} & X_j .
  }$$
  The horizontal composite maps are now ${d_\top}^{n+2-j}$, so the outer
  rectangle is a pullback, and the second square is a pullback.  Hence by the
  basic lemma~\ref{pbk}, also the first square is a pullback, as claimed.
%
%
%
\end{proof}

\begin{lemma}\label{newbonuspullbacks}
  For a decomposition space $X$, the following squares are pullbacks 
  for all $i<j$:
  $$\xymatrix{
  X_n  \drpullback\ar[r]^{s_{j-1}} \ar[d]_{s_i} & X_{n+1} \ar[d]^{s_i} \\
  X_{n+1} \ar[r]_{s_{j}} & X_{n+2}
  }
  $$
\end{lemma}
\begin{proof}
  Just observe that $s_j$ is a section to $d_{j+1}$, and apply the standard 
  argument: if $d_{j+1}$ is an outer face map then the square is a basic
  generic-free pullback; if $d_{j+1}$ is inner, we can use instead the previous lemma.
\end{proof}

\begin{BM}\label{cheaperbonus}
  Note that in the previous two lemmas, the full decomposition-space axiom is
  not needed: it is enough to have pullback squares between free maps and
  degeneracy maps.  This will be useful briefly in Section~\ref{sec:stiff}.
\end{BM}

%% file: ulf.tex
\def\inputfile{ulf.tex}

\subsection{Conservative ULF functors}

\label{sec:cULF}

\begin{deff}
A simplicial map $F: Y \to X$
is called {\em ULF (unique lifting of factorisations)}
if it is a cartesian natural transformation on generic face
maps of $\Delta$.  It is called {\em conservative} if it is cartesian
on degeneracy maps. It is called {\em cULF} if it is both conservative and ULF. 
\end{deff}

\begin{lemma}
  For a simplicial map $F: Y \to X$,
  the following are equivalent.
  \begin{enumerate}
  \item $F$ is cartesian on all generic maps (i.e.~cULF).
  \item $F$ is cartesian on every inner face map and on every degeneracy map.
  \item $F$ is cartesian on every generic map of the form $[1]\to [n]$.
  \end{enumerate}
\end{lemma}

\begin{proof}
  That (1) implies (2) is trivial.  The implication $(2)\Rightarrow (3)$ is easy
  since the generic map $[1]\to [n]$ factors as a sequence of inner face maps
  (or is a degeneracy map if $n=0$).  For the implication $(3) \Rightarrow (1)$, consider 
  a general generic map $[n]\to [m]$, and observe that if $F$ is cartesian on the composite of generic maps $[1]\to[n]\to[m]$ and also on the generic map $[1]\to[n]$, then it is cartesian on $[n]\to[m]$ also.
\end{proof}

\begin{prop}\label{prop:ULF=>cons}
  If $X$ and $Y$ are decomposition spaces (or just stiff simplicial spaces, in 
  the sense of Section~\ref{sec:stiff}), then every ULF map $F:Y \to X$ is 
  automatically conservative.
\end{prop}
\begin{proof}
  In the diagram 
  $$
  \xymatrix  {
  Y_0 \ar[rrd]^(0.7){s_0}
  \ar[r]^{s_0} \ar[d] & Y_1 \ar[rrd]^(0.7){s_1}\ar[d]|\hole&& \\
  X_0 \ar[rrd]_(0.4){s_0} \ar[r]^{s_0} & X_1 \ar[rrd]|\hole^(0.7){s_1} & Y_1  \ar[r]_{s_0}  
  \ar[d] & Y_2 \ar[d] \drpullback \ar@{..>}[r]^{d_1} & Y_1 \ar@{..>}[d]\\
  && X_1 \ar[r]_{s_0} & X_2 \ar@{..>}[r]_{d_1} & X_1
  }
$$
the front square is a pullback since it is a section to the dotted square,
which is a pullback since $F$ is ULF.  The same argument shows that $F$ is 
cartesian on all degeneracy maps that are sections to generic face maps. This 
includes all degeneracy maps except the one appearing in the back square of the 
diagram.  But since the top and bottom slanted squares are bonus pullbacks 
(\ref{newbonuspullbacks}),
also the back square is a pullback.
\end{proof}
The following result is a consequence of Lemma~\ref{lem:s0d1} and 
Proposition~\ref{prop:ULF=>cons}.
\begin{lemma}\label{lem:cULFs0d1}
  A simplicial map $F:Y \to X$ between decomposition spaces 
  (or just stiff simplicial spaces, in 
  the sense of Section~\ref{sec:stiff})
  is cULF if and only if
  it is cartesian on the generic map $[1]\to[2]$
$$\xymatrix{
Y_1 \ar[d] & \ar[l] \dlpullback Y_2 \ar[d] \\
X_1 & \ar[l] X_2}
$$
\end{lemma}


\begin{BM}
  The notion of cULF
  can be seen as an abstraction of coalgebra homomorphism, 
  cf.~\ref{lem:coalg-homo} below: `conservative' corresponds to counit preservation,
  `ULF' corresponds to comultiplicativity.
  
  In the special case where $X$ and $Y$ are fat nerves of $1$-categories, 
  then the condition that the square
  $$
  \xymatrix{
  Y_0 \ar[d] \ar[r] \drpullback & Y_1 \ar[d] \\
  X_0\ar[r] & X_1}$$
  be a pullback is precisely the classical notion of conservative functor
  (i.e.~if $f(a)$ is invertible then already $a$ is invertible).
 
  Similarly, the condition that the square
  $$\xymatrix{
  Y_1 \ar[d] & \ar[l] \dlpullback Y_2 \ar[d] \\
  X_1 & \ar[l] X_2}
  $$
  be a pullback is an up-to-isomorphism version of the classical notion of ULF
  functor, implicit already in Content--Lemay--Leroux~\cite{Content-Lemay-Leroux},
  and perhaps made explicit first by Lawvere~\cite{Lawvere:statecats};
  it is equivalent to the notion of discrete Conduch\'e 
  fibration~\cite{Johnstone:Conduche'}. See
  Street~\cite{Street:categorical-structures} for the $2$-categorical notion.
  In the case of the \M categories of Leroux, where there are no
  invertible arrows around, the two notions of ULF coincide.
\end{BM}

\begin{eks}\label{ex:SIOI}
  Here is an example of a functor which is not cULF in Lawvere's sense
  (is not cULF on classical nerves), but which is cULF
  in the homotopical sense.  Namely, let $\kat{OI}$ denote the category of 
  finite ordered sets and monotone injections.  Let $\kat{I}$ denote the category of 
  finite sets and injections.  The forgetful functor $\kat{OI} \to \kat{I}$
  is not cULF in the classical sense, because the identity monotone map
  $\un 2 \to \un 2$ admits a factorisation in $\kat{I}$ that does not lift
  to $\kat{OI}$, namely the factorisation into two nontrivial transpositions.
  However, it is cULF in our sense, as can easily be verified by checking
  that the square
  $$\xymatrix{
     \kat{OI}_1 \ar[d] & \ar[l]\dlpullback\kat{OI}_2 \ar[d] \\
     \kat{I}_1  & \ar[l]\kat{I}_2
  }$$
  is a pullback by computing the fibres of the horizontal maps
  over a given monotone injection.
%
\end{eks}

\begin{lemma}\label{cULF/decomp}
  If $X$ is a decomposition space and $f: Y \to X$ is cULF
  then also $Y$ is a decomposition space.
\end{lemma}

%% file: dec.tex
\def\inputfile{dec.tex}
\subsection{Decalage}

\begin{blanko}{Decalage.}
  (See Illusie~\cite{Illusie1}).
Given a simplicial space $X$ as the top row in 
the following diagram,
the {\em lower dec} $\Dec_\bot(X)$ is a new simplicial space (bottom row in the 
diagram) obtained by
deleting $X_0$ and shifting everything one place down, 
deleting also all $d_0$ face maps and all $s_0$ degeneracy maps.
It comes equipped with a simplicial map $d_\bot:\Dec_\bot(X)\to X$ given by the original $d_0$:

$$
\xymatrix@C+1em{
X_0  
\ar[r]|(0.55){s_0} 
&
\ar[l]<+2mm>^{d_0}\ar[l]<-2mm>_{d_1} 
X_1  
\ar[r]<-2mm>|(0.6){s_0}\ar[r]<+2mm>|(0.6){s_1}  
&
\ar[l]<+4mm>^(0.6){d_0}\ar[l]|(0.6){d_1}\ar[l]<-4mm>_(0.6){d_2}
X_2 
\ar[r]<-4mm>|(0.6){s_0}\ar[r]|(0.6){s_1}\ar[r]<+4mm>|(0.6){s_2}  
&
\ar[l]<+6mm>^(0.6){d_0}\ar[l]<+2mm>|(0.6){d_1}\ar[l]<-2mm>|(0.6){d_2}\ar[l]<-6mm>_(0.6){d_3}
X_3 
\ar@{}|\cdots[r]
&
\\
\\
X_1  \ar[uu]_{d_0}
\ar[r]|(0.55){s_1} 
&
\ar[l]<+2mm>^{d_1}\ar[l]<-2mm>_{d_2} 
X_2  \ar[uu]_{d_0}
\ar[r]<-2mm>|(0.6){s_1}\ar[r]<+2mm>|(0.6){s_2}  
&
\ar[l]<+4mm>^(0.6){d_1}\ar[l]|(0.6){d_2}\ar[l]<-4mm>_(0.6){d_3}
X_3 \ar[uu]_{d_0}
\ar[r]<-4mm>|(0.6){s_1}\ar[r]|(0.6){s_2}\ar[r]<+4mm>|(0.6){s_3}  
&
\ar[l]<+6mm>^(0.6){d_1}\ar[l]<+2mm>|(0.6){d_2}\ar[l]<-2mm>|(0.6){d_3}\ar[l]<-6mm>_(0.6){d_4}
X_4 \ar[uu]_{d_0}
\ar@{}|\cdots[r]
&
}
$$

Similarly, the upper dec, denoted $\Dec_\top(X)$ is obtained by instead 
deleting, in each degree, the last face map $d_\top$ and the last degeneracy map 
$s_\top$.
\end{blanko}

\begin{blanko}{Decalage in terms of an adjunction.} (See Lawvere~\cite{Lawvere:ordinal}.)
  The functor $\Dec_\bot$ can be described more conceptually as follows.
There is an `add-bottom' endofunctor $b:\Delta\to \Delta$, which sends
$[k]$ to $[k+1]$ by adding a new bottom element.  This is in fact a monad;
the unit $\epsilon: \Id \Rightarrow b$ is given by the bottom coface map 
$d^\bot$.  The lower dec is given by precomposition with $b$:
$$
\Dec_\bot(X) = b\upperstar X
$$
Hence $\Dec_\bot$ is a comonad, and its counit is the bottom face map $d_\bot$.

Similarly, the upper dec is obtained from the `add-top' monad on $\Delta$.  In
Section~\ref{sec:master} we shall exploit crucially the combination of the two
comonads.
\end{blanko}

\begin{blanko}{Slice interpretation.}
  If $X$ is the strict nerve of a category $\C$
  then there is a close relationship between the upper dec and the slice 
  construction.  For the strict nerve, $X=N\C$, $\Dec_\top X$ is the disjoint
  union of all (the nerves of) the slice categories of $\C$:
  $$
  \Dec_\top X = \sum_{x\in X_0} N(\C_{/x}).
  $$
  (In general it is a homotopy sum.)
  
  Any individual slice category
  can be extracted from the upper dec, by exploiting that the upper dec
  comes with a canonical augmentation given by (iterating) the bottom face map.
  The slice over an object $x$ is obtained by pulling back the upper dec
  along the name of $x$:
  $$
  \xymatrix{1 \ar[d]_{\name x} & \ar[l] \dlpullback N\C_{/x} \ar[d] \\
  X_0  & \ar[l]^{d_\bot} \Dec_\top X }
  $$
  
  There is a similar relationship between the lower dec and the coslices.
\end{blanko}

\begin{prop}\label{Dec=Segal+cULF}
If $X$ is a decomposition space then 
$\Dec_\top(X)$ and $\Dec_\bot(X)$ 
are Segal spaces, and the maps 
$d_\top : \Dec_\top(X) \to X$ and  $d_\bot : \Dec_\bot(X) \to X$ 
are cULF.
\end{prop}
\begin{proof}
    We put $Y=\Dec_\top(X)$ and check the pullback 
    condition \ref{segalpqr} (3),
        $$
\vcenter{\xymatrix{
   Y_{n+1}\dto_{{d_\top}}\drpullback 
\rto^-{d_\bot}
 &  Y_n\dto^{d_\top} \\
   Y_{n}\rto_-{d_\bot}  &  Y_{n-1}
  .}}
$$
This is the same as
    $$
\vcenter{\xymatrix{
   X_{n+2}\dto_{{d_{\top-1}}}\drpullback 
\rto^-{d_\bot}
 &  X_{n+1}\dto^{d_{\top-1}} \\
   X_{n+1}\rto_-{d_\bot}  &  Y_{n}
  }}
$$
and since now the horizontal face maps that with respect to $Y$ were 
outer face maps, now become inner face maps in $X$, this square is one
of the decomposition square axiom pullbacks.
The cULF conditions say that the various $d_\top$ form pullbacks with all
generic maps in $X$.  But this follows from the decomposition space axiom for 
$X$.
\end{proof}

\begin{thm}\label{thm:decomp-dec-segal}
For a simplicial $\infty$-groupoid $X: \Delta\op\to\Grpd$, the following are equivalent
\begin{enumerate}
\item $X$ is a decomposition space
\item  both $\Dec_\top(X)$ and $\Dec_\bot(X)$ are 
    Segal spaces, and the two comparison maps back to $X$ are ULF and conservative.
\item  both $\Dec_\top(X)$ and $\Dec_\bot(X)$ are 
    Segal spaces, and the two comparison maps back to $X$ are conservative.
\item  both $\Dec_\top(X)$ and $\Dec_\bot(X)$ are Segal spaces, and 
the following squares are pullbacks:
$$  
\xymatrix{
   X_1\rto^{s_1}\drpullback 
\dto_-{d_\bot}
 & X_2\dto^{d_\bot} \\
  X_0  \rto_-{s_0}  &  X_1,
}\qquad
\xymatrix{
   X_1\dto_{d_\top}\drpullback 
\rto^-{s_0}
 &  X_2 \dto^{d_\top} \\
X_0  \rto_-{s_0}  &  X_1.
  }
$$
\end{enumerate}
\end{thm}

\begin{proof} 
The implication (1) $\Rightarrow$ (2) is just the preceding Proposition, and the implications (2) $\Rightarrow$ (3) $\Rightarrow$ (4) are specialisations. The implication (4) $\Rightarrow$ (1) follows from Proposition \ref{onlyfourdiags}.
\end{proof}

\begin{BM}
  Dyckerhoff and Kapranov~\cite{Dyckerhoff-Kapranov:1212.3563} (Theorem~6.3.2)
  obtain the result that a simplicial space is $2$-Segal (i.e.~a decomposition
  space except that there are no conditions imposed on degeneracy maps)
  if and only if both $\Dec$s are Segal spaces.
\end{BM}

\begin{blanko}{Right and left fibrations.}\label{rightfibSegal}
  A functor of Segal spaces $f:Y \to X$ is called a {\em right fibration} if it
  is cartesian on $d_\bot$ and on all generic maps, or a {\em left fibration} if it
  is cartesian on $d_\top$ and on generic maps.  Here the condition on generic
  degeneracy maps is in fact a consequence of that on the face maps.
These notions are most meaningful when
the Segal spaces involved are Rezk complete.
\end{blanko}

\begin{prop}\label{DecULF-is-right}
  If $f:Y \to X$ is a conservative ULF functor between decomposition spaces, then
  $\Dec_\bot(f) : \Dec_\bot(Y) \to \Dec_\bot(X)$ is a right fibration of Segal
  spaces, cf.~\ref{rightfibSegal}.  Similarly, $\Dec_\top(f) : \Dec_\top(Y) \to
  \Dec_\top(X)$ is a left fibration.
\end{prop}
\begin{proof}
  It is clear that if $f$ is cULF then so is $\Dec_\bot(f)$.  The further claim
  is that $\Dec_\bot(f)$ is also cartesian on $d_0$.  But $d_0$ was originally
  a $d_1$, and in particular was generic, hence has cartesian component.
\end{proof}

%% file: monoids-a.tex
\def\inputfile{monoids-a.tex}

\subsection{Monoidal decomposition spaces}
\label{sec:monoids-a}

  The $\infty$-category of decomposition spaces (as a full subcategory of simplicial
  $\infty$-groupoids), has finite products.  Hence there is a symmetric monoidal structure
  on the $\infty$-category $\Decomp^{\culf}$ of decomposition spaces and cULF maps.
  We still denote this product as $\times$, although of course it is not the
  cartesian product in $\Decomp^{\culf}$.

\begin{deff}
  A {\em monoidal decomposition space} is a monoid object $(X,m,e)$ in
  $(\Decomp^{\culf}, \times, 1)$. 
  A {\em monoidal functor} between monoidal decomposition spaces is a monoid
  homomorphism in $(\Decomp^{\culf}, \times, 1)$.
\end{deff}

\begin{blanko}{Remark.}
  By this we mean a monoid in the homotopy sense, that is, 
an algebra in the sense of
  Lurie~\cite{Lurie:HA}.  We do not wish at this point
  to go into the technicalities of this notion, since
  in our examples, the algebra structure will be given
  simply by sums (or products).
\end{blanko}

\begin{eks}\label{extensive}
  Recall that a category $\EE$ with finite sums is {\em extensive}
  \cite{Carboni-Lack-Walters} when the 
  natural functor $\EE_{/A} \times \EE_{/B} \to \EE_{/A+B}$ is an equivalence.
  The fat nerve of an extensive $1$-category is a monoidal decomposition space.
  The multiplication is given by taking sum, the neutral object by the initial 
  object, and the extensive property ensures precisely that, given a factorisation of
  a sum of maps, each of the maps splits into a sum of maps in a unique way.

  A key example is the category of sets, or of finite sets.  Certain subcategories,
  such as the category of finite sets and surjections, or the category of finite 
  sets and injections, inherit the crucial property $\EE_{/A} \times \EE_{/B} 
  \simeq \EE_{/A+B}$. They fail, however, to be extensive in the strict sense, since the 
  monoidal structure $+$ in these cases is not the categorical sum.  Instead
  they are examples of {\em monoidal extensive} categories, meaning a monoidal
  category $(\EE, \boxplus, 0)$ for which $\EE_{/A} \times \EE_{/B} \to 
  \EE_{/A \boxplus B}$ is an equivalence 
  (and it should then be required separately that also
  $\EE_{/0}\simeq 1$).  The fat nerve of a monoidal extensive $1$-category is 
  a monoidal decomposition space.
\end{eks}

\begin{lemma}\label{dec-mon-is-mon}
  The dec of a monoidal decomposition space has again a natural monoidal 
  structure, and the counit is a monoidal functor.
\end{lemma}



%% file: coass.tex
\def\inputfile{coass.tex}

The goal in this section is to define a coalgebra (with $\infty$-groupoid coefficients)
from any decomposition space.  The following brief discussion explains
the origin of the decomposition space axioms.

The incidence coalgebra associated to a decomposition space $X$ will be a
comonoid object in the symmetric monoidal $\infty$-category $\LIN$ (described 
in the Appendix),
and the underlying object is $\Grpd_{/X_1}$.
Since $\Grpd_{/X_1} \tensor \Grpd_{/X_1} = \Grpd_{/X_1\times X_1}$,
and since linear functors are given by spans, to define a comultiplication
functor is to give a span
$$
X_1 \leftarrow M \to X_1 \times X_1 .
$$

For any simplicial space $X$, the span
$$
\xymatrix{
 X_1  & \ar[l]_{d_1}  X_2\ar[r]^-{(d_2,d_0)} &  X_1\times X_1 
}
$$
defines a linear functor, the {\em comultiplication}
\begin{eqnarray*}
  \Delta : \Grpd_{/ X_1} & \longrightarrow & 
  \Grpd_{/( X_1\times X_1)}  \\
  (S\stackrel s\to X_1) & \longmapsto & (d_2,d_0) \lowershriek \circ d_1 \upperstar(s) .
\end{eqnarray*}
The desired coassociativity diagram (which should commute up to equivalence)
$$
\xymatrix@!C=25ex@R-1.5ex{
\Grpd_{/X_1}\dto_\Delta\rto^-\Delta&\Grpd_{/X_1\times X_1}\dto^{\Delta\times \id}\\
\Grpd_{/X_1\times X_1}\rto_-{\id\times\Delta}&\Grpd_{/X_1\times X_1\times X_1}}
$$
is induced by the spans in the outline of this diagram:
$$\xymatrix@C+6ex@R+0ex{
X_1                  & X_2 \ar[l]_{d_1}\ar[r]^{(d_2,d_0)}   & X_1 \times X_1 \\
X_2 \ar[u]^{d_1} \ar[d]_{(d_2,d_0)}&
X_3\dlpullback\urpullback\ar[l]_{d_2}\ar[u]^{d_1}\ar[d]^{(d_2^2,d_0)}\ar[r]_{(d_3,d_0 d_0)}
          & X_2 \times X_1 \ar[u]_{d_1\times \id}\ar[d]^{(d_2,d_0)\times \id}\\
X_1\times X_1 & X_1 \times X_2 \ar[l]^-{\id\times d_1}
\ar[r]_-{\id\times (d_2,d_0)} &X_1 \times X_1 \times X_1
}
$$
Coassociativity will follow from Beck--Chevalley equivalences if the interior 
part of the diagram can be established, with pullbacks as indicated.
Consider the upper right-hand square: it will be a pullback if and only if
its composite with the first projection is a pullback:
$$\xymatrix@C+6ex@R+0ex{
X_2 \ar[r]^-{(d_2,d_0)}   & X_1 \times X_1  \ar[r]^-{\text{pr}_1} & X_1 \\
X_3\urpullback\ar[u]^{d_1}\ar[r]_{(d_3,d_0 d_0)}
  & X_2 \times X_1 \urpullback \ar[u]^{d_1\times \id} \ar[r]_-{\text{pr}_1} & X_2 
  \ar[u]_{d_1}
}
$$
But demanding the outer rectangle to be a pullback
is precisely one of the basic decomposition space axioms.
This argument is the origin of the decomposition space axioms.

Just finding an equivalence is not enough, though.  Higher coherence has to be 
established, which will be accounted for by the full decomposition space axioms.
To establish coassociativity in a strong homotopy sense we must deal on an equal footing with all `reasonable' spans
$$
\prod X_{n_j}\leftarrow \prod X_{m_j}\rightarrow\prod X_{k_i}
$$
which could arise from composites of products of the comultiplication and counit.
We therefore take a more
abstract approach, relying on some more simplicial machinery.  This also leads
to another characterisation of decomposition spaces, and is actually of
independent interest.

%% file: unDelta.tex

\def\inputfile{unDelta.tex}

\subsection{More simplicial preliminaries}

\label{sec:DD}

\begin{blanko}{The category $\un\Delta$ of finite ordinals (the algebraist's Delta).} 
  We denote by $\un \Delta$ the category of all finite ordinals (including the
  empty ordinal) and monotone maps.  Clearly $\Delta \subset \un\Delta$
  (presheaves on $\un\Delta$ are augmented simplicial sets), but this is not the most useful relationship between the two categories.
 We will thus use a different notation for the objects of $\un\Delta$, given  by their cardinality, with an underline:
  $$
  \un n = \{1,2,\ldots,n\} .
  $$
  The category $\un\Delta$ is monoidal under ordinal sum
  $$
  \un m + \un n := \un{m+n} ,
  $$
  with $\un 0$ as the neutral object.
  
  The cofaces $d^i :  \underline{n\!-\!1} \to \underline n$
and codegeneracies $s^i:\underline  {n\!+\!1} \to \underline n$ in $\un\Delta$ 
 are, as usual, the 
injective and surjective monotone maps which skip and repeat the $i$th element, respectively,  but note that now the index is $1 \leq i \leq n$.

%
%
\end{blanko}

\begin{lemma}\label{lem:Delta-duality}
There is a canonical equivalence of monoidal categories (an isomorphism, if
we consider the usual skeleta of these categories)
\begin{eqnarray*}
(\un \Delta, +, \un 0) &\simeq& (\Deltagen\op, \pm, [0]) \\
  \un k & \leftrightarrow & [k]
\end{eqnarray*}
\end{lemma}

\begin{proof}
  The map from left to right sends $\un k \in \un\Delta$ to 
  $$
  \Hom_{\un\Delta}(\un k, \un 2) \simeq [k] \in \Deltagen\op .
  $$
  The map in the other direction sends 
  $[k]$ to the ordinal
  $$
  \Hom_{\Deltagen}([k],[1]) \simeq \un k .
  $$
  In both cases, functoriality is given by precomposition.
\end{proof}

In both categories we can picture the objects as a line with some dots.
The dots then represent the elements in $\un k$, while the edges represent
the elements in $[k]$;  a map operates on the dots when considered a
map in $\un \Delta$ while it operates on the edges when considered a map in 
$\Deltagen$.
Here is a picture of a certain map $\un 5 \to \un 4$ in
$\un\Delta$ and of the corresponding map $[5] \leftarrow [4]$ in
$\Deltagen$.
\begin{center}
\begin{texdraw}
  \arrowheadtype t:V
  
    \arrowheadsize l:6 w:3  

  \move (0 0) 
  \bsegment
  \move (0 -5)
  \lvec (0 85)
  \move (0 8) \Onedot 
  \move (0 24) \Onedot 
  \move (0 40) \Onedot 
  \move (0 56) \Onedot 
  \move (0 72) \Onedot 
  \esegment
  
  \move (30 0)
    \bsegment
    \move (0 3)
  \lvec (0 77)
  \move (0 16) \Onedot 
  \move (0 32) \Onedot 
  \move (0 48) \Onedot 
  \move (0 64) \Onedot 
  \esegment

  \move   (27 8) \avec (3 0)
  \move  (27 72) \avec (3 80) 
    \move   (27 55) \avec (3 49)
    \move   (27 41) \avec (3 47)
    \move   (27 23.7) \avec (3 16.3)

        \arrowheadsize l:4 w:3  

    \linewd 0.9
    \move (0 8) \avec (15 12)\lvec (30 16)
    \move (0 24)\avec (15 28)\lvec (30 32)
    \move (0 40)\avec (15 36)\lvec (30 32)
    \move (0 56)\avec (15 60)\lvec (30 64)
    \move (0 72)\avec (15 68)\lvec (30 64)
\end{texdraw}
\end{center}

%% file: prop-proof.tex

\def\inputfile{prop-proof.tex}

\begin{blanko}{A twisted arrow category of $\un\Delta$.}
  Consider the category  $\DD$ whose objects are the arrows $\un n \to \un k$ of $\un \Delta$ 
and whose morphisms $(g,f)$ from $a:\un m \to \un h$ to $b:\un n \to \un k$ are 
  commutative squares
      \begin{align}\label{mor-DD}
\vcenter{\xymatrix{
  \un m \ar[d]_a \ar[r]^g \ar@{}[rd]|{(g,f)}& \un n \ar[d]^b \\
\un h &  \ar[l]^f \un k .
  }}\end{align}
That is, $\DD\op$ is the twisted arrow category~\cite{MacLane:categories,BaWi:1985}
 of $\un\Delta$.


There is a canonical factorisation system on $\DD$: any morphism \eqref{mor-DD}
factors uniquely as
      $$\xymatrix@C=3.2pc{
  \un m \ar[d]_{a=fbg} \ar[r]^= \ar@{}[rd]|\varphi
&\un m \ar[d]|{bg} \ar[r]^g \ar@{}[rd]|\gamma
& \un n \ar[d]^b \\
\un h &  \ar[l]^f \un k&  \ar[l]^= \un k
  }$$
The maps $\varphi=(\id,f):fb\to b$ in the left-hand class of the factorisation system are termed {\em segalic},
 \begin{align}\label{cov-DD}\vcenter{\xymatrix{
  \un m \ar[r]^=  \ar[d]_{fb}\ar@{}[rd]|\varphi
& \un m \ar[d]^b \\
  \un h & \ar[l]^f \un k.}}
\end{align}
The maps $\gamma=(g,\id):bg\to b$ in the right-hand class are termed {\em ordinalic} and may be identified with maps in the slice categories $\un \Delta_{/\un h}$
\begin{align}\label{gen-DD}\vcenter{\xymatrix{
  \un m \ar[r]^g  \ar[d]_{bg}\ar@{}[rd]|\gamma
& \un n \ar[d]^{b} \\
  \un h & \ar[l]^= \un h.}}
\end{align}
Observe that $\un\Delta$ is isomorphic to the subcategory of objects with target $\un h=\un 1$, termed the {\em connected objects} of $\DD$, 
\begin{align}\label{unD-in-DD}
\un \Delta\xrightarrow{\;\;=\;\;}\un \Delta_{/\un1}\xrightarrow{\;\;\subseteq\;\;} \DD.\end{align} 

  The ordinal sum operation in $\un\Delta$ induces a monoidal 
  operation in $\DD$:
 the {\em external sum}  $(\un n\to \un k)\oplus(\un n'\to \un k')$ 
of objects in $\DD$ is their ordinal sum $\un n+\un n'\to \un k+\un k'$ as morphisms in $\un\Delta$.  The neutral object is $\un 0 
  \to \un 0$. 
The inclusion functor \eqref{unD-in-DD} is not monoidal, but it is easily seen to be oplax 
  monoidal by means of the codiagonal map $\un1+\un1\to \un1$.

Each object $\un m\xrightarrow{\;a\;}\un k$ of $\DD$ is an external sum of connected objects,
\begin{align}\label{unD-DD}
a\;\;=\;\;a_1\oplus a_2\oplus \dots \oplus a_k\;\;=\;\;\bigoplus_{i\in \un k}
\left(\un m_i\xrightarrow{\;a_i\;}\un1\right),
\end{align}
where $\un m_i$ is (the cardinality of) the fibre of $a$ over $i\in\un k$. 

Any segalic map \eqref{cov-DD} and any ordinalic map \eqref{gen-DD} in $\DD$ may be written uniquely as external sums 
\begin{align}\label{cov-sum}
 \varphi &\;\;= \;\;\varphi_1 \oplus \varphi_2\oplus
\dots \oplus \varphi_h \;\;=\;\;\bigoplus_{j\in \un h}
\left(\vcenter{\xymatrix@R1.4pc{
\un m_j\rto^=\dto
\ar@{}[rd]|{\varphi_j}
&\dto^{b_j} \un m_j\\  \un 1 &\lto \un k_j }}\right)
\\\label{gen-sum}
 \gamma &\;\;= \;\;\gamma_1 \oplus \gamma_2\oplus 
\dots \oplus \gamma_h \;\;=\;\;\bigoplus_{j\in \un h}\left(\un m_j\xrightarrow{\;\gamma_j\;}\un n_j\right)
\end{align}
where each $\gamma_j$ is a map in  $\un\Delta_{/\un1}=\un\Delta$.

\end{blanko}

In fact $\DD$ is a universal monoidal category in the following sense.
\begin{prop}
For any cartesian category $(\CC,\times,1)$, there is an equivalence
$$
\Fun(\Delta\op,\CC)\;\simeq \;\Fun^\otimes((\DD,\oplus,0),(\CC,\times,1))
$$
between the categories of simplicial objects $X$ in $\CC$ and of
monoidal functors $\overline X:\DD\to\CC$.
The correspondence between $X$ and $\ov X$ is determined by following properties.

(a) The functors $X:\Delta\op\to\CC$ and $\ov X:\DD\to\CC$ agree on the common subcategory $\Deltagen\op\cong\un\Delta$,
$$
\xymatrixrowsep{8pt}
\xymatrix@C+1pc{
\ar[dd]_{\cong}
\Deltagen\op\,\ar@{^(->}[r]&\Delta\op\ar[rd]^X\\
&&\CC .\\
\un\Delta\;\ar@{^(->}[r]
&\DD\ar[ru]_{\overline X}& }
$$

(b) Let $(\un m\stackrel a\to \un k)=\bigoplus_i (\un m_i\stackrel a\to \un 1)$ be the external sum decomposition \eqref{unD-DD} of any object of $\DD$, and denote by
$f_i:[m_i]\rat[m_1]\pm\dots\pm[m_k]=[m]$ the canonical free map in $\Delta$, for $i\in\un k$. Then
$$\ov X\!\left(\!\vcenter{\xymatrix@R1.1pc@C1.6pc{\un m\rto^=\dto
\ar@{}[rd]|\varphi
&\dto^a \un m\\  \un 1 &\lto \un k }}
\right)\;=\;(X(f_1),\dots,X(f_k))\;:\;
X_m\longrightarrow\prod_{i\in\un k} X_{m_i}
$$
and each $X(f_i)$ is the composite of $\ov X(\varphi)$ with the projection to $X_i$.
\end{prop}
\begin{proof}
Given $\ov X$, property (a) says that there is a unique way to define 
$X$ on objects and generic maps.
Conversely, given $X$, 
then for any object $a:\un m\to\un k$ in $\DD$ we have  
$$
\overline X_a\;\;=\;\;
\prod_{i\in\un k}
\overline X_{a_i}\;\;=\;\;\prod_{i\in\un k}X_{m_i}
$$
using \eqref{unD-DD}, and for any ordinalic map $\gamma$ we have
$$
\overline X(\gamma)\;\;=\;\;
\prod_{i\in\un k}
\overline X(\gamma_i)\;\;=\;\;\prod_{i\in\un k}X(g_i)
$$
using \eqref{gen-sum}, where $g_i\in\Deltagen\op$ corresponds to $\gamma_i\in\un\Delta$. 

Thus we have a bijection between functors $X$ defined on $\Deltagen\op$ and monoidal functors $\ov X$ defined on the ordinalic subcategory of $\DD$.
Now we consider the free and segalic maps. Given $\ov X$, 
property (b) says that for any free map $f_r:[m_r]\to[m]$ we may define 
$$X(f_r)=
\left(X_m\xrightarrow{\overline X(\varphi)}\prod_{i\in\un k} X_{m_i}\onto X_{m_r}\right)$$
We may
assume $k=3$:
given the factorisation
$$
\varphi=\left(\vcenter{\xymatrix@R1.5pc@C1.6pc{
\un m\rto^-=
\dto\ar@{}[rd]|{\varphi_2}
&\dto \un m_{<r}+\un m_r+\un m_{>r}\rto^-=
\dto\ar@{}[rd]|{\varphi_1\oplus\id\oplus \varphi_3}
&\dto\sum_{i\in \un k} m_i
\\  \un 1 &\lto \un 3&\lto\un k }}\right)
$$
one sees the value $X(f_r)$ is well defined from the following diagram 
$$
\xymatrix@C=6pc{X_m\rto^-{\overline X(\varphi_2)}\ar@/_0.9pc/[rrd]_(0.3){X(f_r)}
  &X_{m_{<r}}\times X_{m_r}\times X_{m_{>r}}
   \rto^-{\overline X(\varphi_1)\times\id\times\overline X(\varphi_3)}
   \ar@{->>}[dr]
       &\prod_{i\in\un k} X_{m_i}
        \ar@{->>}[d]
\\ &   &X_{m_r}\;.
}
$$
Functoriality of $X$ on a composite of free maps, say $[m_3]\rat[\sum_2^4m_i]\rat[\sum_1^5m_i]$,
now follows from the diagram
$$
  \xymatrixrowsep{16pt}
  \xymatrixcolsep{16pt}
\xymatrix
{
X_{\sum_1^5m_i}\drto\rrto &&\ar@{->>}[dr] \prod_1^5X_{m_i}\ar@{->>}[rr]&&X_{m_3}\\
&\!\!\!X_{m_1}\times X_{\sum_2^4m_i}\times X_{m_5}\!\!\!
\ar[ur]\ar@{->>}[dr]&&\prod_2^4X_{m_i}\ar@{->>}[ur]\\
&&X_{\sum_2^4m_i}\urto
}
$$
in which the first triangle commutes by functoriality of $\ov X$.

Conversely, given $X$, then property (b) says how to define $\ov X$ on segalic maps with connected domain and hence, by \eqref{cov-sum}, on all segalic maps.
Functoriality of $\ov X$ on a composite of segalic maps, say $(\id,\un 1\leftarrow\un h\leftarrow\un k)$, follows from functoriality of $X$:
$$
\xymatrix@C=9.5pc{
X_m
\ar@/_2.5pc/[rr]_-{(X([m_i]\rat[m]))_{i\in\un k}}
\rto^-{(X([m_j]\rat[m]))_{j\in\un h}}
&\displaystyle\prod_{j\in \un h}X_{m_j}
\rto^-{\prod_{j\in\un h}(X([m_i]\rat[m_j]))_{i\in\un k_j}}
&\displaystyle\prod_{j\in \un h}\prod_{i\in \un k_j}X_{m_i}
}
$$ 
It remains only to check that the construction of $\overline X$ from $X$ (and of $X$ from $\overline X$) is well defined on composites of ordinalic followed by segalic (free followed by generic) maps. One then has the mutually inverse equivalences required.  
Consider the factorisations in $\DD$,
$$
\vcenter{\xymatrix{
\un m \ar[d] \ar[r]^= \ar@{}[rd]|{\varphi}&\un m \ar[d] \ar[r]^g \ar@{}[rd]|{\gamma}& \un m' \ar[d] \\
\un 1 &  \ar[l] \un k&  \ar[l]^= \un k
  }}
\quad=\quad
\vcenter{\xymatrix{
\un m \ar[d] \ar[r]^g \ar@{}[rd]|{\gamma'}&\un m' \ar[d] \ar[r]^=  \ar@{}[rd]|{\varphi'}&\un m' \ar[d] \\
\un 1 &  \ar[l]^= \un 1&  \ar[l] \un k
  .}}
$$
To show that $\overline X$ is well defined,
we must show that the diagrams
$$
\xymatrix@C=5pc{
X_m\rrto^-{\ov X(\varphi)=(X(f_1),\dots,X(f_k))}\dto_{\ov X(\gamma')=X(\tilde g)}&& \prod
X_{m_i}\dto_{\ov X(\gamma')=\prod X(\tilde g_i)}\ar@{->>}[r]&X_{m_r}
\dto^{X(\tilde g_r)}\\
X_{m'}\rrto_-{\ov X(\varphi')=(X(f'_1),\dots,X(f'_k))}&&\prod
X_{m'_i}\ar@{->>}[r]&X_{m'_r},
}
$$
commute for each $r$, where $\tilde g$, $\tilde g_i$ in $\Deltagen$ correspond to $g$, $g_i$ in $\un\Delta$. This follows by functoriality of $X$, since $\tilde g$ restricted to $m'_r$ is the corestriction of $\tilde g_r$.
Finally we observe that this diagram, with $k=3$ and $r=2$, also serves to show that the construction of $X$ from $\ov X$ is well defined on
$$
\xymatrix@C5pc{
{[m_1+m_2+m_3]} &
 \ar@{ >->}[l]_-{f_2} {[m_2]}\\
{[m_1'+m_2'+m_3']} \ar@{>|}[u]^{\tilde g}&\ar@{>|}[u]_{\tilde g_2}
\ar@{ >->}[l]^-{f'_2}{m_2'}}
$$
\end{proof}

\begin{lemma}\label{lemma:DD-pbk}
  In the category $\DD$, ordinalic and segalic maps admit pullback 
  along each other, and the result is again maps of the same type.
\end{lemma}

\begin{proof}
  This is straightforward: in the diagram below, the map from $a$ to $b$
  is segalic (given essentially by the bottom map $f$) and the map from
  $a'$ to $b$ is ordinalic (given essentially by the top map $g'$):
  \begin{equation} \label{eq:pbk-cube}
  \xymatrix@C3.9pc@R1.35pc{
  && \ar@{-->}[lldd]_{g'} m' \ar@{-->}[d] \ar@{-->}[rrdd]^*-[@]=0+!D{=} && \\
  && h && \\
  m \ar[d]_a \ar[rrdd]^*-[@]=0+!D{=} &&&& m' \ar[lldd]_{g'} \ar[d]^{a'} \\
  h \ar@{-->}[rruu]^*-[@]=0+!D{=} &&&& k \ar@{-->}[lluu]_f \\
  &&m \ar[d]_b && \\
  && \ar[lluu]^f k \ar[rruu]_*-[@]=0+!U{=} &&
  }\end{equation}
  To construct the pullback, we are forced to repeat $f$ and $g'$, completing 
  the squares with 
  the corresponding identity maps.
  The connecting map in the resulting object is 
$fbg': 
m' \to h$.
  It is clear from the presence of the four identity maps that this is a 
  pullback.
\end{proof}

\bigskip

%
%

We now have the following important characterisation of decomposition
spaces.
\begin{prop}\label{prop:DDecomp}
  A simplicial space $X: \Delta\op\to\Grpd$ is a decomposition space if and only if
  the corresponding monoidal functor $\overline X : \DD \to \Grpd$ preserves pullback squares of the kind 
  described in \ref{lemma:DD-pbk}.
\end{prop}
\begin{proof}
  Since an ordinalic map is a sum, it can be decomposed into a sequence of maps 
  in
  which each map has only one nontrivial summand.  This means that a pullback 
  diagram like \eqref{eq:pbk-cube} is a sum of diagrams of the form in which $\un h = 
  \un 1$.  So to prove that these pullbacks are preserved, it is enough to 
  treat the case $h=1$.  In this case, the map $g'$ in the square
  is just a map in $\un \Delta$, so it can be decomposed into face and 
  degeneracy maps.
  The $X$-image is then
  a diagram of the form
  $$
  \xymatrix{
  X_m \ar[r] \ar[d] & X_{m_1} \times \cdots \times X_{m_k} \ar[d] \\
  X_n \ar[r] & X_{n_1} \times \cdots \times X_{n_k},
  }$$
  where the map on the left is a face map or a degeneracy map. It follows that
  the map on the right is a product of 
  maps in which all factors are identity maps except one, say the $i$th factor
  (which is again a face 
  or a degeneracy map).
  Now whether or not this is a pullback can be checked on the projections onto
  the nontrivial factor:
    $$
  \xymatrix{
  X_m \ar[r] \ar[d] & X_{m_1} \times \cdots \times X_{m_k}
  \ar[d] \ar[r] & 
  X_{m_i} \ar[d]\\
  X_n \ar[r] & X_{n_1} \times \cdots \times X_{n_k} 
  \ar[r] & X_{n_i}
  }$$
  But by construction of $\overline X$, 
  the composite horizontal maps are precisely free maps in the sense of the
  simplicial space $X$, and the vertical maps are precisely generic maps 
  in the sense that it is an arbitrary map in $\un\Delta$ and hence (in the 
  other direction) a generic map in $\Delta$, under the duality in \ref{lem:Delta-duality}.
  Since the right-hand square is always a pullback, by the standard pullback
  argument~\ref{pbk}, the total square is a pullback (i.e.~we have a decomposition space)
  if and only if the left-hand square is a pullback (i.e.~the pullback condition on 
  $\overline X$ is satisfied).
\end{proof}



%% file: coalg.tex
\def\inputfile{coalg.tex}

\subsection{Incidence coalgebras}

\begin{blanko}{Comultiplication and counit.}
  For any decomposition space $ X$, the span
  $$
  \xymatrix{
   X_1  & \ar[l]_{m_ X}^{d_1}  X_2\ar[r]^-{p_ X}_-{(d_2,d_0)} &  X_1\times X_1 
  }
  $$
  defines a linear functor, the {\em comultiplication}
  \begin{eqnarray*}
    \Delta : \Grpd_{/ X_1} & \longrightarrow & 
    \Grpd_{/( X_1\times X_1)}  \\
    (S\stackrel s\to X_1) & \longmapsto & p_ X{} \lowershriek \circ m_ X \upperstar(s) .
  \end{eqnarray*}
Likewise, the span
  $$
  \xymatrix{
  X_1  & \ar[l]_{u_ X}^{s_0}  X_0\ar[r]^{t_ X} &  1 
  }
  $$
  defines a linear functor, the {\em counit}
  \begin{eqnarray*}
    \epsilon : \Grpd_{/ X_1} & \longrightarrow & 
    \Grpd  \\
        (S\stackrel s\to X_1) & \longmapsto & t_ X {}\lowershriek \circ u_ X \upperstar(s) .
  \end{eqnarray*}
\end{blanko}

We proceed to establish that this makes $\Grpd_{/X_1}$ a coassociative and 
counital coalgebra in a strong homotopy sense.
We have more generally, for any $n\geq 0$,
the generalised comultiplication maps
\begin{eqnarray}\label{Deltan}
\Delta_n:
\Grpd_{/X_1} & \longrightarrow & \Grpd_{/X_1\times\dots\times X_1}
\end{eqnarray}
given by the spans
\begin{align}\label{Deltanspan}
X_1 \leftarrow X_n \to X_1 \times\dots\times X_1.
\end{align}
The case $n=0$ is the counit map, and $n=1$ gives the identity, and $n=2$ is the
comultiplication we considered above.
The coassociativity will say that all combinations (composites and tensor products) of these agree whenever they have the same source and target.  For this we exploit the category $\DD$ introduced in 
\ref{sec:DD}, designed exactly to encode also cartesian powers of the various spaces
$X_k$. 

\begin{deff}
  A {\em reasonable span} in $\DD$ is a span $a \stackrel g \leftarrow m 
  \stackrel f\to b$ in 
  which $g$ is ordinalic and $f$ is segalic.  Clearly the external sum
  of two reasonable spans is reasonable, and the composite of two 
  reasonable spans is reasonable (by Lemma~\ref{lemma:DD-pbk}).
  
  Let $X:\Delta\op\to\Grpd$ be a fixed decomposition space,
  and interpret it also as a monoidal functor $\ov X:\DD\to\Grpd$.
  A span in $\Grpd$ of the form
$$  \ov X_{a} 	\leftarrow \ov X_{m} \to \ov X_{b}
$$
  is called reasonable if it is induced by a reasonable span in $\DD$.
  
  A linear map between slices of $\Grpd$ is called reasonable if it is
  given by a reasonable span.  That is, it is
  a pullback along a ordinalic map followed by a lowershriek
  along a segalic map.
  
\end{deff}

\begin{lemma} 
  Tensor products of reasonable linear maps are reasonable.  For a decomposition space,
  composites of reasonable linear maps are reasonable.
\end{lemma}
\begin{proof}
Cartesian products of reasonable spans in $\Grpd$ are reasonable since $\ov X$ is monoidal. For decomposition spaces, a composite of reasonable linear maps is induced by the composite reasonable span in $\DD$, using Proposition \ref{prop:DDecomp}.
\end{proof}

  The interest in these notions is of course that the generalised
  comultiplication maps $\Delta_n$ are reasonable, see (\ref{Deltan},\ref{Deltanspan}) above. 
In conclusion:
\begin{prop}
  Any reasonable linear map $$\Grpd_{/X_1} \longrightarrow \Grpd_{/X_1\times\dots\times
  X_1}, \quad n\geq0$$ is canonically equivalent to the $n$th comultiplication map.
\end{prop}

\begin{proof}
    We have to show that the only reasonable span of the form
    $X_1 \leftarrow  \prod X_{m_i}\to X_1 \times\dots\times X_1$
    is \eqref{Deltanspan}.
Indeed, the left leg must come from an ordinalic map, so since 
$X_{1}$ has only one factor, the middle object has also only one 
factor, i.e.\ is the image of $\un m \to \un 1$.  On the other hand,
the right leg must be segalic, which forces $m=n$.
\end{proof}


\begin{theorem}\label{thm:comultcoass}
  For $X$ a decomposition space, the slice $\infty$-category $\Grpd_{/X_1}$ has
  the structure of strong homotopy comonoid in the symmetric monoidal
  $\infty$-category $\LIN$, with the comultiplication defined by the span
  $$
  X_1 \stackrel{d_1}\longleftarrow X_2 \stackrel{(d_2,d_0)}
  \longrightarrow X_1 \times X_1 . 
  $$
\end{theorem}

\begin{blanko}{Covariant functoriality.}
  An important motivation for the notion of decomposition space is that they
  induce coalgebras.  Correspondingly, it is an important feature of
  cULF maps that they induce coalgebra homomorphisms:
\end{blanko}

\begin{lemma}\label{lem:coalg-homo}
  If $F:X \to Y$ is a conservative ULF map  between decomposition spaces
  then $F\lowershriek : \Grpd_{/X_1} \to \Grpd_{/Y_1}$ is a coalgebra 
  homomorphism.
\end{lemma}

\begin{proof}
In the diagram
$$
\xymatrix{
   X_1  \ar[d]_{F_1}& \ar[l]_{g} \dlpullback X_n\ar[r]^-{f} \ar[d]^{F_n}
   &
   X_1^n \ar[d]^{F_1^n}
   \\
   Y_1  & \ar[l]^{g'} Y_n\ar[r]_-{f'} & Y_1^n
   }
   $$
   the left-hand square is a pullback since $F$ is conservative (case $n=0$) and ULF 
   (cases $n>1$).  Hence by the Beck--Chevalley condition we have an equivalence 
   of functors $g'{}\upperstar \circ F_1{}\lowershriek \simeq
   F_n{}\lowershriek\circ g{}\upperstar$, and by postcomposing with 
   $f'\lowershriek$ we arrive at the coalgebra homomorphism condition
$
\Delta'_nF_1{}\lowershriek
\cong
F_1{}\lowershriek^n\Delta_n
$ 
\end{proof}

\begin{BM}
  If $Y$ is a Segal space, then the statement can be improved to an if-and-only-if
  statement.
\end{BM}

\begin{blanko}{Example.}
  An important class of cULF maps are counits of decalage, 
  cf.~\ref{Dec=Segal+cULF}:
  $$
  d_\bot : \Dec_\bot X \to X \qquad \text{ and } \qquad d_\top : 
  \Dec_\top X \to X .
  $$
  We shall see that many coalgebra maps in the classical theory of incidence
  coalgebras, notably reduction maps, are induced from decalage in this way
  (see examples \ref{ex:N&L}, 
\ref{I=DecB}, 
\ref{ex:D&M}, 
\ref{ex:q},
\ref{ex:P=Dec(S)}, 
  \ref{ex:CK} below).
%
%
%
\end{blanko}

%
%
%



\begin{blanko}{Contravariant functoriality.}
  There is also a contravariant
  functoriality for certain simplicial maps, which we briefly explain,
  although it will not be needed elsewhere in this paper.
  
  A functor between decomposition spaces $F: X \to Y$ is called {\em relatively
  Segal} when for all `spines' (i.e.~inclusion of a string of principal edges
  into a simplex)
  $$
  \Delta^1 \underset{\Delta^0}\coprod \dots \underset{\Delta^0}\coprod \Delta^1
  \longrightarrow
  \Delta^n
  $$
  the space of fillers in the diagram
  $$\xymatrix{
    \Delta^1 \underset{\Delta^0}\coprod \dots \underset{\Delta^0}\coprod \Delta^1
    \ar[r]\ar[d] & X \ar[d] \\
     \Delta^n \ar[r]\ar@{..>}[ru] & Y
  }$$
  is contractible.
  Note that the precise condition is that the following square is a pullback:
  $$\xymatrix{
     \Map(\Delta^n, X)\drpullback \ar[r]\ar[d] & \Map(\Delta^1 \underset{\Delta^0}\coprod \dots \underset{\Delta^0}\coprod \Delta^1,X) \ar[d] \\
     \Map(\Delta^n,Y) \ar[r] & \Map(\Delta^1 \underset{\Delta^0}\coprod \dots \underset{\Delta^0}\coprod \Delta^1,Y)
  }$$
  This can be rewritten
  \begin{equation}\label{eq:relSegal}
  \xymatrix{
     X_n\drpullback \ar[r]\ar[d] & X_1 \times_{X_0} \cdots \times_{X_0} X_1 \ar[d] \\
     Y_n \ar[r] & Y_1 \times_{Y_0} \cdots \times_{Y_0} Y_1 .
  }    
  \end{equation}
  (Hence the ordinary Segal condition for a simplicial space $X$ is the case
  where $Y$ is a point.)

\end{blanko}

\begin{prop}
  If $F:X \to Y$ is relatively Segal and $F_0: X_0 \to Y_0$
  is an equivalence, then
  $$
  F\upperstar : \Grpd_{/Y_1} \to \Grpd_{/X_1}
  $$
  is naturally a coalgebra homomorphism.
\end{prop}

\begin{proof}
  In the diagram
  $$
  \xymatrix{
  X_1  \ar[d]_{F_1}& \ar[l]_{g} \drpullback X_n\ar[r]^-{f} \ar[d]_{F_n}
  &
  X_1^n \ar[d]^{F_1^n}
  \\
  Y_1  & \ar[l]^{g'} Y_n\ar[r]_-{f'} & Y_1^n
  }
  $$
  we claim that the right-hand square is a pullback for all $n$.
  Hence by the Beck--Chevalley condition we have an equivalence 
  of functors $f\lowershriek \circ F_n{}\upperstar \simeq
  F_1^n{}\upperstar \circ f'{}\lowershriek$, and by postcomposing with 
  $g'{}\upperstar $ we arrive at the coalgebra homomorphism condition
  $$
  \Delta_n F_1{}\upperstar 
  \cong
  F_1{}\upperstar {}^n\Delta'_n .
  $$
  The claim for $n=0$ amounts to 
    $$
  \xymatrix{
   \drpullback X_0\ar[r]^-{f} \ar[d]_{F_0}
  &
  1 \ar[d]
  \\
   Y_0\ar[r]_-{f'} & 1
  }
  $$
  which is precisely to say that $F_0$ is an equivalence.
  For $n>1$ we can factor the square as
      $$
  \xymatrix{
   \drpullback X_n\ar[r]^-{f} \ar[d]_{F_n}
   & X_1\times_{X_0} \cdots \times_{X_0} X_1 \ar[d]^{F_1^n} \ar[r]
   & X_1 \times \cdots \times X_1 \ar[d]^{F_1^n}
   \\
   Y_n\ar[r]_-{f'} 
   & Y_1\times_{Y_0} \cdots \times_{Y_0} Y_1 \ar[r]
   & Y_1 \times \cdots \times Y_1
  }
  $$
  Here the left-hand square is a pullback since $F$ is relatively Segal.
  It remains to prove that the right-hand square is a pullback.
  For the case $n=2$, this whole square is the pullback of the square
  $$\xymatrix{
     X_0\drpullback \ar[r]\ar[d] & X_0 \times X_0 \ar[d] \\
     Y_0 \ar[r] & Y_0 \times Y_0
  }$$
  which is a pullback precisely when $F_0$ is mono.  But we have
  assumed it is even an equivalence.
  The general case $n>2$ is easily obtained from the $n\!=\!2$ case by an iterative
  argument.
\end{proof}

\begin{blanko}{Remarks.}
  It should be mentioned that in order for contravariant functoriality to preserve
  finiteness, and hence restrict to the
  coefficients in $\grpd$, it is necessary
  furthermore to require that $F$ is finite.
  
  When both $X$ and $Y$ are Segal
  spaces, then the relative Segal condition is automatically satisfied, because
  the horizontal maps in \eqref{eq:relSegal} are then equivalences.
  In this case, we recover the classical results on contravariant functoriality
  by
  Content--Lemay--Leroux~\cite[Prop.~5.6]{Content-Lemay-Leroux} and
  Leinster~\cite{Leinster:1201.0413}, where the only condition is that the
  functor be bijective on objects (in addition to requiring $F$ finite,
  necessary since they work on the level of vector spaces).
\end{blanko}

%% file: monoids-b.tex
\def\inputfile{monoids-b.tex}

\begin{blanko}{Bialgebras.}\label{sec:monoids}
  For a monoidal decomposition space as in \ref{sec:monoids-a} the resulting coalgebra is also a
  bialgebra.  Indeed, the fact that the monoid multiplication is cULF means that
  it induces a coalgebra homomorphism, and similarly with the unit.  Note that
  this notion of bialgebra is not symmetric: while the comultiplication is
  induced from internal, simplicial data in $X$, the multiplication is induced
  by extra structure (the monoid structure).  In the applications, the monoid
  structure will typically be given by categorical sum, and hence is associative
  up to canonical isomorphisms, something that seems much stricter than the
  comultiplication.
\end{blanko}


\begin{prop}\label{bialg-hm}
  If $f:X\to Y$ is a cULF monoidal functor between monoidal decomposition 
  spaces,
  then $f\lowershriek : \Grpd_{/X_1} \to \Grpd_{/Y_1}$ is a bialgebra homomorphism.
\end{prop}

%% file: complete.tex

\def\inputfile{complete.tex}

\section{Complete decomposition spaces}

\label{sec:complete}

\subsection{The completeness condition}

\begin{blanko}{Complete decomposition spaces.}\label{complete}
  A decomposition space $X$ is called {\em complete} if $s_0:X_0 \to X_1$ is a
  monomorphism.
\end{blanko}

\begin{blanko}{Discussion.}\label{complete-discussion}
  It is clear that a Rezk complete Segal space is complete in the sense of
  \ref{complete}.  It makes sense also to state the Rezk completeness condition
  for decomposition spaces.  We prefer the cheaper condition \ref{complete} for
  two reasons: first of all it is sufficient for its purpose, namely to ensure a
  well-behaved notion of nondegenerate simplices.  Second, it covers some
  important examples which are not Rezk complete.  In particular, the classical
  nerve of a group is a complete decomposition space in the sense of
  \ref{complete}, but is clearly not Rezk complete.  The incidence algebra of
  the classical nerve of a group is the group algebra, certainly an example
  worth covering.

  The motivating feature of the notion of complete decomposition space is that
  all issues about
  degeneracy can be settled in terms of the canonical projection maps $X_n \to
  (X_1)^n$ sending a simplex to its principal edges: a simplex is nondegenerate
  precisely when all its principal edges are nondegenerate.
  We shall see that if a \FILT decomposition space is
  a Segal space then it is also Rezk complete (\ref{prop:FILTSegal=Rezk}).

  The completeness condition is necessary to define the Phi functors (the odd
  and even parts of the `\M functor', see \ref{Phi}) and to establish the \M
  inversion principle at the objective level (\ref{thm:zetaPhi}).  The completeness condition is
  also needed to make sense of the notion of length (\ref{length}), and
  to define the length filtration (\ref{def:filt}), which is of independent interest,
  and is also required to be able to take cardinality of \M inversion.
  
\end{blanko}

The following basic result follows immediately from Lemma~\ref{lem:s0d1}.
\begin{lemma}\label{all-s-mono}
  In a complete decomposition space, 
  all degeneracy maps are monomorphisms.
\end{lemma}

\begin{blanko}{Completeness for simplicial spaces.}\label{completesimplicial}
  We shall briefly need completeness also for general simplicial spaces,
and the first batch of results hold in this generality.  We shall say that a
simplicial space $X:\Delta\op\to\Grpd$ is complete if all degeneracy maps are
monomorphisms.  In view of Lemma~\ref{all-s-mono},
this agrees with the previous definition when $X$ is a decomposition space.
\end{blanko}

\begin{blanko}{Word notation.}\label{w}
  Let $X$ be a complete simplicial space.  Since $s_0:X_0\to X_1$ is mono,
  we can identify $X_0$ with a full subgroupoid of $X_1$.  We
  denote by $X_a$ its  complement, the full
  subgroupoid of {\em nondegenerate $1$-simplices}: 
 $$
 X_1 = X_0 + X_a .
 $$
 We extend this notation as follows.
  Consider the alphabet with three letters $\{0,1,a\}$.
  Here $0$ is to indicate degenerate edges $s_0(x)\in X_1$, the letter 
  $a$ denotes
  the edges specified to be nondegenerate, and $1$
  denotes the edges which are not specified to be
  degenerate or nondegenerate.
%
%
For $w$ a word in this alphabet $\{0,1,a\}$, of length $|w|=n$,
  put
  $$
  X^w := \prod_{i\in w} X_i \subset (X_1)^n .
  $$
  This inclusion is full since $X_a \subset X_1$ is full by completeness.
  Denote by
  $X_w$ the $\infty$-groupoid of $n$-simplices whose
  principal edges have the types indicated in the word $w$, 
  or more explicitly, the full subgroupoid of $X_n$ given by the pullback diagram
  \begin{equation}\label{eq:Xw}
\vcenter{  \xymatrix{
  X_w \drpullback \ar[r] \ar[d] & X_n \ar[d] \\
  X^w  \ar[r] &(X_1)^n.
  }}
  \end{equation}
\end{blanko}

\begin{lemma}\label{lem:cons-X1}
  If $X$ and $Y$ are complete simplicial spaces and $f:Y \to X$ is conservative,
  then $Y_a$ maps to $X_a$, and the following square is a
  pullback:
  $$
  \xymatrix{
  Y_1 \ar[d] & Y_a \dlpullback \ar[l] \ar[d]  \\
  X_1 & \ar[l] X_a.}$$
\end{lemma}

\begin{proof}
  This square is the complement of the pullback saying what conservative means.
  But it is general in extensive $\infty$-categories such as $\Grpd$, that in the situation
  $$\xymatrix{ A' \ar[r]\ar[d] & A'+B' \ar[d] & \ar[l] B' \ar[d] \\
  A \ar[r] & A+B & \ar[l] B,}
  $$
  one square is a pullback if and only if the other is.
\end{proof}

\begin{cor}\label{lem:cons-wn}
  If $X$ and $Y$ are complete simplicial spaces and $f:Y \to X$ is
  conservative, then for every word $w \in \{0,1,a\}\upperstar$,
  the following square is a pullback:
  \begin{equation}\label{eq:YnYw}
    \xymatrix{
    Y_n\ar[d] & \ar[l] Y_w \dlpullback \ar[d] \\
    X_n & \ar[l]  X_w.}
  \end{equation}
\end{cor}
\begin{proof}
  The square is connected to 
  \begin{equation}\label{eq:Y1nYw}
    \xymatrix{
    (Y_1)^n\ar[d] & \ar[l] Y^w \dlpullback \ar[d] \\
    (X_1)^n & \ar[l]  X^w} 
  \end{equation}
  by two instances of pullback-square \eqref{eq:Xw}, one for $Y$ and one for
  $X$.  It follows from \ref{lem:cons-X1} that \eqref{eq:Y1nYw} is a pullback,
  hence also \eqref{eq:YnYw} is a pullback, by an application of
  Lemma~\ref{pbk}.
\end{proof}

\begin{prop}\label{prop:cULF-nondegen}
  If $X$ and $Y$ are complete simplicial spaces and $f:Y \to X$ is cULF, then
  for any word $w\in \{0,1,a\}\upperstar$
  the following square is a pullback:
  $$
  \xymatrix{
  Y_1 \ar[d]& \ar[l]  Y_w \dlpullback \ar[d] \\
  X_1 & \ar[l]   X_w.}
  $$
\end{prop}
\begin{proof}
  Just compose the square of Corollary~\ref{lem:cons-wn} with the square 
  $$
  \xymatrix{
  Y_1 \ar[d] & \ar[l] Y_n \dlpullback \ar[d] \\
  X_1 & \ar[l]  X_n,}
  $$
  which is a pullback since $f$ is cULF.
\end{proof}

\begin{lemma}\label{lem:X1w=sum}
  Let $X$ be a complete simplicial space. Then for any words
  $v,v' \in \{0,1,a\}\upperstar$, we have
  $$
  X_{v1v'} = X_{v0v'} + X_{vav'} , 
  $$
  and hence
  $$
  X_n=\sum_{w\in\{0,a\}^n} X_w.
  $$
\end{lemma}
\begin{proof}
  Consider the diagram
  $$
  \xymatrix{
  X_{v0v'} \drpullback  \ar[r]\ar[d] & X_{v1v'} \ar[d] & \ar[l] \dlpullback X_{vav'} \ar[d]\\
  X^{v0v'} \ar[r] & X^{v1v'} & \ar[l] X^{vav'}
  }$$
  The two squares are pullbacks, by Lemma \ref{pbk}, since horizontal composition
  of either with the pullback square \eqref{eq:Xw} for $w=v1v'$ gives again the 
  pullback square \eqref{eq:Xw}, for $w=v0v'$ or $w=vav'$.




  Since the bottom row is a sum diagram, it follows that the top row is also
  (since the $\infty$-category of $\infty$-groupoids is locally cartesian closed, 
  and in particular extensive).
\end{proof}

We now specialise to complete decomposition spaces, although the following
result will be subsumed in Subsection~\ref{sec:stiff} in a more general situation.
\begin{prop}\label{degen-w}
  Let $X$ be a complete decomposition space.
  Then for any words $v,v'$ in the alphabet $\{0,1,a\}$ we have 
  $$
  X_{v0v'}=\Im(s_{|v|}:X_{vv'}\to X_{v1v'}) .
  $$
  That is, the $k$th principal edge of a simplex $\sigma$ is degenerate if and
  only if $\sigma=s_{k-1}d_k\sigma$.
\end{prop}
\noindent 
Recall that $|v|$ denotes the length of the word $v$ and,
as always, the notation $\Im$ refers to the essential image.
\begin{proof}
From \eqref{eq:Xw} we see that (independent of the decomposition-space axiom)
  $X_{v0v'}$ is characterised by the top pullback square in the diagram
  $$\xymatrix{
      X_{v0v'} \drpullback \ar[r]\ar[d] & X_{v1v'} \ar@/^1.5pc/[dd]^{{d_\bot\!}^{|v|}\,{d_\top\!}^{|v'|}} \ar[d]
      \\ 
      X^{v0v'} \drpullback \ar[r]\ar[d] & X^{v1v'} \ar[d]
      \\ 
      X_0 \ar[r]_{s_0} & X_1
  }$$
  But the decomposition-space axiom applied to the exterior pullback diagram
  says that the top horizontal map is $s_{|v|}$, and hence identifies $X_{v0v'}$
  with the image of $s_{|v|}: X_{vv'} \to X_{v1v'}$.  For the final statement,
  note that if $\sigma=s_{k-1}\tau$ then $\tau=d_k\sigma$.
\end{proof}
Combining this with Lemma \ref{lem:X1w=sum} we obtain the following result.
\begin{cor} \label{cor:X1w=sum}
  Let $X$ be a complete decomposition space.  For any words
  $v,v'$ in the alphabet $\{0,1,a\}$ we have
  $$
  X_{v1v'} = s_{|v|}(X_{vv'}) + X_{vav'} .
  $$
\end{cor}

\begin{blanko}{Effective simplices.}\label{effective}
  A simplex in a complete simplicial space $X$ is called {\em effective}
  when all its principal edges are nondegenerate.
  We put
  $$
  \nondeg{X}_n = X_{a\cdots a} \subset X_n ,
  $$
  the full subgroupoid of $X_n$ consisting of the effective simplices.  (Every
  $0$-simplex is effective by convention: $\nondeg X_0 = X_0$.)  It is clear
  that outer face maps $d_\bot,d_\top:X_{n}\to X_{n-1}$ preserve effective
  simplices, and that every effective simplex is nondegenerate, i.e.~is not in
  the image of any degeneracy map.  It is a useful feature of complete {\em
  decomposition} spaces that the converse is true too:
\end{blanko}

\begin{cor}\label{effective=nondegen}
  In a complete decomposition space $X$, 
  a simplex is effective if and only if it is nondegenerate:
  $$
  \nondeg{X}_n = X_n\setminus {\textstyle\bigcup_{i=0}^n}\Im(s_i).
  $$
\end{cor}
\begin{proof}
  It is clear that $\nondeg{X}_n$ is the complement of 
  $X_{01\cdots1} \cup \cdots \cup X_{1\cdots10}$
  and by Proposition~\ref{degen-w}
  we can identify each of these spaces with the image of a degeneracy map.
\end{proof}

\noindent
In fact this feature is enjoyed 
by a more general class of complete simplicial spaces, treated in 
Section~\ref{sec:stiff}.

\medskip

Iterated use of \ref{cor:X1w=sum} yields 
\begin{cor}\label{cor:Xn}
  For $X$ a complete decomposition space we have 
  $$
  X_{n} = \sum s_{j_k}\dots s_{j_1}(\nondeg X_{n-k}) ,
  $$
  where the sum is over all subsets $\{j_1<\dots<j_k\}$ of $\{0,\dots,n-1\}$.
\end{cor}

\begin{lemma}\label{lem:Segal:nondeg}
  If a complete decomposition space $X$ is a Segal space, then $\nondeg X_n
  \simeq \nondeg X_1 \times_{X_0} \cdots \times_{X_0} \nondeg X_1$, the
  $\infty$-groupoid of strings of $n$ composable nondegenerate arrows in $X_n
  \simeq X_1 \times_{X_0} \cdots \times_{X_0} X_1$.
\end{lemma}
\noindent
This follows immediately from \eqref{eq:Xw}.
Note that if furthermore $X$ is Rezk complete, we can say non-invertible instead of 
nondegenerate.

%% file: convolution.tex


\def\inputfile{convolution.tex}

\subsection{Convolution product and \M inversion}

\label{sec:Minv}

Let $X$ be a decomposition space. In this subsection we examine the incidence algebra $\Grpd^{X_1}$ which can be obtained from the incidence coalgebra $\Grpd_{/X_1}$ by taking the linear dual (see~\ref{Lineardual}).

\begin{blanko}{Convolution.}
  The $\infty$-category $\Grpd_{/S}$ plays the role of the vector space with basis 
  $S$.  Just as a linear functional is determined by its values on basis elements,
  linear functors $\Grpd_{/S} \to \Grpd$ correspond to arbitrary functors
  $S \to \Grpd$, hence the presheaf category $\Grpd^S$ can be considered the
  linear dual of the slice category $\Grpd_{/S}$ (see \cite{GKT:HLA} for the 
  precise statements and proofs).


If $X$ is a decomposition space, the coalgebra structure on $\Grpd_{/X_1}$
therefore induces an algebra structure on $\Grpd^{X_1}$.
%
%
%
  The convolution product of two linear functors
  $$
  F,G:\Grpd_{/ X_1}\longrightarrow \Grpd,
  $$
  given by spans $ X_1 \leftarrow M \to 1$ and $ X_1 \leftarrow N\to 1$,
  is the composite of their tensor product $F\otimes G$
  and the comultiplication,
  $$
  F*G: \quad \Grpd_{/ X_1}\stackrel{\Delta}\longrightarrow 
  \Grpd_{/ X_1}\otimes \Grpd_{/ X_1} \stackrel{F\tensor G}\longrightarrow \Grpd.
  $$
  Thus the convolution is given by the composite of spans
  $$
  \xymatrix@!C=9ex{
   X_1 && \\
   X_2 \ar[u] \ar[d]&\ar[l]\ar[d] M *N\ar[lu]\ar[rd]\dlpullback &
  \\
   X_1\times  X_1 &\ar[l]M\times N \ar[r] & 1.
  }$$
  The neutral element for convolution is $\epsilon:\Grpd_{/X_1}\to\Grpd$ defined
  by the span
  $$
  X_1 \stackrel {s_0}\leftarrow  X_0 \to 1\,.
  $$
\end{blanko}

\begin{blanko}{The zeta functor.}\label{zeta}
  The {\em zeta functor}
  $$
  \zeta:\Grpd_{/ X_1} \to \Grpd
  $$
  is the linear functor defined by the span
  $$
  X_1 \stackrel =\leftarrow  X_1 \to 1\,.
  $$
  As an element of the linear dual (\ref{Lineardual}), this corresponds to
  the terminal presheaf.
  We will see later that in the locally finite situation \ref{finitary}, 
  upon taking the homotopy cardinality of the zeta functor
  one obtains the constant function 1 on $\pi_0 X_1$, that is, the classical zeta
  function in the incidence algebra.
  
  It is clear from the definition of the convolution product that 
  the $k$th convolution power of the zeta functor is given by
  $$
  \zeta^k : \; 
  X_1 \stackrel g\leftarrow  X_k \to 1 ,
  $$
  where $g: [1] \to [k]$ is the unique generic map in degree $k$.
\end{blanko}

We also introduce the following elements of the incidence algebra $\Grpd^{X_1}$:
for each $a\in X_1$, let $(X_1)_{[a]}$ be the component of $X_1$ containing $a$, and
let $\delta^a$ be the linear functor given by the span
$$
\delta^a : \quad X_1 \leftarrow (X_1)_{[a]} \to 1 ,
$$
We also have the representable functors 
\begin{eqnarray*}
  h^a := \Map(a, - ) : X_1 & \longrightarrow & \Grpd 
\end{eqnarray*}
which viewed as linear functors $\Grpd_{/X_1} \to \Grpd$ are given by the spans
$$
h^a : \quad X_1 \stackrel{\name{a}}\leftarrow 1 \to 1 .
$$
Hence we have 
$$
\zeta = \sum_{a\in \pi_0 X_1} \delta^a = \int^a h^a.
$$
See \ref{scalar&hosum} for the integral notation for
   homotopy sums.

We are interested in the invertibility of the zeta functor under the convolution
product.  Unfortunately, at the objective level it can practically {\em never}
be convolution invertible, because the inverse $\mu$ should always be given by
an alternating sum (cf.~\ref{thm:zetaPhi})
$$
\mu = \Phieven - \Phiodd 
$$
(of the Phi functors defined below).
We have no minus sign available, but 
following the idea of Content--Lemay--Leroux~\cite{Content-Lemay-Leroux},
developed further by Lawvere--Menni~\cite{LawvereMenniMR2720184},
  we establish the sign-free equations
$$
  \zeta * \Phieven = \epsilon + \zeta * \Phiodd ,
  \qquad\qquad
  \Phieven * \zeta  = \epsilon + \Phiodd * \zeta.
$$

In the category case (cf.~\cite{Content-Lemay-Leroux}
and \cite{LawvereMenniMR2720184}),
$\Phieven$ (resp.~$\Phiodd$) are given by even-length (resp.~odd-length)
chains of non-identity arrows.  (We keep the $\Phi$-notation in honour of 
Content--Lemay--Leroux). In the general setting of decomposition spaces
we cannot talk about chains of arrows,
but in the complete case we can still talk about effective simplices
and their principal edges.

%% file: phi.tex


\def\inputfile{phi.tex}

From now on we assume again that $X$ is complete decomposition space.

\begin{blanko}{`Phi' functors.}\label{Phi}
  We define $\Phi_n$ to be the linear functor given by the span
  $$
  X_1 \stackrel m\longleftarrow \nondeg{X}_n \longrightarrow  1.
  $$
  If $n=0$ then $\nondeg X_0=X_0$ by convention, and $\Phi_0$ is given by the
  span
  $$
  X_1 \stackrel u\longleftarrow X_0 \longrightarrow  1.
  $$
  That is, $\Phi_0$ is the linear functor $\epsilon$.  Note that
  $\Phi_1=\zeta-\epsilon$.  The minus sign makes sense here, since $X_0$
  (representing $\epsilon$) is really a full subgroupoid of $X_1$ (representing
  $\zeta$).
\end{blanko}

To compute convolution with $\Phi_n$, a key ingredient is the following
general lemma (with reference to the word notation of \ref{w}).

\begin{lemma}\label{lem:X1w}
  Let $X$ be a complete decomposition space. Then for any words
  $v,v'$ in the alphabet $\{0,1,a\}$, the square
  $$\xymatrix{
  X_{vv'} \ar[d]\ar[r] & X_2 \ar[d] \\
  X_v \times X_{v'} \ar[r]& X_1\times X_1
  }$$
  is a pullback.
\end{lemma}
\begin{proof}
  Let $m=|v|$ and $n=|v'|$.
  The square is the outer rectangle in the top row of the diagram
  $$\xymatrix{
  X_{vv'} \ar[d]\ar[r] & X_{m+n} \ar[d] \drpullback \ar[r] 
  & X_{1+n} \ar[d] \drpullback \ar[r]& X_2 \ar[d] 
  \\
  X_v \times X_{v'}\drpullback \ar[d] \ar[r]& X_m\times X_n\ar[d]\ar[r]& X_1\times X_n 
  \ar[r]& X_1\times X_1 
  \\
  X^v \times X^{v'} \ar[r] & {X_1}^m\times{X_1}^n 
  }$$
  The left-hand outer rectangle is a pullback by definition of $X_{vv'}$, and
  the bottom square is a pullback by definition of $X_v$ and $X_{v'}$.  Hence
  the top-left square is a pullback.  But the other squares in the top row are
  pullbacks because $X$ is a decomposition space.
\end{proof}
%

\begin{lemma}
    We have $\Phi_n = (\Phi_1)^n = (\zeta-\epsilon)^n$, 
    the $n$th convolution product of $\Phi_1$ with itself.
\end{lemma}
\begin{proof}
    This follows from the definitions and Lemma~\ref{lem:X1w}.
\end{proof}

\begin{prop}
The linear functors $\Phi_n$ satisfy
$$
\zeta*\Phi_n
\;\;=\;\;
\Phi_n+\Phi_{n+1}
\;\;=\;\;
\Phi_n*\zeta.
$$
\end{prop}
\begin{proof}
  We can compute the convolution $\zeta * \Phi_n$ by
  Lemma~\ref{lem:X1w} as
  $$
  \xymatrix@!C=9ex{
   X_1 && \\
   X_2 \ar[u] \ar[d]&\ar[l]\ar[d] X_{1a\cdots a}\ar[lu]\ar[rd]\dlpullback &
  \\
   X_1\times  X_1 &\ar[l]X_1\times \nondeg X_n \ar[r] & 1
  }$$
  But Lemma~\ref{lem:X1w=sum} tells us that
  $X_{1a\cdots a} = X_{0a\cdots a} + X_{aa\cdots a} = \nondeg X_n +
  \nondeg X_{n+1}$, where the identification in the first summand is
  via $s_0$, in virtue of Proposition~\ref{degen-w}.
  This is an equivalence of $\infty$-groupoids over $X_1$
  so the resulting span is $\Phi_n+\Phi_{n+1}$ as desired.
  The second identity claimed follows similarly.
\end{proof}

Put
$$
\Phieven := \sum_{n \text{ even}} \Phi_n , \qquad
\Phiodd := \sum_{n \text{ odd}} \Phi_n .
$$

\begin{theorem}\label{thm:zetaPhi}
  For a complete decomposition space, the following \M inversion 
  principle holds:
  \begin{align*}
\zeta * \Phieven
 &\;\;=\;\; \epsilon\;\; +\;\; \zeta * \Phiodd,\\
=\;\;\Phieven *\zeta &\;\;=\;\; \epsilon \;\;+ \;\; \Phiodd*\zeta.
\end{align*}
\end{theorem}

\begin{proof}
  This follows immediately from the proposition: all four linear functors are in fact
  equivalent to $\sum_{r\geq0}\Phi_r$.
\end{proof}

We note the following immediate corollary of 
Proposition~\ref{prop:cULF-nondegen}, which can be read as saying
`\M inversion is preserved by cULF functors':
\begin{cor}\label{phi=phi}
  If $f:Y \to X$ is cULF, then $f\upperstar \zeta = \zeta$ and
  $f\upperstar \Phi_n = \Phi_n$ for all $n\geq0$.
\end{cor}

%% file: half.tex

\def\inputfile{half.tex}

\subsection{Stiff simplicial spaces}

\label{sec:stiff}

We saw that in a complete decomposition space, degeneracy can be detected on
principal edges.  In the next subsection (\ref{sec:split}) we shall come to
split simplicial spaces, which share this property.  A common generalisation is
that of stiff simplicial spaces, which we now introduce.

\begin{blanko}{Stiffness.}
  A simplicial space $X: \Delta\op\to\Grpd$
  is called {\em stiff} if it sends degeneracy/free 
  pushouts in $\Delta$ to pullbacks in $\Grpd$.  These pushouts are
  examples of generic-free pushouts, so in particular every
  decomposition space is stiff.
\end{blanko}

\begin{lemma}\label{lem:stiff}
  A simplicial space $X$  is stiff if and only if the following diagrams
  are pullbacks for all $0\leq i\leq n$
      $$\xymatrix{
	  X_{n} \drpullback 
	  \ar[r]^{s_{i}}\ar[d]&\ar[d]^{{d_\bot\!}^{i}\,{d_\top\!}^{n-i}} X_{n+1}
	  \\ 
	  X_0 \ar[r]_{s_0} & X_1
      }$$
\end{lemma}
\begin{proof}
  The squares in the lemma are special cases of the degeneracy/free squares.
  On the other hand, every degeneracy/free square sits in between two of the squares
  of the lemma in such a way that the standard pullback argument forces it
  to be a pullback too.  (The proof was essentially done in \ref{onlyfourdiags},
  but note however that while in the case of a decomposition space it is enough
  to check just two squares (namely those with $n=1$), in the present situation
  this is not the case.  The argument used to establish this in the 
  decomposition-space case \ref{onlyfourdiags} exploited pullbacks between
  generic face maps and free maps.)
\end{proof}

We also get, by Remark~\ref{cheaperbonus}:
\begin{cor}
  In a stiff simplicial space, the `bonus pullbacks' of Lemmas~\ref{bonus-pullbacks}
  and \ref{newbonuspullbacks} hold.
\end{cor}

\begin{lemma}
  In a stiff simplicial space $X$, every degeneracy map is
  a pullback of $s_0: X_0 \to X_1$.  In particular, if just $s_0:X_0 \to X_1$
  is mono then all degeneracy maps are mono.
\end{lemma}
\begin{proof}
  The proof is contained in that of Lemma~\ref{lem:s0d1}.
\end{proof}
Finally, by the argument employed in \ref{lem:cULFs0d1} we have
\begin{lemma}\label{lem:conss0}
  A simplicial map $f:Y \to X$ between stiff simplicial spaces is
  conservative if and only if it is cartesian on the first degeneracy map
  $$\xymatrix{
     Y_0 \drpullback \ar[r]^{s_0}\ar[d] & Y_1 \ar[d] \\
     X_0 \ar[r]_{s_0} & X_1 .
  }$$
\end{lemma}

\begin{lemma}
  A stiff simplicial space $X$ is complete if and only if the canonical map from
  the constant simplicial space $X_0$ is conservative.
\end{lemma}
\begin{proof}
  Suppose $X$ is complete.  Then any $s_i : X_k \to X_{k+1}$ is mono,
  and hence in the following diagram the bottom square is a pullback:
  $$\xymatrix{
     X_0 \drpullback \ar[r]^=\ar[d]_s & X_0 \ar[d]^s \\
     X_k \drpullback \ar[r]^=\ar[d]_= & X_k \ar[d]^{s_i} \\
     X_k \ar[r]_{s_i} & X_{k+1} .
  }$$
  Hence $X_0 \to X$ is cartesian on  $s_i$.
  Since this is true for any degeneracy map $s_i$, altogether
  $X_0 \to X$ is conservative.
  Conversely, if $X_0 \to X$ is conservative, then in particular
  we have the pullback square
  $$\xymatrix{
     X_0 \drpullback\ar[r]^=\ar[d]_= & X_0 \ar[d]^{s_0} \\
     X_0 \ar[r]_{s_0} & X_1
  }$$
  which means that $s_0 : X_0 \to X_1$ is a monomorphism.
  %
\end{proof}

For complete simplicial spaces, we can characterise stiffness
also in terms of degeneracy:

\begin{prop}\label{halfdecomp}
The following are equivalent for a complete simplicial space $X$
\begin{enumerate}
\item  $X$ is stiff.
\item  Outer face maps $d_\bot,d_\top:X_{n}\to X_{n-1}$ preserve nondegenerate simplices. 
\item  Any nondegenerate simplex is effective. More precisely,
  $$
  \nondeg{X}_n = X_n\setminus {\textstyle\bigcup_{i=0}^n}\Im(s_{i-1}).
  $$
\item  If the $i$th principal edge of $\sigma\in X_n$ is degenerate, 
then $\sigma=s_{i-1}d_{i-1}\sigma  =s_{i-1}d_{i}\sigma$, that is
  $$X_{1\dots101\dots1} = \Im(s_{i-1}:X_{n-1}\to X_n)$$
\item For each word $w\in\{0,a\}^n$ we have
  $$X_w = \Im(s_{j_k-1}\dots s_{j_1-1}:\nondeg X_{n-k}\to X_n) .
  $$
  where $\{j_1<\dots<j_k\}=\{j:w_j=0\}$.
\end{enumerate}
\end{prop}
\begin{proof}
  $(1)\Rightarrow(2)$: Suppose $\sigma \in X_n$ and that $d_\top \sigma$
is degenerate.  Then $d_\top \sigma$ is in the image
of some $s_i : X_{n-2} \to X_{n-1}$, and hence by
(1) already $\sigma$ is in the image of $s_i : X_{n-1} \to X_n$. 
  

$(2)\Rightarrow(3)$: 
The principal edges of a simplex are obtained by applying outer face maps, 
so nondegenerate simplices are also effective.  For the more precise statement,
just note that both subspaces are full, so are determined by the properties
characteristing their objects.

$(3)\Rightarrow(4)$: 
As $\sigma$ is not effective, we have 
$\sigma=s_j\tau$. If $j>i-1$ then the $i$th principal edge is of $\sigma$ 
is also that of $\tau$, so by induction $\tau\in\Im(s_{i-1})$. Therefore 
$\sigma\in\Im(s_{i-1})$ also, and 
$\sigma=s_{i-1}d_{i-1}\sigma  =s_{i-1}d_{i}\sigma$ as required. 
If $j< i-1$ the argument is similar.   


$(4)\Leftrightarrow(1)$: 
To show that $X$ is stiff, by Lemma~\ref{lem:stiff} it 
is enough to check that this is a pullback:
      $$\xymatrix{
	  X_{n} \drpullback 
	  \ar[r]^{s_{i}}\ar[d]&\ar[d]^{{d_\bot\!}^{i}\,{d_\top\!}^{n-i}} X_{n+1}
	  \\ 
	  X_0 \ar[r]_{s_0} & X_1
      }$$
But the pullback is by definition $X_{1\cdots 101\cdots 1} \subset X_{n+1}$.
But by assumption this is canonically identified with the image of $s_i : X_n 
\to X_{n+1}$, establishing the required pullback.


$(4)\Leftrightarrow(5)$: This is clear, using Lemma~\ref{lem:X1w=sum}.
\end{proof}

In summary, an important feature of stiff complete simplicial spaces
is that all information about degeneracy is encoded in the principal edges.  We
exploit this to characterise conservative maps between stiff complete simplicial
spaces:
\begin{prop}\label{stiff-cons}
  For $X$ and $Y$ stiff complete simplicial spaces, and  $f:Y \to X$ a
  simplicial map, the following are equivalent.
  \begin{enumerate}
    \item $f$ is conservative.  
    \item $f$ preserves the word splitting,
    i.e.~for every word $w\in \{0,a\}\upperstar$, $f$
  sends $Y_w$ to $X_w$.  
    \item $f_1$ maps $Y_a$ to $X_a$.
  \end{enumerate}
\end{prop}
\begin{proof}
  We already saw (\ref{lem:cons-wn}) that conservative maps preserve the word 
  splitting (independently of $X$ and $Y$ being stiff), which proves 
  $(1)\Rightarrow(2)$.  The implication $(2)\Rightarrow(3)$ is trivial.
  Finally assume that $f_1$ maps $Y_a$ to $X_a$.
  To check that $f$ is conservative, it is enough (by \ref{lem:conss0}) 
  to check that the square
    $$\xymatrix{
     Y_0 \drpullback \ar[r]^{s_0}\ar[d] & Y_1 \ar[d] \\
     X_0 \ar[r]_{s_0} & X_1  
  }$$
  is a pullback.  But since $X$ and $Y$ are complete, this square is just
  $$\xymatrix{
     Y_0 \drpullback \ar[r]^-{s_0}\ar[d] & Y_0 + Y_a \ar[d] \\
     X_0 \ar[r]_-{s_0} & X_0 + X_a  ,
  }$$
  which is clearly a pullback when $f_1$ maps $Y_a$ to $X_a$.
\end{proof}

This proposition can be stated more formally as follows.
For $X$ and $Y$ stiff complete simplicial spaces,
the space of conservative maps $\operatorname{Cons}(Y,X)$ is given
as the pullback
$$\xymatrix@C+2pc{
   \operatorname{Cons}(Y,X) \drpullback \ar[r]\ar@{ >->}[d] & {\displaystyle \prod_{n\in \N} \;
   \prod_{w\in \{0,a\}^n}} \Map(Y_w,X_w) \ar@{ >->}[d] \\
   \operatorname{Nat}(Y,X) \ar[r] & {\displaystyle \prod_{n\in \N}} \;\Map(Y_n,X_n) .
}$$
The vertical arrow on the right is given as follows. We have
$$
\Map(Y_n,X_n)\; =\; \Map(\!\underset{w\in\{0,a\}^n}{\textstyle\sum}\!\! Y_w,\! \underset{v\in\{0,a\}^n}{\textstyle\sum}\!\! X_v )
\;= \!\!\!\prod_{w\in \{0,a\}^n} \!\!\!\Map(Y_w,\!\underset{v\in\{0,a\}^n}{\textstyle\sum} \!\!X_v ).
$$
For fixed $w\in 
\{0,a\}^n$, the space $\Map(Y_w,\sum_{v\in\{0,a\}^n} X_v )$
has a distinguished subobject, namely consisting of those maps that
map into $X_w$ for that same word $w$.

%% file: split.tex
\def\inputfile{split.tex}

\subsection{Split decomposition spaces}

\label{sec:split}

\begin{blanko}{Split simplicial spaces.}
  In a complete simplicial space $X$, by definition all
  degeneracy maps are monomorphisms, so in particular it makes sense to
  talk about nondegenerate simplices in degree $n$: these form the full
  subgroupoid of $X_n$ given as the complement of the degeneracy maps $s_i :
  X_{n-1} \to X_n$.  A complete simplicial space, is called {\em split} if the
  face maps preserve nondegenerate simplices.
\end{blanko}

By Proposition~\ref{halfdecomp}, a split simplicial space is stiff, 
so the results from the previous subsection are available for 
split simplicial spaces.  In particular, nondegeneracy can be measured on 
principal edges, and we have
\begin{cor}\label{prop:word=degen}
  If $X$ is a split simplicial space, then the sum splitting
  $$
  X_n\;\; = \sum_{w\in\{0,a\}^n} X_w
  $$
  is realised by the degeneracy maps.
\end{cor}

\begin{blanko}{Non-example.}\label{ex:rsi}
  The strict nerve of any category with a non-trivial section-retraction pair of
  arrows, $r \circ s = \id$, constitutes an example of a complete decomposition
  space which is not split. Indeed, the nondegenerate simplices are the
  chains of composable non-identity arrows, but we have $d_1(s,r)=\id$.
  
  In this way, splitness can be seen as an abstraction of the condition on a
  $1$-category that its identity arrows be indecomposable.  
  (Corollary~\ref{prop:tight=>split} further on
  generalises the classical fact that in a \M category, the identity arrows are
  indecomposable (Leroux~\cite{Leroux:1975}).)
%
%
%
\end{blanko}

\begin{blanko}{Semi-decomposition spaces.}
  Let $\Deltainj \subset \Delta$ denote the subcategory
  consisting of all the objects and only the injective maps.
  A {\em semi-simplicial} space is an object in the functor $\infty$-category
  $\Fun(\Deltainj\op, \Grpd)$.
  A {\em semi-decomposition} space is a semi-simplicial space preserving
  generic-free pullbacks in $\Deltainj\op$.  Since there
  are no degeneracy maps in $\Deltainj$, this means that we
  are concerned only with pullbacks between generic face maps and free face 
  maps.
  
  Every simplicial space has an underlying semi-simplicial space obtained by
  restriction along $\Deltainj \subset \Delta$.  
  The forgetful functor $\Fun(\Delta\op, \Grpd) \to
  \Fun(\Deltainj\op, \Grpd)$ has a left adjoint given by left Kan 
  extension along $\Deltainj \subset \Delta$:
  $$
  \xymatrix {
  \Deltainj\op \ar[d] \ar[r]^Z & \Grpd \\
  \Delta\op \ar@{..>}[ru]_{\overline Z} &
  }
  $$
  The left Kan extension has the following explicit description:
  \begin{align*}
    \overline Z_0 = & Z_0 \\
    \overline Z_1 = & Z_1 + Z_0\\
    \overline Z_2 = & Z_2 + Z_1 +Z_1 +Z_0\\
    \vdots & \\
    \overline Z_k = & \sum_{w\in\{0,a\}^k} Z_{|w|_a}
  \end{align*}
  For $w\in\{0,a\}^k$ and $\sigma\in Z_{|w|_a}$ the corresponding element of $\overline Z_k$ is denoted
$$
s_{i_r}\dots s_{i_2}s_{i_1}\sigma
$$
  where $r=k-|w|_a$ and $i_1<i_2<\dots<i_r$ with $w_{i_j}=0$.
  The faces and degeneracies of such elements are defined in the obvious way.
\end{blanko}

\begin{prop}
  A a simplicial space is split if and only if it is the
  left Kan extension of a semi-simplicial space.
\end{prop}

\begin{proof}
  Given $Z: \Deltainj\op\to\Grpd$, it is clear from the construction
  that the new degeneracy maps in $\overline Z$ are monomorphisms.  Hence 
  $\overline Z$ is complete.
  On the other hand, to say that $\sigma\in \overline Z_n$ is nondegenerate
  is precisely to say that it belongs to the original component $Z_n$, and
  the face maps here are the original face maps, hence map $\sigma$ into
  $Z_{n-1}$ which is precisely the nondegenerate component of $\overline 
  Z_{n-1}$.  Hence $\overline Z$ is split.

  For the other implication, given a split simplicial space $X$, we know that
  nondegenerate is the same thing as effective (\ref{halfdecomp}),
  so we have a sum splitting
  $$
  X_n \;\;= \sum_{w\in \{0,a\}^n} X_w .
  $$
  Now by assumption the face maps restrict to the nondegenerate simplices
  to give a semi-simplicial space $\nondeg X : \Deltainj\op \to \Grpd$.
  It is now clear from the explicit description of the left Kan extension
  that $\overline{(\nondeg X_n)} =X_n$, from where it follows readily that
  $X$ is the left Kan extension of $\nondeg X$.
\end{proof}

\begin{prop}
  A simplicial space is a split decomposition space if and only if it is the
  left Kan extension of a semi-decomposition space.
\end{prop}
\begin{proof}
  It is clear that if $X$ is a split decomposition space then $\nondeg X$
  is a semi-decomposition space.  Conversely, if $Z$ is a semi-decomposition
  space, then one can check by inspection that $\overline Z$ satisfies the
  four pullback conditions in Proposition~\ref{onlyfourdiags}: two of these
  diagrams concern only face maps, and they are essentially from $Z$, with
  degenerate stuff added.  The two diagrams involving degeneracy maps are
  easily seen to be pullbacks since the degeneracy maps are sum inclusions.
\end{proof}


\begin{thm}\label{thm:semisimpl=splitcons}
  The left adjoint from before, 
  $\Fun(\Deltainj\op,\Grpd) \to \Fun(\Delta\op,\Grpd)$,
  induces an equivalence of $\infty$-categories
  $$
\Fun(\Deltainj\op,\Grpd) \simeq \kat{Split}^{\mathrm cons},
$$
the $\infty$-category of split simplicial spaces and conservative maps.
\end{thm}
\begin{proof}
  Let $X$ and $Y$ be split simplicial spaces, then $\nondeg X$ and $\nondeg Y$
  are semi-simplicial spaces whose left Kan extensions are $X$ and $Y$ again.
  The claim is that
  $$
  \operatorname{Cons}(Y,X) \simeq \Nat(\nondeg Y, \nondeg X).
  $$
  Intuitively, the reason this is true can be seen in the first  square as in
  the proof of Lemma~\ref{stiff-cons}: to give a pullback square
    $$\xymatrix{
     Y_0 \drpullback \ar[r]^-{s_0}\ar[d] & Y_0 + Y_a \ar[d] \\
     X_0 \ar[r]_-{s_0} & X_0 + X_a  ,
  }$$
  amounts to giving $Y_0 \to X_0$ and $Y_a\to X_a$ (and of course, in both cases
  this data is required to be natural in face maps), that is to give
  a natural transformation $\nondeg Y 
  \to \nondeg X$.  To formalise this idea, note first that $\Nat(\nondeg Y ,
  \nondeg X)$ can be described as a limit
  $$
  \Nat(\nondeg Y ,\nondeg X)  \longrightarrow \prod_{n\in \N} \Map(\nondeg Y_n, 
  \nondeg X_n) \to \ldots
  $$
  where the rest of the diagram contains vertices indexed by all the face maps,
  expressing naturality.  Similarly
   $\Nat(Y,X)$ is given as a limit
  $$
  \Nat( Y , X)  \longrightarrow \prod_{n\in \N} \Map( Y_n, 
   X_n) \to \ldots
  $$
  where this time the rest of the diagram furthermore
  contains vertices corresponding to degeneracy maps.  The full subspace of 
  conservative maps is given instead as
  $$
  \Cons( Y , X)  \longrightarrow \prod_{w\in \{0,a\}\upperstar} \Map( 
  Y_w, 
   X_w) \to \ldots
  $$
  as explained in connection with Lemma~\ref{stiff-cons}.  Now for each 
  degeneracy map $s_i : X_n \to X_{n+1}$, there is a vertex in the diagram. 
  For ease of notation, let us consider $s_0 : X_n \to 
  X_{n+1}$.  The corresponding vertex sits in the limit diagram as follows:
  for each word $v\in \{0,a\}^n$, we have
  $$\xymatrix{
     {\displaystyle\prod_{w\in \{0,a\}\upperstar} \Map(Y_w,X_w)  }
     \ar[r]^-{\text{proj}}\ar[d]_-{\text{proj}} & \Map(Y_{0v}, X_{0v}) 
     \ar[d]^{\text{pre }s_0} \\
     \Map(Y_v,X_v) \ar[r]_{\text{post }s_0} & \Map(Y_n,X_{n+1}) .     
  }$$
  Now both the pre and post composition maps are monomorphisms with essential
  image $\Map(Y_v, X_{0v})$, so the two projections coincide, which is to say
  that the limit factors through the corresponding diagonal.  Applying this
  argument for every degeneracy map $s_i : X_n \to X_{n+1}$, and for all words,
  we conclude that the limit factors through the product indexed only over the
  words without degeneracies,
  $$
  \prod_{n\in \N} \Map(\nondeg Y_n, \nondeg X_n) .
  $$
  Having thus eliminated all the vertices of the limit diagram that corresponded
  to degeneracy maps, the remaining diagram has precisely the shape of the 
  diagram
  computing $\Nat(\nondeg Y ,\nondeg X)$, and we have already seen that the
  `starting vertex' is the same, $\prod_{n\in \N} \Map(\nondeg Y_n, \nondeg 
  X_n)$.  For the remaining vertices, those corresponding to face maps, 
  it is readily seen that in each case the space is that of the $\Nat(\nondeg 
  Y, \nondeg X)$ diagram, modulo some constant factors that do not play any role in
  the limit calculation.  In conclusion, the diagram calculating $\Cons( Y , X)$
  as a limit is naturally identified with the diagram calculating 
  $\Nat(\nondeg Y, \nondeg X)$ as a limit.
\end{proof}

\begin{prop}\label{prop:split=semi}
  This equivalence restricts to an equivalence between semi-decomposition spaces
  and all maps and split decomposition spaces and conservative maps, and it 
  restricts further to an equivalence between semi-decomposition spaces and
  ULF maps and split decomposition spaces and cULF maps.
\end{prop}



\begin{blanko}{Dyckerhoff--Kapranov $2$-Segal semi-simplicial spaces.}
  Dyckerhoff and Kapranov's notion of $2$-Segal space 
  \cite{Dyckerhoff-Kapranov:1212.3563} does not refer to 
  degeneracy maps at all, and can be formulated already for
  semi-simplicial spaces:
  a $2$-Segal space is precisely a simplicial space whose underlying
  semi-simplicial space is a semi-decomposition space.
  We get the following corollary to the results above.
\end{blanko}

\begin{cor}
  Every split decomposition space is the left 
  Kan extension of a $2$-Segal semi-simplicial space.
\end{cor}

%% file: length.tex
\def\inputfile{length.tex}

\subsection{The length filtration}
\label{sec:length}

The \emph{long edge} of a simplex $\sigma\in X_n$ in a simplicial
space is the element $g(\sigma)\in X_1$, where $g:X_n\to X_1$ is the unique
generic map.


\begin{blanko}{Length.}\label{length}
  Let $a\in X_1$ be an edge in a complete decomposition space $X$.
  The {\em length} of $a$ is defined to be the biggest dimension
  of an effective
  simplex with long edge $a$:
  $$
  \ell(a) := \sup \{ \dim\sigma \mid \sigma\in \nondeg X, g(\sigma)= a \} ,
  $$
  where as usual $g: X_r \to X_1$ denotes the unique generic map.
  More formally: the length is the greatest $r$ such that the pullback 
  $$\xymatrix{
     (\nondeg X_r)_a \drpullback \ar[r]\ar[d] & \nondeg X_r \ar[d]^g \\
     1 \ar[r]_{\name a} & X_1
  }$$
  is nonempty (or $\infty$ if there is no such greatest $r$).
  Length zero can happen only for degenerate edges.
\end{blanko}


\begin{blanko}{Decomposition spaces of locally finite length.}
  A complete decomposition space $X$ is said to have {\em locally finite length}
  when every edge $a\in X_1$
  has finite length.  That is, the pullback
  $$\xymatrix{
     (\nondeg X_r)_a \drpullback \ar[r]\ar[d] & \nondeg X_r \ar[d]^g \\
     1 \ar[r]_{\name a} & X_1
  }$$
  is empty for $r \gg 0$.
  We shall also use the word {\em tight} as synonym for `of locally finite 
  length',
  to avoid confusion with the notion of `locally finite' introduced in 
  Section~\ref{sec:findec}.
\end{blanko}

\begin{eks}
  For posets, the notion of locally finite length coincides with the classical
  notion (see for example Stern~\cite{Stern:1999}), namely that for every $x\leq
  y$, there is an upper bound on the possible lengths of chains from $x$ to $y$.
  When $X$ is the strict (resp.~fat) nerve of a category, locally finite length means that
  for each arrow $a$, there is an upper bound on 
  the length of factorisations of $a$ containing no identity (resp.~invertible)
  arrows.

  A paradigmatic non-example is given by the strict nerve of a category
  containing an idempotent non-identity endo-arrow, $e= e\circ e$: clearly $e$ admits
  arbitrarily long decompositions $e = e\circ \cdots \circ e$.
\end{eks}

\begin{prop}\label{prop:cULF/FILT=FILT}
  If $f:Y \to X$ is cULF and $X$ is a \FILT decomposition space, then also $Y$ 
  is \FILT.
\end{prop}
\begin{proof}
  We know that $Y$ is a decomposition space by Lemma~\ref{cULF/decomp}, and the
  cULF condition in fact ensures it is complete.  It will furthermore be \FILT by
  Proposition~\ref{prop:cULF-nondegen}.
\end{proof}

\begin{prop}\label{prop:FILTSegal=Rezk}
  If a \FILT decomposition space $X$ is a Segal space, then it is Rezk complete.
\end{prop}

\begin{proof}
  If $X$ is not Rezk complete, then there exists a nondegenerate invertible
  arrow $a\in X_1$.  Since for Segal spaces we have 
  $$
  \nondeg X_n \simeq \nondeg X_1 \times_{X_0} \cdots \times_{X_0} \nondeg{X_1}
  $$
  (by \ref{lem:Segal:nondeg}),
  we can use the arrow $a:x\to y$ and its inverse to go back and forth any number of 
  times to create nondegenerate simplices
  of any length (subdivisions of $\id_x$ or $\id_y$).
\end{proof}

\begin{lemma}\label{lem:tight=>split}
  Let $X$ be a tight decomposition space.  Then for every $r\geq 1$
  we have a pullback square
  $$\xymatrix{
     \varnothing\drpullback \ar[r]\ar[d] & \nondeg X_r \ar[d]^g \\
     X_0 \ar[r]_{s_0} & X_1 .
  }$$
  More generally an effective simplex has all of its 1-dimensional 
  faces non-degenerate, so all faces of an effective simplex are effective.
\end{lemma}

\begin{proof}
  For $r=1$ the first statement is simply that $s_0X_0$ and $\nondeg X_1$ are
  disjoint in $X_1$, which is true by construction, so we can assume $r\geq2$.
  Suppose that $\sigma\in \nondeg X_r$ has degenerate long edge $u = g\sigma$.
  The idea is to exploit the decomposition-space axiom to glue together two copies of
  $\sigma$, called $\sigma_1$ and $\sigma_2$, to get a bigger simplex 
  $\sigma_1\#\sigma_2 \in \nondeg X_{r+r}$ again
  with long edge $u$.  By repeating this construction we obtain a contradiction
  to the finite length of $u$.  It is essential for this construction that
  $u$ is degenerate, say $u = s_0 x$, because we glue along the $2$-simplex 
  $\tau=s_0u=s_1u= s_0s_0 x$ which has the property that all three edges are $u$.
  Precisely, consider the diagram
  $$\xymatrix @C3.5pc@R3.5pc {
     X_{r+r}\drpullback \ar[r]_{{d_1}^{r-1}}\ar[dd]_{{d_\top}^r}     \ar@/^1.2pc/[rr]^{{d_\bot}^r} & X_{r+1}\drpullback \ar[d]^{{d_2}^{r-1}} 
     \ar@/_1.2pc/[dd]_{{d_\top}^r}
     \ar[r]_{d_\bot} & X_r \ar[d]^{g={d_1}^{r-1}} \\
     & X_2 \ar[d]^{d_\top} \ar[r]_{d_\bot} & X_1 \\
     X_r \ar[r]_{g={d_1}^{r-1}} & X_1 .
  }$$
  The two squares are pullbacks since $X$ is a decomposition space, and the 
  triangles are simplicial identities.  In the
  right-hand square we have $\sigma_2\in X_r$ and $\tau\in X_2$, with $g\sigma_2
  = u =d_\bot \tau$.  Hence we get a simplex $\rho\in X_{r+1}$.  This simplex has
  ${d_\top}^r \rho =d_\top\tau= u$, which means that in the left-hand square it matches
  $\sigma_1\in X_r$, to produce altogether the desired simplex
  $\sigma_1\#\sigma_2\in X_{r+r}$.  By construction, this simplex belongs to
  $\nondeg X_{r+r}$: indeed, its first $r$ principal edges are the principal 
  edges of $\sigma_1$, and its last $r$ principal edges are those of $\sigma_2$.
  Its long edge is clearly the long edge of $\tau$, namely $u$
  again, so we have produced a longer decomposition of $u$ than the one
  given by $\sigma$, thus contradicting the finite length of $u$.

Now the final statement follows since any 1-dimensional face of an effective simplex $\sigma$ is the long edge of an effective simplex ${d_\bot\!}^i{d_\top}^j\sigma$.
\end{proof}


\begin{cor}\label{prop:tight=>split}
  A \FILT decomposition space is split.
\end{cor}


For the next couple of corollaries, we shall need the following general lemma.

\begin{lemma}\label{lem:all-but-one}
Suppose $X$ is a complete decomposition space and $\sigma\in X_n$ has 
at least $n-1$ of its principal edges degenerate. Then the following 
are equivalent: 
\begin{enumerate}
  \item the long edge $g(\sigma)$ is degenerate, 
  \item all principal edges of $\sigma$ are degenerate, 
  \item $\sigma$ is totally degenerate, $\sigma\in{s_0}^n(X_0)$.
\end{enumerate} 
\end{lemma}
\begin{proof}
  Proposition~\ref{degen-w} says that (2) and (3) are equivalent.  Moreover it
  says that if all principal edges of $\sigma$ except the $j$th are known to be
  degenerate then $\sigma$ is an $(n-1)$-fold degeneracy of its $j$th principal
  edge.  Therefore the long edge of $\sigma$ is equal to its $j$th principal
  edge, and so (1) and (2) are equivalent.
\end{proof}

\begin{cor}\label{cor:degdegdeg}
  For any $\sigma\in X_2$ in a tight decomposition space  $X$, we have
  that $d_1\sigma$ is degenerate if and only if both $d_0\sigma$ and $d_2\sigma$ 
  are degenerate.
\end{cor}
\begin{proof}
  By Lemma~\ref{lem:tight=>split}, if $d_1\sigma$ is degenerate then
  at least one of the two principal edges is degenerate.  The result now
  follows from~\ref{lem:all-but-one}.
\end{proof}

\begin{cor}\label{cor:longedgedegen}
  In a \FILT decomposition space, if the long edge of a simplex is degenerate
  then all its edges are degenerate, and indeed the simplex is totally degenerate.
\end{cor}
\begin{proof}
  Let $\sigma$ be an $n$-simplex of a decomposition space $X$ and consider the
  2-dimensional faces $\tau_j$ of $\sigma$ defined by the vertices $j-1<j<n$.
  Applying Corollary~\ref{cor:degdegdeg} to each $\tau_j$, $j=1\,\dots,n-1$,
  shows that all principal edges of $\sigma$ are degenerate. 
  Lemma~\ref{lem:all-but-one} then says that $\sigma$ is in the image of ${s_0}^n$.
\end{proof}

We can now give alternative characterisations of the length of an arrow in a
\FILT decomposition space:
\begin{prop}\label{prop:alt-length}
  Let $X$ be a \FILT decomposition space, and $f\in X_1$.  Then the following
  conditions on $r\in \N$ are equivalent:
  \begin{enumerate}
    \item For all words $w$ in the alphabet $\{0,a\}$ in which the letter $a$ occurs at least $r+1$ times, the fibre $(X_w)_f$ is empty,
    $$\xymatrix{
       \varnothing \drpullback \ar[r]\ar[d] & X_w \ar[d] \\
       1 \ar[r]_{\name f} & X_1 .
    }$$
    \item For all $k\geq r+1$, the fibre $(\nondeg X_k)_f$ is empty.
    \item The fibre $(\nondeg X_{r+1})_f$ is empty.
  \end{enumerate}
  The length $\ell(f)$ of an arrow in a \FILT decomposition space is the least
  $r\in \N$ satisfying these equivalent conditions.
\end{prop}
\begin{proof}
  Clearly $(1)\Rightarrow(2)\Rightarrow(3)$ and, by definition, the length of
  $f$ is the least integer $r$ satisfying $(2)$.  It remains to show that $(3)$
  implies $(1)$.  Suppose (1) is false, that is, we have $w\in\{0,a\}^n$ with
  $k\geq r+1$ occurrences of $a$ and an element $\sigma\in X_w$ with
  $g(\sigma)=f$.  Then by Corollary~\ref{cor:Xn} we know that $\sigma$ is
  an $(n-k)$-fold degeneracy of some $\tau\in\nondeg X_{k}$, and $\sigma$ and
  $\tau$ will have the same long edge $f$.  Finally we see that (3) is false by
  considering the element ${d_1}^{k-r-1}\tau \in X_{r+1}$, which has long edge
  $f$, and is effective by Lemma~\ref{lem:tight=>split}.
\end{proof}

\begin{blanko}{The length filtration of the space of $1$-simplices.}
  Let $X$ be a \FILT decomposition space.
  We define the $k$th stage of the {\em length filtration} for $1$-simplices to
  consist of all the arrows of length at most $k$:
  $$
  X_1^{(k)} := \{ a \in X_1 \mid \ell(a)\leq k\}.
  $$
  

\begin{cor}
  For a tight decomposition space $X$ we have $X_1^{(0)} = X_0$.
  \qed
\end{cor}


Then $X_1^{(k)}$ is
  the full subgroupoid of $X_1$ given by any of the following equivalent
  definitions:
  \begin{enumerate}
  \item 
  the complement of $\Im( \nondeg X_{k+1} \to X_1)$.
  \item 
  the complement of $\Im( \coprod_{|w|_a>k} X_w \to X_1)$.
  \item
  the full subgroupoid of $X_1$ whose objects $f$ satisfy 
  $(X_{k+1})_f \subset \bigcup s_i X_k$
  \item
  the full subgroupoid of $X_1$ whose objects $f$ satisfy 
  $(\nondeg X_{k+1})_f=\varnothing$
  \item
  the full subgroupoid of $X_1$ whose objects $f$ satisfy 
  $(X_w)_f=\varnothing$ for all $w\in\{0,a\}^{r}$ such that $|w|_a> k$
  \end{enumerate}
\end{blanko}

%
%
%

It is clear from the definition of length that we have a sequence of 
monomorphisms
$$
X_1^{(0)} \into X_1^{(1)} \into X_1^{(2)} \into \dots \into X_1.
$$
The following is now clear.
\begin{prop}
  A complete decomposition space is \FILT if and only if the $X_1^{(k)}$ 
  constitute a filtration, i.e.
  $$
  X_1 = \bigcup_{k=0}^\infty X_1^{(k)} .
  $$
\end{prop}

\begin{blanko}{Length filtration of a \FILT decomposition space.}\label{def:filt}
  Now define the length filtration for all of $X$:
  the length of a simplex $\sigma$ with longest edge $g\sigma=a$ is defined to 
  be the length of $a$:
  $$
  \ell(\sigma) := \ell(a) .
  $$
  In other words, we are defining the filtration in $X_r$ by pulling it back 
  from $X_1$ along the unique generic map $X_r \to X_1$.
  This automatically defines the generic maps in each filtration degree, 
  yielding a generic-map complex
$$
X_\bullet^{(k)} : \Deltagen\op\to\Grpd .
$$

To get the outer face maps, the idea is simply to restrict (since by construction
all the maps $X_1^{(k)}\into X_1^{(k+1)}$ are monos).  We need to check that
an outer face map 
applied to a simplex in $X_n^{(k)}$ again belongs to $X_{n-1}^{(k)}$.
This will be the content of Proposition~\ref{prop:dpresnondeg} below.
Once we have done that, it is clear that we have a sequence of
cULF maps
$$
X_\bullet^{(0)} \into X_\bullet^{(1)} \into \cdots \into X
$$
and we shall see that $X_\bullet^{(0)}$ is the constant simplicial space $X_0$.
\end{blanko}

\begin{prop}\label{prop:dpresnondeg}
  In a tight decomposition space $X$, face maps preserve length: 
  precisely, for any face map $d: X_{n+1} \to X_n$,
  if $\sigma\in X_{n+1}^{(k)}$, then $d\sigma \in X_n^{(k)}$.
\end{prop}

\begin{proof}
  Since the length of a simplex only refers only to its long
  edge, and since a generic face map does not
  alter the long edge, it is enough to treat the case of
  outer face maps, and by symmetry it is enough to treat
  the case of $d_\top$.
  Let $f$ denote the long edge of $\sigma$.  
  Let $\tau$ denote the triangle ${d_1}^{n-1}\sigma$.
  It has long 
  edge $f$ again.  Let $u$ and $v$ denote the short edges of $\tau$,
  \begin{equation*}
    \xymatrixrowsep{10pt}
  \xymatrixcolsep{24pt}
\xymatrix @!=0pt {
  & \cdot \ar[rddd]^v & \\
  &&\\
  & \tau &\\
  \cdot \ar[ruuu]^u \ar[rr]_f && \cdot
  }
  \end{equation*}
  that is $v= d_\bot \tau = {d_\bot}^n \sigma$
  and $u= d_\top \tau$, the long edge of $d_\top \sigma$.
  The claim is that if $\ell(f) \leq k$, then $\ell (u)\leq k$.
  If we were in the category case, this would be true since any
  decomposition of $u$ could be turned into a decomposition of $f$
  of at least the same length, simply by postcomposing with $v$.
  In the general case, we have to invoke the decomposition-space
  condition to glue with $\tau$ along $u$.  Precisely, for any 
  simplex $\kappa \in X_w$ with long edge $u$ we can obtain a 
  simplex $\kappa \#_u \tau \in X_{w1}$
  with long edge $f$: since $X$ is a decomposition space,
  we have a pullback square
  $$\xymatrix{
  \kappa \#_u \tau \ar@{}[r]|\in &X_{w1} \drpullback \ar[r] \ar[d] & X_w \ar[d]^g 
    & \ar@{}[l]|\ni \kappa\\
  \tau \ar@{}[r]|\in & X_2 \ar[r]_{d_\top} \ar[d]_{d_1} & X_1 & 
  \ar@{}[l]|\ni u\\
  f \ar@{}[r]|\in &X_1}
  $$
  and $d_\top \tau = u = g(\kappa)$, giving us the desired simplex in
  $X_{w1}$.  With this construction, any simplex $\kappa$ of 
  length $>k$ violating $\ell(u)=k$ 
  (cf.~the characterisation of length given in (1) of 
  Proposition~\ref{prop:alt-length})
  would also yield a simplex $\kappa \#_u \tau$ (of at least the
  same length) violating $\ell(f) =k$.
\end{proof}

\begin{prop}
  In a \FILT decomposition space $X$, for any generic map $g:X_n \to X_1$ we have
  $$\xymatrix{
      X_0\ar[d]_{s_0} & \ar[l]_= X_0 \ar[d] \dlpullback \\
     X_1 & \ar[l]^g X_n .
  }$$
\end{prop}
\begin{proof}
  By Corollary~\ref{cor:longedgedegen}, if the long edge of $\sigma\in X_n$ is
  degenerate, then $\sigma$ is in the image of the maximal degeneracy map $X_0 \to X_n$.
\end{proof}
\begin{cor}
  For a tight decomposition space, $X_n^{(0)} = X_0$, $\forall n$.
\end{cor}

\begin{blanko}{Coalgebra filtration.}\label{coalgebrafiltGrpd}
  If $X$ is a \FILT decomposition space, the sequence of
  cULF maps
  $$
  X_\bullet^{(0)} \into X_\bullet^{(1)} \into \cdots \into X 
  $$
  defines coalgebra homomorphisms
  $$
  \Grpd_{/X_1^{(0)}} \to \Grpd_{/X_1^{(1)}} \to \cdots \to \Grpd_{/X_1} 
  $$
  which clearly define a coalgebra filtration of $\Grpd_{/X_1}$.
  
  Recall that a filtered coalgebra is called connected if its $0$-stage
  coalgebra is the trivial coalgebra (the ground ring).  In the present
  situation the $0$-stage is $\Grpd_{/X_1^{(0)}} \simeq \Grpd_{/X_0}$, so we
  see that $\Grpd_{/X_1}$ is connected if and only if $X_0$ is contractible.
  
  On the other hand, the $0$-stage elements are precisely the degenerate arrows,
  which almost tautologically are group-like.  Hence the incidence coalgebra of
  a \FILT decomposition space will always have the property that the
  $0$-stage is spanned by group-like elements.  For some purposes, this property
  is nearly as good as being connected (cf.~\cite{Kock:1411.3098}, 
  \cite{Kock:1512.03027} for this 
  viewpoint in the context of renormalisation).
\end{blanko}

\begin{blanko}{Grading.}\label{grading}
  Given a $2$-simplex $\sigma\in X_2$ in a complete decomposition space $X$,
  it is clear that we have
  $$
  \ell(d_2\sigma) + \ell(d_0 \sigma) \leq \ell(d_1\sigma)
  $$
  generalising the case of a category, where $f=ab$  implies 
  $\ell(a)+\ell(b)\leq\ell(f)$.
  In particular, the following configuration of arrows 
  illustrates that one does not in general have equality:
  $$
  \xymatrix @! @R=10pt @C=5pt {
  && \cdot \ar[rr] && \cdot \ar[rrd] && \\
  \cdot \ar[rru] \ar[rrrrrr]^f \ar[rrrd]_a &&&&&& \cdot \\
  &&& \cdot \ar[rrru]_b &&&
  }
  $$
  Provided none of the arrows can be decomposed further, we have
  $\ell(f) = 3$, but $\ell(a) = \ell(b) =1$.
  For the same reason, the length filtration is not in general 
  a grading:  $\Delta(f)$ contains the term $a\tensor b$
  of degree splitting $1+1 < 3$.
  Nevertheless, it is actually common in examples of interest
  to have a grading: this happens when all maximal chains composing to a given 
  arrow $f$ have the same length, $\ell(f)$.  All the examples of 
  Section~\ref{sec:ex} will have this property.
  
  The abstract formulation of the condition for the length filtration to
  be a grading is this:
  For every $k$-simplex $\sigma \in X_k$ with long edge $a$ and principal edges
  $e_1,\ldots,e_k$, we have
  $$
  \ell (a) = \ell(e_1) + \cdots + \ell(e_k) .
  $$
  Equivalently, for every $2$-simplex $\sigma\in X_2$
  with long edge $a$ and short edges
  $e_1,e_2$, we have
  $$
  \ell (a) = \ell(e_1) +\ell(e_2) .
  $$
  
  The length filtration is a grading if and only if the functor $\ell :X_1 \to 
  \N$ extends to a simplicial map to the nerve of the monoid $(\N,+)$ (this map is
  rarely cULF though).  The monoid $(\N,+)$ is studied further in Example~\ref{ex:N&L}
  below, and the special case where $\ell: X \to (\N,+)$ is cULF in 
  \ref{ex:cULF/N}.
  
  If $X$ is the nerve of a poset $P$, then the length filtration is a grading
  if and only if $P$ is {\em ranked}, i.e.~for any $x,y \in P$, every maximal chain
  from $x$ to $y$ has the same length~\cite{Stanley}.
\end{blanko}

%% file: finite-decomp-and-sectioncoeff.tex
\def\inputfile{finite-decomp-and-sectioncoeff.tex}

\section{Locally finite decomposition spaces}
\label{sec:findec}

In order to be able to take cardinality of the $\Grpd$-coalgebra obtained from a
decomposition space $X$ to get a coalgebra at the numerical level (vector
spaces), we need to impose certain finiteness conditions.  First of all, just
for the coalgebra structure to have a cardinality, we need $X$ to be {\em
locally finite} (\ref{finitary}) (but it is not necessary that $X$ be complete).
Secondly, in order for  \M inversion to descend, what we need in addition is
precisely the filtration condition (which in turn assumes completeness).  We
shall define a {\em \M decomposition space} to be a locally finite \FILT
decomposition space (\ref{M}).

\subsection{Incidence (co)algebras and section coefficients}

\begin{blanko}{Locally finite decomposition spaces.}\label{finitary}
  A decomposition space $X:\Delta\op \to \Grpd$ is called
  \emph{locally finite} if $X_1$ is locally finite and both 
  $s_0:X_0 \to X_1$ and $d_1:X_2 \to X_1$ are finite maps.
\end{blanko}



\begin{lemma}
  Let $X$ be a decomposition space.
  \begin{enumerate}
    \item If $s_0:X_0\to X_1$ is finite then so are all degeneracy maps $s_i:X_n\to X_{n+1}$.
    \item If $d_1:X_2\to X_1$  is finite then so are all  generic face maps 
          $d_j:X_{n}\to X_{n-1}$, $j\neq0,n$.
    \item  $X$ is locally finite if and only if $X_n$ is locally finite for every
      $n$ and $g:X_m\to X_n$ is finite for every generic map $g:[n]\to[m]$ in $\Delta$.
  \end{enumerate}
\end{lemma}

\begin{proof}
  Since finite maps are stable under pullback (Lemma\ref{lem:finitemaps}),
  both (1) and (2)
  follow from Lemma~\ref{lem:s0d1}.
  
  Re (3): If $X$ is locally finite, then by definition $X_1$ is locally finite, and 
  for each $n\in \N$ the unique generic map $X_n \to X_1$ is finite by (1) or 
  (2).  It follows that $X_n$ is locally finite by Lemma~\ref{locfinbase}.
  The converse implication is trivial.
\end{proof}

%
%
\begin{BM}
  If $X$ is the nerve of a poset $P$, then it is locally finite in the above
  sense if and only if it is locally finite in the usual sense of posets~\cite{Stanley},
  viz.~for every $x, y\in P$, the interval $[x,y]$ is finite.  The points in
  this interval parametrise precisely the two-step factorisations of the unique 
  arrow $x\to y$, so this condition amounts to $X_2 \to X_1$ having finite 
  fibre over $x \to y$.  (The condition $X_1$ locally finite is void in this 
  case, as any discrete set is locally finite;  the condition on $s_0 : X_0 \to 
  X_1$ is also void in this case, as it is always just an inclusion.)
  
  For posets,
  `locally finite' implies `locally finite length'.  (The converse is not true:
  take an infinite set, considered as a discrete poset, and adjoin a top and a 
  bottom element: the result is of locally finite length but not locally finite.)
  Already for categories, it is not true that locally finite implies locally
  finite length: for example the strict nerve of a finite group is locally
  finite but not of locally finite length.
\end{BM}

\begin{lemma}
  If a decomposition space $X$ is locally finite then so are $\Dec_\bot(X)$ and
  $\Dec_\top(X)$.
\end{lemma}

\begin{blanko}{Numerical incidence algebra.}
  It follows from \ref{finitetypespan} that, for any locally finite decomposition space
  $X$, the comultiplication maps
  \begin{eqnarray*}
    \Delta_n: \Grpd_{/ X_1} & \longrightarrow & 
    \Grpd_{/ X_1\times X_1\times \dots\times X_1  } 
  \end{eqnarray*}
  given for  $n\geq 0$ by the spans
  $$
  \xymatrix@R-15pt{
   X_1  & \ar[l]-_{m}  X_n\ar[r]^-{p} &  X_1\times X_1\times \dots\times X_1  
  }
  $$
  restrict to linear functors
  \begin{eqnarray*}
    \Delta_n: \grpd_{/ X_1} & \longrightarrow & 
    \grpd_{/ X_1\times X_1\times \dots\times X_1  } .
  \end{eqnarray*}
Now we can take
cardinality of the linear functors
$$
\grpd
\stackrel{\epsilon}\lTo
\grpd_{/X_1}
\stackrel{\Delta}\rTo 
\grpd_{/X_1\times X_1}
$$
to obtain a coalgebra structure,
$$
\Q\stackrel{\norm{\epsilon}}
\lTo\Q_{\pi_0 X_1}\stackrel{\norm{\Delta}}\rTo \Q_{\pi_0 X_1}\tensor\Q_{\pi_0 X_1} 
$$
termed the \emph{numerical incidence coalgebra} of $X$.
\end{blanko}

\begin{blanko}{Morphisms.}\label{rmk:morphisms}
  It is worth noticing that for {\em any} conservative ULF functor $F: Y \to X$ between locally finite decomposition spaces,
  the induced coalgebra homomorphism $F\lowershriek: \Grpd_{/Y_1}\to
  \Grpd_{/X_1}$ restricts to a functor $\grpd_{/Y_1} \to \grpd_{/X_1}$.  In other
  words, there are no further finiteness conditions to impose on morphisms.
\end{blanko}

\begin{blanko}{Incidence bialgebras.}
  If a locally finite decomposition space is monoidal, then the incidence coalgebra
  is in fact a {\em bialgebra}.  Note that since the algebra structure in our
  setting is given simply by a lowershriek map, by the previous remark~\ref{rmk:morphisms} there are
  no finiteness conditions needed in order for it to descend to the numerical
  level.  
  
  We also have the notion of incidence {\em algebra}, defined as the 
  (profinite-dimensional) linear dual of the incidence coalgebra.
  In the presence of a monoidal structure on the decomposition space, this
  causes a potential ambiguity regarding algebra structures.  We make the
  convention that {\em incidence bialgebra} always refers to the incidence
  {\em coalgebra} with its extra multiplication.
\end{blanko}

\begin{blanko}{Numerical convolution product.}
  By duality, if $X$ is locally finite, the convolution
  product descends to the profinite-dimensional vector space $\Q^{\pi_0 X_1}$
  obtained by taking cardinality of $\grpd^{X_1}$.  It follows from the general
  theory of homotopy linear algebra (see Appendix, and specifically 
  \ref{metacard}) that the cardinality of
  the convolution product is the linear dual of the cardinality of the 
  comultiplication.  Since it is the same span that defines the comultiplication
  and the convolution product, it is also the exact same matrix
  that defines the cardinalities of these two maps.
  It follows that the structure constants for the convolution product (with 
  respect to the pro-basis $\{\delta^x\}$) are the same as the structure 
  constants for the comultiplication (with respect to the basis
  $\{\delta_x\}$).
  These are classically called the section coefficients, and we proceed to
  derive formulae for them in simple cases.
\end{blanko}

%


\label{sec:tioncoeff}
Let $X$ be a locally finite decomposition space.
The comultiplication at the objective level
\begin{eqnarray*}
  \grpd_{/X_1} & \longrightarrow & \grpd_{/X_1\times X_1}  \\
  \name f & \longmapsto & \big[ R_f : (X_2)_f \to X_2 \to X_1 \times X_1 \big]
\end{eqnarray*}
yields a comultiplication of vector spaces by
taking cardinality (remembering that $\norm{\name f} = \delta_f$):
\begin{eqnarray*}
\Q_{\pi_0 X_1} & \longrightarrow & \Q_{\pi_0 X_1} \tensor \Q_{\pi_0 X_1} \\
\delta_f & \longmapsto & \norm{R_f} \\
&=& 
\int^{(a,b)\in X_1\times X_1} \norm{(X_2)_{f,a,b}} \delta_a \tensor \delta_b \\
&=& \sum_{a,b}  \norm{(X_1)_{[a]}}\norm{(X_1)_{[b]}} \norm{(X_2)_{f,a,b}} 
\delta_a \tensor \delta_b.
\end{eqnarray*}
where $(X_2)_{f,a,b}$ is the fibre over the three face maps.
The integral sign is a sum weighted by homotopy groups.
These weights together with the cardinality of the triple fibre are called
the {\em section coefficients}, denoted
$$
c^f_{a,b} := \norm{(X_2)_{f,a,b}} \cdot \norm{(X_1)_{[a]}}\norm{(X_1)_{[b]}} .
$$


\bigskip

In the case where $X$ is a Segal space (and even more, when $X_0$ is a 
$1$-groupoid), we can be very explicit about the section coefficients.
For a Segal space we have $X_2 \simeq X_1 \times_{X_0} X_1$, which helps
to compute the fibre of $X_2 \to X_1 \times X_1$:

\begin{lemma}\label{lem:Omega(y)}
  The pullback
  $$
  \xymatrix{
  S \drpullback \ar[r] \ar[d] & X_1 \times_{X_0} X_1 \ar[d] \\
  1 \ar[r]_-{\name{a,b}}& X_1 \times X_1
  }
  $$
  is given by
  $$
   S \ = \ 
  \begin{cases}
    \Omega(X_0,y) & \text{ if } d_0 a \simeq y \simeq d_1 b \\
    0 & \text { else.}
  \end{cases}
  $$
\end{lemma}
\begin{proof}
  We can compute the pullback as
  $$
  \xymatrix{
  S \drpullback \ar[r] \ar[d] & X_1 \times_{X_0} X_1 \drpullback\ar[d]  
  \ar[r] & X_0 \ar[d]^{\text{diag}}\\
  1 \ar[r]_-{\name{a,b}}& X_1 \times X_1 \ar[r]_-{d_0\times d_1} & X_0\times X_0 ,
  }
  $$
  and the result follows since in general 
  $$\xymatrix{
  A\times_C B \drpullback \ar[d] \ar[r] & C \ar[d]^{\text{diag}} \\
  A\times B \ar[r] & C\times C.
  }$$
\end{proof}

\begin{cor}
  Suppose $X$ is a Segal space, and that $X_0$ is a $1$-groupoid. 
  Given $a,b,f\in X_1$ such that $d_0a\cong
  y\cong d_1b$ and $ab=f$, then we have
  $$
  (X_2)_{f,a,b} = \Omega(X_0,y) \times \Omega(X_1,f) .
  $$
\end{cor}
\begin{proof}
  In this case, since $X_0$ is a $1$-groupoid, the fibres of the diagonal map $X_0 \to X_0 
  \times X_0$ are $0$-groupoids.  Thus the fibre of the previous lemma is 
  the
  discrete space $\Omega(X_0,y)$.  When now computing the fibre
  over $f$, we are taking that many copies of the loop space of $f$.
\end{proof}

\begin{cor}\label{sectioncoeff!}
  With notation as above, the section coefficients for a locally finite Segal
  $1$-groupoid are
  $$
  c_{a,b}^{ab} = \frac{\norm{\Aut(y)}\norm{\Aut(ab)}}{\norm{\Aut(a)}\norm{ \Aut(b)}} .
  $$
\end{cor}

Coassociativity of the incidence coalgebra says that the section coefficients 
$\{c_{a,b}^{ab}\}$ form a {\em 2-cocycle},
$$c^{ab}_{a,b}c^{abc}_{ab,c}=c^{bc}_{b,c}c^{abc}_{a,bc}.$$
In fact this cocycle is cohomologically trivial, given by the 
coboundary of a 1-cochain,
$$c^{ab}_{a,b}=\partial(\phi)(a,b)=\phi(a)\phi(ab)^{-1}\phi(b),$$
In fact, if one fixes $s,t$ such that $s+t=1$, the 1-cochain
may be taken to be
$$
\phi(x\stackrel a\to y)=\frac{|\Aut(x)|^s|\Aut(y)|^t}{|\Aut(a)|}
$$.

\begin{blanko}{`Zeroth section coefficients': the counit.}
 Let us also say a word about the zeroth section coefficients, i.e.~the
 computation of the counit: the main case is when $X$ is complete (in the sense
 that $s_0$ is a monomorphism).  In this case, clearly we have
  $$
  \epsilon(f) = \begin{cases}
    1 & \text{ if $f$ degenerate } \\
    0 & \text{ else.}
  \end{cases}
  $$
  If $X$ is Rezk complete, the first condition is equivalent to being invertible.

  The other easy case is when $X_0=*$.  In this case
  $$
  \epsilon(f) = \begin{cases}
    \Omega(X_1,f) & \text{ if $f$ degenerate } \\
    0 & \text{ else.}
  \end{cases}
  $$
\end{blanko}

\begin{blanko}{Example.}
  The strict nerve of a $1$-category $\CC$ is a decomposition space which is
  discrete in each degree.  The resulting coalgebra at the numerical level
  (assuming the due finiteness conditions) is the coalgebra of
  Content--Lemay--Leroux~\cite{Content-Lemay-Leroux}, and if the category is
  just a poset, that of Rota et al.~\cite{JoniRotaMR544721}.

  For the fat nerve $X$ of $\CC$, we find
  \label{lem:h*h}
$$h^a * h^b = \begin{cases}
  \Omega(X_0,y) \, h^{ab} & \text{ if $a$ and $b$ composable at $y$} \\
  0 & \text{ else,}
\end{cases}
$$
\noindent
as follows from \ref{lem:Omega(y)}.
Note that the cardinality of the representable $h^a$ is generally different
from the canonical basis element $\delta^a$.
\end{blanko}

%
%

%
%
%
%
%
%

\begin{blanko}{Finite support.}\label{fin-supp}
  It is also interesting to consider the subalgebra of the incidence algebra
  consisting of functions with finite support, i.e.~the full subcategory
  $\grpd^{X_1}_{fin.sup} \subset \grpd^{X_1}$, and numerically $\Q^{\pi_0
  X_1}_{fin.sup} \subset \Q^{\pi_0 X_1}$.  Of course we have canonical
  identifications $\grpd^{X_1}_{fin.sup} \simeq \grpd_{/X_1}$, as well as
  $\Q^{\pi_0 X_1}_{fin.sup} \simeq \Q_{\pi_0 X_1}$, but it is important to keep
  track of which side of duality we are on.
  
  That the decomposition space is locally finite is not the appropriate condition
  for these subalgebras to
  exist.  Instead the requirement is that $X_1$ be locally finite and
  the functor
  $$
  X_2 \to X_1 \times X_1
  $$ be finite.  (This is always the case for a locally finite Segal $1$-groupoid, by 
  Lemma~\ref{lem:Omega(y)}.)
  Similarly, one can ask for the convolution unit to have finite support,
  which is to require $X_0 \to 1$ to be a finite map.
 
  Dually, the same conditions ensure that comultiplication and counit extend
  from $\grpd_{/X_1}$ to $\Grpd_{/X_1}^{rel.fin}$, which numerically is
  some sort of vector space of summable infinite linear combinations.
  An example of this situation is given by the bialgebra of $P$-trees, whose 
  comultiplication does extend to $\Grpd_{/X_1}^{rel.fin}$.  Importantly, this is
  the home for the Green function, an infinite (homotopy)
  sum of trees, and for the Fa\`a di Bruno formula it satisfies, which does not hold
  for any finite truncation.  See \cite{GalvezCarrillo-Kock-Tonks:1207.6404} for these 
  results.

\end{blanko}

\begin{blanko}{Examples.}\label{categoryalg}
  If $X$ is the strict nerve of a $1$-category $\CC$, then 
  the finite-support convolution algebra is precisely the {\em category algebra}
  of $\CC$.  (For a finite category, of course
  the two notions coincide.)  
  
  Note that the convolution unit is
  $$
  \epsilon = \sum_x \delta^{\id_x} = \begin{cases}
    1 & \text{for id arrows} \\
    0 & \text{ else,}
  \end{cases}
  $$
  the sum of all indicator functions of identity arrows, so it will be finite if
  and only if the category has only finitely many objects.
  
  In the case of the fat nerve of a $1$-category, the finiteness condition
  for comultiplication is
  implied by the condition that every object has a finite automorphism
  group (a condition implied by local finiteness).
  On the other hand,
  the convolution unit has finite support precisely when there is only a finite
  number of isoclasses of objects, already a more drastic condition.
  Note the `category algebra' interpretation: compared to the usual category 
  algebra there is a symmetry factor (cf.~\ref{lem:h*h}):
  $$h^a * h^b = \begin{cases}
  \Omega(X_0,y) \, h^{ab} & \text{ if $a$ and $b$ composable at $y$} \\
  0 & \text{ else.}
\end{cases}
$$
  
  Finally, the finite-support incidence algebras are important in the case
  of the Waldhausen $S$-construction: they are the Hall algebras, cf.~\ref{sec:Wald}
  below.  The finiteness conditions are then homological, namely
  finite $\operatorname{Ext}^0$ and $\operatorname{Ext}^1$.
\end{blanko}

%% file: M-condition.tex
\def\inputfile{M-condition.tex}

\subsection{\M decomposition spaces}

\label{sec:M}

\begin{lemma}\label{lem:comp-finite}
  If $X$ is a complete decomposition space then the following conditions are
  equivalent
  \begin{enumerate}
  \item $d_1:X_2\to X_1$ is finite.

  \item $d_1:\nondeg X_2\to X_1$ is finite.

  \item $d_1^{r-1}:\nondeg X_r\to X_1$ is finite for all $r\geq 2$.
  \end{enumerate}
\end{lemma}
\begin{proof}
  We show the first two conditions are equivalent; the third is similar.
  Using the word notation of \ref{w} we consider the map
  $$
  \nondeg X_2 + \nondeg X_1 + \nondeg X_1 + X_0\xrightarrow{~\simeq~}
  \nondeg X_2 + X_{0a} + X_{a0} + X_{00} \xrightarrow{~=~}X_2\xrightarrow{~d_1~}X_1
  $$
  Thus $d_1:X_2\to X_1$ is finite if and only if the restriction of this map to
  the first component, $d_1:\nondeg X_2\to X_1$, is finite.  By completeness the
  restrictions to the other components are finite (in fact, mono).
\end{proof}

\begin{cor}
  A complete decomposition space $X$ is locally finite if and only if $X_1$ is
  locally finite and $d_1^{r-1}:\nondeg X_r\to X_1$ is finite for all $r\geq 2$.
\end{cor}

\begin{blanko}{\M condition.}\label{M}
  A complete decomposition space $X$ is called {\em \M}
  if it is locally finite and \FILT (i.e.~of locally finite length).
  It then follows that  the restricted composition map
  $$
  \sum_r{d_1}^{r-1}:\sum_r \nondeg X_r\to X_1
  $$ 
  is finite.  In other words, the spans defining $\Phieven$ and $\Phiodd$ are of
  finite type, and hence descend to the finite groupoid-slices $\grpd_{/X_1}$.
  In fact we have:
\end{blanko}

\begin{lemma}\label{lem:oldcharacterisationofM}
  A complete decomposition space $X$ is \M if and only if $X_1$ is locally
  finite and the restricted composition map
  $$
  \sum_r{d_1}^{r-1}:\sum_r \nondeg X_r\to X_1
  $$ 
  is finite. 
\end{lemma}
  
\begin{proof}
  `Only if' is clear.  Conversely, if the map $m:\sum_r{d_1}^{r-1}:\sum_r
  \nondeg X_r\to X_1$ is finite, in particular for each individual $r$ the map
  $\nondeg X_r\to X_1$ is finite, and then also $X_r \to X_1$ is finite, by
  Lemma~\ref{lem:comp-finite}.  Hence $X$ is altogether locally finite.  But it
  also follows from finiteness of $m$ that for each $a\in X_1$, the fibre
  $(\nondeg X_r)_a$ must be empty for big enough $r$, so the filtration
  condition is satisfied, so altogether $X$ is \M.
\end{proof}

\begin{BM}
  If $X$ is a Segal space, the \M condition says that for each arrow $a\in X_1$,
  the factorisations of $a$ into nondegenerate $a_i\in \nondeg X_1$ have
  bounded length.   In particular, if $X$ is the strict nerve of a $1$-category, then it is \M in the sense of
  the previous definition if and only if it is \M in the sense of Leroux.  (Note
  however that this would also have been true if we had not included the
  condition that $X_1$ be locally finite (as obviously this is automatic for any
  discrete set).  We insist on including the condition $X_1$ locally finite
  because it is needed in order to have a well-defined cardinality.)
\end{BM}

\begin{blanko}{Filtered coalgebras in vector spaces.}
  A \M decomposition space is in particular length-filtered.
  The coalgebra filtration (\ref{coalgebrafiltGrpd}) at the objective level
  $$
  \Grpd_{/X_1^{(0)}} \to \Grpd_{/X_1^{(1)}} \to \cdots \to \Grpd_{/X_1} 
  $$
  is easily seen to descend to the finite-groupoid coalgebras:
    $$
  \grpd_{/X_1^{(0)}} \to \grpd_{/X_1^{(1)}} \to \cdots \to \grpd_{/X_1}  ,
  $$
  and taking cardinality then yields a coalgebra filtration at
  the numerical level too.
  From the arguments in \ref{coalgebrafiltGrpd}, it follows that this
  coalgebra filtration
  $$
  C_0 \into C_1 \into \cdots \into C
  $$
  has the property that $C_0$ is generated by group-like elements. (This property
  is found useful in the context of perturbative
  renormalisation~\cite{Kock:1411.3098}, \cite{Kock:1512.03027}, where it serves
  as basis for recursive arguments, as an alternative to the more
  common assumption of connectedness.)  Finally, if $X$ is a graded \M
  decomposition space, then the resulting coalgebra at the algebraic level is
  furthermore a graded coalgebra.
\end{blanko}

The following is an immediate corollary to \ref{prop:FILTSegal=Rezk}.  It
extends the classical fact that a \M category in the sense of Leroux does not
have non-identity invertible arrows~\cite[Lemma~2.4]{LawvereMenniMR2720184}.

\begin{cor}\label{cor:MSegal=Rezk}
  If a \M decomposition space $X$ is a Segal space, then it is Rezk complete.
\end{cor}

%% file: M-inversion-algebraic-level.tex
\def\inputfile{M-inversion-algebraic-level.tex}

\begin{blanko}{\M inversion at the algebraic level.}
  Assume $X$ is a locally finite complete decomposition space.
  The span
  $
  \xymatrix@R-15pt{
   X_1  & \ar[l]_=  X_1\ar[r] &  1}
  $
  defines the zeta functor (cf.~\ref{zeta}), which as a presheaf is 
  $\zeta=\int^t h^t$, the homotopy sum of the representables.
  Its cardinality is the usual zeta 
  function in the incidence algebra $\Q^{\pi_0X_1}$.

  The spans $\xymatrix@R-15pt{
   X_1  & \ar[l]  \nondeg X_r\ar[r] &  1}$ define the Phi functors
  \begin{eqnarray*}
    \Phi_r: \Grpd_{/ X_1} & \longrightarrow & \Grpd ,
  \end{eqnarray*}
  with $\Phi_0=\epsilon$. By Lemma~\ref{lem:comp-finite},
  these functors descend to 
    \begin{eqnarray*}
    \Phi_r: \grpd_{/ X_1} & \longrightarrow & \grpd ,
  \end{eqnarray*}
 and we can take cardinality to obtain functions
$
\norm{\zeta}: \pi_0(X_1) \to \Q
$
and $\norm{\Phi_r} : \pi_0(X_1) \to \Q$, elements in the incidence algebra 
$\Q^{\pi_0X_1}$.

Finally, when $X$ is furthermore assumed to be \M,
we can take cardinality of the abstract \M
inversion formula of \ref{thm:zetaPhi}:
\end{blanko}

\begin{thm}\label{thm:|M|}
  If $X$ is a \M decomposition space,
  then the cardinality of the zeta functor, $\norm{\zeta}:\Q_{\pi_0 X_1}\to\Q$,
  is convolution invertible with inverse $\norm{\mu}:= \norm{\Phieven} - \norm{\Phiodd}$:
  $$
  \norm\zeta * \norm\mu = \norm\epsilon = \norm\mu * \norm\zeta .
  $$
\end{thm}


%% file: ex.tex
\def\inputfile{ex.tex}

It is characteristic for the classical theory of incidence (co)algebras of
posets that most often it is necessary to impose an equivalence relation on the
set of intervals in order to arrive at the interesting `reduced' (co)algebras.
This equivalence relation may be simply isomorphism of posets, or equality of
length of maximal chains as in binomial posets \cite{Doubilet-Rota-Stanley}, or
it may be more subtle order-compatible relations \cite{Dur:1986},
\cite{Schmitt:1994}.  Content, Lemay and Leroux~\cite{Content-Lemay-Leroux}
remarked that in some important cases the relationship between the original
incidence coalgebra and the reduced one amounts to a conservative ULF functor,
although they did not make this notion explicit.  From our global simplicial
viewpoint, we observe that very often these cULF functors arise from decalage,
but often of a decomposition space which not a poset and perhaps not even a
Segal space.

\begin{blanko}{Decomposition spaces for the classical series.}
  Classically, the most important incidence algebras 
  are the power series representations.  From the
  perspective of the objective method, these representations appear as
  cardinalities of various monoidal structures on species, realised as incidence
  algebras with $\infty$-groupoid coefficients (or at least $1$-groupoid 
  coefficients).
  We list six examples illustrating some of the various kinds of
  generating functions listed by Stanley~\cite{Stanley:MR513004} (see also
  D\"ur~\cite{Dur:1986}). 
\begin{enumerate}
  \item Ordinary generating functions, the zeta function being $\zeta(z)= 
  \sum_{k\geq 0} z^k$.  This comes from ordered
  sets and ordinal sum, and the incidence algebra is that of ordered species
  with the ordinary product.
  
  \item Exponential generating functions, the zeta function being
  $\zeta(z)=\sum_{k\geq 0} \frac{z^k}{k!}$.
  Objectively, there are two versions of this: one coming from the standard
  Cauchy product of species, and one coming from the shuffle product of 
  $\mathbb L$-species (in the sense of \cite{Bergeron-Labelle-Leroux}).

  \item Ordinary Dirichlet series, the zeta function being $\zeta(z)=\sum_{k>0} k^{-s}$.
  This comes from ordered sets with the cartesian product.

  \item `Exponential' Dirichlet series, the zeta function being $\zeta(z)=
  \sum_{k>0} \frac{k^{-s}}{k!}$.
  This comes from the Dirichlet product of arithmetic species 
  \cite{Baez-Dolan:zeta}, also called the arithmetic product~\cite{Maia-Mendez:0503436}.

  \item $q$-exponential generating series, with zeta function $\zeta(z)=
  \sum_{k\geq 0} \frac{z^k}{[k]!}$.  This comes from 
  the Waldhausen $S$-construction on the category of finite vector 
  spaces.  The incidence algebra is that of $q$-species with a version of the
  external product of Joyal--Street~\cite{Joyal-Street:GLn}.

  \item Some variation with zeta function $\zeta(z)=
  \sum_{k\geq 0} \frac{z^k}{\#\Aut(\F_q^k)}$, which arises from 
  $q$-species with the `Cauchy' product studied  by 
  Morrison~\cite{Morrison:0512052}.
\end{enumerate}
Of these examples, only (1) and (3) have trivial section coefficients
and come from a \M category in the sense of Leroux.  We proceed to the 
details.

\end{blanko}

\subsection{Additive examples}

We start with several easy examples that serve to 
reiterate the importance of having incidence algebras of posets, monoids and 
monoidal $\infty$-groupoids on the same footing, with conservative ULF functors connecting them.

\begin{blanko}{Linear orders and the additive monoid.}\label{ex:N&L}
  Let $\mathbf L$ denote the nerve of the poset $(\N,\leq)$, and let
  $\mathbf N$ be the nerve of the additive monoid $(\N,+)$.
Imposing the equivalence relation `isomorphism of intervals' on the incidence coalgebra of $\mathbf L$ gives that of $\mathbf N$, and Content--Lemay--Leroux observed that this reduction is induced by a
  conservative ULF functor $r:\mathbf L \to \mathbf N$ sending $a\leq
  b$ to $b-a$. In fact we have:
\begin{lemma}
There is an isomorphism of simplicial sets 
  \begin{align*}
   \Dec_\bot (\mathbf N)  & \stackrel\simeq\longrightarrow  \mathbf L  \\
   \intertext{given in degree $k$ by}
    (x_0,\dots,x_k) & \longmapsto  [x_0\leq x_0+x_1 \leq  \dots \leq x_0+\dots+x_k]
  \end{align*}
and the conservative ULF functor $r$ is isomorphic to the structure map   $$
d_\bot: \Dec_{\bot}(\mathbf N) \to \mathbf N,\qquad (x_0,\dots,x_k)\mapsto (x_1,\dots,x_k).$$
\end{lemma}


  %

The comultiplication on $\Grpd_{/\mathbf N_1}$ is given by 
  $$\Delta(\name n) = \sum_{a+b=n} \name a \tensor \name b$$ and, taking cardinality,
the classical incidence coalgebra is the vector space 
  spanned by symbols $\delta_n$ with comultiplication 
  $\Delta(\delta_n) = \sum\limits_{a+b=n} \delta_a \tensor \delta_b$.
  The incidence algebra is the profinite-dimensional vector space
  spanned by the symbols $\delta^n$ with 
  convolution product $\delta^a * \delta^b = \delta^{a+b}$, and is isomorphic
  to the ring of power series in one variable,
  \begin{eqnarray*}
    \Inc{\mathbf{N}} & \stackrel\simeq\longrightarrow & \Q[[z]]  \\
    \delta^n & \longmapsto & z^n \\
    (\N\stackrel f\to\Q) &\longmapsto& \sum f(n)\, z^n  .
  \end{eqnarray*}
\end{blanko}

\begin{blanko}{Powers.}
  As a variation of the previous example, fix $k\in \N$ and 
  let $\mathbf{L}^k$ denote the nerve of the poset $(\N^k, \leq)$ and let
  $\mathbf{N}^k$ denote the nerve of the monoid $(\N^k,+)$.  Again there
  is a cULF functor $\mathbf{L}^k \to \mathbf{N}^k$ defined by coordinatewise
  difference, and again it is given by decalage, via a natural identification
  $\mathbf{L}^k\simeq\Dec_\bot(\mathbf{N}^k)$.  
  In contrast to the $k=1$ case,
  this functor is {\em not} just dividing out by isomorphism of intervals 
  (treated next).
  The incidence algebra of $(\mathbf{N}^k,+)$ is the power series ring 
  in $k$ variables.
\end{blanko}

\begin{blanko}{Symmetric powers.}
   In general, let $M$ be a monoid, considered as a decomposition space via its
   nerve.  For fixed $k\in \N$, the power $M^k$ is again a monoid, as in the
   previous example, we denote its nerve by $X$.  The symmetric group $\mathfrak
   S_k$ acts on $X_1 = M^k$ by permutation of coordinates,  
  and acts on $X_n = X_1^n = (M^k)^n$
  diagonally.  There is induced a simplicial groupoid $X/\mathfrak S_k$ given by
  homotopy quotient: in degree $n$ it is the groupoid $\frac{X_1\times\cdots
  \times X_1}{\mathfrak S_k}$.  Since taking homotopy quotient of a group action
  is a lowershriek operation, it preserves pullbacks, so it follows that this
  new simplicial groupoid again satisfies the Segal condition.  (It is no 
  longer a monoid, though, since in degree zero we have the space $*/\mathfrak 
  S_k$.)
  Furthermore, it
  is easy to check that the quotient map $X \to X/\mathfrak S_k$ is cULF.
  This construction gives a supply of cULF maps which do not arise from decalage.
  
  We now return to the poset $\mathbf{L}^k = (\N^k, \leq)$.  A reduced incidence
  coalgebra is obtained by identifying isomorphic intervals.  The reduced
  incidence coalgebra is the incidence coalgebra of $(\N^k,+)/\mathfrak S_k$,
  and the reduction map is
  $$
  (\N^k,\leq  ) \simeq\Dec_\bot (\N^k,+) \longrightarrow (\N^k,+) \longrightarrow
  (\N^k,+)/\mathfrak S_k .
  $$
  
  The incidence coalgebra of $(\N^2,+)/\mathfrak S_2$ is the simplest example 
  we know of in which the length filtration does not agree with the coradical 
  filtration (see Sweedler~\cite{Sweedler} for this notion).
  The elements $(1,1)$ and $(2,0)\simeq (0,2)$ are clearly of 
  length $2$.  On the other hand, the element
  $$
  P := (1,1) - (2,0) - (0,2)
  $$
  is primitive, meaning
  $$
  \Delta(P) = (0,0) \tensor P + P \tensor (0,0)
  $$
  and is therefore of coradical filtration degree $1$.  (Note that in 
  $(\N^2,+)$ it is not true that $P$ is primitive:
  it is the symmetrisation that make the $(0,1)$ terms cancel out in
  the computation, to make $P$ primitive.)
\end{blanko}

\begin{blanko}{Injections and the monoidal $\infty$-groupoid of sets 
  under sum.}\label{ex:SI&B}\label{Cauchy}\label{I=DecB}
  Let $\mathbf{I}$ be the nerve of the category of finite sets and injections,
  and let $\mathbf{B}$ be the nerve of the monoidal $\infty$-groupoid $(\B, +, 0)$ of
  finite sets and bijections, or of the corresponding 1-object bicategory (see
  Proposition~\ref{prop:BM}). 
D\"ur~\cite{Dur:1986} noted that imposing the equivalence relation `having isomorphic complements' on the incidence coalgebra of $\mathbf{I}$ gives the binomial coalgebra.
  Again, we can see this reduction map as induced by a conservative
  ULF functor from a decalage:
\begin{lemma}
There is an equivalence of simplicial $\infty$-groupoids 
  \begin{align*}
   \Dec_\bot (\mathbf B)  & \stackrel\simeq\longrightarrow  \mathbf{I}  \\
   \intertext{given in degree $k$ by}
    (x_0,\dots,x_k) & \longmapsto  [x_0\subseteq x_0+x_1 \subseteq  \dots \subseteq x_0+\dots+x_k]
  \end{align*}
and a conservative ULF functor $r:\mathbf{I}\to \mathbf B$ is given by   $$
d_\bot: \Dec_{\bot}(\mathbf B) \to \mathbf B,\qquad (x_0,\dots,x_k)\mapsto (x_1,\dots,x_k).$$
\end{lemma}

  The isomorphism may also be represented diagrammatically using diagrams reminiscent
  of those in Waldhausen's $S$-construction (cf.~\ref{sec:Wald} below).
  As an example, both groupoids
$\mathbf{I}_3
$ and $\Dec_\bot(\mathbf B)_3=\mathbf B_4$ are equivalent to the groupoid
of  diagrams
  $$\xymatrix@R-0.8pc{
   &   &   &x_3 \ar[d] 
\\
   &   &x_2 \ar[d] \ar[r] & x_2+x_3 \ar[d] 
\\
   &x_1\ar[r]\ar[d] & x_1+x_2 \ar[d] \ar[r] & x_1+x_2+x_3 \ar[d] 
\\
x_0 \ar[r] & x_0+x_1 \ar[r] & x_0+x_1+x_2\ar[r] & x_0+x_1+x_2+x_3
  }$$
The face maps $d_i:\mathbf{I}_3\to \mathbf{I}_2$ 
and $d_{i+1}:\mathbf{B}_4\to \mathbf{B}_3$ both act by
deleting the column beginning $x_i$  and the row beginning $x_{i+1}$. 
In particular $d_\bot:\mathbf{I}\to\mathbf B$ deletes the bottom 
row, sending a string of injections to the sequence of 
successive complements $(x_1,x_2,x_3)$.  
We will revisit this theme in the treatment of the Waldhausen $S$-construction
(\ref{sec:Wald}).

\medskip

From Lemma~\ref{dec-mon-is-mon} and Proposition \ref{bialg-hm} we have:
\begin{lemma}
  Both $\mathbf{I}$ and $\mathbf{B}$ are monoidal decomposition spaces under
  disjoint union, and $\mathbf{I}\simeq\Dec_\bot (\mathbf B) \to \mathbf{B}$ is
  a monoidal functor inducing a (surjective) homomorphism of bialgebras
  $\Grpd_{/{\mathbf{I}}_1}\to\Grpd_{/\mathbf B_1}$.
\end{lemma}

Formula~\ref{sectioncoeff!} gives the comultiplication on $\Grpd_{/\mathbf B_1}$ as
  \begin{align*}
    \Delta&(\name S) 
\;= \;\sum_{A,B} \frac{\Bij(A+B,S)}{\Aut(A)\times \Aut(B)} \cdot \name A \tensor 
  \name B 
\;= 
\!\!\!\!\!\!\!\!\!\!
\sum_{
  \substack{A ,B\subset S \\  A\cup B = S,\; A\cap B = \emptyset}
  }
\!\!\!\!\!\!\!\!\!\!
 \name A \tensor \name B .
  \end{align*}
It follows that the convolution product on 
$\Grpd^{\B}$
is just the Cauchy product on groupoid-valued species
  $$
  (F*G)[S] = \sum_{A+B=S} F[A] \times G[B] .
  $$
For the representables, the formula says simply
$h^A*h^B=h^{A+B}$.

The decomposition space {\bf B} is locally finite, and taking cardinality gives
the classical binomial coalgebra, spanned by symbols $\delta_n$ with 
  $$\Delta(\delta_n) = \sum\limits_{a+b=n} \frac{n!}{a!\,b!}\,\delta_a \tensor 
  \delta_b.$$ 
As a bialgebra we have $(\delta_1)^n=\delta_n$ and one recovers the comultiplication from $\Delta(\delta_n)= \big( \delta_0 \tensor \delta_1 + \delta_1 \tensor \delta_0 \big)^n$.

Dually, the
incidence algebra $\Q^{\pi_0\mathbf B}$ is the profinite-dimensional vector space
spanned by symbols $\delta^n$ with convolution product
  $$
  \delta^a* \delta^b = \frac{n!}{a!\,b!}\,\delta^{a+b} ,
  $$
  This is isomorphic to the algebra $\Q[[z]]$, where $\delta^n$  corresponds to  $z^n/n!$ and the cardinality of a species $F$ corresponds to its exponential generating series.

\end{blanko}

\begin{blanko}{Finite ordered sets, and the shuffle product of $\mathbb L$-species.}
 Let $\mathbf{OI}$ denote (the fat nerve of) the  category of finite ordered sets and monotone injections. This is the decalage of the decomposition space $\mathbf{Z}$ with 
$\mathbf{Z}_n= \mathbf{OI}_{/\un n}$, the groupoid of arbitrary maps from a finite ordered set $S$ to $\un n$, or equivalently of $n$-shuffles of $S$.
  The incidence coalgebra of $\mathbf Z$ is the {\em shuffle coalgebra}.
  The section coefficients are the binomial coefficients, but  
  on the objective level the convolution algebra
  is the shuffle product of $\mathbb L$-species 
  (cf.~\cite{Bergeron-Labelle-Leroux}).  This example will be subsumed in
  our theory of restriction species, developed in Section~\ref{sec:restriction}.

There is a map $\mathbf Z\to \mathbf B$ that takes an $n$-shuffle to the underlying $n$-tuple of subsets, and the decalage of this functor is  
the cULF functor $\mathbf{OI}\to\mathbf{I}$ given by forgetting the order, see Example \ref{ex:SIOI}. 
\begin{lemma}
There is a commutative diagram of decomposition spaces and cULF functors,
$$\xymatrix{
 \mathbf{OI}\rto^-\simeq\dto& \Dec_\bot (\mathbf{Z})\dto \rto^-{d_\bot} & \mathbf{Z}\dto\\  
 \mathbf{I}\rto^-\simeq&\Dec_\bot (\mathbf{B}) \rto^-{d_\bot} & \mathbf{B}}$$
\end{lemma}



\bigskip

Let $A$ be a fixed set, an alphabet.  The slice category $\un \Delta^{\mathrm{inj}}_{/A}$
is the category of finite words (sequences) in $A$ and subword inclusions
(subsequences), cf.~Lothaire~\cite{Lothaire:MR675953} (see also
D\"ur~\cite{Dur:1986}).  Again it is the decalage of the {\em $A$-coloured
shuffle decomposition space} $\mathbf Z_A$ of $A$-words and complementary
subword inclusions.  More precisely, this space has in degree $k$ the groupoid
of $A$-words equipped with a non-necessarily-order-preserving map to $\un k$.
Precisely, the objects are spans
$$
\un k \leftarrow \un n \to A .
$$

The counit takes a subword inclusion to its 
complement word.  This gives the Lothaire shuffle algebra of words.
Again, it all amounts to observing that $A$-words admit a forgetful cULF functor to 
$1$-words, which is just the decomposition space $\kat Z$ from before, and that 
this in turn admits a cULF functor to $\mathbf B$.
%
%

Note the difference between $\mathbf Z_A$ and the free monoid on $A$: the 
latter is like allowing only the trivial shuffles, where the subword inclusions 
are  only concatenation inclusions.  In terms of the structure maps $\un n \to 
\un k$, the free-monoid nerve allows only monotone maps, whereas the shuffle 
decomposition space allows arbitrary set maps.
\end{blanko}

%
%
%
%

\begin{blanko}{Alternative, strict, version of $\mathbf{B}$.}\label{B/k}
  The following strict version of
  $\mathbf B$ will come in handy in the treatment of restriction species in 
  Section~\ref{sec:restriction}.
  First, an application of the Grothendieck construction gives an
  equivalence of groupoids over $\B$,
$$\xymatrix{
\B^k \ar[rd] \ar[rr]^\sim && \B_{/\underline k} \ar[ld] \\
&\B,&}
$$
that takes a $k$-tuple of finite sets to their disjoint union
$\sum_{i\in \underline k} S_i$ with the obvious projection map to 
$\underline k$.  Conversely, a map $S \to \underline k$
defines a the $k$-tuple $(S_1,\ldots,S_k)$ by taking fibres.
Unlike the groupoids $\B^k$, the groupoids $\B_{/\underline k}$ form a strict simplicial groupoid.
The generic maps (generated by inner faces and degeneracies) are given by postcomposition of  $S \to \underline k$ with the corresponding map
$\underline k \to \underline k'$. 
The outer faces
$d=d_\bot,d_\top : \B_{/\underline k}  \to 
\B_{/\underline {k\!-\!1}}$ take $S\to \underline k$
to the pullback
$$\xymatrix{
S' \drpullback \ar[r] \ar[d] & S \ar[d] \\
\underline{k\!-\!1} \ar[r]_-{d} & \underline k
}$$
The simplicial identities can be arranged to
hold on the nose: the only subtlety is the pullback construction
involved in defining the outer face maps, but these pullbacks
can all be chosen in terms of subset inclusions.
It is clear that the simplicial groupoid $\B_{/\underline k}$
is equivalent to the fat nerve of the classifying space of $\B$.
\end{blanko}

\subsection{Multiplicative examples}

\begin{blanko}{Divisibility poset and multiplicative monoid.}\label{ex:D&M}
  In analogy with \ref{ex:N&L}, 
let $\mathbf D$ denote the nerve of the divisibility poset $(\N^\times,
  |)$, and let $\mathbf M$ be the nerve 
of the multiplicative monoid $(\N^\times, \cdot)$.
Imposing the equivalence relation `isomorphism of intervals' on
the incidence coalgebra of $\mathbf D$ gives that of $\mathbf M$, 
and Content--Lemay--Leroux~\cite{Content-Lemay-Leroux}
observed that this reduction is induced by the
  conservative ULF functor $r:\mathbf D \to \mathbf M$ sending $d|
  n$ to $n/d$. In fact we have:

\begin{lemma}\label{lem:DtoM}
There is an isomorphism of simplicial sets 
  \begin{align*}
   \Dec_\bot (\mathbf M)  & \stackrel\simeq\longrightarrow  \mathbf D  \\
   \intertext{given in degree $k$ by}
    (x_0,x_1,\dots,x_k) & \longmapsto  [x_0| x_0x_1 | \dots | x_0x_1\cdots x_k]
  \end{align*}
and the conservative ULF functor $r$ is isomorphic to the structure map   $$
d_\bot: \Dec_{\bot}(\mathbf M) \to \mathbf M,\qquad (x_0,\dots,x_k)\mapsto (x_1,\dots,x_k).$$
\end{lemma}

This example can be obtained from Example  \ref{ex:N&L} directly, since
$\mathbf M= \prod_p \mathbf N$ and $\mathbf D=\prod_p
  \mathbf L$, where the (weak) product is over all primes $p$.  Now 
  $\Dec_\bot$ is a right adjoint, so preserves products, and
 Lemma \ref{lem:DtoM}
follows from Lemma~\ref{ex:N&L}.
  
  We can use the general formula \ref{sectioncoeff!}:
  since there are no nontrivial automorphisms the convolution product is 
  $
  \delta^m * \delta^n = \delta^{mn}
  $, and the incidence  
  algebra is isomorphic to the Dirichlet algebra:
  \begin{eqnarray*}
    \Inc(\mathbf D) & \longrightarrow & \D =\{ \sum_{k>0} a_k k^{-s} \}   \\
    \delta^n & \longmapsto & n^{-s} \\
    f & \longmapsto & \sum_{n>0} f(n) n^{-s} .
  \end{eqnarray*}
  
\end{blanko}

\begin{blanko}{Arithmetic species.}
  The Dirichlet coalgebra (\ref{ex:D&M})
  also has a fatter version: consider now instead
  the monoidal groupoid $(\B^\times, \times, 1)$ of non-empty finite sets under the 
  cartesian product.  It gives the classifying space
  $\mathbf A$ with $A_k := (\B^\times)^k$,
where this time the inner face maps take the cartesian product of two adjacent 
factors, and the outer face maps project away an outer factor.

The resulting coalgebra structure is
$$
\Delta(S) = \sum_{A\times B \simeq S} A \tensor B .
$$
Some care is due to interpret this correctly: 
the homotopy fibre over $S$ is the groupoid whose objects
are triples $(A,B,\phi)$ consisting of sets $A$ and $B$ equipped with a 
bijection
$\phi: A \times B \isopil S$, and whose morphisms are pairs of isomorphisms
$\alpha: A \isopil A'$, $\beta: B \isopil B'$ forming a commutative square
with $\phi$ and $\phi'$.

The corresponding incidence algebra $\grpd^{\B^\times}$ with the convolution
product is the algebra of arithmetic species \cite{Baez-Dolan:zeta} under the
Dirichlet product (called the arithmetic product of species by Maia and
M\'endez~\cite{Maia-Mendez:0503436}).

%
%
The section coefficients are given directly by \ref{sectioncoeff!},
and we find
$$
\delta^m * \delta^n = \frac{(mn)!}{m!n!} \, \delta^{mn}
$$
It follows that we can get an isomorphism with the Dirichlet algebra,
namely
\begin{eqnarray*}
  \Inc(\mathbf A) & \longrightarrow & \D=\{ \sum_{k>0} a_k k^{-s} \}  \\
  \delta^m & \longmapsto & \frac{m^{-s}}{m!} \\
  f & \mapsto & \sum_{n>0} f(n) \frac{k^{-s}}{n!};
\end{eqnarray*}
these are the 
`exponential' (or modified) Dirichlet series (cf.\  
Baez--Dolan~\cite{Baez-Dolan:zeta}.)
So the incidence algebra zeta function in this setting is
$$
\zeta = \sum_{k>0} \delta^k \mapsto \sum_{k>0} \frac{k^{-s}}{k!}
$$
(which is not the usual Riemann zeta function).
\end{blanko}

\subsection{Linear examples}

The following classical examples lead us to classes of decomposition spaces which
are not Segal spaces, namely  Waldhausen's $S$-construction (\ref{sec:Wald}).

\begin{blanko}{$q$-binomials: $\F_q$-vector spaces.}\label{ex:q}
  Let $\mathbb F_q$ denote a finite field with $q$ elements.  Let
  $\mathbf W$ denote the fat nerve of the category $\vect$ of
  finite-dimensional $\mathbb F_q$-vector spaces and $\mathbb
  F_q$-linear injections.  Impose the equivalence relation identifying
  two injections if their cokernels are isomorphic.  This gives the
  $q$-binomial coalgebra (see D\"ur~\cite[1.54]{Dur:1986}).
  
  The same coalgebra can be obtained without reduction as follows.  Put $\mathbf
  V_0 = *$, let $\mathbf V_1$ be the maximal groupoid of $\vect$, and let
  $\mathbf V_2$ be the groupoid of short exact sequences.  The span
  $$
   \xymatrixrowsep{4pt}
  \xymatrix {
  \mathbf V_1 & \ar[l] \mathbf V_2 \ar[r]  &\mathbf V_1\times\mathbf V_1 \\
  E & \ar@{|->}[l] [ E' \!\to\! E \!\to\! E''] \ar@{|->}[r]  & (E',E'')
  }$$
  (together with the span $\mathbf V _1 \leftarrow \mathbf V_0 \to 1$)
  defines a coalgebra on $\grpd_{/\mathbf V_1}$ which (after taking cardinality)
  is the $q$-binomial coalgebra, without further reduction.
  The groupoids and maps involved are part of a simplicial groupoid $\mathbf V: \Delta\op\to\Grpd$,
  namely the  Waldhausen $S$-construction of $\vect$, studied in more detail in 
  the next section (\ref{sec:Wald}), where we'll see that this is a decomposition
  space but not a Segal space.  The lower dec of $\mathbf V$ is naturally 
  equivalent 
  to the fat nerve $\mathbf W$ of the category of injections, and the
  comparison map $d_0$ is the reduction map of D\"ur.
  
  We calculate the section coefficients of 
  $\mathbf V$.
%
%
  From Section~\ref{sec:tioncoeff} we have the following formula
  for the section coefficients (which is precisely the standard formula
  for the {\em Hall numbers}, as explained further in \ref{Hall}):
   $$
  \frac{\norm{\text{SES}_{k,n,n-k}}}{\norm{\Aut(\F_q^k)}\norm{\Aut(\F_q^{n-k})}}  .
  $$
  Here $\text{SES}_{k,n,n-k}$ is the groupoid of short exact sequence with specified
  vector spaces of dimensions $k$, $n$, and $n-k$.  This is just a discrete 
  space, and it has
  $(q-1)^n q^{k\choose 2} q^{{n-k}\choose 2} [n]!$ elements.
  Indeed, there are $(q-1)^k q^{k \choose 2} 
  \frac{[n]!}{[n-k]!}$ choices for the injection $\F_q^k \into \F_q^n$, and then
  $(q-1)^n q^{n \choose 2} 
  [n]!$ choices for identifying the cokernel with $\F_q^{n-k}$.
  Some $q$-yoga yields altogether the $q$-binomials as section coefficients:
  $$
  =  { n \choose k}_q .
  $$
%
From this description we see that there is an isomorphism of algebras
\begin{eqnarray*}
  \Inc(\mathbf V) & \longrightarrow & \Q[[z]]  \\
  \delta^k & \longmapsto & \frac{z^k}{[k]!} .
\end{eqnarray*}

  Clearly this algebra is commutative.
  However, an important new aspect is revealed on the objective
  level: here the convolution product is the external product of $q$-species
  of Joyal-Street~\cite{Joyal-Street:GLn}.  They show (working with vector-space
  valued $q$-species), that this product has a natural 
  non-trivial braiding (which of course reduces to commutativity upon taking 
  cardinality).
\end{blanko}

\begin{blanko}{Direct sums of $\F_q$-vector spaces and `Cauchy' product of 
  $q$-species.}
  A coalgebra which is the $q$-analogue of $\mathbf B$ can be obtained 
  from the classifying space
of the monoidal groupoid $(\kat{vect}_{\mathbb 
  F_q}, \oplus, 0)$ of
  finite-dimensional $\mathbb F_q$-vector spaces under direct sum.
  Comultiplication of a vector space $V$ is the groupoid consisting of
  triples $(A,B,\phi)$ where $\phi$ is a linear isomorphism $A\oplus B\isopil 
  V$.  This groupoid projects to $\kat{vect}\times \kat{vect}$: the fibre over
  $(A,B)$ is discrete of cardinality $\norm{\Aut(V)}$, giving
   altogether the following section 
  coefficient
   $$
  \frac{\norm{\Aut(\F_q^n)}}{\norm{\Aut(\F_q^k)}\norm{\Aut(\F_q^{n-k})}} 
  = q^{k(n-k)} \, { n \choose k}_q .
  $$
 
    At the objective level, this convolution product corresponds to the `Cauchy'
    product of $q$-species in the sense of Morrison~\cite{Morrison:0512052}.

%
  The resulting coalgebra is therefore, if we let $\delta_n$ denote
  the cardinality of the name of an $n$-dimensional vector space $V$:
  $$
  \Delta(\delta_n) = \sum_{k\leq n} q^{k(n-k)}\, {n \choose k}_q \cdot\; 
  \delta_k 
  \tensor \delta_{n-k} .
  $$
  Hence this one also has a power series representation, this time not
  with $\varphi(n)=[n]!$, but rather with $\varphi(n) = \# \Aut(\F_q^n)$.
\end{blanko}

\subsection{\fdb bialgebra and variations}
\label{sec:fdb}


\begin{blanko}{\fdb bialgebra.}\label{ex:P=Dec(S)}
  Classically (cf.~Doubilet~\cite{Doubilet:1974}) the \fdb bialgebra is
  constructed by imposing a {\em type-equivalence} relation on the incidence
  coalgebra of the poset $\mathbf P$ of all partitions of finite sets.
  Joyal~\cite{JoyalMR633783} observed that it can also be realised directly from
  the category $\mathbf S$ of finite sets and surjections.  (See 
  also~\cite{GalvezCarrillo-Kock-Tonks:1207.6404} for further development of 
  this viewpoint.)
  Let $\mathbf S$ denote the fat nerve of the category of finite sets and 
  surjections. That is, $\mathbf S_k$ is the groupoid of strings of $k$ consecutive 
  surjections.  
  
%
%
%
  A partition $\sigma$ of a finite set $X$ is encoded by the surjection $X\onto
  S$, where $S$ is the set of parts.  Conversely, any surjection constitutes a
  partition of its domain.  There is an equivalence of groupoids between
  partitions and surjections.  Under this correspondence, if partition $\tau$
  refines partition $\sigma$, then the corresponding surjections $X \onto T$ and
  $X \onto S$ fit into a sequence of surjections $X \onto T \onto S$.  Hence we
  can write the partition poset nerve as having $P_0$ the groupoid of finite
  partitions (i.e.~surjections), and $P_k$ the groupoid of $k+1$ strings of
  surjections. 
Under this identification, the conservative ULF functor $F:\mathbf
  P \to \mathbf S$ simply deletes the first surjection in the string.  
  Precisely,
  the partition-poset nerve is
  simply the decalage of the surjections nerve:
  $$
  \mathbf P = \Dec_\bot (\mathbf S).
  $$
  Finally note that the functor $F$ is precisely reduction modulo type 
  equivalence:
  recall that an interval
  $[\tau,\sigma]$ has {\em type} $1^{\lambda_{1}} 2^{\lambda_{2}} \cdots$ if
  $\lambda_k$ is the number of blocks of $\sigma$ that consist of exactly $k$
  blocks of $\tau$.  Two intervals have the same type if and only if
  their images under $F:\mathbf P \to \mathbf S$ are isomorphic.
%
%
%

\end{blanko}

\begin{blanko}{\fdb section coefficients.}
  The category $\kat{Surj}$ of finite sets and surjections is monoidal under $+$,
  and is `monoidal extensive' in the sense that the canonical map
  $$
  \kat{Surj}_{/A} \times \kat{Surj}_{/B} \to \kat{Surj}_{/A+B}
  $$
  is an equivalence, just like in the ambient category $\Set$ which is extensive. 
  (We do not say that $\kat{Surj}$ is extensive in the strict sense of the word, since
  $+$ is not the categorical coproduct.)  It follows that the 
  fat nerve $\mathbf S$ is a monoidal decomposition space (under $+$),
  hence the incidence coalgebra is a bialgebra.
  Note also that automatically the decalage of a monoidal decomposition space
  is monoidal, and the counit cULF.  Hence the partition poset nerve
  is monoidal, and the reduction function a bialgebra homomorphism.
%
  Since $\mathbf S$ is monoidal, it is enough to describe the section
  coefficients on connected input.  (A connected surjection is one with codomain
  $1$.)
%
%
  Our general formula \ref{sectioncoeff!} gives
  $$
  \Delta(\un n\stackrel f\onto \un 1) = 
  \sum_{
  \substack{
  a:\un n \onto \un k\\  b:\un k\onto \un 1 
  }
  }
  \frac{\#\Aut(\un k)\cdot \# \Aut(ab)}{\# \Aut(a) \cdot 
  \#\Aut(b)} \,\name a\otimes\name b .
  $$
  The order of the automorphism group of $\un k$ and of a  surjection $\un k\onto\un 1$ is $k!$, and 
  for a general surjection 
  $a:\un n\onto \un k$ of type $1^{\lambda_1}2^{\lambda_2} \cdots$,
  $$
  \#\Aut(a)\;=\;\prod_{j=1}^\infty\lambda_j ! (j!)^{\lambda_j}
  $$
  and hence
  $$
  \Delta(\un n\stackrel f\onto \un 1) = 
  \sum_{
  \substack{
  a:\un n \onto \un k\\  b:\un k\onto \un 1 
  }
  }
  \frac{n! }{\prod_{j=1}^k\lambda_j ! (j!)^{\lambda_j}} 
   \,\name a\otimes\name b .
  $$
  The section coefficients, called the \fdb section coefficients,
  are the coefficients
  ${n \choose \lambda;k}$ of the Bell polynomials, cf.\ \cite[(2.5)]{Figueroa-GraciaBondia:0408145}.
\end{blanko}

\begin{blanko}{A decomposition space for the \fdb\ {\em Hopf} algebra.}
  The \fdb\ {\em Hopf} algebra is obtained by further reduction, classically
  stated as identifying two intervals in the partition poset if they are
  isomorphic as posets.  This is equivalent to forgetting the value of
  $\lambda_1$.  There is also a decomposition space that yields this Hopf
  algebra directly, obtained by quotienting the decomposition space $\mathbf S$
  by the same equivalence relation.  This means identifying two surjections (or
  sequences of composable surjections) if one is obtained from the other by
  taking disjoint union with a bijection.  One may think of this as `levelled 
  forests modulo linear trees'.  It is straightforward to check that this
  reduction respects the simplicial identities so as to define a simplicial
  groupoid, that it is a monoidal decomposition space, and that the quotient
  map from $\mathbf S$ is monoidal and cULF.
\end{blanko}

\begin{blanko}{Ordered surjections.}
  Let $\mathbf{OS}$ denote the fat nerve of the category of finite ordered set
  and monotone surjections.  It is a monoidal decomposition space under ordinal
  sum.  Hence to describe the resulting comultiplication, it is enough to say
  what happens to a connected ordered surjection, say $f: \un n \onto \un 1$,
  which we denote simply $n$: since there are no automorphisms around, we find
  $$
  \Delta(n) = \sum_{k=1}^n \sum_a a \tensor k
  $$
  where the second sum is over the ${n-1 \choose k-1}$ possible surjections 
  $a:n\onto k$. This comultiplication has appeared in
\cite{Baues:Hopf} and \cite{Goncharov:Hopf}.
\end{blanko}


%% file: trees.tex
\def\inputfile{trees.tex}

\subsection{Graphs and trees}

Various bialgebras of graphs and trees can be realised as incidence bialgebras
of decomposition spaces which are not Segal.  These examples will be 
subsumed in general classes of decomposition spaces,
namely coming from restriction species, and the new notion of
{\em directed restriction species} introduced in Section~\ref{sec:restriction}.

\medskip

All the examples in this section are naturally bialgebras, with the monoidal
structure given by disjoint union.

%


\begin{blanko}{Graphs and restriction species.}\label{ex:graphs}
  The following coalgebra of graphs seems to be due to Schmitt~\cite{Schmitt:1994}, \S 12.
  For a graph $G$ with vertex set $V$ (admitting multiple edges and loops),
  and a subset $U \subset V$, define
  $G|U$ to be the graph whose vertex set is $U$, and whose edges are those
  edges of $G$ both of whose incident vertices belong to $U$.
  On the vector space spanned by isoclasses of graphs,
  define a comultiplication by the rule
  $$
  \Delta(G) = \sum_{A+B=V} G|A \tensor G|B .
  $$
  
  This coalgebra is the cardinality of the coalgebra of a decomposition space
  but not directly of a category.  Indeed, define a simplicial groupoid with
  $\mathbf G_1$ the groupoid of graphs, and more generally let $\mathbf G_k$ be
  the groupoid of graphs with an ordered partition of the vertex set into $k$
  (possibly empty) parts.  In particular, $\mathbf G_0$ is the contractible
  groupoid consisting only of the empty graph.  The outer face maps delete the
  first or last part of the graph, and the inner face maps join adjacent parts.
  The degeneracy maps insert an empty part.  It is clear that this is not a
  Segal space: a graph structure on a given set cannot be reconstructed from
  knowledge of the graph structure of the parts of the set, since chopping up
  the graph and restricting to the parts throws away all information about edges
  going from one part to another.  One can easily check that it is a
  decomposition space.  It is clear that the resulting coalgebra is Schmitt's
  coalgebra of graphs.
  Note that disjoint union of graphs makes this into a bialgebra.

  The graph example is typical for a big family of decomposition spaces, which
  can be treated uniformly, namely decomposition spaces of restriction species,
  in the sense of Schmitt~\cite{Schmitt:hacs} (see also \cite{Aguiar-Mahajan}).
  We develop this theory further in Section~\ref{sec:restriction}.
\end{blanko}

\begin{blanko}{Butcher-Connes-Kreimer Hopf algebra.}\label{ex:CK}
  D\"ur~\cite{Dur:1986} (Ch.IV, \S3) constructed what was later called the Connes-Kreimer 
  Hopf algebra of rooted trees, after~\cite{Connes-Kreimer:9808042}: he starts with the notion of (combinatorial) tree
  (i.e.~connected and simply connected graphs with a specified root vertex);
  then a forest is a disjoint union  of rooted trees.  He 
  then considers the category of root-preserving inclusions of forests.
  A coalgebra is induced from this (in our language it is given by the 
  simplicial groupoid $\mathbf R$, where $\mathbf R_k$ is the groupoid of strings of
  $k$ root-preserving forest inclusions) but it is not
  the most interesting one.
  The Connes--Kreimer coalgebra is obtained by the reduction that
  identifies two root-preserving forest inclusions if their complement crowns
  are isomorphic forests.  
  
  We can obtain this coalgebra directly from a 
  decomposition space: let $\mathbf H_1$ denote the groupoid of forests, and 
  let $\mathbf H_2$ 
  denote the groupoid of forests with an admissible cut.  More generally, 
  $\mathbf H_0 $ is defined to be a point, and 
  $\mathbf H_k$ is the groupoid of forests with $k-1$ compatible admissible cuts.
  These form a simplicial groupoid in which the inner face maps forget a cut,
  and the outer face maps project away either the crown or the bottom layer 
  (the part of the forest below the bottom cut).
  The notion of admissible cut is standard,
  see for example~\cite{Connes-Kreimer:2000}.
  One convenient way to define what it means is
  to say that it is a root-preserving inclusion of forests: then the cut is
  interpreted as the division between the included forest and its complement.
  In this way we see that $\mathbf H_k$ is the groupoid of $k-1$ consecutive 
  root-preserving inclusions.
  
  
  There is a natural conservative ULF functor from $\mathbf R$ to
  $\mathbf H$: on $\mathbf R_1 \to \mathbf H_1$ it sends
  a root-preserving forest inclusion to its crown.  More generally, on 
  $\mathbf R_k \to \mathbf H_k$ it deletes the first inclusion in the string.
  Once again we see that $\mathbf R \simeq \Dec_\bot(\mathbf H)$, and that
  the reduction is just the counit of decalage.

%
  
  It is clear that $\mathbf H$ is not a Segal space:
  a tree with a cut cannot
  be reconstructed from its crown and its bottom tree, which is to say that 
  $\mathbf H_2$
  is not equivalent to $\mathbf H_1 \times_{\mathbf H_0} \mathbf H_1$.  
  It is straightforward to check  
  that it {\em is} a decomposition space.
\end{blanko}

\begin{blanko}{Operadic trees and $P$-trees.}\label{Ptrees}
  There is an important variation on the Connes-Kreimer Hopf algebra (but it is 
  only a bialgebra): instead of considering combinatorial trees one considers
  operadic trees (i.e.~trees with open incoming edges), or more generally
  $P$-trees for a finitary polynomial endofunctor $P$.  For details on this
  setting, see \cite{Kock:0807}, \cite{Kock:1109.5785}, \cite{Kock:MFPS28}
  \cite{GalvezCarrillo-Kock-Tonks:1207.6404}; it suffices here to note that the
  notion covers planar trees, binary trees, effective trees, Feynman diagrams,
  etc.

  There is a functor from operadic trees or $P$-trees to
  combinatorial trees which is {\em taking core} \cite{Kock:1109.5785}: it
  amounts to shaving off all open-ended edges (and forgetting the
  $P$-decoration).  This is a conservative ULF functor which realises the core
  bialgebra homomorphism from the bialgebra of operadic trees or $P$-trees to
  the Hopf algebra of combinatorial trees.

  For operadic trees, when copying over the description of the nerve $X$
  where $X_k$ is the groupoid of forests with $k-1$ compatible admissible cuts,
  there are two important differences: one is that $X_0$ is not just a point: it is
  the groupoid of node-less forests.  The second is that unlike $\mathbf
  H$, this one is a Segal space; this follows from the Key Lemma of 
  \cite{GalvezCarrillo-Kock-Tonks:1207.6404}.
  Briefly this comes from the fact that
  the cuts do not throw away the edges cut, and hence there is enough data to
  reconstruct a tree with a cut from its bottom tree and crown by grafting.
  More precisely, the Segal maps $X_k\to X_1 \times_{X_0} \cdots \times_{X_0}
  X_1$ simply return the layers seen in between the cuts.  It is easy to see
  that this is an equivalence:  given the layers separately, and a match of
  their boundaries, one can glue them together to reconstruct the original
  forest, up to isomorphism.  In this sense the operadic-forest decomposition
  space is a `category' with node-less forests as objects.  In this perspective,
  the combinatorial-forest decomposition space is obtained by throwing away the
  object information, i.e.~the data governing the possible ways to compose.
  These two differences are crucial in the work on Green functions and \fdb
  formulae in \cite{GalvezCarrillo-Kock-Tonks:1207.6404} and \cite{Kock:1512.03027}.
\end{blanko}

\begin{blanko}{Note about symmetries.}
  It may be worth stressing here that one can {\em not} obtain the same
  bialgebra (either the combinatorial or the operadic) by taking isomorphism
  classes of each of the groupoids $X_k$: doing this would destroy symmetries
  that constitute an essential ingredient in the Connes--Kreimer bialgebra.
  Indeed, define a simplicial set $Y$ in which $Y_k =
  \pi_0(X_k)$, the set of iso-classes of forests with $k$ compatible admissible
  cuts.  Consider the tree $T$ 
  \newcommand{\mytree}{ \bsegment \move (0 0) \lvec
  (0 8) \onedot \lvec (-5 17) \onedot \move (0 8) \lvec (5 17) \onedot \esegment
  }
\begin{center}\begin{texdraw}
  \move (0 0) \mytree
\end{texdraw}
\end{center}
belonging to $X_1$.  The fibre in $X_2$ is the (discrete) groupoid of
all possible cuts in this tree:
\begin{center}\begin{texdraw}
  \move (0 0) \mytree \move (0 21) 
      \bsegment \move (-10 0)\clvec (-4 4)(4 4)(10 0) \esegment
  \move (30 0) \mytree \move (30 10)
      \bsegment \move (-10 -2)\clvec (2 0)(-2 14)(10 14) \esegment
  \move (60 0) \mytree \move (60 10)
      \bsegment \move (10 -2)\clvec (-2 0)(2 14)(-10 14) \esegment
  \move (90 0) \mytree \move (90 10)
      \bsegment \move (-8 0)\clvec (-4 4)(4 4)(8 0) \esegment
  \move (120 0) \mytree \move (120 4)
      \bsegment \move (-6 0)\clvec (-4 -2)(4 -2)(6 0) \esegment
\end{texdraw}
\end{center}
The thing to notice here is that while the second and third
cuts are isomorphic as abstract cuts, and therefore get identified in 
$Y_2 = \pi_{0}(X_2)$, this isomorphism is not vertical over the underlying tree 
$T$, so in the comultiplication formula at the groupoid level of $X$ both
cuts appear, and there is a total of $5$ terms, whereas at the level of
$Y$ there will be only $4$ terms.  (Put in another way, the functor $X \to Y$
given by taking components
is not cULF.)

It seems that there is no way to circumvent this discrepancy directly at the
isoclass level: attempts involving ingenious decorations by natural numbers and
actions by symmetric groups will almost certainly end up amounting to actually
working at the groupoid level, and the conceptual clarity of the groupoid
approach seems much preferable.
\end{blanko}

\begin{blanko}{Free categories and free multicategories.}
  Let $G$ be a directed graph $E \rightrightarrows V$.  Consider the polynomial
  endofunctor $P$ given by $V \leftarrow E \stackrel = \to E \to V$.  Then the
  groupoid of $P$-trees (\ref{Ptrees}) (necessarily linear trees, since the
  middle map is an identity) is precisely (equivalent to) the set of arrows in
  the free category on $G$, and the decomposition space of $P$-trees described
  in \ref{Ptrees} coincides with the nerve of this category.
  
  More generally, for an arbitrary polynomial endofunctor $P$ given by a diagram
  of sets $I \leftarrow E \to B \to I$, the groupoid of $P$-trees is the
  groupoid of operations of the free monad on $P$.  Thinking of $P$ as
  specifying a signature, we can equivalently think of $P$-trees as operations
  for the free (coloured) operad on that signature, or as the multi-arrows of
  the free multicategory on $P$ regarded as a multigraph.  To a multicategory
  there is associated a monoidal category~\cite{Hermida:repr-mult}, whose object
  set is the free monoid on the set of objects (colours).  The decomposition space of
  $P$-trees is naturally identified with the nerve of the monoidal category
  associated to the multicategory of $P$-trees.
\end{blanko}

\begin{blanko}{Polynomial monads.}
  The decomposition space of $P$-trees for $P$ a polynomial endofunctor
  (\ref{Ptrees}) can be regarded as the decomposition space associated to the
  free monad on $P$, and as such the construction works for any (cartesian,
  discrete-finitary) polynomial monad, not just free ones, as we now proceed to
  explain.  Namely, let $P$ be a polynomial monad
  in groupoids, given by $I \leftarrow E \to B \to I$ with $E\to B$  finite and
  discrete, and assume that the monad is cartesian (i.e.~the naturality squares
  for the monad multiplication and unit are cartesian).  Following the graphical
  interpretation given in \cite{Kock-Joyal-Batanin-Mascari:0706}, one can regard $I$
  as the groupoid of decorated unit trees (i.e.~trees without nodes), and $B$ as
  the groupoid of corollas (i.e.~trees with exactly one node) decorated with
  $B$ on the node and $I$ on the edges, compatibly.  The arity of a
  corolla labeled by $b\in B$ is then the cardinality of the fibre $E_b$.  We
  can now form a simplicial groupoid $X$ in which $X_0$ is the groupoid of
  disjoint unions of decorated unit trees, $X_1$ is the groupoid of disjoint
  unions of decorated corollas, and where more generally $X_n$ is the groupoid of
  $P$-forests of height $n$.  For example, $X_2$ is the groupoid of forests of
  height $2$, which equivalently can be described as configurations consisting
  of a disjoint unions of bottom corollas whose leaves are decorated with other
  corollas, in such a way that the roots of the decorating corollas match the
  leaves of the bottom corollas.  This groupoid can more formally be described
  as the free symmetric monoidal category on $P(B)$.  Similarly, $X_n$ is the
  free symmetric monoidal category on $P^{n-1}(B)$.  The outer face maps project
  away the top or bottom layer in a level-$n$ forest.  For example $d_1: X_1
  \to X_0$ sends a disjoint union of corollas to the disjoint union of their
  root edges, while $d_0: X_1 \to X_0$ sends a disjoint union of corollas to the
  forest consisting of all their leaves.  The generic face maps (i.e.~inner face
  maps) join two adjacent layers by means of the monad multiplication on $P$.
  The degeneracy maps insert unary corollas by the unit operation of the monad.
  Associativity of the monad law ensures that this simplicial groupoid is
  actually a Segal space.  The operation of disjoint union makes this a monoidal
  decomposition space, and altogether an incidence bialgebra results from the
  construction.

  The example (\ref{Ptrees}) of $P$-trees (for $P$ a polynomial endofunctor) and
  admissible cuts is an example of this construction, namely corresponding to
  the free monad on $P$: indeed, the operations of the free monad on $P$ is the
  groupoid of $P$-trees, which now plays the role of $B$.  Level-$n$ trees in
  which each node is decorated by objects in $B$ is the same thing as $P$-trees
  equipped with $n-1$ compatible admissible cuts, and grafting of $P$-trees (as
  prescribed by the generic face maps in \ref{Ptrees}) is precisely the monad
  multiplication in the free monad on $P$.  It should be stressed that while
  the free thing is always automatically of locally finite length, the general case
  is not automatically so.  This condition must be imposed separately if
  numerical examples are to be extracted.

  The simplicial groupoid constructed can be regarded as a generalisation of the
  monoidal category associated to a multicategory, since a multicategory is
  essentially the same thing as a (discrete) polynomial monad equipped with a
  cartesian monad map to the free-monoid monad
  (cf.~\cite{Gambino-Kock:0906.4931}).
\end{blanko} 

\begin{blanko}{Directed graphs and free PROPs.}\label{ex:MM}
  The foregoing constructions readily generalise from trees to directed graphs
  (although the attractive polynomial interpretation does not).  By a
  directed graph
  we understand a
  finite oriented graph with a certain number of open input edges, a certain
  number of open output edges, and prohibited to contain an oriented cycle (see 
  \cite{Kock:1407.3744} for a recent elegant categorical formalism).  In
  particular, a directed graph has an underlying poset.  The directed
  graphs form a groupoid $G_1$.  We allow graphs without vertices, these form a
  groupoid $G_0$.  Let $G_2$ denote the groupoid of directed graphs with an
  {\em admissible cut}: by this we mean a partition of the set of vertices into
  two disjoint parts: a poset filter $F$ (i.e.~an upward closed subset) and a poset ideal
  $I$ (i.e.~a downward closed subset).  The edges connecting the two parts
  become the output edges of $F$ and input edges of $I$; hence $F$ and $I$
  become directed graphs again. 
  Similarly, let $G_k$ denote the groupoid of
  directed graphs with $k-1$ compatible admissible cuts, just like we did for
  forests.  It is clear that this defines a simplicial groupoid $\mathbf G$,
  easily verified to be a decomposition space and in fact a Segal space.  The
  directed graphs form the set of operations of 
  the free PROP with one generator in each input/output
  degree $(m,n)$.  The Segal space is the nerve of the associated monoidal
  category. The resulting coalgebra (in fact a bialgebra) has been studied in the
  context of Quantum Field Theory by Manchon~\cite{Manchon:MR2921530}.  Certain
  decorated directed graphs, and the resulting bialgebra have been studied by
  Manin~\cite{Manin:MR2562767}, \cite{Manin:0904.4921} in the theory of
  computation: his directed graphs are decorated by operations on partial
  recursive functions and switches.
  The decorating data is called a tensor 
  scheme in \cite{Joyal-Street:tensor-calculus}, and the class of decorated
  graphs form the set of operations of the free (coloured) PROP on the tensor scheme.
  Again, the resulting decomposition space is naturally identified with
  the  nerve of the associated monoidal category.
\end{blanko}

%
%

%% file: hall.tex

\def\inputfile{hall.tex}

\subsection{Waldhausen $S$-construction}

\label{sec:Wald}

%
%
%

\begin{blanko}{Waldhausen $S$-construction of an abelian category.}
  We follow Lurie~\cite[Subsection 1.2.2]{Lurie:HA} for the account of
  Waldhausen $S$.  For $I$ a linearly ordered set,
let $\Ar(I)$ denote the category of arrows in $I$: the objects
are pairs of elements $i\leq j$ in $I$, and the morphisms are relations
$(i,j) \leq (i',j')$ whenever $i\leq i'$ and $j\leq j'$.
A {\em gap complex} in an abelian category $\mathscr A$ is a functor
$F: N(\Ar(I)) \to \mathscr A$ such  that

\begin{enumerate}
  \item For each $i \in I$, the object $F(i,i)$ is zero.

  \item For every $i \leq j \leq k$, the associated diagram
  $$\xymatrix{
 \!\!\!\!\!\!\!\!\!\!0= F(j,j) \ar@{ >->}[r] & F(j,k) \\
  F(i,j) \ar@{->>}[u] \ar@{ >->}[r] & F(i,k) \ar@{->>}[u]
  }$$
  is a pushout (or equivalently a pullback).
\end{enumerate}

Remark:
since the pullback of a monomorphism is always a monomorphism, and the pushout
of an epimorphism is always an epimorphism, it follows that automatically the
horizontal maps are monomorphisms and the vertical maps are epimorphisms, as
already indicated with the arrow typography.  Altogether, it is just a fancy but
very convenient way of saying `short exact sequence' or `(co)fibration sequence'.





Let $\operatorname{Gap}(I,\mathscr A)$ denote the full subcategory of
$\Fun(\Ar(I),\mathscr A)$ consisting of the gap complexes.
This is a $1$-category, since $\mathscr A$ was assumed to be an abelian $1$-category.

The assignment
$$
[n] \mapsto \operatorname{Gap}([n],\mathscr A)^{\eq}
$$
defines a simplicial space $S\mathscr A: \Delta\op\to\Grpd$, which by
definition is the Waldhausen $S$-construction on $\mathscr A$.
Intuitively (or essentially), the
groupoid $\operatorname{Gap}([n],\mathscr A)^{\eq}$ has as objects
staircase diagrams like the following (picturing $n=4$):
$$\xymatrix{
&&& A_{34} \\
&& A_{23} \ar@{ >->}[r] & A_{24} \ar@{->>}[u] \\
& A_{12} \ar@{ >->}[r] & A_{13} \ar@{->>}[u]\ar@{ >->}[r] & A_{14} \ar@{->>}[u] \\
A_{01} \ar@{ >->}[r] & A_{02} \ar@{->>}[u]\ar@{ >->}[r] & A_{03} 
\ar@{->>}[u]\ar@{ >->}[r] & A_{04} \ar@{->>}[u]
}$$
\end{blanko}
The face map $d_i$ deletes all objects containing an $i$ index.
The degeneracy map $s_i$ repeats the $i$th row and the $i$th column.

%


A string of composable monomorphisms 
$(A_{1}\rightarrowtail A_{2}\rightarrowtail \dots \rightarrowtail A_{n})$
determines, up to canonical isomorphism,  short exact  
sequences $A_{ij}\rightarrowtail A_{ik}\twoheadrightarrow 
A_{jk}=A_{ij}/A_{ik}$ with $A_{0i}=A_i$.
Hence the whole diagram can be reconstructed up to isomorphism from the
bottom row.  (Similarly, since epimorphisms have uniquely determined kernels,
the whole diagram can also be reconstructed from the last column.)

We have  $s_0(*)=0$, and
\begin{align*}
d_0(A_1\rightarrowtail A_2\rightarrowtail \dots \rightarrowtail A_{n})&= 
(A_2/A_1\rightarrowtail \dots \rightarrowtail A_{n}/A_1)
\\s_0(A_1\rightarrowtail A_2\rightarrowtail \dots \rightarrowtail A_{n})&= 
(0\rightarrowtail 
A_1\rightarrowtail 
A_2\rightarrowtail \dots \rightarrowtail A_{n}) 
\end{align*}
The simplicial maps $d_i,s_i$ for $i\geq1$ are more straightforward: the 
simplicial set $\Dec_\bot (S\mathscr A)$ is just the nerve of $\mathrm{mono}(\mathscr A)$.

\begin{lemma}
  The projection $S_{n+1} \mathscr A \to \Map([n], \mathrm{mono}(\mathscr A))$
  is a trivial Kan fibration.  Similarly the projection 
  $S_{n+1} \mathscr A \to \Map([n], \mathrm{epi}(\mathscr A))$.
\end{lemma}
More precisely (with reference to the fat nerve):

\begin{prop}
  These equivalences assemble into levelwise simplicial 
equivalences
$$
\Dec_\bot(S\mathscr A) \simeq N( \mathrm{mono}(\mathscr A))
$$
$$
\Dec_\top(S\mathscr A) \simeq N( \mathrm{epi}(\mathscr A)) .
$$
\end{prop}

\begin{theorem}
  The Waldhausen $S$-construction of an abelian category $\mathscr A$
  is a decomposition space.
\end{theorem}
\begin{proof}
  For convenience we write $S\mathscr A$ simply as $S$.
The previous proposition already implies that the two $\Dec$s of $S$
  are Segal spaces.  By Theorem~\ref{thm:decomp-dec-segal}, it is 
  therefore enough to
  establish that the squares
  $$
  \xymatrix{S_1 \ar[r]^{s_1} \ar[d]_{d_0} & S_2 \ar[d]^{d_0} \\
  S_0 \ar[r]_{s_0} & S_1
  }
  \qquad 
  \xymatrix{S_1 \ar[r]^{s_0} \ar[d]_{d_1} & S_2 \ar[d]^{d_2} \\
  S_0 \ar[r]_{s_0} & S_1
  }
  $$
  are pullbacks. Note that we have $S_0 = *$ and $S_1 = \mathscr 
  A^{\operatorname{iso}}$,
  and that $s_0 : S_0 \to S_1$ picks out the zero object, and since the zero object
  has no nontrivial automorphisms, this map is fully faithful.  The map $d_0 : S_2 \to S_1$
  sends a monomorphism to its quotient object.  We need to compute the fibre over the zero 
  object, but since $s_0$ is fully faithful, we are just asking for the full subgroupoid
  of $S_2$ consisting of those monomorphisms whose cokernel is zero.  Clearly these
  are precisely the isos, so the fibre is just $\mathscr A^{\operatorname{iso}} = S_1$.
  The other pullback square is established similarly, but arguing with epimorphisms
  instead of monomorphisms.
\end{proof}
\begin{BM}
  Waldhausen's $S$-construction was designed for more general categories than
  abelian categories, namely what are now called Waldhausen categories, where
  the cofibrations play the role of the monomorphisms, but where there is no stand-in
  for the epimorphisms.  The theorem does not generalise to Waldhausen categories in
  general, since in that case $\Dec_\top(S)$ is not necessarily
  a Segal space of any class of arrows.
\end{BM}

\begin{blanko}{Waldhausen $S$ of a stable $\infty$-category.}
  The same construction works in the $\infty$-setting, by considering stable
  $\infty$-categories instead of abelian categories.  Let $\mathscr A$ be a
  stable $\infty$-category (see Lurie~\cite{Lurie:HA}).  Just as in the abelian
  case, the assignment
  $$
  [n] \mapsto \operatorname{Gap}([n],\mathscr A)^{\eq}
  $$
  defines a simplicial space $S\mathscr A: \Delta\op\to\Grpd$, which by
  definition is the Waldhausen $S$-construction on $\mathscr A$.
  Note that in the case of a stable $\infty$-category, in contrast to the abelian 
  case, every map can arise as either horizontal or vertical arrow in a gap 
  complex.  Hence the role of monomorphisms (cofibrations)
  is played by all maps, and the role of epimorphisms is also played by all maps.
\end{blanko}

\begin{lemma}
  For each $k\in \N$, the two projection functors
$S_{k+1}\mathscr A \to \Map(\Delta[k], \mathscr A)$
are equivalences.
\end{lemma}
From the description of the face and degeneracy maps, 
the following more precise result follows readily,
comparing with the fat nerves:

\begin{prop}
  We have natural (levelwise) simplicial equivalences
$$
\Dec_\bot(S\mathscr A) \simeq N(\mathscr A)
$$
$$
\Dec_\top(S\mathscr A) \simeq N(\mathscr A).
$$
\end{prop}

\begin{theorem}\label{thm:WaldhausenS}
  Waldhausen's $S$-construction of a
  stable $\infty$-category $\mathscr A$ is a decomposition space.
\end{theorem}
\begin{proof}
  The proof is exactly the same as in the abelian case, relying on the following
  three facts: 
  \begin{enumerate}
    \item The $\Dec$s are Segal spaces.
  
    \item $s_0 : S_0 \to S_1$ is fully faithful.
  
    \item A map (playing the role of monomorphisms) is an equivalence if and only if its cofibre
    is the zero object, and a map (playing the role of epimorphism) is an equivalence 
    if and only if its fibre is the zero object.
  \end{enumerate}
\end{proof}


\begin{BM}
  This theorem was proved independently (and first) by Dyckerhoff and 
  Kapranov~\cite{Dyckerhoff-Kapranov:1212.3563}, Theorem 7.3.3.
  They prove it more generally for exact $\infty$-categories, a notion they
  introduce.  Their proof that Waldhausen's $S$-construction of an exact 
  $\infty$-category is a decomposition space is somewhat more complicated than ours above.
  In particular their proof of unitality (the pullback condition on
  degeneracy maps) is technical and 
  involves Quillen model structures on certain marked 
  simplicial sets \`a la Lurie~\cite{Lurie:HTT}.
  We do not wish to go into exact $\infty$-categories here, and refer
  instead the reader to \cite{Dyckerhoff-Kapranov:1212.3563}, but we wish
  to point out that our simple proof above works as
  well for exact $\infty$-categories.  This follows since the three points in
  the proof hold also for exact $\infty$-categories,
  which in turn is a consequence of 
  the definitions and basic results provided in 
  \cite[Sections~7.2 and 7.3]{Dyckerhoff-Kapranov:1212.3563}.
\end{BM}

%
%

\begin{blanko}{Hall algebras.}\label{Hall}
  The finite-support incidence algebra of a decomposition space $X$ was
  mentioned in \ref{fin-supp}.  In order for it to admit a cardinality, the
  required assumption is that $X_1$ be locally finite, and 
  that $X_2 \to X_1 \times X_1$ be a finite map.  In the
  case of $X= S(\mathscr A)$ for an abelian category $\mathscr A$, this
  translates into the condition that $\Ext^0$ and $\Ext^1$ be finite (which in
  practice means `finite dimension over a finite field').  The finite-support
  incidence algebra in this case is the {\em Hall algebra} of $\mathscr A$
  (cf.~Ringel~\cite{Ringel:Hall}; see also 
  \cite{Schiffmann:0611617},
  although these sources twist the multiplication by the so-called Euler form).

  For a stable $\infty$-category $\mathscr A$, with mapping spaces assumed to be
  locally finite (\ref{finite}), the finite-support incidence algebra of
  $S(\mathscr A)$ is the {\em derived Hall algebra}.  These were introduced by
  To\"en~\cite{Toen:0501343} in the setting of dg-categories.
  
  Hall algebras were one of the main motivations for Dyckerhoff and 
  Kapranov~\cite{Dyckerhoff-Kapranov:1212.3563} to
  introduce $2$-Segal spaces.  We refer to their work for development 
  of this important topic; see in particular the lecture notes of 
  Dyckerhoff~\cite{Dyckerhoff:CRM}.
\end{blanko}

%% file: cancellation.tex
\def\inputfile{cancellation.tex}

\subsection{\M functions and cancellation}

\label{sec:cancellation}

We compute the \M functions in some of our examples.  While the formula $\mu =
\Phieven-\Phiodd$ seems to be the most general and uniform expression of the \M
function, it often not the most economical.  At the numerical level, it is
typically the case that much more practical expressions for the \M functions can
be computed with different techniques.  The formula $\Phieven-\Phiodd$ should not be
dismissed on these grounds, though: it must be remembered that it constitutes a
natural `bijective' account, valid at the objective level, in contrast to many
of the elegant cancellation-free expressions in the classical theory which are
often the result of formal algebraic manipulations, typically power-series 
representations.

Comparison with the economical formulae raises the question whether these too can
be realised at the objective level.  This can be answered (in a few cases) by
exhibiting an explicit cancellation between $\Phieven$ and $\Phiodd$, which in
turn may or may not be given by a {\em natural} bijection.  

Once a more economical expression has been found for some \M decomposition space
$X$, it can be transported back along any cULF functor $f:Y \to X$ to yield also
more economical formulae for $Y$.

%
%
%
%
%
%

\begin{blanko}{Natural numbers.}
  For the decomposition space $\mathbf N$ (see~\ref{ex:N&L}),
  the incidence algebra is $\grpd^{\N}$, spanned by the representables
  $h^n$, and with convolution product
  $$
  h^a * h^b = h^{a+b}.
  $$

  To compute the \M functor,
  we have
    $$
  \Phieven = \sum_{r \text{ even}} (\N\shortsetminus \{0\})^r ,
  $$
  hence $\Phieven(\un n)$ is
  the set of ordered compositions of the ordered 
  set $\un n$ into an even number of parts, or equivalently
  $$
  \Phieven(\un n) = \{ \un n \onto \un r \mid r \text{ even } \} ,
  $$
  the set of monotone surjections. 
%
%
%
%
  In conclusion, with an abusive sign notation,
  the \M functor is
  $$
  \mu(\un n) = \sum_{r\geq 0} (-1)^r  \{ \un n \onto \un r \} .
  $$
  
  At the numerical level, this formula simplifies to
  $$
  \mu(n)
  = \sum_{r\geq 0} (-1)^r {n-1 \choose r-1} = \begin{cases}
    1 & \text{ for } n=0 \\
    -1 & \text{ for } n=1 \\
    0 & \text{ else, }
  \end{cases}
  $$
%
%
  (remembering that ${-1 \choose -1} =  1$, and ${k \choose -1} = 0$ for $k\geq 
  0$).
  
  \bigskip
  
  On the other hand, since clearly the incidence algebra is isomorphic to the 
  power series ring under the identification $\norm {h^n} = \delta^n 
  \leftrightarrow z^n \in \Q[[z]]$, and since the zeta function corresponds to the 
  geometric 
  series $\sum_n x^n = \frac{1}{1-x}$, we find that the \M function is $1-x$.
  This corresponds the functor $\delta^0-\delta^1$.

  At the objective level, there is indeed a cancellation of $\infty$-groupoids taking
  place.  It amounts to an equivalence of the Phi-groupoids restricted to
  $n\geq2$:
$$\xymatrix{
\Phieven{}_{\mid r\geq 2} \ar[rd]\ar[rr]^\sim && \Phiodd{}_{\mid r\geq 2} \ar[ld] \\
& \N_{\geq 2} & }
$$
which cancels out most of the stuff, leaving us with
the much more economical \M function
$$
\delta^0 - \delta^1
$$
supported on $\N_{\leq 1}$.
Since $\N$ is discrete, this equivalence (just a bijection)
can be established fibrewise: 
  
  
  {\em For each $n\geq 2$ there is a natural fibrewise
  bijection}
  $$
  \Phieven(n) \simeq \Phiodd(n) .
  $$
  To see this, encode the elements $(x_1,x_2,\ldots,x_k)$ 
  in $\Phieven(n)$ as binary strings 
  of length $n$ and starting with $1$ as follows: each coordinate 
  $x_i$ is represented as a string of length $x_i$ 
  whose first bit is $1$ and whose other bits are 
  $0$, and all these strings are concatenated.
  In other words, thinking of the
  element $(x_1,x_2,\ldots,x_k)$ as a ordered partition of the 
  ordered set $n$, in the binary representation the $1$-entries
  mark the beginning of each part.  
  (The binary strings must start with $1$ since the first part must
  begin at the beginning.)  
  For example, with
  $n=8$, the element $(3,2,1,1,1)\in \Phiodd(8)$, 
  is encoded as the binary string $10010111$.
  Now the bijection between
  $\Phieven(n)$ and $\Phiodd(n)$ can be taken to simply flip the
  second bit in the binary representation.  In the example,
  $10010111$ is sent to $11010111$, meaning that
  $(3,2,1,1,1)\in \Phiodd(8)$ is sent to $(1,2,2,1,1,1)\in 
  \Phieven(8)$.  Because of this cancellation which occurs for
  $n\geq 2$ (we need the second bit in order to flip), the difference
  $\Phieven - \Phiodd$ is the same as $\delta_0 -\delta_1$, which is
  the cancellation-free formula.
  

  The minimal solution $\delta^0-\delta^1$ can also be checked immediately at 
  the objective level to satisfy the defining equation for the \M function:
  $$
  \zeta * \delta^0 = \zeta * \delta^1 + \delta^0
  $$
  This equation says
  $$
  \xymatrix{\N\times \{0\} \ar[d]_{\text{add}} \\ \N}
  =
  \xymatrix{(\N\times \{1\}) + \{0\} \ar[d]_{\text{add}+\text{incl}} \\ \N}
  $$
  
  In conclusion, the classical formula lifts to the objective level.
\end{blanko}

\begin{blanko}{Finite sets and bijections.}
  Already for the `next' example, that of the monoidal groupoid $(\B,+,0)$,
  whose incidence algebra is the algebra of species under the Cauchy convolution
  product (cf.~\ref{Cauchy}), the situation is much more subtle. 
  
  Similarly to the previous example, we have
  $\Phi_r(S) = \operatorname{Surj}(S, \un r)$, but this time we are 
  dealing with arbitrary surjections, as $S$ is just an abstract set.
  Hence the \M functor is given by
  $$
  \mu(S) = \sum_{r\geq 0} (-1)^r \operatorname{Surj}( S, \un r) .
  $$
  
  
  Numerically, 
  this is much more complicated than what is obtained from the observation that
  the incidence algebra, at the $\Q$-level, is just the power series algebra
  $\Q[[z]]$: since this time the zeta function is the exponential $\exp(z)$, 
  the \M function is $\exp(-z)$, corresponding to
  $$
  \mu(n) = (-1)^n .
  $$

%
%
%

The economical \M function suggests the existence of the following
equivalence at the groupoid level:
$$
\mu(S) = \int^r (-1)^r h^r(S) = \Beven(S) - \Bodd(S) ,
$$
where 
$$\Beven = \sum_{r \text{ even}} \B_{[r]} \quad \text{ and }\quad
\Bodd = \sum_{r \text{ odd}} \B_{[r]}$$
are the full subgroupoids of $\B$ consisting of the even and odd sets,
respectively.  However, it seems that such an equivalence is not possible, at
least not over $\B$: while we are able to exhibit a bijective proof, this
bijection is {\em not} natural, and hence does not assemble into a groupoid
equivalence.

\end{blanko}

\begin{prop}
  For a fixed set $S$, there are monomorphisms $\Beven(S) \into \Phieven(S)$
  and $\Bodd(S) \into \Phiodd(S)$, and a residual bijection
  $$
  \Phieven(S)-\Beven(S)= \Phiodd(S) -\Bodd(S) .
  $$
  This is {\em not} natural in $S$, though, and hence does not constitute
  an isomorphism of species, only an equipotence of species.
\end{prop}

\begin{cor}
  For a fixed $S$ there is a bijection
  $$
  \mu(S) \simeq \Beven(S) - \Bodd(S)
  $$
  but it is {\em not} natural in $S$.
\end{cor}

\begin{proof*}{Proof of the Proposition.}
  The map $\Beven \to \B$ is a  monomorphism, so for each set $S$ of even cardinality
  there is a single element to subtract from $\Phieven(S)$.  The groupoid
  $\Phieven$ has as objects finite sets $S$ equipped with a surjection $S \onto \un k$
  for some even $k$.  If $S$ is itself of even cardinality $n$, then among such
  partitions there are $n!$ possible partitions into $n$ parts.  If there were
  given a total order on $S$, among these $n!$ $n$-block partitions, there is
  one for which the order of $S$ agrees with the order of the $n$ parts.  We
  would like to subtract that one and then establish the required bijection.
  This can be done fibrewise: over a given $n$-element set $S$, we can establish
  the bijection by choosing first a bijection $S \simeq \un n =
  \{1,2,\ldots, n\}$, the totally ordered set with $n$ elements.

  {\em For each $n$, there is an explicit bijection
  $$
  \{ \text{surjections } p: \un n \onto \un k \mid k \text{ even}, p 
  \text{ not the identity map} \}
  $$
$$  \leftrightarrow$$
  $$
    \{ \text{surjections } p: \un n \onto \un k \mid k \text{ odd}, p 
  \text{ not the identity map} \} 
  $$
  }

  Indeed, define first the bijection on the subsets for which $p^{-1}(1)\neq \{1\}$,
  i.e.~the element $1$ is not alone in the first block.  In this case the 
  bijection goes as follows.  If the elements $1$ is alone in a block,
  join this block with the previous block.  (There exists a previous block as
  we have excluded the case where $1$ is alone in block $1$.)  If $1$ is not
  alone in a block, separate out $1$ to a block on its own, coming just after
  the original block.  Example
  $$
  (34,1,26,5) \leftrightarrow (134,2,6,5)
  $$
  For the remaining case, where $1$ is alone in the first block, we just leave 
  it
  alone, and treat the remaining elements inductively, considering now the
  case where the element $2$ is not alone in the second block.  In the end,
  the only case not treated is the case where for each $j$, we have 
  $p^{-1}(j)=\{j\}$, that is, each element is alone in the block with the same
  number.  This is precisely the identity map excluded explicitly in the 
  bijection.  (Note that for each $n$, this case only appears on one of the 
  sides
  of the bijection, as either $n$ is even or $n$ is odd.)
\end{proof*}



In fact, already subtracting the groupoid $\Beven$ from $\Phieven$ 
is not possible naturally.  We would have first to find a monomorphism
$\Beven\into\Phieven$ over $\B$.  But the automorphism group of an
object $\un n \in \B$ is $\mathfrak S_n$, whereas the 
automorphism
group of any overlying object in $\Phieven$ is a proper subgroup of 
$\mathfrak S_n$.  In fact it is the subgroup of those permutations 
that
are compatible with the surjection $\un n \onto \un k$.
So locally the fibration $\Phieven \to \B$ is a group monomorphism,
and hence it cannot have a section.
So in conclusion, we cannot even realise $\Beven$ as a full 
subgroupoid in $\Phieven$, and hence it doesn't make sense to 
subtract it.

%
%

\bigskip

One may note that  it is not logically necessary to be able to subtract
the redundancies from $\Phieven$ and $\Phiodd$ in order to find the economical
formula.
It is enough to establish directly (by a separate proof) that the economical 
formula holds, by actually convoluting it with the zeta functor.
At the object level the simplified \M function would be the groupoid
$$
\Beven - \Bodd .
$$
We might try to establish directly that
$$
\zeta * \Beven = \zeta * \Bodd + \epsilon .
$$
This should be a groupoid equivalence over $\B$.
But again we can only establish this
fibrewise.  This time, however, rather than exploiting a non-natural total
order, we can get away with a non-natural base-point.
On the left-hand side, the fibre over an $n$-element set $S$, consists
of an arbitrary set and an even set whose disjoint union is $S$.  In other 
words,
it suffices to give an even subset of $S$.  Analogously, on the right-hand
side, it amounts to giving an odd subset of $S$ --- or in the special case of
$S=\emptyset$, we also have the possibility of giving that set, thanks to the
summand $\epsilon$.  This is possible, non-naturally:

{\em 
For a fixed nonempty set $S$, there is an explicit bijection between even subsets of
  $S$ and odd subsets of $S$.}

  Indeed, fix an element $s\in S$.  The bijection consists of adding $s$ to the subset 
  $U$
  if it does not belong to $U$, and removing it if it already belongs to $U$.
  Clearly this changes the parity of the set.

Again, since the bijection involves the choice of a basepoint, it seems impossible to lift
it to a natural bijection.
%

\begin{blanko}{Finite vector spaces.}
  We calculate the \M function in the incidence algebra of the Waldhausen
  decomposition space of $\F_q$-vector spaces, cf.~\ref{ex:q}.
  In this case, $\Phi_r$ is the
  groupoid of strings of $r-1$ nontrivial injections.  The fibre over $V$ is the
  discrete groupoid of strings of $r-1$ nontrivial injections whose last space
  is $V$.  This is precisely the set of nontrivial $r$-flags in $V$, i.e.~flags
  for which the $r$ consecutive codimensions are nonzero.
  In conclusion,
$$
\mu(V) =  \sum_{r=0}^n (-1)^r  \{ \text{ nontrivial $r$-flags in $V$} \} .
$$
(That's in principle a groupoid, but since we have fixed $V$, it is just a 
discrete groupoid: a flag inside a fixed vector space has no automorphisms.)

The number of flags with codimension sequence $p$ is the $q$-multinomial 
coefficient
$$
{ n \choose p_1, p_2, \ldots, p_r }_q .
$$
In conclusion, at the numerical level we find
$$\mu(V) = \mu(n) = \sum_{r=0}^n (-1)^r 
\sum_{\substack {p_1+\cdots + p_r=n \\ p_i > 0}}
{ n \choose p_1, p_2, \ldots, p_r }_q .
$$
%
%
%

On the other hand, it is classical that from the power-series representation
(\ref{ex:q}) one gets the numerical \M function
$$\mu(n) = (-1)^n q^{n\choose 2}.$$
While the equality of these two expressions can easily be established at the
numerical level (for example via a zeta-polynomial argument, cf.~below), we do
not know of a objective interpretation of the expression $\mu(n) = (-1)^n
q^{n\choose 2}$.  Realising the cancellation on the objective level would
require first of all to being able to impose extra structure on $V$ in such a
way that among all nontrivial $r$-flags, there would be $q^{r\choose 2}$ special
ones!
%
%
%
%
%
%
\end{blanko}

\begin{blanko}{\fdb.}
  Recall (from \ref{ex:P=Dec(S)})
  that the incidence coalgebra of the category of surjections
  is the \fdb coalgebra.  Since this is a monoidal decomposition space,
  we have at our disposal the notion of multiplicative function, and these
  are determined by their values on the connected surjections.  The 
  multiplicative functions form a subalgebra of the incidence algebra, and
  clearly this subring contains both $\zeta$ and $\epsilon$, and hence 
  $\mu$.  It is therefore sufficient to calculate the \M function on
  connected surjections.
 
  The general formula gives
$$
\mu( \un n \onto \un 1) = \sum_{r=0}^n (-1)^n \kat{Tr}(n,r)
$$
where $\kat{Tr}(n,r)$ is the (discrete) groupoid of $n$-leaf $r$-level trees
with no trivial level (in fact, more precisely, strings of $r$ nontrivial
surjections composing to $n \onto 1$).

%
%
%

On the other hand, classical theory (see Doubilet--Rota--Stanley~\cite{Doubilet-Rota-Stanley})
gives the following `connected 
\M function':
$$
\mu(n) = (-1)^{n-1}(n-1)! .
$$
In conjunction, the two expressions yield
the following combinatorial identity:
$$
(-1)^{n-1}(n-1)! \ = \ \sum_{r=0}^n (-1)^r \# \kat{Tr}(n,r).
$$

We do not know how to realise the cancellation at the objective level.
This would require developing first the theory of monoidal decomposition spaces 
and incidence bialgebras a bit further, a task we plan to take up in the near 
future.

\end{blanko}


\begin{blanko}{Zeta polynomials.}
  For a complete decomposition space $X$, we can write
  $$
  X_r = \sum_w X_w = \sum_{k=0}^r {r \choose k} \nondeg X_k .
  $$
  where $w$ runs over the words of length $r$ in the alphabet $\{0,a\}$
  as in \ref{w}, and the binomial coefficient is an abusive shorthand
  for that many copies of  $\nondeg X_k$, embedded disjointly into $X_r$ by specific
  degeneracy maps.  Now we fibre over a fixed arrow $f\in X_1$, to obtain
  $$
  (X_r)_f =\sum_{k=0}^\infty {r \choose k} (\nondeg X_k)_f ,
  $$
  where we have now allowed ourselves to sum to $\infty$.
  
  The {\em `zeta polynomial}' of a decomposition space $X$ is the function
  \begin{eqnarray*}
    \zeta^r(f) : X_1 \times \N & \longrightarrow & \Grpd  \\
    (f,r) & \longmapsto & (X_r)_f
  \end{eqnarray*}
  assigning to each arrow $f$ and length $r$ the $\infty$-groupoid of $r$-simplices with 
  long edge $f$.
  We don't actually know whether in general this is a polynomial in $r$, but when we
  know how to compute it, and it is a polynomial, 
  then we can substitute  $r= -1$ into it to find (assuming of course that $X$ 
  is complete):
  $$
  \zeta^{-1}(f) = \sum_{k=0}^\infty (-1)^k \Phi_k(f) 
  $$
  Hence $\zeta^{-1}(f) = \mu(f)$, as the notation suggests.
  
  In some cases there is a polynomial formula for $\zeta^r(f)$.
  For example, in the case $X= (\N,+)$ we find
  $\zeta^r(n) = {n+r-1\choose n}$,
  and therefore $\mu(n) = {n-2\choose n}$, in agreement with the other 
  calculations (of this trivial example).
  
%
    
  \noindent
  In the case $X= (\B,+)$, we find $\zeta^r(n) = r^n$,
  and therefore $\mu(n) = (-1)^n$ again.
 
  Sometimes, even when a formula for $\zeta^r(n)$ cannot readily be found, the
  $(-1)$-value can be found by a power-series
  representation argument.  For example in the case of the Waldhausen $S$
  of $\vect$, we have that $\zeta^r(n)$ is the set of $r$-flags of $\F_q^n$
  (allowing trivial steps).  We have
  $$
\zeta^r(n) = \sum_{\substack{p_1+\cdots+p_r=n \\ p_i \geq 0}}
\frac{[n]!}{[p_1]!\cdots [p_r]!},
$$
and therefore 
$$
\sum_{n=0}^\infty \zeta^r(n) \frac{z^n}{[n]!} = \left( \sum_{n=0}^\infty 
\frac{z^n}{[n]!}\right)^r ,
$$
Now $\zeta^{-1}(n)$ can be read off as the $n$th coefficient in the inverted
series $\big( \sum_{n=0}^\infty 
\frac{z^n}{[n]!}\big)^{-1}$.  In the case at, hand, these coefficients are $(-1)^n q^{n
\choose 2}$, as we already saw.

\end{blanko}

%
%
%
%
%

Once a more economical \M function has been found for a decomposition space $X$,
it can be exploited to yield 
more economical formulae for any decomposition
space $Y$ with a cULF functor to $X$.  This is the content of the following
obvious lemma:
\begin{lemma}
  Suppose that for the complete decomposition space $X$ we have found
  a \M inversion formula
  $$
  \zeta * \Psi_0 = \zeta * \Psi_1 + \epsilon .
  $$
  Then for every decomposition space cULF over $X$, say $f:Y \to X$, we
  have the same formula
  $$
  \zeta * f\upperstar \Psi_0 = \zeta * f\upperstar \Psi_1 + \epsilon 
  $$  
  for $Y$.
\end{lemma}


\begin{blanko}{Length.}\label{ex:cULF/N}
  In most of the examples treated, the length filtration is actually a grading.
  Recall from \ref{grading} that this amounts to having a simplicial map from
  $X$ to the nerve of $(\N,+)$.  In the rather special situation when this is
  cULF, the economical \M function formula
  $$
  \mu = \delta^0 - \delta^1
  $$
  for $(\N,+)$
  induces the same formula for the \M functor of $X$.
  This is of course a very restrictive condition; in fact, 
  for nerves of categories, this happens only for free categories on
  directed graphs (cf.~Street~\cite{Street:categorical-structures}).
  For such categories, there is for each $n\in \N$ a linear span
  $\delta^n$ consisting of all the arrows of length $n$.
  In particular, $\delta^0$ is the span $X_1 \leftarrow X_0 \to 1$
  (the inclusion of the vertex set into the set of arrows),
  and $\delta^1$ is the span $X_1 \leftarrow E \to 1$,
  the inclusion of the original set of edges into the set of
  all arrows.
  The simplest example is the free monoid on a set $S$, i.e.~the monoid
  of words in the alphabet $S$.  The economical \M function is then
  $\delta^0-\delta^1$, where $\delta^1 = \sum_{s\in S} \delta^s$.  In
  the power series ring, with a variable $z_s$ for each letter $s\in S$,
  it is the series $1-\sum_{s\in S} z_s$.
\end{blanko}

\begin{blanko}{Decomposition spaces over $\mathbf B$ (\ref{ex:SI&B}).}
  Similarly, if a decomposition space $X$ admits a cULF functor $\ell : X \to 
  \mathbf B$ (which may be thought of as a `length function with symmetries')
  then at the numerical level  and at the objective level, locally for each 
  object $S\in X_1$, 
  we can pull back the economical \M `functor' $\mu(n) = (-1)^n$
  from $\mathbf B$ to $X$, yielding the numerical  \M function on $X$
  $$
  \mu(f) = (-1)^{\ell(f)} .
  $$
  An example of this is the coalgebra of graphs \ref{ex:graphs} of 
  Schmitt~\cite{Schmitt:hacs}: the functor from the decomposition space of
  graphs to $\mathbf B$ which to a graph associates its vertex set is cULF.
  Hence the \M function for this decomposition space is
  $$
  \mu(G) = (-1)^{\# V(G) } .
  $$
  In fact this argument  works for any restriction species.
  
%
%
%
\end{blanko}

We finish with a kind of non-example which raises certain interesting
questions.
\begin{eks}
  Consider the strict nerve of the category
  $$
  \xymatrix {
  x \ar@(lu,ld)[]_e  \ar@/_0.6pc/[r]_r  & y \ar@/_0.6pc/[l]_s
  }
  $$
  in which $r\circ s = \id_y$, $s\circ r = e$
  and $e\circ e = e$.
  This decomposition space $X$ is clearly locally finite, so it defines
  a vector-space coalgebra, in fact a finite-dimensional one.
  One can check by linear algebra (see 
  Leinster~\cite[Ex.6.2]{Leinster:1201.0413}),
  that this coalgebra has \M inversion.
  On the other hand, $X$ is not of locally finite length,
  because the identity arrow $\id_y$ can be written as an 
  arbitrary long string $\id_y = r\circ s \circ \cdots \circ r\circ s$.
  (It is not even split as then the identity arrow would have no decomposition at 
  all.)  In particular $X$ is not a \M decomposition space.
  So we are in the following embarrassing situation: on the objective level,
  $X$ has \M inversion (as it is complete), but the formula does not
  have a cardinality.  At the same time, at the numerical level \M inversion 
  exists nevertheless.  Since inverses are unique if they exist, it is
  therefore likely that the infinite \M inversion formula of the objective
  level admits some drastic cancellation at this level, yielding a finite
  formula, whose cardinality is the numerical formula.  Unfortunately we
  have not been able to pinpoint such a cancellation.
%
%
%
%
\end{eks}

%% file: restriction.tex
\def\inputfile{restriction.tex}

\newcommand{\BB}{\mathscr{B}}
\newcommand{\Span}{\operatorname{Span}}

\section{Restriction species and directed restriction species}

\label{sec:restriction}

We show that restriction species and their associated coalgebras in the sense of
Schmitt~\cite{Schmitt:hacs} are examples of decomposition spaces.  Then we 
introduce the notion of {\em directed restriction species}, which covers various 
classical combinatorial coalgebras (such as for example the Connes-Kreimer 
bialgebra) and show that they also come from decomposition spaces.
We unify the proofs of these results by giving a general construction of decomposition
spaces from what we call sesquicartesian fibrations over the ordinal category 
$\un\Delta$, involving covariant functoriality in all maps, and contravariant 
functoriality in convex inclusions.

The general construction can be viewed as follows.  Since a monoid can be
considered a one-object category, it yields in particular a decomposition space.
Instead of regarding a monoid as a Segal space $X : \Delta\op\to\Grpd$ with the
property that $X_0 = 1$, monoids can be encoded as monoidal functors
$$
(\un\Delta, +,0) \to (\Grpd,\times,1),
$$
and hence in particular are certain kinds of left fibrations
$X \to \un\Delta$.
In this setting, a weaker structure than monoid
is sufficient to obtain a decomposition space.

\subsection{Restriction species (in the sense of Schmitt)}

\begin{blanko}{Restriction species.}
  The notion of restriction species was introduced by Schmitt~\cite{Schmitt:hacs}:
  it is simply a presheaf on the category $\I$ of finite sets and injections.
  Compared to a classical species~\cite{JoyalMR633783}, a restriction species 
  $R$ is thus
  functorial not only on bijections but also on injections, meaning that
  a given structure on a set $S$ induces also such a structure on every subset $A 
  \subset S$ (denoted with a restriction bar):
  \begin{eqnarray*}
    R[S] & \longrightarrow & R[A]  \\
    X & \longmapsto & X|A .
  \end{eqnarray*}
  
  The Schmitt construction associates to a restriction species $R: \I\op\to\Set$
  a coalgebra structure on the vector space spanned by the
  isoclasses of $R$-structures: the comultiplication is
  $$
  \Delta(X)=
  \sum_{A+B=S} X|A \tensor X|B , \qquad X \in R[S] ,
  $$
  and counit sending only the empty structures to $1$.
  
  A morphism of restriction species is just a natural transformation 
  $R\Rightarrow R'$ of functors
  $\I\op\to\Set$, i.e.~for each finite set $S$ a map $R[S]\to R'[S]$, natural 
  in $S$.  Since the summation in the comultiplication formula only involves 
  the underlying sets, it is clear that a morphism of restriction species 
  induces a coalgebra homomorphism.

  A great many combinatorial coalgebras can be realised by the Schmitt
  construction (see \cite{Schmitt:hacs} and also \cite{Aguiar-Mahajan}).  For
  example, graphs (\ref{ex:graphs}), matroids,
  posets, lattices, categories, etc., form restriction species and hence 
  coalgebras.
\end{blanko}

%
%

\begin{blanko}{Restriction species as decomposition spaces.}
  Let $R: \I\op\to \Set$ be a restriction species.  It corresponds
  by the Grothendieck construction to a (discrete) right fibration
  $$
  \R \to \I ,
  $$
  where the total space $\R$ is the category of all $R$-structures and their
  structure-preserving injections.  Precisely, a structure-preserving injection
  from 
  $X\in R[S]$ to $X'\in R[S']$ consists of
  an injection of
  underlying sets $S \subset S'$ such that $X'|S = X$.
  
  We construct a simplicial groupoid
  $\mathbf R$ where $\mathbf R_k$ is the groupoid
  of $R$-structures with an ordered partition of the underlying set
  into $k$ (possibly empty) parts.
%
%
%
%
  Precisely,  with reference to the strict version \ref{B/k} of the 
  finite-sets-and-bijections-nerve $\mathbf B$,
  we define $\mathbf R_k$ 
  as the pullback 
  $$
  \mathbf R_k = \B_{/k} \times_{\B} \R^\iso .
  $$
  The pullback construction delivers all the generic maps in $\mathbf 
  R$, and
  so far the construction works for any species.  To define
  also the free maps (i.e.~outer face maps) we need the restriction structure 
  on $R$:
for example, 
  the outer face map $d_\bot : \B_{/k} \to \B_{/k-1}$ is defined by sending $S
  \to \underline k$ to the pullback
  $$\xymatrix{
  S' \ar[d]\drpullback \ar[r]^\subset & S \ar[d]\\
  \underline{k\!-\!1} \ar[r] & \underline{k} .
  }$$
  Since $S' \into S$ is an injection, we can use functoriality of $R$
  (the fact that $R$ is a restriction species) to get also the face 
  map
  for $\mathbf R_k$.
  We shall formalise these constructions in \ref{restr-decomp-formal}. 
  Note that by construction, as cULF over a decomposition space (the 
  decomposition space $\mathbf B$ (cf.~\ref{I=DecB})),
  $\mathbf R$ is again a decomposition space.
    
  Note that the subtlety in getting the free maps involves
  projecting
  away some parts of the underlying set.  This means that maps lying over
  free maps are not vertical with respect to the projection down to $\I$.
   We shall develop theory to deal with this kind of problem.

\bigskip

  A morphism of restriction species $R \to R'$ corresponds to a morphism of right 
  fibrations $\R\to\R'$, and it is clear that the construction is functorial
  so as to induce a cULF functor of decomposition spaces.
\end{blanko}

\begin{theorem}
  Given a restriction species $R$, the corresponding simplicial groupoid
  $\mathbf R$ is a decomposition space, and the (cardinality of the)
  associated coalgebra is the Schmitt coalgebra of $R$.  A morphism of
  restriction species induces a cULF functor, whose cardinality is the
  coalgebra homomorphism resulting from
  the Schmitt construction.
\end{theorem}

  \bigskip
  
  We have already exploited (\ref{I=DecB}) that lower dec of $\mathbf B$ is 
  $\mathbf I$, the 
  nerve of the category of injections $\mathbb I$.  Similarly, it is 
  straightforward to check  that:
\begin{lemma}
  The lower dec of the decomposition space of a restriction species
  $\mathbb R$ is the fat nerve of $\mathbb R$.
\end{lemma}

\begin{blanko}{Convex poset inclusions.}
  Recall that a subposet $V\subset P$ is {\em convex} if $a,b\in V$ and $a\leq x
  \leq b$ imply $x \in V$.  Let $\C$ denote the category of finite posets and
  convex poset inclusions.
  
  An {\em ordered monotone partition} of a poset $X$ is by definition a
  monotone map $X\to\un k$ for $\un k \in \un\Delta$.  Note that the fibres of
  such a map are convex subposets of $X$.
\end{blanko}

\begin{blanko}{Directed restriction species.}
  We introduce a new notion of directed restriction species,
  which is a generalisation of well-known constructions with
  lattices --- see for example Schmitt~\cite{Schmitt:1994} and also 
  Figueroa and Gracia-Bond\'ia~\cite{Figueroa-GraciaBondia:0408145}.

  A {\em directed
  restriction species} is by definition a functor
  $$
  R:\C\op\to\Grpd ,
  $$
  or equivalently, by the Grothendieck construction, a right fibration
  $\R\to\C$.  The idea is that the value on a poset $S$ is the groupoid of
  all possible $R$-structures that have $S$ as underlying poset.
  A morphism of directed restriction species is just a natural transformation.
\end{blanko}
  
\begin{eks}
  The category of posets and convex inclusions is the terminal directed
  restriction species.  Similarly there is a directed restriction species
  of lattices with convex inclusions, or
  categories with fully faithful cULF functors.  (Note that a category has an
  underlying poset, namely by $(-1)$-truncation of all hom sets.)  Rooted forests
  and convex maps form a directed restriction species.  Similarly for directed
  graphs.  In all these cases, there is a notion of underlying poset, which
  inherits the given structure from the ambient one.  Note that in each case
  there is also a plain restriction species: in fact any subset of elements,
  convex or not, inherits the given structure.
\end{eks}

\begin{blanko}{Coalgebras from directed restriction species.}
  Let $R$ be any directed restriction species.
  An {\em admissible cut} of an object $X\in R[S]$ is by definition 
  a monotone map from the underlying poset $S$ to $\un 2$.
 That is, an admissible cut is an ordered monotone partition
$A+B=S$.
 This agrees
  with the notion of admissible cut in Connes--Kreimer, and in related examples.
  Let $\mathbf R_2$ be the groupoid of $R$-structures with an admissible cut.
  
  A coalgebra is defined by the rule
  \begin{equation}\label{eq:comultRS}
  \Delta(X) = \sum_{A+B=S}  X|A \tensor X|B,  \qquad X \in R[S].
  \end{equation}
  Here the sum is over $\pi_0\mathbf R_2$, that is, all isomorphism classes of admissible cuts.
  
  A special case of this construction is the Connes--Kreimer coalgebra of
  (combinatorial) trees (\ref{ex:CK}).  And also the Manchon--Manin coalgebra of
  directed graphs (\ref{ex:MM}).  Various examples of cobordism categories can
  also be envisioned.
\end{blanko}

\begin{blanko}{Decomposition spaces from directed restriction species.}
  If $\R\to\C$ is a directed restriction species, let $\mathbf R_k$ be
  the groupoid of $R$-structures on posets $S$ with
  ordered monotone partitions into $k$ possibly empty parts. In other words, $\mathbf R_2$
  is the groupoid of $R$-structures with an admissible cut, and  $\mathbf R_k$ is the groupoid of $R$-structures with $k-1$ compatible admissible cuts. 
  The $\mathbf R_k$ form a simplicial groupoid.  The functoriality in generic maps is
  clear, as these do not alter the underlying poset $S$.  Functoriality in 
  free maps comes from the
  structural restrictions, noting that free maps correspond to convex 
  inclusions.
\end{blanko}

\begin{theorem}
  The construction just outlined defines a decomposition space, whose incidence
  coalgebra coincides with Formula~\eqref{eq:comultRS}.  Morphisms of directed restriction
  species induce cULF functors and hence coalgebra homomorphisms. 
\end{theorem}

The theorem can be proved by a direct verification.  The only subtlety is to
establish functoriality in free maps of $\Delta$.  Rather than rendering this
verification we prefer to take a rather abstract approach in the following 
subsections, establishing a
general method for providing functoriality in free maps.

\begin{blanko}{Decalage.}
  Taking upper or lower dec of the decomposition space of a directed restriction
  species yields Segal spaces.  The lower dec gives the (fat nerve of the)
  subcategory of $\R$ consisting of the maps that are order ideal inclusions
  (i.e.~convex inclusions which are also downward closed).  For example, in the
  case of the directed restriction species of forests, we get the
  category of forests and root-preserving inclusions of D\"ur~\cite{Dur:1986}.
  Similarly, the upper dec yields  the (fat nerve of the)
  subcategory of $\R$ consisting of the maps that are order filter inclusions
  (i.e.~convex inclusions which are also upward closed).
\end{blanko}
\subsection{Further simplicial preliminaries}

%


\begin{blanko}{Finite ordinals.}
  Recall that $\un \Delta$ is the category whose objects
are the finite (possibly empty) ordinals
$
\un k := \{1,2,\ldots,k\}$,
and whose arrows are the monotone maps. 
The distance-preserving maps in $\un\Delta$ (which in the subcategory 
$\Delta\subset \un\Delta$ we call `free maps') are called {\em convex}:
they are those $i: \un k' \to \un k$
such that $i(x+1)= i(x) + 1$, for all $1\leq x< k'$.
We denote the convex maps by arrows $\rat$.
Observe that the convex maps are just the canonical inclusions
$$f:\un n\to \un a+\un n+\un b,$$


\end{blanko}

\begin{lemma}
  Convex maps in $\un \Delta$ admit basechange
  along any map. 
  In other words, given the solid cospan consisting of $f$ and $i$,
  with $i$ convex,
  $$\xymatrix{
     \cdot \drpullback\ar@{-->}[r]\ar@{ >-->}[d]_{i'} & \cdot \ar@{ >->}[d]^i \\
     \cdot \ar[r]_f & \cdot
  }$$
  the pullback exists and $i'$ is again convex.  
\end{lemma}

\begin{blanko}{Convex correspondences.}
%
%
  Denote by $\nabla$ the category of {\em convex correspondences} in $\un
  \Delta$: the objects are those of $\un\Delta$, and a morphism is a span
  $$
  \xymatrix{
  \un k' & \ar@{ >->}[l]_i \un k \ar[r]^f & \un n}
  $$
  where $i$ is convex.
  Composition of such spans is given by pullback, as allowed by the lemma.
  By construction, $\nabla$ has a factorisation system in which the left-hand
  class (called {\em backward convex maps}) consists of spans of the form $\xymatrix{\cdot & \ar@{ >->}[l] \cdot 
  \ar[r]^= &\cdot}$, and the right class (called {\em ordinalic})
  consists of spans of the form $\xymatrix{\cdot & \ar[l]_= \cdot 
  \ar[r] &\cdot}$; the right hand class forms of course a subcategory isomorphic 
  to $\un \Delta$.
  Note that $\nabla$ has a zero object, namely $\un 0$.  The zero maps are 
$\un n \lat \un 0 \to \un k$.
\end{blanko}

    Note that pullback squares in $\un\Delta$ along a convex map
    are commutative squares in $\nabla$.
  A pullback square in $\un\Delta$ 
    $$\xymatrix{
     \un n \ar@{ >->}[r]\ar[d] & \un n' \ar[d] \\
     \un    k \ar@{ >->}[r] &\un  k'
  }$$
  has to be interpreted as a square in $\nabla$
  $$\xymatrix{
     \un n \ar[d] & \ar[l]\un n' \ar[d] \\
  \un    k  & \ar[l]\un  k'
  }$$
  in which the leftwards maps are backward convex maps, and the vertical maps
  are ordinalic.  That a square of this form
  commutes as a diagram of convex correpondences is
  precisely to say that it is a pullback in $\un\Delta$.
  \bigskip

A map in $\nabla$ can be understood as
a monotone map, but defined possibly only on a certain middle convex part of an 
ordinal.  The complement of the domain of definition consists of a bottom part
and a top part.  We can make such partial maps total by introducing
new artificial bottom and top elements,
and understand that the undefined parts are mapped there.  Hence we are led
to consider finite ordinals with a bottom and a top element:

\begin{blanko}{Finite strict intervals.}
  Let $\Xi$ denote the category of {\em finite strict intervals}
  (cf.~Joyal~\cite{Joyal:disks}): its objects are finite ordinals with a bottom
  and a top element required to be distinct, and the arrows are the monotone
  maps that also preserve bottom and top.  We denote an object by the number of
  inner points, so as to write for example
$$
\un k := \{\bot,1,2,\ldots,k,\top\} .
$$
(This naming convention is different from that we will use in Section~\ref{sec:master},
where our viewpoint on the same category is a bit different.)

There is a canonical embedding
$$
\un \Delta \into \Xi
$$
which to an ordinal adjoins a new bottom and a new top element.
In particular the indexing convention is
designed to reflect this embedding.
$\Xi$ has a factorisation system in which the left-hand class consists of maps
for which the inverse image of every inner point is singleton (called 
{\em coconvex}), and whose 
right-hand class are the maps for which the inverse image of each of the outer 
points is singleton, in other words, they are the maps coming from
$\un \Delta$ (called ordinalic).
\end{blanko}

From the descriptions we see that the categories $\nabla$ and $\Xi$ are almost the same;
the only difference is for maps factoring through $\un 0$: in $\nabla$ each hom 
set $\Hom_\nabla(\un n,\un k)$ contains exactly one such map, namely the zero 
map $\un n \lat \un 0 \to 
\un k$, whereas in $\Hom_\Xi(\un n,\un k)$ there are $n+1$ maps through $\un 0$,
depending on which elements map to top and bottom in the first step $\un n \to 
\un 0$.  
%
%

\begin{lemma}
  There is a canonical functor $\Xi\to \nabla$, which is bijective on objects,
  and restricts to an isomorphism on the common subcategory $\un\Delta$,
  and also restricts to an isomorphism $\Xi^{\geq 1}_{\mathrm{coconv.}} 
  \isopil \nabla^{\geq 1}_{\mathrm{back.conv.}}$:
  $$\xymatrix{
  &\ar[ld]\un\Delta\ar[rd] &\\
  \Xi \ar[rr] && \nabla \\
  \Xi^{\geq1}_{\mathrm{coconv.}}  \ar[u]
  \ar[rr]^\simeq&& \nabla^{\geq1}_{\mathrm{back.conv.}}\ar[u]
  }$$
  All maps $\un n \to \un 0$ in $\Xi$ are sent to the zero map
  $\un n \lat \un 0 \to \un 0$ in $\nabla$.
\end{lemma}
\noindent
%

The following is standard \cite{Joyal:disks}:
\begin{lemma}
  There is a canonical isomorphism of categories
  $$
  \Delta\op \simeq \Xi
  $$
  restricting to an isomorphism
  $$
  \Deltagen\op \simeq \un \Delta .
  $$
\end{lemma}
\noindent
The generic maps in $\Delta$ 
correspond to the ordinalic maps in $\Xi$, and the free maps in $\Delta$ correspond to 
the coconvex maps in $\Xi$.


Combining these maps we get

\begin{cor}\label{DeltaNabla}
%
%
    There is a canonical functor $\Delta\op\to \nabla$, which is bijective on objects,
  and restricts to an isomorphism on the common subcategories $\Deltagen\op$,
  takes the free maps to the backward-convex maps in $\nabla$, restricting
   to an isomorphism $(\Delta\op_{\mathrm{free}})^{\geq 1} 
  \isopil \nabla^{\geq 1}_{\mathrm{back.conv.}}$, as indicated here:
  $$\xymatrix{
  &\ar[ld]\Deltagen\op\ar[rd] &\\
  \Delta\op \ar[rr] && \nabla \\
  (\Delta\op_{\mathrm{free}})^{\geq 1}  \ar[u]
  \ar[rr]^\simeq&& \nabla^{\geq1}_{\mathrm{back.conv.}}\ar[u]
  }$$
  All maps $[0] \to [n]$ in $\Delta$ are sent to the zero map
  $\un n \lat \un 0 \to \un 0$ in $\nabla$.
\end{cor}
%
%

\begin{cor}
  A simplicial space $X: \Delta\op\to\Grpd$ with $X_0=1$ can be realised
  from a $\nabla$-diagram.
\end{cor}
Indeed, since $X_0$ is terminal, all the maps $X_n \to X_0$ coincide,
so $X$ factors through $\nabla$.

\begin{blanko}{Identity-extension squares.}\label{iesq}
  A square in $\un \Delta$
  $$\xymatrix{
     \un n \ar[r]^j\ar[d]_f & \un n' \ar[d]^g \\
  \un    k \ar@{ >->}[r]_i &\un  k'
  }$$
  in which the bottom map $i$ is a convex map
  is called an {\em identity-extension square (iesq)} if is it of the form
    $$\xymatrix{
     \un n \ar@{ >->}[r]^-j\ar[d]_f & \un a+\un n+\un b \ar[d]^{\id_a+f+\id_b} \\
     \un k \ar@{ >->}[r]_-i & \un a+\un k+\un b .
  }$$
\end{blanko}
  
%

\begin{lemma}
  An identity-extension square is both a pullback and a pushout.
\end{lemma}


  Any identity-extension square in $\un\Delta$ is a commutative square in $\nabla$,
  again called an identity-extension square.

  
\begin{lemma}\label{lemma:ie}
      \begin{enumerate}
	  \item
      An identity-extension square is uniquely determined by $i$ and 
      $f$.
      
      \item An identity-extension square is uniquely determined by $j$ and 
      $f$, provided $n>0$.
      \end{enumerate}
  \end{lemma}
  Note a special case:
  $$\xymatrix{
     \un 0 \ar@{ >->}[r]\ar[d] & \un k' \ar[d]^{\id} \\
  \un    0 \ar@{ >->}[r]_i &\un  k'
  }$$
  is an identity-extension square, but there is
  more than one way to choose the
  $a$ and $b$ parts.

  \bigskip
  
  Recall from Lemma~\ref{genfreepushout} that in $\Delta$
the pushout  of a generic map along a free map is an iesq,
and every iesq  in which $g$ is generic is such a pushout.

\begin{prop}
  Under the correspondence of Corollary~\ref{DeltaNabla}, there is a bijection
  between the set of 
  identity-extensions squares in
  $\un \Delta$ and the set of identity-extension squares in $\Delta$ in which 
  the
  vertical maps are generic
  $$
  \left\{\vcenter{\xymatrix{
     \un n' \ar[d]&\ar@{ >->}[l]  \un n \ar[d] \\
  \un    k' &\ar@{ >->}[l] \un  k}} \quad \text{ in } 
  \un\Delta\right\}
  \qquad = \qquad
  \left\{\vcenter{\xymatrix{
     [n']  & \ar@{ >->}[l] [n] \\
      [k'] \ar@{->|}[u] & \ar@{ >->}[l] [k] \ar@{->|}[u]
}}
\quad \text{ in } \Delta \right\}
  $$
  except in the case $k=0$. 
\end{prop}

In the case $k=0$, we necessarily have $n=0$ and $n'=k'$, but there is not 
even a bijection on the bottom arrows.

\begin{proof}
%
    The bijection is the composite of the three 
    bijections
  $$
  \left\{\vcenter{\xymatrix{
     \un n' \ar[d]&\ar@{ >->}[l]  \un n \ar[d] \\
  \un    k' &\ar@{ >->}[l] \un  k}} \right\}
   =   
   \left\{\vcenter{\xymatrix{
      \un n' \ar[d] & \\
  \un k' &\ar@{ >->}[l] \un  k}} \right\}
   = 
  \left\{\vcenter{\xymatrix{
      [n'] &\\
[k'] \ar@{->|}[u] & \ar@{ >->}[l]  [k] 
}} \right\} = 
  \left\{\vcenter{\xymatrix{
     [n']  & \ar@{ >->}[l] [n] \\
      [k'] \ar@{->|}[u] & \ar@{ >->}[l] [k] \ar@{->|}[u]
}} \right\}
  $$
   where the first bijection is by Lemma~\ref{lemma:ie} (1), the 
   second is by Corollary~\ref{DeltaNabla} (here we use that $k\neq0$), and the
   third is by Lemma~\ref{lemma:ie} (2) restricted to the subcategory 
   $\Delta$.
\end{proof}

%
%


\begin{blanko}{Iesq condition on functors.}
  For a functor $X:\nabla \to \Grpd$, the image of a backward convex map
  is denoted by upperstar.  Precisely, if the backward convex map correponds to
  $i: \un k \rat \un k'$ in $\un\Delta$, we denote its image by $i\upperstar : X_{k'} 
  \to X_k$.  Similarly, the image of an ordinalic map, corresponding to
  $f: \un n \to \un k$ in $\un\Delta$ is denoted $f\lowershriek : X_n \to X_k$.
  For any identity-extension square in $\un\Delta$,
  \begin{equation}\label{eq:iesq}
    \xymatrix{
     \un a+\un n+\un b \ar[d]_{\id_a+f+\id_b=g}
                        &   \un n \ar@{ >->}[l]_-j\ar[d]^f  \\
     \un a+\un k+\un b  &   \un k ,\ar@{ >->}[l]^-i 
    }
  \end{equation}
  since it is a commutative square in $\nabla$,
  it is automatic just from 
  functoriality that the corresponding square in $\Grpd$ commutes:
  \begin{equation}\label{eq:BC}
 \xymatrix{
     X_{a+n+b}\drpullback \ar[r]^-{j\upperstar }\ar[d]_{g\lowershriek}
     & X_{n} \ar[d]^{f\lowershriek} \\
      X_{a+k+b}\ar[r]_-{i\upperstar } &X_{k} .
  }
  \end{equation}
  (This is the `Beck--Chevalley condition' (BC).)
  We say that $X$ satisfies the {\em iesq condition} when \eqref{eq:BC}
  is furthermore a pullback for every identity-extension square \eqref{eq:iesq}.
\end{blanko}

\begin{prop}\label{NablaDecomp}
  If a covariant functor $M: \nabla \to \Grpd$ sends
  identity-extension squares to pullbacks then the composite
  $$
  \Delta\op\to \nabla \to \Grpd
  $$
  is a decomposition space.
\end{prop}

Similarly:
\begin{prop}
  Let $u: M'\Rightarrow M : \nabla \to \Grpd$ be a natural transformation
  between functors that send identity-extension squares to pullbacks.  If $u$
  is cartesian on arrows in $\un \Delta\subset \nabla$, then
  it induces a cULF functor between decomposition spaces.
\end{prop}


\begin{blanko}{Example: monoids.}
  {\em A monoid viewed as a  
  monoidal functor $(\un\Delta, +,0) \to (\Grpd,\times,1)$
  defines a $\nabla$-space which sends iesq to pullbacks.} 
  The contravariant functoriality on the convex maps is given as follows.  
  The value on a convex map $\un n \rat \un a + \un n + \un b$ is simply the 
  projection
  $$
  X_{a+n+b} \simeq X_a \times X_n \times X_b \longrightarrow X_n ,
  $$
  where the first equivalence expresses that $X$ is monoidal.
  For any identity-extension square
      $$\xymatrix{
  \un a+\un n+\un 
     b \ar[d]_{\id_a+f+\id_b=g} &   \un n \ar@{ >->}[l]_-j\ar[d]^f & \\
\un a+\un k+\un b  &   \un k \ar@{ >->}[l]^-i 
  }$$
  the diagram
  $$\xymatrix{
     X_{a+n+b}\drpullback \ar[r]^-{j\upperstar }\ar[d]_{g\lowershriek}
     & X_{n} \ar[d]^{f\lowershriek} \\
      X_{a+k+b}\ar[r]_-{i\upperstar } &X_{k}
  }$$
  is a pullback, since the upperstar functors are just projections.
\end{blanko}

\begin{blanko}{Functors out of $\nabla$.}
  In view of the previous propositions, we are interested in defining functors
  out of $\nabla$.  By its construction as a category of spans, this amounts to
  defining a covariant functor on $\un\Delta$ and a contravariant functor on $\un
  \Delta_{\text{convex}}$ which agree on objects, and such that for every
  pullback along a convex map the Beck--Chevalley condition holds.  Better
  still, we can describe these as certain fibrations over $\un\Delta$, called
  sesquicartesian fibrations, introduced in the next subsection.  The fact that
  $\nabla$ is not the whole bicategory of spans, and that the fibrations are
  similarly restricted, are just a minor distracting point.  The essential
  points of the equivalence are well-understood and documented in the
  literature, as we proceed to explain. 
%
\end{blanko}

The following technical result seems to be due to
Hermida~\cite{Hermida:repr-mult}, with more detailed statement and proof given
by Dawson-Par\'e-Pronk~\cite{Dawson-Pare-Pronk:MR2116323}.
Our dependence on this result (which we don't quite know how to prove in the
$\infty$-setting) means that the rest of this section should be interpreted 
only in $1$-groupoids and $1$-categories.

\begin{prop}
  Let $\D$ be a $1$-category with pullbacks, and let $\BB$ be a bicategory.
  The natural functor $\D \to 
  \Span(\D)$ 
  induces an equivalence of categories
  $$
  \Hom(\Span(\D), \BB) \simeq \operatorname{Sin}_{\operatorname{BC}}(\D,\BB) .
  $$
\end{prop}
\noindent
Here on the left we have pseudo-functors and pseudo-natural transformations,
and on the right we have the category whose objects are sinister pseudofunctors
satisfying the Beck-Chevalley condition (BC), and whose morphisms are the
sinister pseudo-natural transformations.  A
pseudofunctor is {\em sinister} \cite{Dawson-Pare-Pronk:MR2116323}
if it sends all arrows to left adjoints, and
it is {\em BC} if the image of any comma square has invertible mate.  A {\em
sinister pseudo-natural transformation} (between sinister pseudo-functors)
is one whose naturality squares have invertible mate.

On the other hand, when $\BB=\kat{Cat}$ we have:
\begin{prop}
  There is a natural equivalence of categories
  $$
  \operatorname{Sin}_{\operatorname{BC}}(\D,\kat{Cat})\simeq 
  \operatorname{Bicart}_{\operatorname{BC}}(\D) .
  $$
\end{prop}
\noindent
Here on the right we have the category whose objects are bicartesian fibrations
over $\D$ satisfying the Beck-Chevalley condition,
and whose morphisms are functors over $\D$
preserving both cartesian and cocartesian arrows.

The proof of this result can be found (in the $\infty$-case) in 
Lurie~\cite{Lurie:HA}, Proposition~6.2.3.17.  Note however that
Lurie does not consider the Beck-Chevalley condition (although he uses this name for 
something similar).  More precisely he proves that bicartesian fibrations 
correspond to sinister functors and sinister transformations (called by him
right-adjointable squares).  It is clear though that the Beck-Chevalley
condition goes on top of his result.

\bigskip

In the case at hand, the base category is $\un\Delta$, but we only allow
pullbacks along convex maps.

\subsection{Sesquicartesian fibrations}


A functor $X \to S$ is called a {\em bicartesian fibration} (\cite{Lurie:HA}, 
6.2.3.1) when it is
simultaneously a cartesian and a cocartesian
fibration.
We are interested in bicartesian fibrations over $\un\Delta$, except that we
only require the cartesianness over $\un\Delta_{\text{convex}}$.  We call these
{\em sesquicartesian fibrations}.  


A sesquicartesian fibration $X\to\un\Delta$ is said to have the {\em iesq 
property} if for every 
identity-extension square 
      $$\xymatrix{
  \un a+\un n+\un 
     b \ar[d]_{\id_a+f+\id_b=g} &   \un n \ar@{ >->}[l]_-j\ar[d]^f & \\
\un a+\un k+\un b  &   \un k \ar@{ >->}[l]^-i 
  }$$
   the diagram
  $$\xymatrix{
     X_{a+n+b}\drpullback \ar[r]^-{j\upperstar }\ar[d]_{g\lowershriek}
     & X_{n} \ar[d]^{f\lowershriek} \\
      X_{a+k+b}\ar[r]_-{i\upperstar } &X_{k}
  }$$
  not only commutes (that's BC) but is furthermore  a pullback.

\begin{prop}\label{NablaSesq}
  There is an equivalence of categories
  $$
  \Hom(\nabla,\kat{Cat}) \simeq \operatorname{Sesq}_{BC}(\un\Delta) ,
  $$
  under which the iesq conditions correspond to each other.
\end{prop}
This is just a variation of the previous result.

\bigskip

So in order to construct nabla spaces satisfying the iesq property, we can
construct sesquicartesian fibrations satisfying iesq, and then take maximal
sub-groupoid.

\bigskip

\begin{blanko}{Two-sided fibrations.}
  Classically (the notion is due to Street), 
  a {\em two-sided fibration} is 
%
%
%
%
  a span of functors
%
%
  $$
  \xymatrix{X \ar[d]_p\ar[r]^q & T \\
  S&
  }$$
 such that
 
 ---
 $p$ is a cocartesian fibration whose
 $p$-cocartesian arrows are precisely the $q$-vertical arrows, 
 
 ---
 $q$ is a cartesian fibration whose
  $q$-cartesian arrows are precisely the $p$-vertical arrows
  
  --- for $x\in X$, an arrow $f: px\to s$ in $S$ and $g:t \to qx$ in $T$, the canonical map 
  $f\lowershriek g\upperstar  x \to g\upperstar f \lowershriek x$ is an 
  isomorphism.
  
In the setting of $\infty$-categories, Lurie~\cite{Lurie:HA}, Section~2.4.7
(using the terminology `bifibration') characterise two-sided fibrations as
functors $X \to S \times T$ subject to a certain horn-filling condition,
which among other technical advantages makes it clear that the notion is
stable under base change $S' \times T' \to S \times T$.
The classical axioms are derived from the horn-filling condition.
%
%
%
\end{blanko}

\begin{blanko}{The category of arrows }
  $$
  \Ar(\CC) \stackrel{(\mathrm{codom}, \mathrm{dom})}\longrightarrow \CC\times\CC
  $$
  is a two-sided fibration.
  Assuming that $\CC$ has pullbacks, the codomain cocartesian fibration
  $$
  \Ar(\CC) \stackrel{\mathrm{codom}}\to \CC
  $$
  is a bicartesian fibration, and it satisfies BC.  
\end{blanko}


\begin{blanko}{Comma categories.}
  Given functors 
  $$\xymatrix{
  & B \ar[d]^G \\
  A \ar[r]_F & I
  }$$
  the {\em comma category} $A \comma B$
  is the category whose objects are triples $(a,b,\phi)$, where $a\in A$, $b\in 
  B$,
  and $\phi:Fa\to Gb$.  More formally it is defined as 
  the pullback two-sided fibration
  $$\xymatrix{
    A \comma B \drpullback  \ar[r]\ar[d] & \Ar(I) 
    \ar[d]^{(\mathrm{codom},\mathrm{dom})} \\
     B \times A \ar[r]_{G\times F} & I \times I .
  }$$
    Note that the factors come in the opposite order: $A\comma B \to B$ is the
  cocartesian fibration, and $A\comma B \to A$ the cartesian fibration.
\end{blanko}


\begin{lemma}
  Given a two-sided fibration $X \to S \times T$, and let
  $R \to T$ be any map.  Then the left-hand composite
  $$
  \xymatrix{X \times_T R \drpullback \ar[r] \ar[d] & R \ar[d]\\
  X \ar[d]_p\ar[r]^q & T \\
  S&
  }$$
  is a cocartesian fibration. 
\end{lemma}
\begin{proof}
  It is the pullback two-sided fibration of $X \to S\times T$ along $S \times R \to S 
  \times T$.
\end{proof}

\begin{cor}
  In the situation of the previous lemma, 
  if $X \to S$ is furthermore a bicartesian fibration and if
  $R \to T$ is a cartesian 
  fibration, then the left-hand 
  composite is a bicartesian fibration.  
  If $X \to S$ satisfies BC, then
  so does the left-hand composite. 
\end{cor}

%

We don't actually need this result, but rather the following more special case.
\begin{lemma}\label{sesquilemma}
  If $X \to \un\Delta \times T$ is a two-sided fibration such that
  $X \to \un\Delta$ is a sesquicartesian fibration, then for any cartesian 
  fibration $R \to T$, the left-hand composite in the diagram
    $$
  \xymatrix{X \times_T R \drpullback \ar[r] \ar[d] & R \ar[d]\\
  X \ar[d]\ar[r] & T \\
  \un\Delta&
  }$$
  is a sesquicartesian fibration. 
Furthermore, if $X \to \un\Delta$ is iesq,
then so is the left-hand composite.
\end{lemma}
\begin{proof}
  The only non-trivial statement is about the iesq condition:
  given the pullback square 
  $$\xymatrix{
     X_{a+n+b}\drpullback \ar[r]^-{j\upperstar }\ar[d]_{g\lowershriek}
     & X_{n} \ar[d]^{f\lowershriek} \\
      X_{a+k+b}\ar[r]_-{i\upperstar } &X_{k}}$$
      expressing that $X \to S$ has the iesq property, the corresponding square 
      for  $X\times_T R \to S$ is simply obtained applying
  $- \times_T R$ to it, hence is again a pullback.
\end{proof}

\subsection{Decomposition spaces from sesquicartesian fibrations}

\begin{blanko}{Restriction species and directed restriction species.}
  Recall that a restriction species is a right fibration $\R\to\I$, where $\I$
  is the category of finite sets and injections, and that a {\em
  directed restriction species} is a right fibration $\R\to\C$, where $\C$
  denotes the category of posets and convex poset inclusions.
\end{blanko}


\begin{prop}\label{IDelta}
  The projection $\I \comma \un\Delta \to \un\Delta$ is an iesq sesquicartesian
  fibration.
\end{prop}
\begin{proof}
  The comma category is taken over $\Set$.  The objects of $\I\comma \un\Delta$
  are maps $S \to \un k$, and the arrows are squares in $\Set$
  $$\xymatrix{
     T \ar[r]\ar[d] & S \ar[d] \\
     \un n \ar[r] & \un k
  }$$
  with $T\to S$ injective and $\un n \to \un k$ monotone.
  Just from being a comma category projection, $\I \comma \un\Delta \to \un\Delta$ is a
  cocartesian fibration.  The cocartesian arrows are squares in $\Set$ of the form
  $$\xymatrix{
     S \ar[r]^=\ar[d] & S \ar[d] \\
     \un n \ar[r] & \un k .
  }$$
  Over $\un\Delta_{\text{convex}}$ it is also a
  cartesian fibration, as follows readily from the fact that the pullback lying
  over a convex map is injective: the cartesian arrows over a convex map are 
  squares in $\Set$ of the form
  $$\xymatrix{
     S'\drpullback \ar[r]\ar[d] & S \ar[d] \\
     \un k' \ar@{ >->}[r] & \un k
  }$$
  Beck-Chevalley is a consequence of the iesq property.
%
For the latter we need to check that given
      $$\xymatrix{
  \un a+\un n+\un 
     b \ar[d]_{\id_a+f+\id_b=g} &   \un n \ar@{ >->}[l]_-j\ar[d]^f & \\
\un a+\un k+\un b  &   \un k \ar@{ >->}[l]^-i 
  }$$
the resulting square
  $$\xymatrix{
     \I_{/a+n+b}\drpullback \ar[r]^-{j\upperstar }\ar[d]_{g\lowershriek}
     & \I_{/n} \ar[d]^{f\lowershriek} \\
      \I_{/a+k+b}\ar[r]_-{i\upperstar } &\I_{/k}
  }$$
  is a pullback.   But this is clear since $\I$ is a monoidal  extensive 
  category in the sense of \ref{extensive}.
\end{proof}

\begin{cor}\label{restr-decomp-formal}
  For any restriction species $\R\to\I$ the comma category projection
  $\R \comma \un\Delta \to \un\Delta$ is an iesq sesquicartesian fibration.
\end{cor}
\begin{proof}
  This follows from Lemma~\ref{sesquilemma}.
\end{proof}

\begin{prop}
  The projection $\C \comma \un\Delta \to \un\Delta$ is an iesq sesquicartesian
  fibration.
\end{prop}
\begin{proof}
  The comma category is taken over $\kat{Poset}$.  The objects of $\C\comma \un\Delta$
  are poset maps $S \to \un k$, and the arrows are squares in $\kat{Poset}$
  $$\xymatrix{
     T \ar[r]\ar[d] & S \ar[d] \\
     \un n \ar[r] & \un k
  }$$
  with $T\to S$ a convex poset inclusion and $\un n \to \un k$ a monotone 
  map.
  Just from being a comma category projection, $\C \comma \un\Delta \to \un\Delta$ is a
  cocartesian fibration.  The cocartesian arrows are squares in $\kat{Poset}$ of the form
  $$\xymatrix{
     S \ar[r]^=\ar[d] & S \ar[d] \\
     \un n \ar[r] & \un k
  }$$
  Over $\un\Delta_{\text{convex}}$ it is also a
  cartesian fibration, as follows readily from the fact that the pullback lying
  over a convex map is a convex poset inclusion: the cartesian arrows over a convex map are 
  squares in $\kat{Poset}$ of the form
  $$\xymatrix{
     S'\drpullback \ar@{ >->}[r]\ar[d] & S \ar[d] \\
     \un k' \ar@{ >->}[r] & \un k
  }$$
  Beck-Chevalley is obvious from the fact that the
  cartesian arrows are pullback squares.
Finally for the iesq property, here the argument is trickier than in the case 
of finite sets and injections.
We 
need to check
that given
      $$\xymatrix{
  \un a+\un n+\un 
     b \ar[d]_{\id_a+f+\id_b=g} &   \un n \ar@{ >->}[l]_-j\ar[d]^f & \\
\un a+\un k+\un b  &   \un k \ar@{ >->}[l]^-i 
  }$$
the resulting square
  $$\xymatrix{
     \C_{/a+n+b}\drpullback \ar[r]^-{j\upperstar }\ar[d]_{g\lowershriek}
     & \C_{/n} \ar[d]^{f\lowershriek} \\
      \C_{/a+k+b}\ar[r]_-{i\upperstar } &\C_{/k}
  }$$
  is a pullback.   This time it is not the case that $\C$ is extensive.
  Nevertheless, the iesq property is a direct check:
  an element in the pullback $\C_{/n} \times_{\C_{/k}} \C_{/k'}$ amounts 
  of a diagram
$$\xymatrix{
S \ar[r]
\ar[d] \drpullback & T \ar[dd] \\
\un n \ar[d]  & \\
\un k \ar@{ >->}[r] & \un k'
}$$
  Here the part $S\to \un n$ is the element in $\C_{/n}$, and
  $T \to a+k+b$ is the element in $\C_{/k'}$, and saying that they
  have the same image in $\C_{/k}$ is to say that we have the 
  pullback diagram.
  The claim is that given this diagram, there is a unique way to 
  complete it to
  $$\xymatrix{
S \ar[r]
\ar[d] \drpullback & T \ar[dd]|\hole \ar@{-->}[rd] &\\
\un n \ar@{ >->}[rr]\ar[d]  && \un n' \ar[ld] \\
\un k \ar@{ >->}[r] & \un k'
}$$
Namely, on the element level $T$ has three parts, namely the inverse 
images $T_a$, $T_k$ and 
$T_b$.  (We don't need to worry about the poset structure, since
we already know all of $T$.  The crucial thing is therefore
that the covariant functoriality does not change the total space!).
We now define $T\to n'=a+n+b$ as follows: we use $T_a \to a$ and
$T_b \to b$ on the outer parts.  On the middle part we know that
$T_k = S$, so here we just use the map $S\to n$. 
\end{proof}

\begin{cor}
  For any directed restriction species $\R\to\C$ the comma category projection
  $\R \comma \un\Delta \to \un\Delta$ is an iesq sesquicartesian fibration.
\end{cor}
\begin{proof}
  This follows from Lemma~\ref{sesquilemma}.
\end{proof}

%

Now, by Propositions~\ref{NablaDecomp} and~\ref{NablaSesq}, iesq sesquicartesian fibrations over $\un\Delta$ define decomposition spaces. The previous two corollaries therefore imply:

\begin{cor}
  Restriction species and directed restriction species define decomposition 
  spaces.
\end{cor}

\begin{blanko}{Towards decomposition categories.}
  An iesq sesquicartesian fibration defines actually a decomposition {\em category},
  not just a decomposition {\em space}.  In fact we started this section
  observing that we are generalising the notion of monoid, which in our 
  terminology includes monoidal groupoids.  But many of our
  examples were actually monoidal categories, not just monoidal groupoids.  It
  is therefore natural that the constructions meant to generalise these give
  actually simplicial diagrams in categories, not just in spaces or sets.  We
  leave for another occasion the study of decomposition categories.
\end{blanko}



%% file: complete-fish.tex
\def\inputfile{complete-fish.tex}

Lawvere (in 1988, unpublished until Lawvere-Menni~\cite{LawvereMenniMR2720184})
observed that there is a coalgebra (in fact a Hopf algebra)
of isoclasses of \M intervals, which receives a canonical coalgebra homomorphism
from any incidence coalgebra of a \M category.  Furthermore, this 
Hopf algebra has \M inversion, and therefore \M inversion in all other incidence
algebras (of \M categories) are induced from this master inversion formula.

Here is the idea: a {\em \M interval} is a \M category (in the sense of Leroux)
having an initial and a terminal object (not necessarily distinct).  (It follows
that it is actually a finite category.)  An arrow $a:x\to y$ in a \M category
$\C$ determines (\cite{Lawvere:statecats}) a \M interval $I(a)$ (mimicking the
identification of arrows and intervals in a poset), namely the category of
factorisations of $a$: this category has an initial object $0$ given by the
factorisation $\id$-followed-by-$a$, and a terminal object $1$ given by the
factorisation $a$-followed-by-$\id$.  There is a
canonical conservative ULF functor $I(a)\to\C$, sending $0$ to $x$,
sending $1$ to $y$, and sending $0\to 1$ to $a$. 
The longest arrow $0\to 1$ in $I(a)$ has
the same decomposition structure as $a$ in $\C$, and hence
the comultiplication of $a$ can be computed inside $I(a)$.

Any collection of \M intervals closed under subintervals defines a coalgebra.  
It is an interesting integrability condition for such a collection to come from a
single \M category.  
The Lawvere--Menni coalgebra is simply the collection of {\em
all} isomorphism classes of \M intervals.

Now, the coalgebra of \M intervals cannot be the coalgebra of a single Segal
space, because such a Segal space $U$ would have to have $U_1$ the space
of all \M intervals, and $U_2$ the space of all subdivided \M intervals.  But a
\M interval with a subdivision (i.e.~a `midpoint') contains more information than
the two parts of the subdivision: one from $0$ to the midpoint, and one from the
midpoint to $1$:

\begin{center}\begin{texdraw}

  \setunitscale 1
  \arrowheadtype t:V
  \arrowheadsize l:5 w:4
  \move (0 0) 
  \bsegment
  \move (0 0) \clvec (12 10)(28 10)(40 0) 
  \move (0 0) \clvec (12 -10)(28 -10)(40 0)
  \move (0 0) \onedot
  \move (20 0) \onedot
  \move (40 0) \onedot
  \esegment

  \move (70 0)\htext{$\neq$}
  \move (80 0)
  \bsegment
  \move (20 0)
  \clvec (26 7)(34 7)(40 0)
  \move (20 0)
  \clvec (26 -7)(34 -7)(40 0)
  \move (40 0)
  \clvec (46 7)(54 7)(60 0)
  \move (40 0)
  \clvec (46 -7)(54 -7)(60 0)
  \move (20 0) \onedot
  \move (40 0) \onedot
  \move (60 0) \onedot
  \move(60 12)
  \move(60 -12)
  \esegment
\end{texdraw}\end{center}
This is to say that the Segal condition is not satisfied: we have
\begin{eqnarray*}
  U_2 & \neq & U_1 \times_{U_0} U_1 .
\end{eqnarray*}

We shall prove that the simplicial space of all intervals and their subdivisions
{\em is} a decomposition space, as suggested by this figure:

\begin{center}\begin{texdraw}

  \setunitscale 0.9
  \arrowheadtype t:V
  \arrowheadsize l:5 w:4
  \move (0 90)
  \move (0 75) 
  \bsegment
  \move (0 0) \clvec (20 16)(40 16)(60 0)
  \move (0 0) \clvec (20 -16)(40 -16)(60 0)
  \move (0 0) \onedot
  \move (20 0) \onedot
  \move (40 0) \onedot
  \move (60 0) \onedot
  \esegment
  
  \move (80 -3) \rlvec (0 6)
  \move (80 0) \ravec (40 0)
  \move (80 72) \rlvec (0 6)
  \move (80 75) \ravec (40 0)
  \move (27 55) \rlvec (6 0)
  \move (30 55) \ravec (0 -35)
  \move (167 58) \rlvec (6 0)
  \move (170 58) \ravec (0 -40)
  
  \move ( 60 55) \rlvec (5 0) \rlvec (0 5)
  \move (0 0)
  \bsegment
  \move (0 0) \clvec (20 16)(40 16)(60 0)
  \move (0 0) \clvec (20 -16)(40 -16)(60 0)
  \move (0 0) \onedot
  \move (40 0) \onedot
  \move (60 0) \onedot
  \esegment
  
  \move (140 75) 
  \bsegment
  \move (0 0) \clvec (12 10)(28 10)(40 0) \clvec (46 7)(54 7)(60 0)
  \move (0 0) \clvec (12 -10)(28 -10)(40 0) \clvec (46 -7)(54 -7)(60 0)
  \move (0 0) \onedot
  \move (20 0) \onedot
  \move (40 0) \onedot
  \move (60 0) \onedot
  \esegment

    \move (140 0) 
  \bsegment
  \move (0 0) \clvec (12 10)(28 10)(40 0) \clvec (46 7)(54 7)(60 0)
  \move (0 0) \clvec (12 -10)(28 -10)(40 0) \clvec (46 -7)(54 -7)(60 0)
  \move (0 0) \onedot
  \move (40 0) \onedot
  \move (60 0) \onedot
  \esegment\move (0 -12)
\end{texdraw}\end{center}
meant to indicate that this diagram is a pullback:
$$\xymatrix @C=44pt @R=30pt{
   U_3 \drpullback \ar[r]^-{(d_3,d_0d_0)}\ar[d]_{d_1} & U_2\times_{U_0} U_1 \ar[d]^{d_1 \times 
   \id} \\
   U_2 \ar[r]_-{(d_2,d_0)} & U_1 \times_{U_0} U_1
}$$
which in turn is one of the conditions involved in the decomposition-space 
axiom. 

While the ideas outlined have a clear intuitive content, a considerable amount of
machinery is needed actually to construct the universal decomposition space, and to get
sufficient hold of its structural properties to prove the desired results about
it.  The main technicalities concern factorisation systems.  We start with a
subsection on general theory about factorisation systems, some results of which
are already available in Lurie's book~\cite{Lurie:HTT}.

We then develop the theory of intervals, and construct the decomposition space
of all intervals.  We do it first without finiteness conditions, which we impose
at the end.

\subsection{Factorisation systems and cartesian fibrations}
\label{sec:fact}

For background to this subsection, see Lurie~\cite{Lurie:HTT},  \S~5.2.8.

\begin{blanko}{Factorisation systems.}
  A {\em factorisation system} on an $\infty$-category $\DD$ consists of two classes $E$ and $F$
  of maps, that we shall depict as $\onto$ and $\rat$, such that
  \begin{enumerate}
    \item The classes $E$ and $F$ are closed under equivalences.
  
    \item The classes $E$ and $F$ are orthogonal, $E\bot F$.  That is, given $e\in E$ and $f\in F$, for every solid square
    $$\xymatrix{
       \cdot \ar[r]\ar@{->>}[d]_e & \cdot \ar@{ >->}[d]^f \\
       \cdot \ar@{-->}[ru]\ar[r] & \cdot
    }$$
    the space of fillers is contractible.

    \item Every map $h$ admits a factorisation
  $$
  \xymatrix@R-1em@C+1em{\cdot\ar[rr]^h\ar@{->>}[rd]_e&&\cdot\\&\cdot\ar@{ >->}[ru]_f}
  $$
  with $e\in E$ and $f\in F$.
  \end{enumerate}

  (Note that in \cite[Definition 5.2.8.8]{Lurie:HTT},
  the first condition is given as `stability under formation of retracts'.  In 
  fact this stability follows from the three conditions above.  Indeed, suppose 
  $h \bot F$;  
  factor $h = f \circ e$ as above.  Since $h\bot f$, there is a diagonal filler in
  $$\xymatrix{
     \cdot \ar@{->>}[r]^e\ar[d]_h & \cdot \ar@{ >->}[d]^f \\
     \cdot \ar@{-->}[ru]^-d \ar[r]_{\id} & \cdot
  }$$
  Now $d$ belongs to ${}^\bot F$ since $e$ and $h$ do, and $d$ belongs to
  $E^\bot$ since $f$ and $\id$ do.  Hence $d$ is an equivalence, and therefore
  $h \in E$, by equivalence stability of $E$.  Hence $E= {}^\bot F$,
  and is therefore closed under retracts.  Similarly for $F$.  It also follows 
  that the two classes are closed under composition.)
\end{blanko}

\begin{blanko}{Set-up.}\label{setup}
  In this subsection, fix an $\infty$-category $\DD$ with a factorisation system
  $(E,F)$ as above.  
Let $\Ar(\DD)= \Fun(\Delta[1],\DD)$, 
 whose $0$-simplices we depict vertically, then the domain projection $\Ar(\DD) 
  \to\DD$ (induced by the inclusion $\{0\} \into \Delta[1]$) is a cartesian 
  fibration; the cartesian arrows are the squares of the form
  $$\xymatrix{
  \cdot \ar[r] \ar[d] & \cdot \ar[d] \\
  \cdot 
  \ar[r]^\sim & \cdot}
  $$

  Let $\Ar^E(\DD) \subset \Ar(\DD)$ denote the full subcategory spanned by the
  arrows in the left-hand class $E$.
\end{blanko}

\begin{lemma}\label{ArED}
  The domain projection $\Ar^E(\DD) \to \DD$ is a cartesian fibration. The cartesian
  arrows in $\Ar^E(\DD)$ are given by squares of the form
  $$\xymatrix{
  \cdot \ar[r] \ar@{->>}[d] & \cdot \ar@{->>}[d] \\
  \cdot 
  \ar@{ >->}[r] & \cdot}$$
\end{lemma}
\begin{proof}
  The essence of the argument is to
  provide uniquely the dashed arrow in
  $$\xymatrix{A \ar[rrd]\ar[rd]\ar@{->>}[d] &&\\
  S \ar[rrd]|\hole\ar@{-->}[rd] & \cdot \ar[r]\ar@{->>}[d] & \cdot \ar@{->>}[d] \\
  & X \ar@{ >->}[r] & Y
  }$$
  which amounts to filling
  $$\xymatrix{ A \ar[r]\ar@{->>}[d] & X \ar@{ >->}[d] \\
  S \ar@{-->}[ru] \ar[r] & Y ,}$$
  in turn uniquely fillable by  orthogonality $E \bot F$.
\end{proof}


\begin{lemma}\label{coreflection:w}
  The inclusion $\Ar^E(\DD)\to \Ar(\DD)$ admits a right adjoint $w$.
  This right adjoint $w :\Ar(\DD) \to \Ar^E(\DD)$ sends an arrow $a$ to its $E$-factor.
  In other words, if $a$ factors as $a= f\circ e$ then $w(a)=e$.  
\end{lemma}

\begin{proof}
  This is dual to \cite[5.2.8.19]{Lurie:HTT}.
\end{proof}

\begin{lemma}
  The right adjoint $w$ sends cartesian arrows 
in $\Ar(\DD)$ 
to cartesian arrows
in $\Ar^E(\DD)$.
\end{lemma}

\begin{proof}
  This can be seen from the factorisation:
$$
\vcenter{\xymatrix{
  \cdot \ar[r] \ar[d] & \cdot \ar[d] \\
  \cdot 
  \ar[r]^\sim & \cdot}}\qquad\mapsto\qquad
\vcenter{
\xymatrix{  \cdot \ar@{->>}[d] \ar[r] & \cdot \ar@{->>}[d] \\
  \cdot \ar[r] \ar@{ >->}[d] & \cdot \ar@{ >->}[d] \\
  \cdot 
  \ar[r]^\sim & \cdot}
}
$$
The middle horizontal arrow is forced into $F$ by the closure properties of 
right classes.
\end{proof}

Let $\Fun'(\Lambda^1_2, \DD) = \Ar^E(\DD) \times_{\DD} \Ar^F(\DD)$ denote the 
$\infty$-category whose objects are pairs of composable arrows where the first
arrow is in $E$ and the second in $F$.  Let $\Fun'(\Delta[2],\DD)$ denote the
$\infty$-category of $2$-simplices in $\DD$ for which the two `short' edges are in
$E$ and $F$ respectively.  The projection map 
$\Fun'(\Delta[2],\DD) \to \Fun'(\Lambda^1_2,\DD)$ is always a trivial Kan 
fibration, just because $\DD$ is an $\infty$-category.
\begin{prop}\label{HTT:5.2.8.17}
  (\cite[5.2.8.17]{Lurie:HTT}.) The projection $\Fun'(\Delta[2],\DD) \to 
  \Fun(\Delta[1],\DD)$ induced by the long edge $d_1 : [1] \to [2]$
  is a trivial Kan fibration.
\end{prop}

\begin{cor}\label{cor:ArD=EF}
  There is an equivalence of $\infty$-categories
  $$
  \Ar(\DD) \isopil \Ar^E(\DD) \times_{\DD} \Ar^F(\DD)
  $$
  given by $(E,F)$-factoring an arrow.
\end{cor}
\begin{proof}
  Pick a section  to the map in \ref{HTT:5.2.8.17} and compose with the 
  projection discussed just prior.
\end{proof}

%
%

\bigskip

Let $x$ be an object in $\DD$, and denote by $\DD^E_{x/}$ the $\infty$-category
of $E$-arrows out of $x$.  More formally it is given by the pullback 
$$\xymatrix{
  \DD^E_{x/} \drpullback \ar[r] \ar[d] & \Ar^E(\DD) \ar[d]^{\mathrm{dom}} \\
  {*} \ar[r]_{\name{x}} & \DD}$$
\begin{cor}\label{cor:DxF}
  We have a pullback
  $$\xymatrix{
  \DD_{x/} \drpullback \ar[r] \ar[d] & \Ar^F(\DD) \ar[d]^{\mathrm{dom}} \\
  \DD^E_{x/} \ar[r] & \DD}$$
 \end{cor}
\begin{proof}
  In the diagram
  $$\xymatrix{
    \DD_{x/} \drpullback \ar[r] \ar[d] & \Ar(\DD) \drpullback \ar[r]\ar[d]_w & \Ar^F(\DD) \ar[d]^{\mathrm{dom}} \\
  \DD^E_{x/} \drpullback \ar[r] \ar[d]& \Ar^E(\DD) \ar[r]_{\mathrm{codom}} \ar[d]^{\mathrm{dom}} & \DD\\
  {*} \ar[r]_{\name{x}} & \DD
  }$$
  the right-hand square is a pullback by \ref{cor:ArD=EF};
  the bottom square and the left-hand rectangle are clearly pullbacks, hence the
  top-left square is a pullback, and hence the top rectangle is too.
\end{proof}

\begin{lemma}\label{lem:Dx'Dx}
  Let $e: x\to x'$ be an arrow in the class $E$. Then we have a pullback square
  $$\xymatrix{
  \DD_{x'/} \drpullback \ar[d]_w \ar[r]^{e\uppershriek} & \DD_{x/} \ar[d]_w \\
  \DD^E_{x'/}   \ar[r]_{e\uppershriek} & \DD^E_{x/}
  }$$
  Here $e\uppershriek$ means `precompose with $e$'.
\end{lemma}

\begin{proof}
  In the diagram 
    $$\xymatrix{
  \DD_{x'/} \ar[d]_w \ar[r]^{e\uppershriek} & \DD_{x/} \drpullback \ar[d]_w  
  \ar[r] & \Ar^F(\DD) \ar[d]^{\mathrm{dom}}\\
  \DD^E_{x'/}   \ar[r]_{e\uppershriek} & \DD^E_{x/} \ar[r]_{\mathrm{codom}} & \DD
  }$$
  the functor $\DD_{x/}\to \Ar^F(\DD)$ is `taking $F$-factor'.
  Note that the horizontal composites are again `taking $F$-factor' and 
  codomain, respectively, since 
  precomposing with an $E$-map does not change the $F$-factor.
  Since both the right-hand  square and the rectangle are pullbacks by 
  \ref{cor:DxF},
  the left-hand square is a pullback too.
\end{proof}

\begin{blanko}{Restriction.}\label{setup-end}
  We shall need a slight variation of these results.  We continue the assumption
that $\DD$ is a $\infty$-category with a factorisation system $(E,F)$.
Given a full subcategory $\A \subset \DD$,  we denote by
$\A\comma \DD$ the `comma category of arrows in $\DD$ with domain in $\A$'.  More precisely
it is defined as the pullback 
$$\xymatrix{
\A\comma \DD \drpullback \ar[d]_{\mathrm{dom}}\ar[r]^{\mathrm{f.f}} & \Ar(\DD) \ar[d]^{\mathrm{dom}} \\
\A \ar[r]_{\mathrm{f.f}} & \DD}$$
(This is dual to Artin gluing (cf.~\cite{Gepner-Kock}).)
The map $\A\comma \DD \to \A$ is a cartesian fibration.
Similarly, let $\Ar^E(\DD)_{|\A}$ denote the comma category of $E$-arrows with domain 
in $\A$, defined as the pullback 
$$\xymatrix{
\Ar^E(\DD)_{|\A} \drpullback \ar[d]_{\mathrm{dom}}\ar[r]^{\mathrm{f.f}} & \Ar^E(\DD) \ar[d]^{\mathrm{dom}} \\
\A \ar[r]_{\mathrm{f.f}} & \DD}$$
Again $\Ar^E(\DD)_{|\A} \to \A$ is a cartesian fibration (where the cartesian arrows
are squares whose top part is in $\A$ and whose bottom horizontal arrow belongs
to the class $E$).  These two fibrations are just the restriction to $\A$ of
the fibrations $\Ar(\DD) \to \DD$ and $\Ar^E(\DD) \to \DD$. Since the coreflection
$\Ar(\DD) \to \Ar^E(\DD)$ is vertical for the domain fibrations, it restricts to
a coreflection $w:\A\comma \DD \to \Ar^E(\DD)_{|\A}$.

Just as in the unrestricted  situation (Corollary~\ref{cor:ArD=EF}),
we have a pullback square
$$
\xymatrix{\A\comma \DD \drpullback \ar[r] \ar[d]_w & \Ar^F(\DD) \ar[d] \\
\Ar^E(\DD)_{|\A} \ar[r] & \DD}
$$
saying that an arrow in $\DD$ factors like before, also if it starts in
an object in $\A$.
Corollary~\ref{cor:DxF} is the same in the restricted situation --- just assume
that $x$ is an object in $\A$.   Lemma~\ref{lem:Dx'Dx} is also the same, just
assume that $e:x'\to x$ is an $E$-arrow between $\A$-objects.
\end{blanko}

\bigskip

The following easy lemma expresses the general idea of extending a 
factorisation system.

\begin{lemma}\label{fact-ext}
  Given an adjunction $\xymatrix{L:\DD \ar@<+3pt>[r] & \ar@<+3pt>[l] \CC:R}$
  and given
  a factorisation system $(E,F)$ on $\DD$ with
  the properties

  --- $RL$ preserves the class $F$;

  --- $R\epsilon$ belongs to $F$;

  \noindent
  consider the full subcategory $\wtil \DD \subset \CC$ spanned by the image of
  $L$.  (This can be viewed as the Kleisli category of the monad $RL$.)  Then
  there is an induced factorisation system $(\wtil E, \wtil F)$ on $\wtil
  \DD\subset \CC$ with $\wtil E := L(E)$ (saturated by equivalences), and $\wtil
  F := R^{-1}F \cap \wtil \DD$.
\end{lemma}

\begin{proof}
  It is clear that the classes $\wtil E$ and $\wtil F$ are closed under
  equivalences.  The two classes are orthogonal: given $Le\in \wtil E$ and
  $\tilde f \in \wtil F$ we have $Le \bot \tilde f$ in the full subcategory
  $\wtil\DD \subset\CC$ if and only if $e \bot R\tilde f$ in $\DD$, and the
  latter is true since $R\tilde f \in F$ by definition of $\wtil F$.  Finally,
  every map $g: LA \to X$ in $\wtil\DD$ admits an $(\wtil E, \wtil
  F)$-factorisation: indeed, it is transpose to a map $A \to RX$, which we
  simply $(E,F)$-factor in $\DD$,
  $$
  \xymatrix@R-1em@C+1em{A\ar[rr]\ar[rd]_e&&RX,\\&D\ar[ru]_f}
  $$
  and transpose back the factorisation (i.e.~apply $L$ and postcompose with the counit):
  $g$ is now the composite
  $$
  \xymatrix{ LA \ar[r]^{Le} & LD \ar[r]^{Lf} & LRX \ar[r]^\epsilon & X ,}
  $$
  where clearly $Le \in \wtil E$, and we also have $\epsilon \circ Lf \in \wtil F$
  because of the two conditions imposed.
\end{proof}


\begin{blanko}{Remarks.}\label{fact-decomp}
  By general theory (\ref{coreflection:w}),
  having the factorisation system $(\wtil E , \wtil F)$ 
  implies the existence of a right adjoint to the inclusion
  $$
  \Ar^{\wtil E}(\wtil\DD)  \longrightarrow \Ar(\wtil \DD).
  $$
  This right adjoint returns the $\wtil E$-factor of an arrow.

  Inspection of the proof of \ref{fact-ext} shows that we have the 
  same factorisation property 
  for other maps in $\CC$ than those between objects in $\Im L$, namely
  giving up the requirement that the codomain
  should belong to $\Im L$: it is enough that the domain belongs to $\Im L$:
  {\em every map in $\CC$ whose domain belongs to $\Im L$
  factors as a map in $\wtil E$ followed by a map in $\wtil F:=R^{-1}F$, and we
  still have $\wtil E \bot \wtil F$, without restriction on the codomain
  in the right-hand class.}  This result amounts to a coreflection:
\end{blanko}

\begin{theorem}\label{fact-theorem}
  In the situation of Lemma~\ref{fact-ext}, let $\wtil\DD \comma\CC \subset \Ar(\CC)$
  denote the {\em full} subcategory spanned by the maps with domain in $\Im L$.
  The inclusion functor
$$
\Ar^{\wtil E}(\wtil\DD) \into \wtil\DD\comma\CC
$$
has a right adjoint,
given by factoring any map with domain in $\Im L$
and returning the $\wtil E$-factor.
Furthermore, the right adjoint preserves cartesian arrows (for the domain 
projections).
\end{theorem}

\begin{proof}
  Given that the factorisations exist as explained above, the proof now follows
  the proof of Lemma~5.2.8.18 in Lurie~\cite{Lurie:HTT}, using the dual of his
  Proposition~5.2.7.8.
\end{proof}

%
%

\bigskip

The following restricted version of these results will be useful.

\begin{lemma}\label{fact-Kl-E}
  In the situation of Lemma~\ref{fact-ext}, assume
%
%
%
   there is a full subcategory $J:\A \into\DD$ such that
  
  --- All arrows in $\A$ belong to $E$.
  
  --- If an arrow in $\DD$ has its domain in $\A$, then its $E$-factor
  also belongs to $\A$.
    
  \noindent
  Consider the full subcategory $\wtil \A \subset \CC$ spanned by the image of
  $LJ$.  (This can be viewed as some kind of restricted Kleisli category.)  Then
  there is induced a factorisation system $(\wtil E, \wtil F)$ on $\wtil
  \A\subset \CC$ with $\wtil E := LJ(E)$ (saturated by equivalences), and $\wtil
  F := R^{-1}F \cap \wtil \A$.
\end{lemma}

\begin{proof}
    The proof is the same as before.
\end{proof}

\medskip


\begin{blanko}{A basic factorisation system.}\label{eq1-cart}
    Suppose $\CC$ is any $\infty$-category, and $\DD$ is an $\infty$-category 
    with a terminal object $1$.  Then evaluation on
    $1$ defines a cartesian fibration
    $$ev_1:\Fun(\DD,\CC) \to \CC$$ 
    for which the cartesian arrows are precisely the cartesian natural
    transformations.  The vertical arrows are the natural 
    transformations whose component at $1$ is an equivalence.
    Hence the functor $\infty$-category has a factorisation system in which the
    left-hand class is the class of vertical natural transformations, 
    and the right-hand class is the class of cartesian natural 
    transformations:
  $$
  \xymatrix@R-1em@C+1em{X\ar[rr]\ar[rd]_{\mathrm{eq. on 1}}&&Y\\
  &Y'\ar[ru]_{\mathrm{cartesian}}}
  $$
\end{blanko}

\bigskip

Finally we shall need the following general result (not related to factorisation 
systems):
\begin{lemma}\label{preservespullback}
  Let $\DD$ be any $\infty$-category.  Then the functor
  \begin{eqnarray*}
    F: \DD\op & \longrightarrow & \Grpd  \\
    D & \longmapsto & (\DD_{D/})^{\eq},
  \end{eqnarray*}
  corresponding to the right fibration $\Ar(\DD)^{\mathrm{cart}}\to\DD$,
  preserves pullbacks.
\end{lemma}

\begin{proof}
  Observe first that $F = \colim_{X\in \DD^\eq} \Map ( - , X)$,
  a homotopy sum of representables.  
  Given now a pushout in $\DD$,
  $$\xymatrix{ D \drpullback & \ar[l] B \\
\ar[u] A & \ar[l] \ar[u] C}$$
  we compute, using the distributive law:
  \begin{eqnarray*}
     F(A \coprod_C B )& = & \colim_{X\in \DD^\eq} \Map ( A \coprod_C B , X)   \\
 & = & \colim_{X\in \DD^\eq} \big( \Map (A,X) \times_{\Map(C,X)} \Map(B,X) \big)  \\
 & = & \colim_{X\in \DD^\eq}  \Map (A,X) \times_{\colim\Map(C,X)} \colim_{X\in \DD^\eq}\Map(B,X)  \\
 & = & F(A) \times_{F(C)} F(B) .
  \end{eqnarray*}
%
%
%
%
\end{proof}

\subsection{Flanked decomposition spaces}
\label{sec:flanked}

\begin{blanko}{Idea.}
  The idea is that `interval' should mean complete decomposition space (equipped)
  with both an initial and a terminal object.  An object $x\in X_0$ is {\em
  initial} if the projection map $X_{x/} \to X$ is a levelwise equivalence.
  Here the {\em coslice} $X_{x/}$ is defined as the pullback of the lower dec
  $\Dec_\bot X$ along $1 \stackrel{\name{x}} \to X_0$.  Terminal objects are
  defined similarly with slices, i.e.~pullbacks of the upper dec.  It is not
  difficult to see (compare Proposition~\ref{prop:i*flanked=Segal} below) that
  the existence of an initial or a terminal object forces $X$ to be a Segal
  space.
  
  While the intuition may be helpful, it is not obvious that the above
  definition of initial and terminal object should be meaningful for Segal
  spaces that are not Rezk complete.  In any case, it turns out to be practical
  to approach the notion of interval from a more abstract viewpoint, which will
  allow us to get hold of various adjunctions and factorisation systems that are
  useful to prove things about intervals.  We come to intervals in the next 
  subsection.  First we have to deal with flanked decomposition spaces.
\end{blanko}

\begin{blanko}{The category $\Xi$ of finite strict intervals.}
  We denote by $\Xi$ the category of finite strict intervals
  (cf.~\cite{Joyal:disks}, see also Section~\ref{sec:restriction} where we took a slightly 
  different viewpoint), that is, a skeleton of the category whose objects
  are nonempty finite linear orders with a bottom and a top element, required to
  be distinct, and whose arrows are the maps that preserve both the order and
  the bottom and top elements.  We depict the objects as columns of dots, with
  the bottom and top dot white, then the maps are the order-preserving maps that
  send white dots to white dots, but are allowed to send black dots to white
  dots.


  There is a forgetful functor $u:\Xi\to\Delta$ which forgets that there is anything
  special about the white dots, and just makes them black.  This functor has a left
  adjoint $i:\Delta\to\Xi$ which to a linear order (column of black dots) adjoins a
  bottom and a top element (white dots).

  Our indexing convention for $\Xi$ follows
  the free functor $i$: the object in $\Xi$ with $k$ black dots (and two outer white
  dots) is denoted $[k-1]$.  Hence the objects in $\Xi$ are $[-1]$, $[0]$, $[1]$, etc.
  Note that $[-1]$ is an initial object in $\Xi$.
  The two functors can therefore be described on objects as $u([k])=[k+2]$
  and $i([k])=[k]$, and the adjunction is given by the following 
  isomorphism:
  \begin{equation}\label{eq:XiDelta}
    \Xi([n],[k]) = \Delta([n],[k\!+\!2]) \qquad n\geq 0, k\geq -1 .
  \end{equation}
\end{blanko}

\begin{blanko}{New outer degeneracy maps.}
  Compared to $\Delta$ via the inclusion $i:\Delta\to \Xi$, the category $\Xi$ 
  has one extra 
  coface map in $\Xi$,
  namely $[-1]\to [0]$.  It also has, in each degree, 
  two extra {\em outer codegeneracy 
  maps}:  $s^{\bot-1}: [n]\to [n-1]$ sends the bottom black dot to the bottom white 
  dot, and $s^{\top+1} : [n]\to [n-1]$ sends the top black dot to the top white 
  dot.  (Both maps are otherwise bijective.)  
\end{blanko}

\begin{blanko}{Basic adjunction.}
  The adjunction   
$ i \isleftadjointto u $
induces an adjunction   $
  i\upperstar \isleftadjointto u\upperstar
  $
$$\xymatrix{
\Fun(\Xi\op,\Grpd) \ar@<+3pt>[r]^-{i\upperstar} & \ar@<+3pt>[l]^-{u\upperstar} 
\Fun(\Delta\op,\Grpd)
}$$
which will play a central role in all the constructions in this section.

The functor $i\upperstar$ takes {\em 
underlying simplicial space}:
concretely, applied to a $\Xi\op$-space $A$, the functor $i\upperstar $ deletes $A_{-1}$
and removes all the extra outer degeneracy maps.

On the other hand, the functor $u\upperstar $,
applied to a simplicial space $X$, 
deletes $X_0$ and removes all outer face maps (and then reindexes).

The comonad 
$$
i\upperstar u\upperstar : \Fun(\Delta\op,\Grpd) \to \Fun(\Delta\op,\Grpd)
$$
is precisely the double-dec construction $\Dec_\bot\Dec_\top$, and the
counit of the adjunction is precisely the comparison map
$$
\varepsilon_X=d_\top d_\bot:i\upperstar u\upperstar X=\Dec_\bot\Dec_\top X
\longrightarrow X .
$$

On the other hand, the monad
$$
u\upperstar i\upperstar : \Fun(\Xi\op,\Grpd) \to \Fun(\Xi\op,\Grpd)
$$
is also a kind of double-dec, removing first the extra outer degeneracy
maps, and then the outer face maps.  The unit 
$$\eta_A = s_{\bot-1} s_{\top+1}: A \to 
u\upperstar i\upperstar A$$ 
will also play an important role.

%
%
%
\end{blanko}

\begin{lemma}\label{u*f(cULF)=cart}
    If $f: Y \to X$ is a cULF map of simplicial spaces,
  then  $u\upperstar f :u\upperstar Y \to u\upperstar X $ is cartesian.
\end{lemma}
\begin{proof}
  The cULF condition on $f$ says
  it is cartesian on `everything' except outer face 
  maps, which are thrown away when taking $u\upperstar f$.
\end{proof}
\noindent
Note that the converse is not always true: if $u\upperstar f$ is cartesian then
$f$ is ULF, but there is no information about $s_0: Y_0 \to Y_1$, so we 
cannot conclude that $f$ is conservative.

Dually:
\begin{lemma}\label{i*cart}
  If a map of  $\Xi\op$-spaces  $g : B \to A$ is cartesian
 (or just cartesian on inner face and degeneracy maps),
  then $i\upperstar g : i\upperstar B \to i\upperstar A$ is cartesian.
\end{lemma}

\begin{blanko}{Representables.}
  The representables on $\Xi$ we denote by $\Xi[-1]$, $\Xi[0]$, etc.  By
  convention we will also denote the terminal presheaf on $\Xi$ by $\Xi[-2]$,
  although it is not representable since we have chosen not to include $[-2]$ (a
  single white dot) in our definition of $\Xi$.  Note that \eqref{eq:XiDelta}
  says that $i\upperstar$ preserves representables:
  \begin{equation}\label{eq:Deltak+2}
  i\upperstar (\Xi[k]) = \Delta[k\!+\!2], \qquad k\geq-1 .
  \end{equation}

%
%
\end{blanko}

\begin{blanko}{Wide-cartesian factorisation system.}
  Call an arrow in $\Fun(\Xi\op,\Grpd)$ {\em wide} if its $[-1]$-component is an
  equivalence.  Call an arrow {\em cartesian} if it is a cartesian natural
  transformation of $\Xi\op$-diagrams.  By general theory (\ref{eq1-cart}) we
  have a factorisation system on $\Fun(\Xi\op,\Grpd)$ where the left-hand class
  is formed by the wide maps and the right-hand class consists of the cartesian maps.
  In concrete terms, given any map $B \to
  A$, since $[-1]$ is terminal in $\Xi\op$, one can pull back the whole
  diagram $A$ along the map $B_{-1} \to A_{-1}$.  The resulting
  $\Xi\op$-diagram $A'$ is cartesian over $A$ by construction,
  and by the universal property of the pullback it receives a map from
  $B$ which is manifestly the identity in degree $-1$, hence
  wide.
  $$
  \xymatrix@R-1em@C+1em{B\ar[rr]\ar[rd]_{\mathrm{wide}}&&A\\&A'\ar[ru]_{\mathrm{cartesian}}}
  $$
\end{blanko}

\begin{blanko}{Flanked $\Xi\op$-spaces.}\label{flanked}
    A $\Xi\op$-space $A$ is called \emph{flanked} if the extra outer degeneracy
  maps form cartesian squares with opposite outer face maps. 
  Precisely, for $n\geq 0$
  $$
  \xymatrix{
  A_{n-1}   \ar[d]_{s_{\bot-1}} & \ar[l]_{d_\top} \dlpullback A_n \ar[d]^{s_{\bot-1}} \\
  A_n & \ar[l]^{d_\top} A_{n+1}
  }
  \qquad
  \xymatrix{
  A_{n-1}   \ar[d]_{s_{\top+1}} & \ar[l]_{d_\bot} \dlpullback A_n 
  \ar[d]^{s_{\top+1}} \\
  A_n & \ar[l]^{d_\bot} A_{n+1}
  }
  $$
  Here we have included the special extra face map $A_{-1} \leftarrow A_0$ both as
  a top face map and a bottom face map.
%
\end{blanko}

\begin{lemma}\label{s-1-newbonuspullbacks}
  (`Bonus pullbacks' for flanked spaces.)
  In a flanked $\Xi\op$-space $A$, all the following squares are pullbacks:
$$
\xymatrix @C=12pt @C=12pt {
A_{n-1} \ar[d]_{s_{\bot-1}} & \ar[l]_-{d_i} \dlpullback A_n \ar[d]^{s_{\bot-1}} \\
A_n  & \ar[l]^-{d_{i+1}} A_{n+1}
}
\
\xymatrix @C=12pt @C=12pt {
A_{n-1} \drpullback \ar[r]^-{s_j} \ar[d]_{s_{\bot-1}} & A_n \ar[d]^{s_{\bot-1}} \\
A_n \ar[r]_-{s_{j+1}} & A_{n+1}
}
\
\xymatrix @C=12pt @C=12pt {
A_{n-1} \ar[d]_{s_{\top+1}} & \ar[l]_-{d_i} \dlpullback A_n \ar[d]^{s_{\top+1}} \\
A_n  & \ar[l]^-{d_i} A_{n+1}
}
\
\xymatrix @C=12pt @C=12pt {
A_{n-1} \drpullback \ar[r]^-{s_j} \ar[d]_{s_{\top+1}} & A_n \ar[d]^{s_{\top+1}} \\
A_n \ar[r]_-{s_j} & A_{n+1}
}
$$
This is for all $n\geq 0$, and the running indices are $0 \leq i \leq n$ and $-1 \leq j \leq n$.
\end{lemma}

\begin{proof}
  Easy argument with pullbacks, as in \ref{bonus-pullbacks}.
\end{proof}

Note that in the upper rows, all face or degeneracy maps are present, whereas
in the lower rows, there is one map missing in each case.  In particular,
all the `new' outer degeneracy maps appear as pullbacks of `old' degeneracy
maps.

\begin{blanko}{Flanked decomposition spaces.}
  By definition, a {\em flanked decomposition space} is a $\Xi\op$-space
  $A:\Xi\op\to\Grpd$ that is flanked and whose underlying $\Delta\op$-space
  $i\upperstar A$ is a decomposition space.  Let $\FD$ denote the full
  subcategory of $\Fun(\Xi\op,\Grpd)$ spanned by the flanked
  decomposition spaces.
\end{blanko}

\begin{lemma}\label{lem:flanked}
  If $X$ is a decomposition space, then 
  $u\upperstar X$ is a flanked decomposition space.
\end{lemma}
\begin{proof}
  The underlying simplicial space is clearly a decomposition space (in fact a 
  Segal space), since all we have done is to throw away some outer face maps 
  and reindex.  The flanking condition comes from the `bonus pullbacks' of $X$,
  cf.~\ref{newbonuspullbacks}.
\end{proof}

It follows that the basic adjunction $i\upperstar \isleftadjointto u\upperstar $
restricts to an adjunction 
$$\xymatrix{
i\upperstar  : \FD \ar@<+3pt>[r] & \ar@<+3pt>[l]\Decomp : u\upperstar 
}$$
between flanked decomposition spaces (certain $\Xi\op$-diagrams)
and decomposition spaces.

\begin{lemma}\label{counit-cULF}
  The counit
  $$
  \epsilon_X : i\upperstar u\upperstar X \to X
  $$
  is cULF, when $X$ is a decomposition space.
\end{lemma}
\begin{proof}
  This was proved in Theorem~\ref{thm:decomp-dec-segal}.
\end{proof}

\begin{lemma}\label{unit-cart}
  The unit
  $$
  \eta_A: A \to u\upperstar i\upperstar A
  $$
  is cartesian, when $A$ is flanked.
\end{lemma}
\begin{proof}
  The map $\eta_A$ is given by $s_{\bot-1}$ followed by $s_{\top+1}$.
  The asserted pullbacks are precisely the `bonus pullbacks' of 
  Lemma~\ref{s-1-newbonuspullbacks}.
\end{proof}

From Lemma~\ref{unit-cart} and Lemma~\ref{counit-cULF} we get:

\begin{cor}\label{u*i*(cart)=cart}
  The monad $u\upperstar i\upperstar : \FD \to \FD$ preserves cartesian maps.
\end{cor}

\begin{lemma}\label{cart&cULF}
  $i\upperstar A \to X$ is cULF in $\Decomp$ if and only if the transpose
  $A \to u\upperstar  X$ is cartesian in $\FD$.
\end{lemma}
\begin{proof}
  This follows since the unit is cartesian (\ref{unit-cart}), the counit is 
  cULF (\ref{counit-cULF}), and $u\upperstar $ and $i\upperstar $ send those
  two classes to each other (\ref{u*f(cULF)=cart} and \ref{i*cart}).
\end{proof}

\begin{prop}\label{prop:i*flanked=Segal}
  If $A$ is a flanked decomposition space, then $i\upperstar A $ is a Segal 
  space.
\end{prop}
\begin{proof}
  Put $X = i\upperstar A$.  We have the maps
  $$
  \xymatrix{
  i\upperstar A \ar[r]^-{i\upperstar  \eta_{A}} 
  & i\upperstar u\upperstar i\upperstar  A =
  u\upperstar  i\upperstar X \ar[r]^-{\epsilon_{X}} 
  & X= i\upperstar 
  A}
  $$
  Now $X$ is a decomposition space by assumption, so $i\upperstar 
  u\upperstar X = \Dec_\bot \Dec_\top X$ is a Segal space and the counit is cULF
  (both statements by Theorem~\ref{thm:decomp-dec-segal}).
  On the other hand, since $A$ is flanked, the unit $\eta$ is 
  cartesian by Lemma~\ref{unit-cart}, hence $i\upperstar \eta$ is 
  cartesian by Lemma~\ref{i*cart}.  Since $i\upperstar A$ is thus cartesian over a Segal 
  space, it is itself a Segal space (\ref{cart/Segal=Segal}).
\end{proof}

\begin{lemma}\label{Flanked/cart} 
  If $B\to A$ is a cartesian map of $\Xi\op$-spaces and $A$ is a flanked
  decomposition space then so is $B$.
\end{lemma}

\begin{cor}
  The wide-cartesian factorisation system
  restricts to a factorisation system on $\FD$.
\end{cor}

\begin{lemma}\label{repr=flanked}
  The representable functors $\Xi[k]$ are flanked.
\end{lemma}

\begin{proof}
  Since the pullback squares required for a presheaf to be flanked
  are images of pushouts in $\Xi$, this follows since representable
  functors send colimits to limits.
\end{proof}

\subsection{Intervals and the factorisation-interval construction}
\label{sec:fact-int}

\begin{blanko}{Complete $\Xi\op$-spaces.}
  A $\Xi\op$-space is called {\em complete} if all degeneracy maps are
  monomorphisms.  We are mostly interested in this notion for flanked
  decomposition spaces.  In this case, if just $s_0: A_0 \to A_1$ is a
  monomorphism, then all the degeneracy maps are monomorphisms.  This follows
  because on the underlying decomposition space, we know from \ref{all-s-mono} that
  $s_0: A_0 \to A_1$ being a monomorphism implies that all the simplicial
  degeneracy maps are monomorphisms, and by flanking we then deduce that also the new
  outer degeneracy maps are monomorphisms.  Denote by $\cFD \subset\FD$ the full
  subcategory spanned by the complete flanked decomposition spaces.
  
  It is clear that if $X$ is a complete decomposition space, then $u\upperstar X$
  is a complete flanked decomposition space, and if $A$ is a complete flanked
  decomposition space then $i\upperstar A $ is a complete decomposition space.
  Hence the fundamental adjunction
  $\xymatrix{
  i\upperstar  : \FD \ar@<+3pt>[r] & \ar@<+3pt>[l]\Decomp : u\upperstar 
  }$
  between flanked decomposition
  spaces and decomposition spaces restricts to an adjunction
  $$\xymatrix{
  i\upperstar  : \cFD \ar@<+3pt>[r] & \ar@<+3pt>[l]\cDecomp : u\upperstar 
  }$$
 between complete flanked decomposition
  spaces and complete decomposition spaces.
  
  Note that anything cartesian over a 
  complete $\Xi\op$-space is again complete.
\end{blanko}

\begin{blanko}{Reduced $\Xi\op$-spaces.}
  A $\Xi\op$-space $A:\Xi\op\to\Grpd$ is called {\em reduced} when $A[-1]\simeq *$.
\end{blanko}

\begin{lemma}\label{reduced/wide}
  If $A\to B$ is a wide map of $\Xi\op$-spaces and $A$ is reduced then $B$ is 
  reduced.
\end{lemma}

\begin{blanko}{Algebraic intervals.}\label{aInt}
  An {\em algebraic interval} is  by definition a reduced complete flanked 
  decomposition space.
  We denote by $\aInt$ the full subcategory of $\Fun(\Xi\op,\Grpd)$
  spanned by the algebraic intervals.  In other words, a {\em morphism of
  algebraic intervals} is just a natural transformation of functors
  $\Xi\op\to\Grpd$.
  Note that the underlying decomposition space of an interval is always a Segal
  space.
\end{blanko}

\begin{lemma}
  All representables $\Xi[k]$ are algebraic intervals (for $k\geq -1$), 
  and also the terminal presheaf $\Xi[-2]$ is an algebraic interval.
\end{lemma}

\begin{proof}
  It is clear that all these presheaves are contractible in degree $-1$, and
  they are flanked by Lemma~\ref{repr=flanked}.  It is also clear from
  \eqref{eq:Deltak+2} that their underlying simplicial spaces are complete
  decomposition spaces (they are even Rezk complete Segal spaces).
\end{proof}

\begin{lemma}
  $\Xi[-1]$ is an initial object in $\aInt$.
\end{lemma}

\begin{lemma}\label{aInt:map=wide}
  Every morphism in $\aInt$ is wide.
\end{lemma}

\begin{cor}
  If a morphism of algebraic intervals is cartesian, then it is an equivalence.
\end{cor}

\begin{blanko}{The factorisation-interval construction.}
  We now come to the important notion of factorisation
  interval $I(a)$ of a given arrow $a$ in a decomposition space $X$. 
  In the case where $X$ is a $1$-category the construction is due to
  Lawvere~\cite{Lawvere:statecats}:
  the objects of $I(a)$ are the two-step
  factorisations of $a$, with initial object id-followed-by-$a$ and terminal object
  $a$-followed-by-id.  The $1$-cells are arrows between such
  factorisations, or equivalently $3$-step factorisations, and so on.

  For a general (complete) decomposition space $X$, the idea is this:
  taking the double-dec of $X$ gives a simplicial object starting at $X_2$, but
  equipped with an augmentation $X_1 \leftarrow X_2$.   Pulling back this 
  simplicial object along $\name a: 1 \to X_1$ yields a new simplicial object
  which is $I(a)$.  This idea can be formalised in terms of the basic adjunction
  as follows.
  
  By Yoneda, to give an arrow $a\in X_1$ is to give $\Delta[1] \to X$ in
  $\Fun(\Delta\op,\Grpd)$, or in the full subcategory $\cDecomp$.
  By adjunction, this is equivalent to giving $\Xi[-1] \to u\upperstar X$ in
  $\cFD$.  Now factor this map as a wide map followed by a cartesian map:
  $$\xymatrix{
  \Xi[-1] \ar[rr] \ar[rd]_{\mathrm{wide}} && u\upperstar X  .\\
  & A \ar[ru]_{\mathrm{cart}} &
  }$$
  The object appearing in the middle is an algebraic
  interval since it is wide under $\Xi[-1]$ (\ref{reduced/wide}).
  By definition, the factorisation interval of $a$ is
  $I(a) := i\upperstar A$, equipped with a cULF map to $X$, as seen in the diagram
  $$\xymatrix{
  \Delta[1] \ar[rr] \ar[rd] && i\upperstar u\upperstar X 
  \ar[r]^-{\epsilon}_-{\mathrm{cULF}} & X .\\
  & I(a) \ar[ru]_{\mathrm{cULF}} &
  }$$
  The map $\Delta[1] \to I(a)$ equips $I(a)$ with two endpoints, and a longest 
  arrow between them.
  The cULF map $I(a) \to X$ sends the longest arrow of $I(a)$ to $a$.
  
  More generally, by the same adjunction argument,
  given an $k$-simplex $\sigma: \Delta[k]\to X$ with long edge $a$,
  we get a $k$-subdivision of $I(a)$, i.e.~a wide map $\Delta[k]\to I(a)$.

  The construction shows, remarkably, that 
  as far as comultiplication is concerned, any decomposition space is locally
  a Segal space, in the sense that the comultiplication of an arrow $a$
  may as well be performed inside $I(a)$, which is a Segal space by 
  \ref{prop:i*flanked=Segal}.
  So while there may be no global way to compose arrows even if their source
  and targets match, the {\em decompositions} that exist do compose again.
\end{blanko}

We proceed to formalise the factorisation-interval construction. 

\begin{blanko}{Coreflections.}
%
  Inside the
  $\infty$-category of arrows $\Ar(\cFD)$, denote by
  $\Ar^{\mathrm w}(\cFD)$ the full subcategory spanned by the
  wide maps.  The wide-cartesian factorisation system amounts to a coreflection
  $$
  w: 
  \Ar(\cFD) \longrightarrow \Ar^{\mathrm 
  w}(\cFD) ;
  $$
  it sends an arrow $A\to B$ to its wide factor $A \to B'$, and in particular can
  be chosen to have $A$ as domain again (\ref{coreflection:w}).
  In particular, for each algebraic interval $A\in\aInt\subset\cFD$,
  the adjunction restricts to an adjunction between coslice categories, with coreflection
  $$
  w_A: \cFD _{A/} \longrightarrow \cFD^{\mathrm{w}}_{A/}  .
  $$
  The first $\infty$-category is that of flanked decomposition spaces under $A$, and the
  second $\infty$-category is that of flanked decomposition spaces with a wide map from
  $A$.  Now, if a flanked decomposition space receives a
  wide map from an algebraic interval then it is itself an algebraic interval
  (\ref{reduced/wide}), and all maps of algebraic intervals are wide
  (\ref{aInt:map=wide}).  So in the end the cosliced adjunction takes the form
  of the natural full inclusion functor
  $$
  v_A: \aInt_{A/} \to \cFD_{A/}
  $$
  and a right adjoint
  $$
  w_A: \cFD_{A/} \to \aInt_{A/} .
  $$
  
\begin{blanko}{Remark.}
  These observations amount to saying that
  the functor $v: \aInt \to \cFD$ is a {\em colocal left adjoint}.
  This notion is dual to the important concept
  of local right adjoint~\cite{Weber:TAC18}.
\end{blanko}
  
  We record the following obvious lemmas:
  \begin{lemma}
    The coreflection $w$ sends cartesian maps to equivalences.
  \end{lemma}
  \begin{lemma}
    The counit is cartesian.
  \end{lemma}

\begin{blanko}{Factorisation-interval as a comonad.}\label{I=LvwR}
  We also have the basic adjunction $i\upperstar \isleftadjointto u\upperstar $
  between complete decomposition spaces and complete flanked decomposition spaces.
  Applied to coslices over an algebraic interval $A$, and its underlying 
  decomposition space $\un A =i\upperstar A$ we get the 
  adjunction
  $$\xymatrix{
  L : \cFD_{A/} \ar@<+3pt>[r] & \ar@<+3pt>[l]\cDecomp_{\un A/} : R   .
  }$$
  Here $L$ is simply the functor $i\upperstar$,
  while the right adjoint $R$ is given by applying $u\upperstar$
  and precomposing with the unit $\eta_A$.
  Note that the unit of this adjunction $L \isleftadjointto R$ at an object
  $f:A \to X$ is given by
  $$
  \xymatrix {& A \ar[ld]_f \ar[rd]^{u\upperstar i\upperstar f \circ \eta_A} \\
  X \ar[rr]_{\eta_X} && u\upperstar i\upperstar  X
  }
  $$
  
  We now combine the two adjunctions:
  $$\xymatrix{
  \aInt_{A/} \ar@<+3pt>[r]^v & \ar@<+3pt>[l]^w \cFD_{A/} 
  \ar@<+3pt>[r]^-L
  & \ar@<+3pt>[l]^-R \cDecomp_{\un A/}  \, .
  }$$

  The factorisation-interval functor is the $\un A= \Delta[k]$ instantiation:
  $$
  I := L \circ v \circ w \circ R
  $$
  Indeed, this is precisely what we said in the construction, 
  just phrased more functorially.
  It follows that the factorisation-interval construction is a comonad
  on $\cDecomp_{\un A/}$.
\end{blanko}

\begin{lemma}
  The composed counit is cULF.
\end{lemma}
\begin{proof}
  Follows readily from \ref{counit-cULF}.
\end{proof}

\end{blanko}

\begin{prop}
  The composed unit $\eta : \Id \Rightarrow w \circ R \circ L \circ v$ is an
  equivalence.
\end{prop}

\begin{proof}
  The result of applying the four functors to an algebraic interval map $f:A\to B$
  is the wide factor in
    $$\xymatrix{
  A \ar[rr] \ar@{->>}[rd]_{\mathrm{wide}} && u\upperstar i\upperstar B \\
  & D \ar@{ >->}[ru]_{\mathrm{cart}} &
  }$$
  The unit on $f$ sits in this diagram
  $$
  \xymatrix{
  & \ar@{->>}[ld]_f  A  \ar@{->>}[rd] & \\
  B \ar@{-->}[rr]_{\eta_f} \ar@{ >->}[rd]_{\eta_B}&& D \ar@{ >->}[ld]\\
  & u\upperstar i\upperstar B ,}
  $$
  where $\eta_B$ is cartesian by \ref{unit-cart}.
  It follows now from 
  orthogonality of the
  wide-cartesian factorisation system
  that $\eta_f$ is an equivalence.
\end{proof}

\begin{cor}\label{i*v:ff}
  The functor $i\upperstar \circ v : \aInt \to \cDecomp_{\Delta[1]/}$
  is fully faithful.
\end{cor}

%

\begin{prop}\label{prop:I(cULF)=eq}
  $I$ sends cULF maps to equivalences.  In detail, for a cULF functor $F: Y\to 
  X$ and any arrow $a \in Y_1$
  we have a natural equivalence of intervals (and hence of underlying Segal spaces)
  $$
  I(a) \isopil I(Fa).
  $$
\end{prop}
\noindent
\begin{proof}
  $R$ sends cULF maps to cartesian maps, and $w$ send cartesian maps to 
  equivalences.
\end{proof}

\begin{cor}
  If $X$ is an interval, with longest arrow $a\in X_1$, then $X\simeq I(a)$.
\end{cor}

%

\begin{prop}\label{aInt-faithful}
  The composed functor 
  $$
  \aInt \to \cDecomp_{\Delta[1]/} \to \cDecomp
  $$
  is faithful (i.e.~induces a monomorphism on mapping spaces).
\end{prop}
\begin{proof}
  Given two algebraic intervals $A$ and $B$, 
  denote by $f:\Delta[1] \to i\upperstar A$ and $g: \Delta[1] \to i\upperstar B$
  the images in $\cDecomp_{\Delta[1]/}$.
  The claim is that the map
  $$
  \Map_{\aInt}(A,B) \longrightarrow \Map_{\cDecomp_{\Delta[1]/}}(f,g) \longrightarrow
  \Map_{\cDecomp}(i\upperstar A, i\upperstar B)
  $$
  is a monomorphism.
  We already know that the first part is an equivalence (by Corollary~\ref{i*v:ff}).
  The second map will be a monomorphism because of the special nature of $f$ and $g$.
  We have a pullback diagram (mapping space fibre sequence for coslices):
  $$\xymatrix{
     \Map_{\cDecomp_{\Delta[1]/}}(f,g)\drpullback \ar[r]\ar[d] & 
     \Map_{\cDecomp}(i\upperstar A, i\upperstar B)\ar[d]^{\mathrm{precomp. }f} \\
     1 \ar[r]_-{\name g} & \Map_{\cDecomp}(\Delta[1], i\upperstar B).
  }$$
  Since $g:\Delta[1]\to i\upperstar B$ is the image of the canonical map 
  $\Xi[-1]\to B$,
  the map 
  $$
  \xymatrix{
  1 \ar[r]^-{\name g}& \Map_{\cDecomp}(\Delta[1], i\upperstar B)}
  $$
  can be identified with
  $$
  \xymatrix @C=48pt{
  B_{-1} \ar[r]^{s_{\bot-1}s_{\top+1}} & B_1 ,
  }
  $$
  which is a monomorphism since $B$ is complete.  It follows that the top map
  in the above pullback square is a monomorphism, as asserted.
  (Note the importance of completeness.)
\end{proof}

\subsection{The decomposition space of intervals}

\begin{blanko}{Interval category as a full subcategory in $\cDecomp$.}\label{towards-Int}
  We now invoke the general results about Kleisli categories (\ref{fact-Kl-E}).
  Let 
  $$
  \Int := \wtil \aInt
  $$
  denote the restricted Kleisli category for the adjunction
  $$\xymatrix{
  i\upperstar  : \cFD \ar@<+3pt>[r] & \ar@<+3pt>[l]\cDecomp : u\upperstar 
  },$$ as in~\ref{fact-Kl-E}.  Hence $\Int \subset \cDecomp$ is the full 
  subcategory of
  decomposition spaces underlying algebraic intervals.
  Say a map in $\Int $ is {\em wide} if it is the $i\upperstar $ image of a 
  map in $\aInt$ (i.e.~a wide map in $\cFD$).
\end{blanko}
  
\begin{prop}\label{fact-Int}
  The wide maps as left-hand class and the cULF maps as right-hand class form a 
  factorisation system on $\Int$.
\end{prop}
\begin{proof}
  The wide-cartesian factorisation system on $\cFD$ is compatible with the
  adjunction $i\upperstar \isleftadjointto u\upperstar$ and the subcategory 
  $\Int$ precisely as required to apply the general 
  Lemma~\ref{fact-Kl-E}. Namely, we have:
  
  --- $u\upperstar i\upperstar$
  preserves cartesian maps by Corollary~\ref{u*i*(cart)=cart}.
  
  ---  $u\upperstar \epsilon$ is cartesian by \ref{u*f(cULF)=cart}, since
  $\epsilon$ is cULF by \ref{counit-cULF}.
  
  
  --- If $A \to B$ is wide, $A$ an algebraic interval, then so is $B$,
  by \ref{reduced/wide}.
  
  The general Lemma~\ref{fact-Kl-E} now tells us that there is a
  factorisation system on $\Int$ where the left-hand class are the maps of
  the form $i\upperstar $ of a wide map.  The right-hand class of $\Int$,
  described by 
  Lemma~\ref{fact-Kl-E} as those maps $f$ for which $u\upperstar f$ is 
  cartesian,
  is seen by Lemma~\ref{cart&cULF} to be precisely the cULF maps.
\end{proof}

We can also restrict the Kleisli category and the factorisation system
to the category $\Xi+$ consisting of
the representables together with the terminal object $\Xi[-2]$.

\begin{lemma}\label{IntDelta}
  The restriction of the Kleisli category to $\Xi+$ gives $\Delta$, and the
   wide-cULF factorisation systems on $\Int$ restricts to the
  generic-free factorisation system on $\Delta$.
\end{lemma}
$$\xymatrix{
\Delta \ar[r]^-{\mathrm{f.f.}} & \Int \ar[r]^-{\mathrm{f.f.}} & \cDecomp 
\ar@<+3pt>[d]^{u\upperstar}\\
\Xi+ \ar[u] \ar[r] &\aInt \ar[u] \ar[r] & \cFD 
\ar@<+3pt>[u]^{i\upperstar }
}$$

\begin{proof}
  By construction the objects are $[-2], [-1], [0], [1],\ldots$.
  and the mapping
  spaces are
  \begin{eqnarray*}
    \Map_{\Int}(\Xi[k],\Xi[n]) & = & \Map_\Decomp( i\upperstar \Xi[k], i\upperstar \Xi[n])\\
     & = &\Map_{\widehat\Delta}(\Delta[k+2],\Delta[n+2]) \\
     & = & \Delta([k+2],[n+2]).
  \end{eqnarray*}
  It is clear by the explicit description of $i\upperstar $ that it takes
  the maps in $\Xi+$ to the generic maps in $\Delta$.  On the other hand,
  it is clear that the cULF maps in $\Delta$ are the free maps.
\end{proof}

%
%
%

\begin{blanko}{Arrow category and restriction to $\Delta$.}
  Let $\Ar^{\mathrm w}(\Int) \subset \Ar(\Int)$ denote the full subcategory of the arrow 
  category spanned by the wide maps.  Recall (from \ref{ArED}) that 
  $\Ar^{\mathrm w}(\Int)$ is a
  cartesian fibration over $\Int$ via the domain projection.
  We now restrict this cartesian fibration to $\Delta \subset \Int$ as in 
  \ref{setup-end}:
  $$\xymatrix{
     \Ar^{\mathrm{w}}(\Int)_{|\Delta} \drpullback \ar[r]^-{\mathrm{f.f.}}\ar[d]_{\mathrm{dom}} & 
     \Ar^{\mathrm{w}}(\Int) \ar[d]^{\mathrm{dom}} \\
     \Delta \ar[r]_-{\mathrm{f.f.}} & \Int
  }$$
  We put
  $$
  \mathscr U := \Ar^{\mathrm{w}}(\Int)_{|\Delta} .
  $$
  $\mathscr U \to \Delta$ is the {\em Cartesian fibration of subdivided 
  intervals}: the objects of $\mathscr U$ are the wide 
  interval maps $\Delta[k] \onto A$, which we think of as 
  subdivided intervals.
  The arrows are
  commutative squares
  $$\xymatrix{
  \Delta[k] \ar[r] \ar@{->>}[d] & \Delta[n] \ar@{->>}[d] \\
  A \ar[r] & B
  }$$
  where the downwards maps are wide, and the
  rightwards maps are in $\Delta$ and in $\cDecomp$, respectively.
  (These cannot be realised in the world of $\Xi\op$-spaces, and the necessity
  of having them was the whole motivation for constructing $\Int$.)
  By \ref{ArED}, the cartesian maps are squares
  $$\xymatrix{
  \Delta[k] \ar[r] \ar@{->>}[d] & \Delta[n] \ar@{->>}[d] \\
  A \ar@{ >->}[r] & B   .
  }$$
  Hence, cartesian lift are performed by precomposing and then coreflecting
  (i.e.~wide-cULF factorising and keeping only the wide part).
  For a fixed domain $\Delta[k]$, we have (in virtue of 
  Proposition~\ref{aInt-faithful})
  $$
  \Int^{\mathrm{w}}_{\Delta[k]/} \simeq \aInt_{\Xi[k-2]/} .
  $$
\end{blanko}

\begin{blanko}{The (large) decomposition space of intervals.}\label{U}
  The cartesian fibration $\mathscr U =\Ar^{\mathrm{w}}(\Int)_{|\Delta} \to \Delta$ determines a right fibration,
$U := \mathscr U^\mathrm{cart} = \Ar^{\mathrm{w}}(\Int)_{\mid
\Delta}^{\mathrm{cart}} \to \Delta$, and hence by straightening 
(\cite{Lurie:HTT}, Ch.2)
a simplicial $\infty$-groupoid 
$$
U:\Delta\op\to\widehat{\Grpd},
$$ 
where $\widehat{\Grpd}$ is the very large $\infty$-category of not necessarily small $\infty$-groupoids.
We shall see that it is a complete decomposition space.

We shall not actually need the straightening, as it is more convenient to work
directly with the right fibration $U \to \Delta$.
Its fibre over $[k]\in\Delta$ is the $\infty$-groupoid $U_k$ of $k$-subdivided 
intervals.  That is, an interval $\un A$ equipped with a wide map
$\Delta[k] \onto A$.
Note that $U_1$ is equivalent to the $\infty$-groupoid $\Int^\eq$.
Similarly, $U_2$ is equivalent to the $\infty$-groupoid of subdivided intervals,
more precisely intervals with a wide map from $\Delta[2]$.
Somewhat more exotic is
$U_0$, the $\infty$-groupoid of intervals with a wide 
map from $\Delta[0]$.  This means that the endpoints must coincide.
This does not imply that the interval is trivial.
For example, any category with a zero object provides an example of an object
in $U_0$.

\begin{blanko}{A remark on size.}
  The fibres of the right fibration $U \to \Delta$ are large
  $\infty$-groupoids.  Indeed, they are all variations of $U_1$, the
  $\infty$-groupoid of intervals, which is of the same size as the
  $\infty$-category of simplicial spaces, which is of the same size as
  \Grpd.  Accordingly, the corresponding presheaf takes values in
  large $\infty$-groupoids, and $U$ is therefore a large decomposition space.
  These technicalities do not affect the following results, but will
  play a role from \ref{conjecture} and onwards.
\end{blanko}

Among the generic maps in $U$, in each degree the unique map $g:U_r \to U_1$
consists in forgetting the subdivision.  The space $U$ also has the codomain
projection $U \to \Int$.  In particular we can describe the $g$-fibre over a
given interval $A$:
\begin{lemma}\label{lem:U_r_A}
  We have a pullback square
  $$\xymatrix{
      (A_r)_a \drpullback \ar[r]\ar[d] &  U_r \ar[d]^g \\
     {*} \ar[r]_{\name A} & U_1
  }$$ 
  where $a\in A_1$ denotes the longest edge.
\end{lemma}
\begin{proof}
  Indeed, the fibre over a coslice is the mapping space, so the pullback is
at first
$$
\Map_{\mathrm{wide}}(\Delta[r], A)
$$
But that's the full subgroupoid inside $\Map(\Delta[r], A) \simeq A_r$
consisting of the wide maps, but that means those whose 
restriction to the long edge is $a$.
\end{proof}

\end{blanko}

\begin{thm}\label{UcompleteDecomp}
  The simplicial space $U: \Delta\op \to \widehat{\Grpd}$ is a (large) complete
  decomposition space.
\end{thm}

\begin{proof}
  We first show it is a decomposition space.
  We need to show that for a generic-free  pullback square in $\Delta\op$,
  the image under $U$ is a pullback:
  $$
  \xymatrix{U_k \drpullback\ar[r]^{f'} \ar[d]_{g'} & U_m \ar[d]^g \\
  U_n \ar[r]_f & U_s}$$
  This square is the outer rectangle in
   $$
  \xymatrix{\Int_{\Delta[k]/}^{\mathrm w} \ar[r]^j \ar[d]_{g'} 
& \Int_{\Delta[k]/} \ar[r]^{f'} \ar[d]_{g'} 
& 
\Int_{\Delta[m]/}\ar[d]^g \ar[r]^w  & \Int_{\Delta[m]/}^{\mathrm w}\ar[d]^g \\
  \Int_{\Delta[n]/}^{\mathrm w} \ar[r]_j & 
  \Int_{\Delta[n]/}{} \ar[r]_f & 
\Int_{\Delta[s]/}\ar[r]_w & \Int_{\Delta[s]/}^{\mathrm w}
  }$$
  (Here we have omitted taking maximal $\infty$-groupoids, but it doesn't affect the 
  argument.)
  The first two squares consist in precomposing with the free 
  maps $f$, $f'$.  The result will no longer be a
  wide map, so in the middle columns we allow arbitrary maps.
  But the final step just applies the coreflection to take the wide part.  Indeed this is how cartesian lifting goes in
  $\Ar^{\mathrm{w}}(\Int)$.
  The first square is a pullback since $j$ is fully
  faithful.  The last square is a pullback since it is a special case of
  Lemma~\ref{lem:Dx'Dx}.  The main point is the second square
  which is a pullback by Lemma~\ref{preservespullback} --- this is where we use that
  the generic-free square in $\Delta\op$ is a pullback.
  
  To establish that $U$ is complete, we need to check
  that the map $U_0 \to U_1$ is a monomorphism.  This map is just the forgetful functor
  $$
  (\Int^{\mathrm w}_{*/})^\eq \to \Int^\eq
 .
  $$
  The claim is that its fibres are empty or contractible.
  The fibre over an interval $\un A= i\upperstar  A$ is
  $$
  \Map_{\mathrm{wide}}(*,\un A)  = \Map_{\aInt}( \Xi[-2], A) = \Map_{\widehat 
  \Xi}(\Xi[-2],A).
  $$
  Note that in spite of the notation, $\Xi[-2]$ is not a representable:
  it is the terminal object, and it is hence the colimit of all the 
  representables.  It follows that $\Map_{\widehat \Xi}(\Xi[-2],A) = \lim A$.
  This is the limit of a cosimplicial diagram
  $$
  \lim A \stackrel e \longrightarrow * \rightrightarrows A_0 \cdots
  $$
  In general the limiting map of a cosimplicial diagram 
  does not have to be a monomorphism, but in this case
  it is, as all the coface maps (these are the degeneracy maps of $A$)
  are monomorphisms by completeness of $A$, and since $A_{-1}$ is contractible.
  Since finally $e$ is a monomorphism into the contractible space
  $A_{-1}$, the limit must be empty or contractible.  Hence $U_0 \to U_1$ is a 
  monomorphism,
  and therefore $U$ is complete.
\end{proof}

\subsection{Universal property of $U$}

The refinements discussed in \ref{fact-decomp} now pay off to give us the
following main result.  Let $\Int\comma\cDecomp$ denote the comma category (as
in \ref{fact-theorem}).  It is the full subcategory in $\Ar(\cDecomp)$ spanned
by the maps whose domain is in $\Int$.  Let $\Ar^{\mathrm w}(\Int)$ denote the full
subcategory of $\Ar(\Int)$ spanned by the wide maps.  Recall (from \ref{ArED})
that both $\Int\comma\cDecomp$ and $\Ar^{\mathrm w}(\Int)$ are cartesian fibrations over
$\Int$ via the domain projections, and that the inclusion $\Ar^{\mathrm w}(\Int) \to
\Int\comma\cDecomp$ commutes with the projections (but does not preserve
cartesian arrows).

\begin{theorem}\label{Thm:I}
  The inclusion functor $\Ar^{\mathrm w}(\Int) \into \Int\comma\cDecomp$
  has a right adjoint
  $$I:\Int\comma\cDecomp \to \Ar^{\mathrm w}(\Int)  ,
  $$
  which takes cartesian arrows to cartesian arrows.
\end{theorem}
\begin{proof}
  We have already checked, in the proof of \ref{fact-Int},
  that the conditions of the general Theorem~\ref{fact-theorem}
  are satisfied by the adjunction $i\upperstar \isleftadjointto u\upperstar $
  and
  the wide-cartesian factorisation system on $\cFD$.  It remains to restrict
  this adjunction to the full subcategory 
  $\aInt \subset \cFD$.
%
%
%
%
\end{proof}
Note that over an interval $\un A$, the adjunction restricts to the adjunction
of \ref{I=LvwR} as follows:
$$\xymatrix{
   \Int^{\mathrm{w}}_{\un A/} \ar@<+2.5pt>[r]\ar[d]_\simeq
   & \ar@<+2.5pt>[l]^-I \cDecomp_{\un A/} \ar@<+2.5pt>[d]^R \\
   \aInt_{A/} \ar@<+2.5pt>[r]^v &\ar@<+2.5pt>[l]^w \cFD_{A/} \ar@<+2.5pt>[u]^L
}$$

\bigskip

We now restrict these cartesian fibrations further to $\Delta \subset \Int$.
We call the coreflection $I$, as it is the factorisation-interval construction:
$$\xymatrix{
\mathscr U=\Ar^{\mathrm{w}}(\Int)_{|\Delta}\drto_{\mathrm{dom}}
\ar@<+3pt>[rr] && \ar@<+3pt>[ll]^I
\Delta\comma\cDecomp\dlto^{\mathrm{dom}}\\
&\Delta
}
$$

The coreflection 
$$
I : \Delta\comma\cDecomp  \to \mathscr U
$$
is a morphism of cartesian fibrations over $\Delta$ (i.e.~preserves
cartesian arrows).  Hence it induces a morphism of right fibrations
$$
I : (\Delta\comma\cDecomp)^\mathrm{cart}  \to  U .
$$
\begin{theorem}\label{thm:IU=cULF}
  The morphism of right fibrations
  $$
  I : (\Delta\comma\cDecomp)^\mathrm{cart}  \to  U
  $$ 
  is cULF.
\end{theorem}

\begin{proof}
  We need to establish that for the unique generic map $g: \Delta[1] \to \Delta[k]$,
  the following square is a pullback:
  $$\xymatrix{
     {\cDecomp_{\Delta[k]/}}^\eq \ar[r]^{\mathrm{pre. } g}\ar[d]_{I_k} & 
     {\cDecomp_{\Delta[1]/}}^\eq  \ar[d]^{I_1} \\
     {\Int^{\mathrm{w}}_{\Delta[k]/}}^\eq \ar[r]_{\mathrm{pre. }g} & {\Int^{\mathrm{w}}_{\Delta[1]/}}^\eq .
  }$$
  Here the functors $I_1$ and $I_k$ are the coreflections of Theorem~\ref{Thm:I}.
  We compute the fibres of the horizontal maps over a point $a: \Delta[1]\to 
  X$.  For the first row, the fibre is
  $$
  \Map_{\cDecomp_{\Delta[1]/}}(g,a) .
  $$
  For the second row, the fibre is
  $$
  \Map_{\Int^{\mathrm{w}}_{\Delta[1]/}}(g, I_1(a)). 
  $$
  But these two spaces are equivalent by the adjunction of 
  Theorem~\ref{Thm:I}.
\end{proof}

Inside $\Delta\comma\cDecomp$, we have the fibre over $X$, for the codomain
fibration (which is a cocartesian fibration).  This fibre is just $\Delta_{/X}$
the Grothendieck construction of the presheaf $X$.
%
%
%
%
This fibre clearly includes into the
cartesian part of $\Delta\comma\cDecomp$.  

\begin{lemma}
  The associated morphism of right fibrations
  $$
  \Delta_{/X} \to (\Delta\comma\cDecomp)^\mathrm{cart}
  $$
  is cULF.
\end{lemma}

\begin{proof}
  For $g: \Delta[k] \to \Delta[1]$ the unique generic map in degree $k$,
  consider the diagram
  $$\xymatrix{
   \Map(\Delta[k],X) \drpullback \ar[r]^-{\mathrm{pre. } g}\ar[d] & 
   \Map(\Delta[1],X)\drpullback \ar[d] 
   \ar[r] & 1 \ar[d]^{\name X}\\
   {\cDecomp_{\Delta[k]/}}^\eq \ar[r]_-{\mathrm{pre. } g}& 
   {\cDecomp_{\Delta[1]/}}^\eq \ar[r]_-{\mathrm{codom}} & \cDecomp^\eq .}
   $$
  The right-hand square and the outer rectangle are obviously pullbacks, 
  as the fibres of coslices are the mapping spaces.  Hence the left-hand
  square is a pullback, which is precisely to say that the vertical map is
  cULF.  
\end{proof}

So altogether we have cULF map
$$
\Delta_{/X} \to (\Delta\comma\cDecomp)^\mathrm{cart} \to U ,
$$
or, by straightening, a cULF map of complete decomposition spaces
$$
I: X \to U ,
$$
the {\em classifying map}.
It takes a $k$-simplex in $X$ to a $k$-subdivided interval,
as already detailed in \ref{sec:fact-int}.

\medskip

The following conjecture expresses the idea that $U$ should be terminal in the 
category of complete decomposition spaces and cULF maps,
but since $U$ is large this cannot literally be true,
and we have to formulate it slightly differently.
\begin{blanko}{Conjecture.}\label{conjecture}
  {\em $U$ is the universal complete decomposition space for cULF maps.
  That is, for each (legitimate) complete decomposition space $X$, the space 
  $\Map_{\cDecomp^{\culf}}(X,U)$ is contractible.
}
\end{blanko}
At the moment we are only able to prove the following weaker statement.

\begin{theorem}\label{thm:connected}
  For each (legitimate) complete decomposition space $X$, the space 
  $\Map_{\cDecomp^{\culf}}
(X,U)$ is connected.
\end{theorem}
\begin{proof}
  Suppose $J: X\to U$ and $J': X \to U$ are two cULF functors.  
  As in the proof of Theorem~\ref{thm:IU=cULF}, cULFness is equivalent to saying
  that we have a pullback
    $$\xymatrix{
     \Map_{\cDecomp}(\Delta[k],X) \drpullback \ar[r]^{\mathrm{pre. } g}\ar[d]_{J_k} & 
     \Map_{\cDecomp}(\Delta[1],X) \ar[d]^{J_1} \\
     {\Int^{\mathrm{w}}_{\Delta[k]/}}^\eq \ar[r]_{\mathrm{pre. }g} & 
     {\Int^{\mathrm{w}}_{\Delta[1]/}}^\eq  .
  }$$
  We therefore have equivalences between the fibres over
  a point $a: \Delta[1]\to 
  X$:
  $$
  \Map_{\cDecomp_{\Delta[1]/}}(g,a) \simeq
   \Map_{\Int^{\mathrm{w}}_{\Delta[1]/}}(g, J_1(a)).
   $$
  But the second space is equivalent to $\Map_{\Int^{\mathrm{w}}}(\Delta[k],J_1(a))$.
  Since these equivalences hold also for $J'$, we get
  $$
  \Map_{\Int^{\mathrm{w}}}(\Delta[k],J_1(a)) \simeq 
  \Map_{\Int^{\mathrm{w}}}(\Delta[k],J'_1(a)),
  $$
  naturally in $k$.  This is to say that $J_1(a)$ and $J_1'(a)$ are
  levelwise equivalent simplicial spaces.  But a cULF map is determined by its
  $1$-component, so $J$ and $J'$ are equivalent in the functor category.  In
  particular, every object in $\Map^{\culf}(X,U)$ is equivalent to the canonical $I$
  constructed in the previous theorems. 
\end{proof}

\begin{blanko}{Size issues and cardinal bounds.}\label{kappa}
  We have observed that the decomposition space of intervals is large, in the
  sense that it takes values in the very large $\infty$-category of large
  $\infty$-groupoids.  This size issue prevents $U$ from being a terminal object in
  the category of decomposition spaces and cULF maps.

  A more refined analysis of the situation is possible by standard techniques,
  by imposing cardinal bounds, as we briefly explain.
  For $\kappa$ a regular uncountable cardinal, say that a simplicial space
  $X:\Delta\op\to\Grpd$ is {\em $\kappa$-bounded}, when for each $n\in\Delta$
  the space $X_n$ is $\kappa$-compact.  In other words, $X$ takes values in
  the (essentially small) $\infty$-category
  $\Grpd^\kappa$ of $\kappa$-compact $\infty$-groupoids.  Hence the $\infty$-category
  of $\kappa$-bounded simplicial spaces is essentially small.  The attribute
  $\kappa$-bounded now also applies to decomposition spaces and intervals.
  Hence the $\infty$-categories of $\kappa$-bounded decomposition spaces and
  $\kappa$-bounded intervals are essentially small.  Carrying the $\kappa$-bound
  through all the constructions, we see that there is an essentially small
  $\infty$-category $U_1$ of $\kappa$-bounded intervals, and a legitimate
  presheaf $U^\kappa: \Delta\op\to\Grpd$ of $\kappa$-bounded intervals.

  It is clear that if $X$ is a $\kappa$-bounded decomposition space, then
  all its intervals are $\kappa$-bounded too.  It follows that if
  Conjecture~\ref{conjecture} is true then it is also true that $U^\kappa$,
  the (legitimate) decomposition space of all $\kappa$-bounded intervals,
  is universal for $\kappa$-bounded decomposition spaces, in the sense that
  for any $\kappa$-bounded decomposition space $X$, the space
  $\Map_{\cDecomp^{\culf}}(X,U^\kappa)$ is contractible.
\end{blanko}

%
%

\subsection{\M intervals and the universal \M function}
\label{sec:MI}

We finally impose the \M condition.

\begin{blanko}{Nondegeneracy.}
  Recall from \ref{effective} that for a complete decomposition space $X$ we have
  $$
  \nondeg X_r \subset X_r
  $$
  the full subgroupoid of $r$-simplices none of whose principal edges are 
  degenerate.   These can also be described as the full subgroupoid
  $$
  \nondeg X_r \simeq \Map_{\operatorname{nondegen}}(\Delta[r],X) \subset
  \Map(\Delta[r],X)\simeq X_r
  $$
  consisting of the nondegenerate maps, i.e.~maps for which the restriction to
  any principal edge $\Delta[1]\to \Delta[r]$ is nondegenerate.

  Now assume that $A$ is an interval.  Inside
  $$
  \Map_{\operatorname{nondegen}}(\Delta[r],A) \simeq \nondeg A_r
  $$
  we can further require the maps to be wide.  It is clear that this 
  corresponds to considering only nondegenerate simplices whose
  longest edge is the longest
    edge $a \in A_1$:
  \begin{lemma}\label{wide+nondegA}
    $$
    \Map_{\operatorname{wide+nondegen}}(\Delta[r],A) \simeq (\nondeg A_r)_a .
    $$
  \end{lemma}
%
\end{blanko}

\begin{blanko}{Nondegeneracy in $U$.}
  In the case of $U: \Delta\op\to\widehat{\Grpd}$, it is easy to describe the
  spaces $\nondeg U_r$.  They consist of wide maps $\Delta[r] \to A$ for which
  none of the restrictions to principal edges $\Delta[1] \to A$ are degenerate.
  In particular we can describe the fibre over a given interval $A$
  (in analogy with \ref{lem:U_r_A}):
\begin{lemma}\label{lem:nondegU_r_A}
  We have a pullback square
  $$\xymatrix{
     (\nondeg A_r)_a \ar[r]\ar[d] & \nondeg U_r \ar[d] \\
     {*} \ar[r]_{\name A} & U_1   .
  }$$ 
\end{lemma}

\end{blanko}

\begin{blanko}{\M intervals.}
  Recall (from~\ref{M}) that a complete decomposition space $X$ is called \M when
  it is  locally finite and of locally finite length (i.e.~tight).
  A {\em \M interval} is an interval which is \M as a decomposition space.
\end{blanko}

\begin{prop}\label{prop:Mint=Rezk}
  Any \M interval is a Rezk complete Segal space.
\end{prop}
\begin{proof}
  Just by being an interval it is a Segal space (by \ref{prop:i*flanked=Segal}).
  Now the filtration condition implies the Rezk condition by
  Proposition~\ref{prop:FILTSegal=Rezk}.
\end{proof}

\begin{lemma}\label{XFILT=>IFILT}
  If $X$ is a \FILT decomposition space, then for each $a\in X_1$, the
  interval $I(a)$ is a \FILT decomposition space.
\end{lemma}
\begin{proof}
  We have a cULF map $I(a)\to X$.  Hence by Proposition~\ref{prop:cULF/FILT=FILT},
  $I(a)$ is again \FILT.
\end{proof}

\begin{lemma}
  If $X$ is a locally finite decomposition space then for each $a\in X_1$,
  the interval $I(a)$ is a locally finite decomposition space.
\end{lemma}

\begin{proof}
  The morphism of decomposition spaces $I(a) \to X$ was constructed by pullback
  of the map $1 \stackrel{\name a}\to X_1$ which is finite (by
  Lemma~\ref{lem:1->B}) since $X_1$ is locally finite.  Hence $I(a) \to X$
  is a finite morphism of decomposition spaces.  So $I(a)$ is locally finite 
  since $X$ is.
\end{proof}
From these two lemmas we get
\begin{cor}\label{XM=>IM}
  If $X$ is a \M decomposition space, then for each $a\in X_1$, the
  interval $I(a)$ is a \M interval.
\end{cor}

\begin{prop}\label{prop:Mint=finite}
  If $A$ is a \M interval then for every $r$, the space $A_r$ is finite.
\end{prop}

\begin{proof}
  The squares
  $$\xymatrix{
     A_0 \drpullback \ar[r]^-{s_{\top+1}}\ar[d] & A_1 \drpullback 
     \ar[r]^-{s_{\bot-1}}\ar[d]_{d_0} & A_2 \ar[d]^{d_1} \\
     1 \ar@/_1pc/[rr]_{\name a} \ar[r]^-{s_{\top+1}} & A_0\ar[r]^-{s_{\bot-1}} & A_1
  }$$
  are pullbacks by the flanking condition \ref{flanked} (the second is a bonus
  pullback, cf.~\ref{s-1-newbonuspullbacks}).  The bottom composite arrow picks
  out the long edge $a\in A_1$.  (That the outer square is a pullback can be
  interpreted as saying that the $2$-step factorisations of $a$ are parametrised
  by their midpoint, which can be any point in $A_0$.)  Since the generic maps 
  of $A$ are finite (simply by the assumption that $A$ is locally finite) 
  in particular the map $d_1:A_2 \to A_1$ is finite, hence
  the fibre $A_0$ is finite.  The same argument works for arbitrary $r$, by
  replacing the top row by $A_r \to A_{r+1} \to A_{r+2}$, and letting the columns
  be $d_0^r$, $d_0^r$ and $d_1^r$.
\end{proof}
\noindent
(This can be seen as a homotopy version of \cite{LawvereMenniMR2720184} Lemma 
2.3.)

\begin{cor}
  For a \M interval, the total space of all nondegenerate simplices $\sum_r \nondeg 
  A_r$ is finite.
\end{cor}
\begin{proof}
  This follows from Proposition~\ref{prop:Mint=finite} and
  Lemma~\ref{lem:oldcharacterisationofM}.
\end{proof}

\begin{cor}\label{cor:<kappa}
A \M interval is $\kappa$-bounded for any uncountable
cardinal $\kappa$.
\end{cor}

%
%
%
%

\begin{blanko}{The decomposition space of \M intervals.}  
  There is a decomposition space $\mathbf {MI}\subset U$ consisting of all \M
  intervals.  In each degree, $\MI_k$ is the full subgroupoid of $U_k$
  consisting of the wide maps $\Delta[k] \to A$ for which $A$ is \M. While $U$
  is large, $\MI$ is a legitimate decomposition space by \ref{cor:<kappa} and
  \ref{kappa}.
\end{blanko}

\begin{theorem}\label{thm:MI=M}
  The decomposition space $\MI$ is \M.
\end{theorem}
\begin{proof}
  We first prove that the map $\sum_r \nondeg \MI_r \to \MI_1$ is a finite 
  map. Just check the fibre: fix a \M interval $A\in \MI_1$, with longest edge
  $a\in A_1$.
  From Lemma~\ref{lem:nondegU_r_A} we see that the fibre over $A$ is 
  $(\sum_r \nondeg A_r)_a=\sum_r (\nondeg A_r)_a$.  But this is the fibre 
  over $a\in  A_1$ of the map $\sum_r \nondeg A_r \to A_1$, which is finite
  by the assumption that $A$ is \M.
  
  Next we show that the $\infty$-groupoid $\MI_1$ is locally finite.
  But $\MI_1$ is the space of
  \M intervals, a full subcategory of the space of all decomposition spaces,
  so we need to show, for any \M interval $A$, that $\Eq_{\Decomp}(A)$ is 
  finite.  Now $A$ is in particular split, so we can compute this space of 
  equivalences inside the $\infty$-groupoid of split decomposition spaces,
  which by Proposition~\ref{prop:split=semi} is equivalent to the $\infty$-groupoid
  of semi-decomposition spaces.  So we have reduced to computing
  $$
  \Map_{\Fun(\Deltainj\op,\Grpd)}(\nondeg A,\nondeg A) .
  $$
  Now we know that all $\nondeg A_k$ are finite, so the mapping space can 
  be computed in the functor category with values in $\grpd$.
  On the other hand we also know that these groupoids are
  are empty for $k$ big enough, say 
  $\nondeg A_k = \emptyset$ for $k > r$.  Hence we can compute this mapping 
  space as a functor category on the truncation $\Deltainj^{\leq r}$.
  So we are finally talking about a functor category over a finite simplicial set
  (finite in the sense: only finitely many nondegenerate simplices), and with 
  values in finite groupoids.
  So we are done by the following lemma.
\end{proof}

\begin{lemma}
  Let $K$ be a finite simplicial set, and let $X$ and $Y$ be finite-groupoid-valued 
  presheaves on $K$. Then
  $$
  \Map_{\Fun(K\op,\grpd)}(X,Y)
  $$
  is finite.
\end{lemma}
\begin{proof}
This mapping space may be calculated as the limit of the diagram
$$
\widetilde 
K\xrightarrow{\;f\;}K\times K^{\op}\xrightarrow{X^{\op}\times Y}
\grpd\op\times\grpd\xrightarrow{\Map}\Grpd
$$
See for example \cite[Proposition 2.3]{glasman-1408.3065v3} for a proof. 
Here $\widetilde K$ is the edgewise subdivision of $K$, introduced 
in \cite[Appendix 1]{Segal-conf-sp} as follows:
$$
\widetilde K_n=K_{2n+1},
\qquad \widetilde d_i=d_id_{2n+1-i},
\quad \widetilde s_i=s_is_{2n+1-i}.
$$ 
and $f:\widetilde K\to K\times K\op$ is defined by $(d_{n+1},d_0)^{n+1}:K_{2n+1}\to K_n\times K_n$.
Now $\widetilde K$ is also finite: for each nondegenerate simplex $k$ of $K$, only a finite number of the degeneracies $s_{i_j}\dots s_{i_1}k$ will be nondegenerate in $\widetilde K$. 
Furthermore, mapping spaces between finite groupoids are 
  again finite, since $\grpd$ is cartesian closed by 
  Proposition~\ref{prop:cartesianclosed}. Thus 
  the mapping space in question can be computed as a finite limit of finite groupoids, so it is finite by 
  Proposition~\ref{prop:closedunderfinlims}.
\end{proof}

\begin{prop}
  Let $X$ be a  decomposition space $X$ with locally finite $X_1$.  Then the 
  following are equivalent.
  \begin{enumerate}
    \item $X$ is \M.
  
    \item All the intervals in $X$ are \M.
  
    \item Its classifying map factors through  $\MI\subset U$.
  \end{enumerate}
\end{prop}

\begin{proof}
  If the classifying map factors through $X\to \MI$, then $X$ is cULF over a \M
  space, hence is itself \FILT (by \ref{prop:cULF/FILT=FILT}), and has finite
  generic maps.  Since we have assumed $X_1$ locally finite, altogether $X$ is
  \M. We already showed (\ref{XM=>IM}) that if $X$ is \M then so are all its
  intervals.  Finally if all the intervals are \M, then clearly the classifying
  map factors through $\MI$.
\end{proof}


\begin{BM}
  For $1$-categories, Lawvere and Menni~\cite{LawvereMenniMR2720184}
  show that a category is \M if and only
  if all its intervals are \M. This is not quite true in our setting: even if
  all the intervals of $X$ are \M, and in particular finite, there is no
  guarantee that $X_1$ is locally finite.
\end{BM}

\begin{blanko}{Conjecture.}
  The decomposition space $\MI$ is terminal in the category of \M decomposition
  spaces and cULF maps.
\end{blanko}
This would follow from Conjecture \ref{conjecture}, but could be strictly weaker.

\begin{blanko}{The universal \M function.}
  The decomposition space $U$ of all intervals is complete, hence it has \M
  inversion at the objective level, in the sense of \ref{thm:zetaPhi}.  The \M
  function is the formal difference $\Phieven-\Phiodd$.  Note that the map
  $m:\nondeg U_k\to U_1$ in $\widehat{\Grpd}$, that defines $\Phi_k$, has fibres
  in {\Grpd} by Lemma~\ref{lem:U_r_A}.  Since every complete decomposition space
  $X$ has a canonical cULF map to $U$, it follows from Corollary \ref{phi=phi}
  that the \M function of $X$ is induced from that of $U$.  The latter can
  therefore be called the {\em universal \M function}.

  The same reasoning works in the \M situation, and implies the existence of a
  universal \M function numerically.  Namely, since $\MI$ is \M, its \M 
  inversion formula admits a cardinality by Theorem~\ref{thm:|M|}:
\end{blanko}

\begin{theorem}
  In the incidence algebra $\Q^{\pi_0 \MI}$, the zeta function $\norm \zeta: 
  \pi_0 \MI \to \Q$ is invertible under convolution, and its inverse is
  the universal \M function
  $$
  \norm \mu := \norm{\Phieven}-\norm{\Phiodd} .
  $$
  The \M function in the (numerical) incidence algebra of any \M decomposition
  space is induced from this universal \M function via the classifying map.
\end{theorem}

%% file: appendix.tex

\def\inputfile{appendix.tex}

\newcommand{\PP}{\mathscr{P}}

\appendix
\setcounter{lemma}{-1}

\section{Homotopy linear algebra and homotopy cardinality}

\begin{blanko}{Objective algebraic combinatorics.}
  One may say that algebraic combinatorics is the study of combinatorial
  structures via algebraic objects associated to them.  In the classical theory
  of \M inversion of Rota et al., the combinatorial objects are locally finite posets,
  and the associated algebraic structures are their incidence coalgebras and
  algebras, whose underlying vector spaces are freely generated by intervals in
  the poset.  In our theory, decomposition spaces are viewed as a generalisation
  of the notion of poset.  Instead of vector spaces to support the associated
  algebraic structures, we work with certain
%
%
%
%
  linear structures generated directly by the
  combinatorial objects (with coefficients in $\infty$-groupoids).
  This is the so-called `objective method', advocated in particular by
  Lawvere and Schanuel (see Lawvere--Menni~\cite{LawvereMenniMR2720184} for an explicit
  objective treatment of the Leroux theory of \M categories); the next level of
  objectivity is often called `groupoidification', developed in particular by
  Baez, Hoffnung and Walker~\cite{Baez-Hoffnung-Walker:0908.4305},
  where the scalars are (suitably finite) $1$-groupoids.  In the present work we 
  take coefficients in $\infty$-groupoids, and hence incorporate homotopy theory. 
  At the same time,
  the abstract viewpoints forced upon us by this setting lead to some 
  conceptual simplifications even at the $1$-groupoid level.
\end{blanko}

\begin{blanko}{Groupoid slices as vector spaces.}
  To deal with algebraic structures at the objective level requires at least to
  be able to form sums (linear combinations).  In analogy with taking the free
  vector space on a set, we can take the homotopy-sum completion of
  an $\infty$-groupoid $S$: this is (cf.~\ref{HoSum}) the homotopy slice $\infty$-category
  $\Grpd_{/S}$, whose objects are groupoid maps $X\to S$.
  It stands in
  for the free vector space on a set $\pi_0S$: just as a vector is a
  (finite) $\pi_0S$-indexed family of scalars (namely its coordinates with
  respect to the basis), an object $X\to S$ in $\Grpd_{/S}$ is interpreted as
  an $S$-indexed family of $\infty$-groupoids $X_s$, hence the fibre $X_s$ plays the role of
  the $s$th coordinate.
  
  The groupoid slices form an $\infty$-category in which the morphisms are the
  homotopy-sum preserving functors, the objective analogue of linear maps.  They
  are given by spans of $\infty$-groupoids, i.e.~doubly indexed families of $\infty$-groupoids,
  just as ordinary linear maps are given by matrices of numbers (once a basis 
  has been chosen).

  To really mimic vector spaces, where linear combinations are finite sums, we
  should require the total space $X$ to be finite in a suitable sense (while the base
  is allowed to be infinite).  Then 
  one can take homotopy cardinality, and recover linear algebra over $\Q$.
%
%
%
%
  The finiteness conditions are needed to be able to take homotopy cardinality.
  However, as long as we are working at the objective level, it is not necessary
  to impose the finiteness conditions, and in fact, the theory is simpler
  without them.  Furthermore, the notion of homotopy
  cardinality is not the only notion of size: Euler characteristic and various
  multiplicative cohomology theories are other alternatives, and it is
  reasonable to expect that the future will reveal more comprehensive and
  unified notions of size and measures.  For these reasons, we begin
  (\ref{sec:LIN}) with `linear algebra' without finiteness conditions, and then
  proceed to incorporate finiteness conditions expressed in terms of homotopy
  groups.
%
\end{blanko}

\begin{blanko}{Overview.}
%

  In Subsection~\ref{sec:LIN} we define the $\infty$-category $\LIN$ of groupoid slices 
  and linear functors, without imposing any finiteness conditions.
  
  For the finiteness conditions, the goals are:
  
  (1) Define `finite $\infty$-groupoid' and define homotopy cardinality of a 
  finite $\infty$-groupoid.
  
  (2) Define homotopy cardinality of `finite' families (for example 
  elements in the 
  incidence coalgebras): if $x:X \to S$ is a family with
  $X$ finite, its cardinality should be an element in the vector space 
  $\Q_{\pi_0 S}$ freely generated by the set $\pi_0 S$.

  (3) Define homotopy cardinality of finite presheaves (this is needed for the
  incidence algebras): these will take values in profinite-dimensional vector
  spaces.

To set this up uniformly, we follow
  Baez-Hoffnung-Walker~\cite{Baez-Hoffnung-Walker:0908.4305} and define a
  cardinality {\em functor} from a certain $\infty$-category of finite slices and linear
  functors to vector spaces.  From this `meta cardinality', all the individual
  notions of cardinality of families and presheaves are induced, by observing
  that vectors are the same thing as linear maps from the ground field.

  The `linear' $\infty$-categories of groupoid slices are introduced as follows.
  There is an $\infty$-category $\lin$ whose objects are $\infty$-categories of the form 
  $\grpd_{/\alpha}$ where $\alpha$ is a finite $\infty$-groupoid.  The morphisms are
  given by
  finite spans $\alpha \leftarrow \mu \to \beta$.  This $\infty$-category corresponds to
  the category $\vect$ of
  finite-dimensional vector spaces. We need infinite indexing,
  so the following two extensions are introduced.
  
  There is an $\infty$-category $\ind\lin$ whose objects are $\infty$-categories of the form
  $\grpd_{/S}$ with $S$ an `arbitrary' $\infty$-groupoid, and whose morphisms are
  spans of finite type (i.e.~the left leg has finite fibres).  This $\infty$-category 
  corresponds to the category $\ind\vect$ of general vector spaces
  (allowing infinite-dimensional ones).
  
  Finally we have the $\infty$-category $\pro\lin$ whose objects are $\infty$-categories of the form
  $\grpd^S$ with $S$ an `arbitrary' $\infty$-groupoid, and whose morphisms are
  spans of profinite type (i.e.~the right leg has finite fibres).  This $\infty$-category 
  corresponds to the category $\pro\vect$ of pro-finite-dimensional vector spaces.

\bigskip

\noindent
{\bf Remark.} To set up all this in order to define meta cardinality,
   it is actually only necessary to have $1$-categories.  This means that it
   is enough to consider equivalence classes of spans.  However, although
   cardinality is a main motivation, we are equally interested in 
   understanding
   how all this works at the objective level.  This turns out to throw light
   on the deeper meaning of ind and pro, and actually to  understand vector 
   spaces better.
   
   In order to introduce $\lin$, $\ind\lin$ and $\pro\lin$ as 
   $\infty$-categories, we first `extend scalars' from $\grpd$ to $\Grpd$,
   where there is more elbow room to perform the constructions. We work in
   the ambient $\infty$-category
   $\LIN$.  So we define,
   as subcategories of $\LIN$: the $\infty$-category
   $\Lin$ consisting of $\Grpd_{/\alpha}$ and finite spans, the $\infty$-category
   $\ind\Lin$ consisting of $\Grpd_{/S}$ and spans of finite
   type, and the $\infty$-category $\pro\Lin$ consisting of $\Grpd^S$ and spans of
   profinite type.  In the latter case, we can characterise the mapping
   spaces in terms of an attractive continuity condition (\ref{prop:cont}).
   
   The three $\infty$-categories constructed with $\Grpd$ coefficients 
   are in fact equivalent to the three $\infty$-categories with $\grpd$ coefficients
   introduced heuristically.
   
   There is a perfect pairing $\grpd_{/S} \times \grpd^S \to \grpd$,
   which upon taking cardinality yields the pairing $\Q_{\pi_0 S} \times 
   \Q^{\pi_0 S} \to \Q$.

\end{blanko}

\subsection{Homotopy linear algebra without finiteness conditions}
\label{sec:LIN}

\begin{blanko}{Fundamental equivalence.}\label{fund-eq}
  Fundamental to many constructions and arguments is the canonical equivalence
$$
\Grpd_{/S} \simeq \Grpd^S
$$
which is the homotopy version of the equivalence $\Set_{/S} \simeq \Set^S$ (for 
$S$ a set),
expressing the two ways of encoding a family of sets $\{X_s \mid s\in S\}$: 
either regarding the 
members of the family as the fibres of a map $X \to S$, or as a parametrisation
of sets $S \to \Set$.
To an object $X \to S$ one associates the functor $S\op\to\Grpd$ sending
$s\in S$ to the $\infty$-groupoid $X_s$.  The other direction is the Grothendieck construction,
which works as follows:
to any presheaf $F:S\op \to \Grpd$, that sits over the terminal presheaf $*$,
one associates the object $\colim (F)\to\colim(*)$.
It remains to observe that $\colim (*)$ is equivalent to $S$ itself.
More formally, the Grothendieck construction equivalence is a consequence of
a finer result, namely Lurie's straightening theorem
(\cite[Theorem 2.1.2.2]{Lurie:HTT}).  Lurie constructs 
a Quillen equivalence between the category of right
fibrations over $S$ and the category of (strict) simplicial presheaves on
$\mathfrak{C}[S]$.  
Combining this result with the fact that simplicial presheaves on $\mathfrak{C}[S]$
is a model for the functor $\infty$-category $\Fun(S\op,\Grpd)$ (see \cite{Lurie:HTT}, 
Proposition 5.1.1.1), the Grothendieck 
construction equivalence follows.
Note that when $S$ is just an $\infty$-groupoid (i.e.~a Kan complex), $X \to S$
is a right fibration if and only if $X$ itself is an $\infty$-groupoid.  Hence
altogether
$
\Grpd_{/S} \simeq \Fun(S\op,\Grpd)$, and 
since $S\op$ is canonically equivalent to $S$ (since it is just an 
$\infty$-groupoid), this establishes the fundamental equivalence
from this fancier viewpoint.
\end{blanko}

\begin{blanko}{Scalar multiplication and homotopy sums.}\label{scalar&hosum}
  The  `lowershriek' operation
$$
f\lowershriek:\Grpd_{/I}\to\Grpd_{/J}
$$
along a map $f:I\to J$ has two special cases, which play the role of 
  scalar multiplication (tensoring with an $\infty$-groupoid)
  and vector addition (homotopy sums):
  
The $\infty$-category $\Grpd_{/I}$
is {\em tensored} over $\Grpd$. Given $S \in \Grpd$ and $g:X \to I$ in $\Grpd_{/I}$
we have
$$
S\otimes g\;\;:=\;\; g\lowershriek
(S\times X\to X):S\times X\to I
\text{ in }\Grpd_{/I}. 
$$
It also has {\em homotopy sums}, by which we mean colimits 
indexed by an $\infty$-groupoid.
The colimit of a functor $F: B \to \Grpd_{/I}$ is a special case of the 
lowershriek.  Namely, the functor $F$ corresponds by adjunction to an object
$g: X \to B \times I$ in $\Grpd_{/B\times I}$, and we have
$$
\colim(F) = p\lowershriek(g)
$$
where $p: B\times I \to I$ is the projection.
We interpret this as the homotopy sum of the family $g:X\to B\times I$
with members $$g_b:X_b\longrightarrow \{b\}\times I=I,$$
and we denote  the homotopy sum by an integral sign:
\begin{equation}\label{fibsum}
\int^{b\in B}g_b \ := \ p\lowershriek g \;\text{ in }\Grpd_{/I}.
\end{equation}

\begin{eks}
  With $I=1$, 
this
  gives the important formula
  $$
  \int^{b\in B} X_b = X ,
  $$
  expressing the total space of $X\to B$ as the homotopy sum of its fibres.
\end{eks}

Using the above, we can define the $B$-indexed {\em linear
combination} of a family of vectors $g:X\to B\times I$ and scalars $f:
S \to B$, 
$$
\int^{b\in B} S_b \tensor g_b = 
p\lowershriek ( g \lowershriek (f'))
\;: S\times_B X\to I
\;\text{ in }\Grpd_{/I},
$$
as illustrated in the first row of the following diagram
 \begin{equation}\label{lincomb}\vcenter{
 \xymatrix@R1.8pc{ 
 S\times_B X
 \drpullback  \ar[r]^-{f'
} \ar[d] &X \ar[dr]\ar^-g[r]& B\times I \dto^q \ar[r]^-p & I\\
 S \ar^{f
}[rr]& & B .
 }}\end{equation}
Note that the members of the family $g\lowershriek(f'
)$  are  just $(g\lowershriek(f'
))_b=S_b\otimes g_b$. 

\begin{blanko}{Basis.}
  In $\Grpd_{/S}$, the names $\name s : 1 \to S$ play the role of a basis.
  Every object $X\to S$ can be written uniquely as a linear combination of basis
  elements; or, by allowing repetition of the basis elements instead of scalar 
  multiplication, as a homotopy sum of basis elements:
\end{blanko}
%
\end{blanko}

\begin{lemma}\label{BasisRep}
  Given $f:S\to B$  in $\Grpd_{/B}$ , we have
  $$
  f = \int^{s\in S} \name{f(s) } = \int^{b\in B} S_b\otimes \name{b} .
  $$
\end{lemma}

\begin{proof}
  For the first expression, take as family $S \stackrel{(\id, f)} 
  \longrightarrow S \times B$. 
  Then the members of the family are the names $\name {f(s)}$, and the formula follows from \eqref{fibsum}.
  For the second expression, take as family
  $g:B \stackrel{(\id, \id)}
  \longrightarrow B \times B$, and as scalars $f:
S\to B$ itself.
  Then the members of $g$ are the names $\name b$, 
  and the scalars are $S_b$,
  and in \eqref{lincomb} 
we have $p\lowershriek ( g \lowershriek (f'
))=f
$ since $pg$ 
  and $qg$ are the identity. 
\end{proof}

The name $\name b : 1 \to B$ corresponds under the
Grothendieck construction to the representable functor
\begin{eqnarray*}
  B & \longrightarrow & \Grpd  \\
  x & \longmapsto & \Map(b,x) .
\end{eqnarray*}
Thus, interpreted in the presheaf category $\Grpd^B$, the
Lemma is the standard result expressing any presheaf as a colimit of representables.

\begin{prop}\label{HoSum}
  $\Grpd_{/S}$ is the homotopy-sum completion of $S$.
  Precisely, for $\CC$ an $\infty$-category admitting
  homotopy sums, precomposition with the Yoneda embedding
  $S \to \Grpd_{/S}$ induces an equivalence of $\infty$-categories
  $$
  \Fun^{\int}(\Grpd_{/S},\CC) \isopil \Fun(S,\CC) ,
  $$
  where the functor category on the left consists of homotopy-sum preserving functors.
\end{prop}

\begin{proof}
  Since every object in $\Grpd_{/S}$ can be written
  as a homotopy sum of names, to preserve
  homotopy sums is equivalent to preserving all colimits, so the natural
  inclusion $\Fun^{\colim}(\Grpd_{/S},\CC) \to \Fun^{\int}(\Grpd_{/S},\CC)$
  is an equivalence.  It is therefore enough to establish the equivalence
  $$\Fun^{\colim}(\Grpd_{/S},\CC) \isopil \Fun(S,\CC).$$
  In the case where $\CC$ is cocomplete, this is true
  since
  $\Grpd_{/S} \simeq \Fun(S\op,\Grpd)$ is the colimit completion of $S$.  The
  proof of this statement (Lurie~\cite{Lurie:HTT}, Theorem~5.1.5.6) goes as
  follows: it is enough to prove that left Kan extension of any functor $S \to
  \CC$ along Yoneda exists and preserves colimits.  Existence 
  follows from \cite[Lemma~4.3.2.13]{Lurie:HTT}
  since $\CC$ is assumed cocomplete, and the fact
  that left Kan extensions preserve colimits \cite[Lemma~5.1.5.5 (1)]{Lurie:HTT}
  is independent of the cocompleteness of $\CC$.  In our case $\CC$ is not
  assumed to be cocomplete but only to admit homotopy sums.  But since $S$ is
  just an $\infty$-groupoid in our case, 
  this is enough to apply Lemma~4.3.2.13 of \cite{Lurie:HTT} 
  to guarantee the existence of the left Kan extension.
\end{proof}

\begin{blanko}{Linear functors.}\label{linearfunctorsfromspans}
  A span 
$$
I \stackrel p \leftarrow M \stackrel q \to J
$$
defines a {\em linear functor} 
 \begin{equation}\label{linfunctor}
 \Grpd_{/I} \stackrel {p^{\upperstar} } \rTo 
 \Grpd_{/M} \stackrel{q\lowershriek} \rTo \Grpd_{/J} .
 \end{equation}
\end{blanko}

\begin{lemma}\label{linearity}
  Linear functors preserve linear combinations.
\end{lemma}
\begin{proof}
Suppose $\int^{b\in B} S_b\otimes g_b
$ is the $B$-indexed linear combination of $f:
S\to B$  and $g:X\to B\times I$ 
in $\Grpd_{/I}$. This is shown in the middle row of the following diagram, and in the top row is shown the result of applying a linear functor $L$ given by \eqref{linfunctor}
$$
\xymatrix{
L(\int^{b\in B} S_b\otimes g_b):\!\!\!\!\!\!\!\!\!\!\!\!\!
&E\drpullback\rto^{f''
}\dto&\dto X'
\drpullback\rto^-{g'}&\drpullback\dto_{B\times p} B\times M\rto&M\dto^p\rto^q&J
\\
\int^{b\in B} S_b\otimes g_b:\!\!\!\!\!\!\!\!\!\!\!\!\!
&
S\times_BX\ar[d]\drpullback\rto^-{f'
}&X\rto^-g\ar[rd]&\dto B\times I\rto&I
\\&S\ar^-f
[rr]&&B&
}
$$
Now observe that  $f''
$ is the pullback of $f
$ along
$X'\stackrel{g'}\to B\times M
\to B$, and that the family $L(g)$ is given by 
$X'\stackrel{g'}\to B\times M
\stackrel{B\times q}\longrightarrow B\times J$.
The result is now clear, since the first row of the diagram coincides with
$$
\xymatrix{
\int^{b\in B} S_b\otimes L(g)_b:\!\!\!\!\!\!\!\!\!
&E\rto^{f''
}& X'
\rto^-{g'}&B\times M\rto^{B\times q}&B\times J\rto&J
}
$$
as required.
  %
  %
  %
\end{proof}

\begin{blanko}{Coordinates.}
  Coming back to the span
$$I\stackrel p\longleftarrow M \stackrel q\longrightarrow J$$
and the linear functor
$$
q\lowershriek p\upperstar
:\Grpd_{/I}\longrightarrow\Grpd_{/J},
$$
consider an element $\name i:1\to I$. Then we have, by Lemma~\ref{BasisRep},
$$
q\lowershriek p\upperstar\name i
=(M_i\to J)=
\int^{j\in J}{M_{i,j}}\otimes \name j
$$
$$
\xymatrix{&M_i\drto\dlto\\1\drto_{\name i}&&\dlto^pM\drto_q\\&I&&J}
$$
For  a more general element $f:X\to I$ we have $f=\int^iX_i\otimes \name{i}$ and so by homotopy linearity \ref{linearity}
\begin{align*}
q\lowershriek p\upperstar f&=
\int^{i,j}X_i\otimes M_{i,j}\otimes \name j.
\end{align*}
\end{blanko}


\begin{blanko}{The symmetric monoidal closed $\infty$-category $\kat{Pr}^{\mathrm L}$.}
  There is an $\infty$-category $\kat{Pr}^{\mathrm L}$, defined and studied in
  \cite[Section 5.5.3]{Lurie:HTT}, whose objects are the presentable
  $\infty$-categories, and whose morphisms are the left adjoint functors, or
  equivalently colimit-preserving functors.
  More precisely, given presentable
  $\infty$-categories $\CC$ and $\DD$, the mapping space in $\kat{Pr}^{\mathrm 
  L}$
  is
  $\Fun^{\mathrm L}(\CC,\DD)\subset\Fun(\CC,\DD)$,
  the full subcategory spanned by the colimit-preserving functors.
  The mapping space $\Map_{\kat{Pr}^{\mathrm L}}(\CC,\DD) = \Fun^{\mathrm 
  L}(\CC,\DD)^\eq$.
  
  The $\infty$-category $\kat{Pr}^{\mathrm L}$ has an internal hom
  (\cite[5.5.3.8]{Lurie:HTT}): for two presentable $\infty$-categories $\CC$ and
  $\DD$, the functor $\infty$-category $\Fun^{\mathrm L}(\CC,\DD)$ is again
  presentable.  Finally, $\kat{Pr}^{\mathrm L}$ has a canonical symmetric monoidal
  structure, left adjoint to the closed structure.  See Lurie~\cite{Lurie:HA}, 
  subsection 4.8.1, and in particular 4.8.1.14 and 4.8.1.17.  The tensor product
  can be characterised as universal recipient of functors in two variables that
  preserve colimits in each variable, and we have an evaluation functor
  $$
  \CC \tensor \Fun^{\mathrm L}(\CC,\DD) \to \DD
  $$
  which exhibits $\Fun^{\mathrm L}(\CC,\DD)$ as an exponential of $\DD$ by 
  $\CC$.
  
  This tensor product has an easy description in the case of presheaf 
  categories (cf.~\cite[4.8.1.12]{Lurie:HA}): if $\CC= \PP(\CC_0)$ and $\DD= \PP(\DD_0)$ for small
  $\infty$-categories $\CC_0$ and $\DD_0$, then we have
  \begin{equation}\label{eq:tensor-times}
    \PP(\CC_0) \tensor \PP(\DD_0) \simeq \PP(\CC_0 \times \DD_0)  .
  \end{equation}
\end{blanko}

\begin{blanko}{The $\infty$-category $\LIN$.}\label{internalLin}
%
 We define $\LIN$ to be the full
  subcategory of $\kat{Pr}^{\mathrm L}$ whose objects are the
  $\infty$-categories (equivalent to) groupoid slices $\Grpd_{/S}$.
  We call the functors {\em linear}.  The mapping spaces in $\LIN$ are
\begin{eqnarray*}
  \LIN(\Grpd_{/I}, \Grpd_{/J}) &=& \Fun^{\mathrm L}(\Grpd_{/I}, \Grpd_{/J})^\eq \\
  & \simeq & \Fun^{\mathrm L}(\Grpd^I, \Grpd^J)^\eq \\
  & \simeq & \Fun(I,\Grpd^J)^\eq \\
  & \simeq & (\Grpd^{I\times J})^\eq \\
  & \simeq & (\Grpd_{/I\times J} )^\eq.
\end{eqnarray*}
This shows in particular that the linear functors are given by spans.
Concretely, tracing through the chain of equivalences, a span defines a
left adjoint functor as described above in \ref{linearfunctorsfromspans}.
Composition in $\LIN$ is given by composing spans, i.e.~taking a
pullback.  This amounts to the Beck--Chevalley condition.

The $\infty$-category $\LIN$ inherits a symmetric monoidal closed structure 
from $\kat{Pr}^{\mathrm L}$.  For the `internal hom':
\begin{eqnarray*}
  \un\LIN(\Grpd_{/I}, \Grpd_{/J}) &:=& \Fun^{\mathrm L}(\Grpd_{/I}, \Grpd_{/J}) \\
  & \simeq & \Fun(I, \Grpd^{J}) \\
  &\simeq & \Fun( I \times J, \Grpd) \\
  &\simeq & \Grpd_{/I\times J} .
\end{eqnarray*}

Also the tensor product restricts, and we have the convenient formula
$$
\Grpd_{/I} \tensor \Grpd_{/J} = \Grpd_{/I\times J}
$$
with neutral object $\Grpd$.
This follows from formula~\eqref{eq:tensor-times} combined with the
fundamental equivalence $\Grpd_{/S} \simeq \Grpd^S$.

Clearly we have
$$
\LIN( \Grpd_{/I} \tensor \Grpd_{/J}, \Grpd_{/K}) \simeq \LIN( \Grpd_{/I}, 
\un\LIN(\Grpd_{/J},\Grpd_{/K})) 
$$
as both spaces are naturally equivalent to $(\Grpd_{/I\times J \times 
K})^\eq$.
\end{blanko}

\begin{blanko}{The linear dual.}\label{Lineardual}
`Homming' into the neutral object defines a contravariant autoequivalence
of $\LIN$:
\begin{eqnarray*}
  \LIN & \longrightarrow & \LIN\op  \\
  \Grpd_{/S} & \longmapsto & \un\LIN(\Grpd_{/S}, \Grpd) \simeq \Grpd_{/S}
  \simeq \Grpd^{S} .
\end{eqnarray*}
%
%
%
%
  Here there right-hand side should be considered the dual of $\Grpd_{/S}$.
  (Since our vector spaces are fully coordinatised, the difference between a
  vector space and its dual is easily blurred.  We will see a clearer difference
  when we come to the finiteness conditions, in which situation the dual of a `vector
  space' $\grpd_{/S}$ is $\grpd^S$ which should rather be thought of as a
  profinite-dimensional vector space.)
\end{blanko}

\begin{blanko}{Remark.}
  It is clear that there is actually an $\infty$-$2$-category in play here,
  with the $\un\LIN(\Grpd_{/S}, \Grpd_{/T})$ as hom $\infty$-categories.
  This can be described as a Rezk-category object in the `distributor'
  $\kat{Cat}$, following the work of Barwick and Lurie~\cite{Lurie:GoodwillieI}.
  Explicitly, let
  $\Lambda_k$ denote the full subcategory of $\Delta_k\times\Delta_k$
  consisting of the pairs $(i,j)$ with $i+j\leq k$.  These are the shapes of
  diagrams
  of $k$ composable spans.  They form a cosimplicial category.  
  Define $\operatorname{Sp}_k$ to be the full subcategory of
  $\Fun(\Lambda_k,\Grpd)$ consisting of those diagrams $S : \Lambda_k \to \Grpd$
  for which for all $i'< i$ and $j' < j$ (with  $i+j \leq k$)
  the square 
  $$\xymatrix{
     S_{i',j'} \drpullback \ar[r]\ar[d] & S_{i,j'} \ar[d] \\
     S_{i',j} \ar[r] & S_{i,j}
  }$$
  is a pullback.
  Then we claim that
  \begin{eqnarray*}
    \Delta\op  & \longrightarrow & \kat{Cat}  \\
    {}[k] & \longmapsto & \operatorname{Sp}_k 
  \end{eqnarray*}
  defines a Rezk-category object in $\kat{Cat}$ corresponding to $\un\LIN$.
  We leave the claim unproved, as the result is not necessary for our purposes.
%
%
%
\end{blanko}


\label{sec:card}

\subsection{Cardinality of finite $\infty$-groupoids}
\label{sec:finite}

\begin{blanko}{Finite $\infty$-groupoids.}\label{finite}
  An $\infty$-groupoid $B$ is called {\em locally finite} if at each base point
  $b$ the homotopy groups $\pi_i (B,b)$ are finite for $i\geq1$ and are trivial
  for $i$ sufficiently large.  An $\infty$-groupoid is called {\em finite} if it
  is locally finite and has finitely many components. 
  Note that $B$ is locally finite iff it is a filtered colimit of finite
  $\infty$-groupoids.  An example of a non locally finite $\infty$-groupoid is $B\Z$.

  Let $\grpd\subset \Grpd$ be the full subcategory spanned by the finite 
  $\infty$-groupoids.  For $S$ any $\infty$-groupoid, let $\grpd_{/S}$ be 
  the `comma $\infty$-category' defined by the
  following pullback diagram of $\infty$-categories:
    $$\xymatrix{
       \grpd_{/S} 
  \ar[r]\ar[d]\drpullback &       \Grpd_{/S} 
  \ar[d] \\
       \grpd \ar[r] & \Grpd .
    }$$
\end{blanko}

\begin{blanko}{Cardinality.} \cite{Baez-Dolan:finset-feynman}
  The {\em (homotopy) cardinality} of a finite $\infty$-groupoid $B$
is the nonnegative  rational number given by the formula
  $$
  \norm{B} := \sum_{b\in \pi_0 B} \prod_{i>0} \norm{\pi_i(B,b)}^{(-1)^i} .
  $$
  Here the norm signs on the right refer to order of homotopy groups.
\end{blanko}

If $G$ is a $1$-groupoid, that is, an $\infty$-groupoid
having trivial homotopy groups $\pi_i(G)=0$ for $i>1$, its cardinality is
  $$
  \norm{G} = \sum_{x\in \pi_0 G} \frac1{\norm{\Aut_G(x)}} .
  $$
The notion and basic properties of homotopy cardinality have been around
for a long time. See Baez--Dolan~\cite{Baez-Dolan:finset-feynman}. 
The first printed reference we know of is 
Quinn~\cite{Quinn:TQFT}.

\begin{blanko}{Remark.}\label{sumcard}
  It is clear from the definition that 
  a finite sum of finite $\infty$-groupoids is again finite, and that
  cardinality is compatible with finite sums:
  $$
  \norm{\sum_{i=1}^n X_i}= \sum_{i=1}^n \norm{X_i} .
  $$
\end{blanko}

\begin{lemma}\label{FEB} Suppose $B$ is connected.
  Given a fibre sequence
  $$\xymatrix{
     F \ar[r]\ar[d]\drpullback & E \ar[d] \\
     1 \ar[r] & B ,
  }$$
  if two of the three spaces are finite then so is the third, and in that case
$$\norm E\;=\;\norm F\,\norm B.$$
\end{lemma}
\begin{proof}
  This follows from the homotopy long exact sequence of a fibre sequence.
\end{proof}

For $b\in B$, we denote by $B_{[b]}$ the connected component of $B$ containing
$b$.  Thus an $\infty$-groupoid $B$ is locally finite if and only if each connected
component $B_{[b]}$ is finite.
  
\begin{lemma}
  Suppose $B$ locally finite.  Given a map $E\to B$, then $E$ is finite if and
  only if all fibres $E_b$ are finite, and are nonempty for only finitely many 
  $b\in\pi_0B$.  In this situation,
  $$\norm E\;=\;\sum_{b\in\pi_0(B)}\norm{E_b}\,\norm{B_{[b]}}.$$
\end{lemma}
\begin{proof}
  Write $E$ as the sum of the full fibres $E_{[b]}$,
  and apply 
  Lemma~\ref{FEB} to the fibrations
  $E_b\to E_{[b]}\to B_{[b]}$  for each $b\in\pi_0(B)$. Finally sum 
  (\ref{sumcard}) over those $b\in \pi_0 B$ with 
  non-empty $E_b$.
\end{proof}

\begin{cor}
  Cardinality preserves (finite) products.
\end{cor}

\begin{proof}
  Apply the previous lemma to a projection.
\end{proof}

\begin{blanko}{Notation.}
  Given any $\infty$-groupoid $B$ and a function $q: \pi_0 B \to 
\Q$, we write
$$
\int^{b\in B}q_b\;:=\; \sum_{b\in \pi_0 B} q_b\,\norm{B_{[b]}} 
$$ 
if the sum is finite. Then the previous lemma says
$$
\norm E\;=\;\int^{b\in B}\norm{E_b}
$$
for any finite $\infty$-groupoid $E$ and a map $E\to B$.
Two important special cases are given by fibre products and loop spaces:
%
\end{blanko}

\begin{lemma}\label{finite-products}
  In the situation of a pullback
  $$
\xymatrix{
X\times_B Y \ar[r] \drpullback \ar[d] & X\times Y  \ar[d] \\
B \ar[r]_-{\text{diag}} & B\times B,}$$
  if $X$ and $Y$ are finite, and $B$ is locally finite,
then $X\times_B Y$ is finite and 
$$\norm{X\times_B Y}\;=\;\int^{b\in B}\norm{X_b}\norm{Y_b}.$$
\end{lemma}

\begin{prop}\label{prop:closedunderfinlims}
  The $\infty$-category $\grpd$ is closed under finite limits.
\end{prop}
\begin{proof}
  It is closed under pullbacks by the previous lemma, and it also contains the 
  terminal object, hence it is closed under all finite limits.
\end{proof}

\begin{lemma}\label{loop-finite}
  In the situation of a loop space
$$  \xymatrix{
\Omega(B,b) \ar[r] \drpullback \ar[d] & 1 \ar[d]^{\name b} \\
1 \ar[r]_-{\name b} & B_{[b]}\,.}$$ we have that
  $B$ is locally finite if and only if each $\Omega(B,b)$ is finite, and in that
  case
  $$
  \norm{\Omega(B,b)}\cdot\norm{B_{[b]}}=1.
  $$
\end{lemma}

\begin{blanko}{Finite maps.}
We say that a map $p: E \to B$ is  \emph{finite} if any pullback 
to a finite base $X$  has  finite total space $X'$, as in the diagram
\begin{align}\label{relfinpb}\vcenter{\xymatrix{
X' \ar[r] \drpullback \ar[d] & E  \ar[d]^p \\
X \ar[r]_c & B .}}\end{align}
\end{blanko}

\begin{lemma}\label{lem:finitemaps}\begin{enumerate}
\item Pullbacks of finite maps are finite.
\item 
\label{item:fibrefinite}
A map $E\to B$ is finite if and only if each fibre $E_b$ is finite.
\end{enumerate}\end{lemma}
\begin{proof}
Statement (1), and one direction of (2), are clear.  In the other direction, the map $X'\to X$ in the pullback diagram \eqref{relfinpb}
has finite fibres $X'_x=E_{c(x)}$, so $X$ finite implies $X'$ finite by Lemma~\ref{FEB}.
\end{proof}

\begin{lemma}\label{locfinbase}
Suppose $p:E\to B$ has locally finite base. 
\begin{enumerate}\item If $p$ is finite then $E$ is locally finite. \item If $E$ is finite then $p$ is finite.\end{enumerate}
\end{lemma}
\begin{proof}
A full fibre $E_{[b]}$ of $p$ is finite if and only if $E_{b}$ is, by
Lemma~\ref{FEB}.  If each full fibre $E_{[b]}$ is finite, then each component
$E_{[e]}$ is, and if $E$ is finite then each full fibre is.
\end{proof}

\begin{lemma}\label{lem:1->B}
  $B$ is locally finite iff each name $1 \to B$ is a finite map.
\end{lemma}

\begin{prop}\label{prop:cartesianclosed}
  The $\infty$-category $\grpd$ is cartesian closed.
\end{prop}
\begin{proof}
  We already know that $\Grpd$ is cartesian closed.
  We need to show that for $X$ and $Y$ finite groupoids, the mapping space
  $\Map(X,Y)$ is again finite.  We can assume  $X$ and $Y$ connected:
  indeed,
  if we write them as sums of their connected components,
  $X = \sum X_i$ and $Y= \sum Y_j$, then we have
  $$
  \Map(X,Y) = \Map (\sum X_i, Y) = \prod_i \Map(X_i,Y) 
  = \prod_i \sum_j \Map(X_i, Y_j)
  $$
  Since these are finite products and sums, if we can prove that each 
  $\Map(X_i,Y_j)$ is finite, then we are done.  
  Since $Y$ is finite, $\Map(S^k,Y)$ is finite for all $k\geq 0$, and
  there is $r \geq 0$ such that $\Map(S^k,Y)= *$
  for all $k\geq r$.  This is to say that $Y$ is $r$-truncated.
  On the other hand, since $X$ is finite, it has the homotopy type of
  a CW complex with finitely many cells in each dimension.  Write
  $$
  X = \colim_{i\in I} E_i
  $$
  for its realisation as a cell complex.
  Write $X' = \colim_{i\in I'} E_i$ for the colimit obtained by the same
  prescription 
  but omitting all cells of dimension $> r$; this is now a finite colimit,
  and the comparison map
  $X \to X'$ is $r$-connected.  Since $Y$ is $r$-truncated, we have
  $$
  \Map(X',Y) \isopil \Map(X,Y),
  $$
  and the first space is finite: indeed,
  $$
  \Map(X',Y) = \Map( \colim_{i\in I'} E_i, Y) = \lim_{i\in I'} \Map(E_i, Y)
  $$
  is a finite limit of finite spaces, hence is finite by 
  Proposition~\ref{prop:closedunderfinlims}.
 \end{proof}

\begin{theorem}
  For each locally finite $\infty$-groupoid $S$, the comma $\infty$-category
  $\grpd_{/S}$ is cartesian closed.
\end{theorem}
\begin{proof}
  This is essentially a corollary of Proposition~\ref{prop:cartesianclosed}
  and the fact that the bigger $\infty$-category $\Grpd_{/S}$ is cartesian closed.
  We just need to check that the internal mapping object in $\Grpd_{/S}$
  actually belongs to $\grpd_{/S}$.   
  Given $a:A\to S$ and $b: B \to S$, the internal mapping object is
  $$
  \underline{\Map}_{/S}(a,b) \to S
  $$
  given fibrewise by
  $$
  \underline{\Map}_{/S}(a,b) _s = \Map(A_s, B_s)
  $$
  Since $A_s$ and $B_s$ are finite spaces, also the mapping space is finite,
  by \ref{prop:cartesianclosed}.
\end{proof}
\begin{cor}
  The $\infty$-category $\grpd$ is locally cartesian closed.
\end{cor}

\subsection{Finiteness conditions on groupoid slices}

In this subsection, 
after some motivation and background from linear algebra, we first explain the
finiteness conditions imposed on slice categories in order to model vector
spaces and profinite-dimensional vector spaces.  Then afterwards we assemble
all this into $\infty$-categories using more formal constructions.

\begin{blanko}{Linear algebra rappels.}\label{vect-rappels}
  There is a fundamental duality 
  $$
  \Vect \simeq \pro\vect\op
  $$
  between vector spaces and profinite-dimensional
vector spaces: given any vector space $V$, the linear dual $V\upperstar $ is a
profinite-dimensional vector space, and conversely, given a
profinite-dimensional vector space, its continuous dual is a vector space.
This equivalence is a formal consequence of the observation that the category 
$\vect$ of finite-dimensional vector spaces is self-dual: $\vect\simeq 
\vect\op$, and the fact that $\Vect= \ind\vect$, the ind completion of $\vect$.

In the fully coordinatised situation typical to algebraic combinatorics, the vector
space arises from a set $S$ (typically an infinite set of 
isoclasses of combinatorial objects): the vector space is then
$$V= \Q_S = 
\left\{\,\sum_{s\in S} c_s\,\delta_s\;:\;c_s\in\Q\text{ almost all zero}
\right\},$$
the vector space with basis the symbols $\delta_s$ for 
each $s\in S$.
The linear dual is then the function
space $V\upperstar =\Q^S$, having a canonical pro-basis consisting of
the functions $\delta^s$, taking the value 1 on $s$ and 0 elsewhere.
%
%
%
%
%
%

Vectors in $\Q_S$
are finite linear combinations of the $\delta_s$, and we represent a
vector as an infinite column vector $\vec v$ 
with only finitely many non-zero
entries.  A linear map $f:\Q_S \to \Q_T$ is given by 
matrix multiplication
$$
\vec v \mapsto A \cdot \vec v .
$$
for $A$ an infinite
$2$-dimensional matrix with $T$-many rows and $S$-many columns, and
with the crucial property that it is {\em column finite}: in each
column there are only finitely many non-zero entries. 
More generally, the matrix multiplication of two column-finite
matrices makes sense and is again a column-finite matrix.  The
identity matrix is clearly column finite.  A basis element
$\delta_s$ is identified with the
column vector all of whose entries are zero, except the one of index
$s$.

On the other
hand, elements in the function space $\Q^S$ are represented as
infinite row vectors. 
A continuous linear map $\Q^T \to \Q^S$,
dual to the linear map $f$, is represented by the {\em same matrix} 
$A$, but viewed now as sending a row vector $\vec w$ (indexed by $T$)
to the matrix product $\vec w \cdot A$.  Again the fact that $A$ is 
column finite ensures that this matrix product is well defined.

There is a canonical perfect pairing
\begin{eqnarray*}
    \Q_S \times \Q^S  &\longrightarrow& \Q \\
    (\vec v , f) &\longmapsto& f(\vec v)
\end{eqnarray*}
given by evaluation. In matrix terms, it is just a matter of 
multiplying $f\cdot \vec v$.
%
%
%
%
\end{blanko}

\begin{blanko}{Remark.}
  In the theory of \M inversion, the incidence coalgebra is on the
vector-space side of the duality: the coalgebra is the free vector
space on some objects, and the formula for comultiplication is a
finite sum, reflecting the fact that an object decomposes in finitely
many ways.  The incidence algebra is the linear dual, the
pro-finite-dimensional vector space of functions on the objects. 
In many interesting cases the
incidence algebra (a monoid object in a function space)
restricts to a monoid in the space of functions with finite support,
which can be regarded as a kind of Hall algebra.  This happens under
different finiteness conditions on the
combinatorial structures.  Note that the zeta function is not finitely
supported (except in degenerate cases), and that \M inversion does not
make sense in the context of functions with finite-support.
\end{blanko}

\bigskip

This duality has a very neat description in homotopy linear algebra.
While the vector space $\Q_{\pi_0 S}$ is modelled by the
$\infty$-category $\grpd_{/S}$, the function space $\Q^{\pi_0 S}$ is
modelled by the $\infty$-category $\grpd^S$.
The classical
duality results from taking cardinality of a duality on the 
categorical level that we proceed to explain.
For the most elegant definition of  cardinality we first need to
introduce the objective versions of $\Vect$ and $\pro\vect$.

\bigskip

Let $S$ be a locally finite $\infty$-groupoid,
and consider the following $\infty$-categories.

\begin{itemize}
\item denote by $\grpd^S$ the full subcategory of $\Grpd^S$
spanned by the presheaves $S\to\Grpd$ whose images lie in $\grpd$, and 
\item denote by $\Grpd^{\text{rel.fin.}}_{/S}$ the full subcategory of $\Grpd_{/S}$ 
spanned by the finite maps $p:X\to S$.
\end{itemize}


\begin{lemma}\label{fundamentalfinite}
  The fundamental equivalence $\Grpd^S
\simeq \Grpd_{/S} 
$ 
restricts to an 
  equivalence
  $$\grpd^S\simeq\Grpd^{\text{rel.fin.}}_{/S}
$$
\end{lemma}
\begin{proof}
  The inclusions $\grpd_{/S} \subset \Grpd_{/S}$ and $\grpd^S \subset \Grpd^S$
  are both full, and the objects characterising them correspond to each other
  under the fundamental equivalence because of 
  Lemma~\ref{lem:finitemaps}~\eqref{item:fibrefinite}.
\end{proof}

From the definition of finite map we have the following result.
\begin{lemma}\label{finitetypespan}
    For a span $S \stackrel p\leftarrow M \stackrel q\to T$
    defining a linear map $F:\Grpd_{/S}\to \Grpd_{/T}$,
the following are equivalent:
\begin{enumerate}
\item $p$ is finite,
\item $F$ restricts to
$$
\grpd_{/S}\stackrel{p\upperstar}\longrightarrow 
\grpd_{/M}\stackrel{q\lowershriek}\longrightarrow \grpd_{/T}
$$
\item  $F$ restricts to
$$
\Grpd^{\text{rel.fin.}}_{/T}
\stackrel{q\upperstar}\longrightarrow 
\Grpd^{\text{rel.fin.}}_{/M}\stackrel{p\lowershriek}\longrightarrow 
\Grpd^{\text{rel.fin.}}_{/S}
$$
\end{enumerate}
\end{lemma}
\noindent

\bigskip

The $\infty$-category
$\grpd_{/S}$ has finite homotopy sums: for $I$ finite and $F:I \to \grpd_{/S}$
we have $\colim F = p\lowershriek (X \to I \times S)$,  where $p:I\times S \to 
S$ is the projection.  A family $X \to I \times S$ comes from some $F:I \to 
\grpd_{/S}$ and admits a homotopy sum in 
$\grpd_{/S}$ when for each $i\in I$, the partial fibre $X_i$ is finite.
Since already $I$ was assumed finite, this is equivalent to having $X$ finite.

\bigskip

The following is the finite version of Proposition~\ref{HoSum}
\begin{lemma}\label{hosum}
The $\infty$-category  $\grpd_{/S}$ is the finite-homotopy-sum completion of $S$.
\end{lemma}

\subsection{Categories of linear maps with infinite-groupoid coefficients}

Our main interest is in the linear $\infty$-categories with finite-groupoid coefficients,
but it is technically simpler to introduce first the infinite-coefficients 
version of these $\infty$-categories, since they can be defined as subcategories in
$\LIN$, and can be handled with the ease of presentable $\infty$-categories.

\bigskip

Recall that a span $(S\stackrel p\longleftarrow M \stackrel q\longrightarrow T)$
defines a linear functor
$$
L:\Grpd_{/S}\stackrel{p\upperstar}\longrightarrow 
\Grpd_{/M}\stackrel{q\lowershriek}\longrightarrow \Grpd_{/T}.
$$
by pullback and postcomposition, as shown in the following diagram
$$\xymatrix{
\!\!\!\!\!\!\!\!\!\!\!\!\!\!\!\!
L(x):X'\rto^{p\upperstar x}\dto\drpullback&M\rto^q\dto^p&T\\
\!\!\!\!\!\!\!\!\!
x:X\rto&S.\!
}
$$

Let $\Lin \subset \LIN$ be the $\infty$-category whose objects are the slices
$\Grpd_{/\sigma}$, with $\sigma$ finite.  Its morphisms are 
those linear functors between them which preserve finite objects.
Clearly these are given by
the spans
of the form $\sigma \leftarrow \mu \rightarrow \tau$ where $\sigma, \tau$ and
$\mu$ are finite.  Note that there are equivalences of $\infty$-categories
$\Grpd_{/\sigma}\simeq \Grpd^{\sigma}$ for each $\sigma$.

Let $\ind\Lin$ be the $\infty$-category whose objects are the slices $\Grpd_{/S}$ with 
$S$
locally finite, and whose
morphisms are the linear functors between them that preserve finite objects.
These correspond to the spans of the form $S\stackrel p\leftarrow M \rightarrow 
T$ with $p$ finite.   

Let $\pro\Lin$ be the $\infty$-category whose objects are the presheaf categories
$\Grpd^{S}$ with $S$ locally finite, and whose morphisms are the continuous 
linear functors:

A linear functor $F: \Grpd^{T}\to \Grpd^{S}$ is called \emph{continuous} when for all $\epsilon \subset S$ there exists $\delta \subset T$ and a factorisation 
$$\xymatrix{
      \Grpd^{T} \ar[r]\ar[d]_F &  \Grpd^{\delta} \ar[d]^{F_{\delta}} \\
      \Grpd^{S}\ar[r] &  \Grpd^{\epsilon}
  }$$
where the horizontal maps are the projections of the canonical pro-structures.  


\begin{prop}\label{prop:cont}
    For a linear functor $F: \Grpd^T \to \Grpd^S$ in $\LIN$, represented by a 
    span
    $$
    S \stackrel p \leftarrow M \stackrel q \to T,
    $$
    the following are equivalent.
    
    \begin{enumerate}
    
        \item The span is of finite type (i.e.~$p$ is a finite map).
    
        \item $F$ is continuous.
    \end{enumerate}
\end{prop}
\begin{proof}
    It is easy to see that if the span is of finite type then  $F$ is continuous: 
    for any given
    finite $\epsilon \subset S$ with inclusion $j$, the pullback $\mu$
    is finite, and we can take $\delta$ to be the essential full image 
    of the composite $q \circ m$: 
     \begin{equation}\label{epsdeltdiag}
	 \xymatrix{
    \epsilon \ar[d]_j & \mu \dlpullback \ar[l]_{\bar p} 
    \ar[d]^m  \ar[r]^{\bar q} & \delta\ar[d]^i\\
    S & \ar[l]^p M \ar[r]_q & T .}
    \end{equation}
    Now by Beck-Chevalley,
    $$j\upperstar 
    p\lowershriek q\upperstar = \bar p\lowershriek m\upperstar q \upperstar =
     \bar p \lowershriek \bar q{}\upperstar i \upperstar$$ 
    which is precisely the continuity condition.

    Conversely, if the factorisation in the continuity diagram exists,
    let $\epsilon \leftarrow \mu \to \delta$ be the span (of finite 
    $\infty$-groupoids) representing $f_{\delta_\epsilon}$.  Then we have the
    outer rectangle of the diagram~\eqref{epsdeltdiag} and an 
    isomorphism
    $$j\upperstar 
    p\lowershriek q\upperstar = 
     \bar p \lowershriek \bar q{}\upperstar i \upperstar$$ 
    Now a standard argument implies the existence of $m$ completing
    the diagram: namely take the pullback of $j$ and $p$, with the 
    effect of interchanging the order of upperstar and lowershriek.
    Now both
    linear maps are of the form upperstars-followed-by-lowershriek,
    and by uniqueness of this representation, the said pullback must 
    agree with $\mu$ and in particular is finite.  Since this is true
    for every $\epsilon$, this is precisely to say that $p$ is 
    a finite map.
\end{proof}

%

The continuity condition is precisely continuity for the profinite 
slice topology, as we proceed to explain.
Every locally finite $\infty$-groupoid $S$ is canonically the filtered colimit of
its finite (full) subgroupoids:
$$
S = \colim_{\alpha\subset S} \alpha .
$$
We use Greek letters here to denote finite $\infty$-groupoids.
Similarly, $\Grpd^S$ is a cofiltered limit of $\infty$-categories $\Grpd^\alpha \simeq 
\Grpd_{/\alpha}$:
$$
\Grpd^S = \lim_{\alpha\subset S} \Grpd^\alpha .
$$


This leads to the following `categorical' description of the mapping spaces:
$$
\pro\Lin (\Grpd^T, \Grpd^S) : = \lim_{\epsilon\subset S} \colim_{\delta\subset 
T} \Lin(\Grpd^\delta, \Grpd^\epsilon) .
$$

%

\subsection{Categories of linear maps with finite-groupoid coefficients}

\begin{blanko}{The $\infty$-category $\lin$.}
  We define $\lin$ to be the subcategory of $\bigcat$ whose objects are those
$\infty$-categories equivalent to $\grpd_{/\sigma}$ for some finite $\infty$-groupoid $\sigma$,
and whose mapping spaces are the full subgroupoids of those of $\bigcat$ given by
the functors which are restrictions of functors in $\Lin(\Grpd_{/\sigma}, 
\Grpd_{/\tau})$.  Note that the latter mapping space was exactly defined as those
linear functors in $\LIN$ that preserved finite objects.  Hence, by 
construction there is an equivalence of mapping spaces 
$$
\lin(\grpd_{/\sigma},\grpd_{/\tau}) \simeq 
\Lin(\Grpd_{/\sigma}, 
\Grpd_{/\tau}) ,
$$
and in particular, the mapping spaces are given by spans of finite $\infty$-groupoids.
The maps can also be described as those functors that 
preserve finite homotopy sums.  By construction we have an equivalence
of $\infty$-categories
$$
\lin \simeq \Lin .
$$
\end{blanko}

\begin{blanko}{The $\infty$-category $\ind\lin$.}
  Analogously, we define $\ind\lin$ to be the subcategory of $\bigcat$, whose
  objects are the $\infty$-categories equivalent to $\grpd_{/S}$ for some locally finite
  $\infty$-groupoid $S$, and whose mapping spaces are the full subgroupoids of the
  mapping spaces of $\bigcat$ given by the functors that are restrictions of
  functors in $ \Lin(\Grpd_{/S}, \Grpd_{/T})$; in other words 
  (by~\ref{finitetypespan}), they are the
  $\infty$-groupoids of spans of finite type.  Again by construction we have
  $$
  \ind\lin \simeq \ind\Lin .
  $$
\end{blanko}

\begin{blanko}{Categories of prolinear maps.}
  We denote by $\pro\lin$ the $\infty$-category whose objects are the $\infty$-categories
  $\grpd^{S}$, where $S$ is locally finite, and whose morphisms are restrictions
  of continuous linear functors.
  We have seen that the mapping spaces are given by spans of finite type:
  $$\pro\lin(\grpd^T  ,\grpd^{S}  )\;\;=\;\;\left\{  
  (T\stackrel q\longleftarrow M \stackrel p\longrightarrow S
  )\;:\; p\text{ finite}\right\}.$$

  As in the ind case we have
  $$
  \pro\lin \simeq \pro\Lin ,
  $$
  and by combining the previous results we also find
  $$
  \pro\lin (\grpd^T, \grpd^S) : = \lim_{\epsilon\subset S} \colim_{\delta\subset 
  T} \lin(\grpd^\delta, \grpd^\epsilon) .
  $$
\end{blanko}

\begin{blanko}{Mapping categories.}
  Just as $\bigcat$ has internal mapping {\em categories} (of which the mapping spaces
  are the maximal subgroupoids), we also have internal mapping categories in $\lin$, denoted
  $\un{\lin}$:
  $$
  \un\lin (\grpd_{/\sigma},\grpd_{/\tau}) \simeq \grpd_{/\sigma\times\tau}.
  $$

  Also $\ind\lin$ and $\pro\lin$ have mapping categories, but due to the
  finiteness conditions, they are not internal.  Just as the mapping spaces are
  given (in each case) as $\infty$-groupoids of spans of finite type, the mapping
  categories are given as $\infty$-categories of spans of finite type.  Denoting the
  mapping categories with underline, we content ourselves to record the 
  important case of
  `linear dual':
\end{blanko}

\begin{prop}\label{linearduals}
  \begin{eqnarray*}
    \un{\ind\lin}(\grpd_{/S},\grpd) &=& \grpd^S\\
    \un{\pro\lin}(\grpd^T,\grpd) &=& \grpd_{/T}.
  \end{eqnarray*}
\end{prop}

\normalsize


\begin{blanko}{Remark.}
  It is clear that the correct viewpoint here would be that there is 
  altogether a $2$-equivalence between the 
  $\infty$-$2$-categories
  $$
  \ind\lin\op\cong\pro\lin
  $$ 
  given on objects by $\grpd_{/S}\mapsto\grpd^{S}$, 
  and by the identity on homs.  It all comes formally from the ind-pro duality
  starting with the anti-equivalence
  $$
  \lin \simeq \lin\op.
  $$
  Taking $S=1$ we see that $\grpd$ is an object of both $\infty$-categories, and
  mapping into it gives the duality isomorphisms of 
  Proposition~\ref{linearduals}.
\end{blanko}

\begin{blanko}{Monoidal structures.}
  The $\infty$-category $\ind\lin$ has two monoidal structures: $\oplus$ and
  $\tensor$, where $\grpd_{/I} \oplus \grpd_{/J} = \grpd_{/I+J}$ and $\grpd_{/I}
  \tensor \grpd_{/J} = \grpd_{/I\times J}$.  The neutral object for the first is
  clearly $\grpd_{/0} = 1$ and the neutral object for the second is
  $\grpd_{/1}=\grpd$.  The tensor product distributes over the direct sum.  The
  direct sum is both the categorical sum and the categorical product (i.e.~is a
  biproduct).  There is also the operation of infinite direct sum: it is the
  infinite categorical sum but not the infinite categorical product.  (Just as
  it happens for vector spaces.)

  Similarly, also the $\infty$-category $\pro\lin$ has two monoidal structures, $\oplus$
  and $\tensor$, given as $\grpd^I \oplus \grpd^J = \grpd^{I+J}$ and $\grpd^I
  \tensor \grpd^J = \grpd^{I\times J}$.  The $\tensor$ should be considered the
  analogue of a completed tensor product.  Again $\oplus$ is both the
  categorical sum and the categorical product, and $\tensor$ distributes over
  $\oplus$.  Again the structures allow infinite versions, but this times the
  infinite direct sum is a categorical infinite product but is not an infinite
  categorical sum.

  (To see the difference between the role of infinite $\oplus_\alpha$ in $\ind\lin$ and
  in $\pro\lin$: in $\pro\lin$ there is a diagonal map $\grpd^I \to \oplus_\alpha \grpd^I
  = \grpd^{\sum_\alpha I}$ given by sending $X \to I$ to $\sum_\alpha X \to
  \sum_\alpha I$.  This makes sense for a finite map $X\to I$, since the
  infinite sum of copies of that map is still finite, but it does not
  make sense in $\ind\lin$ since that $\sum_\alpha X$ is not finite.  On the other
  hand, $\ind\lin$ sports a codiagonal
  $\oplus_\alpha \grpd_{/I} = \grpd_{/\sum\alpha I} \to \grpd_{/I}$ given by
  sending $A \to \sum_\alpha I$ to the composite $A \to \sum_\alpha I \to I$
  (where the second map is the codiagonal for the infinite sum of $\infty$-groupoids).
  Since $X$ is finite it remains finite so there is no problem.  In contrast
  this construction does not work in $\pro\lin$: even if $A \to \sum_\alpha I$ is
  finite, $A \to \sum_\alpha I \to I$ will generally not be so.)
\end{blanko}

\begin{blanko}{Summability.}
  In algebraic combinatorics, the profinite stuff is often expressed in terms of
  notions of summability.  We briefly digress to see the constructions from this
  angle.

  
  For $B$ a locally finite $\infty$-groupoid, a $B$-indexed family $g : E \to B \times I$
  (as in \ref{scalar&hosum}) is called {\em summable} if the composite $E \to B
  \times I \to I$ is a finite map.  The condition implies that in fact the
  members of the family were already finite maps.  Indeed, with reference to the
  diagram
    $$\xymatrix{
    E_{b,i} \drpullback \ar[r] \ar[d] & E_i \drpullback \ar[r] \ar[d] 
    & E \ar[d] \\
    \{b\}\times \{i\} \ar[r] & B \times \{i\} \ar[r] \ar[d] 
    \drpullback & B\times I \ar[d] \\
    & \{i\} \ar[r] & I
    }$$
    summability 
    implies (by Lemma~\ref{lem:finitemaps}.\ref{item:fibrefinite}) that 
    each $E_i$ is finite,
    and therefore (by Lemma~\ref{lem:1->B} since $B$ is locally finite) we 
    also conclude that each $E_{b,i}$ is finite, which is precisely to
    say that the members $g_b :E_b \to I$ are finite  maps
    (cf.~\ref{lem:finitemaps}.\ref{item:fibrefinite} again).
    It thus makes sense to interpret the family as a family of 
    objects in $\Grpd_{/I}^{rel.fin.}$.  And finally we can say that 
    a summable family is a family $g:E\to B \times I$
    of finite maps $g_b : E_b \to I$, whose 
    homotopy sum
    $p\lowershriek(g)$ is again a finite map.
If $I$ is finite, then the only summable families are the finite 
families (i.e.~$E \to B \times I$ with $E$ finite).
%
%
%
%
A family $g : E \to B\times I$, given equivalently as a functor
$$
F: B \to \grpd^I,
$$
is summable if and only if it is a cofiltered limit of diagrams $F_\alpha: B \to
\grpd^{\alpha}$ (with $\alpha$ finite).

It is easy to check that
    a map $q:M \to T$ (between locally finite $\infty$-groupoids) is 
    finite if and only if for every finite map 
    $f:X \to M$ we have that also $q\lowershriek f$ is 
    finite.
%
Hence we find
\end{blanko}
\begin{lemma}
A span $I \stackrel p \leftarrow M \stackrel q \to J$
preserves summable families if and only if $q$ is finite.
\end{lemma}

\subsection{Duality}

\begin{blanko}{The perfect pairing.}
We have a perfect pairing
\begin{eqnarray*}
    \grpd_{/S} \times \grpd^S  &\longrightarrow& \grpd \\
    (p , f) &\longmapsto& f(p)
\end{eqnarray*}
given by evaluation.
In terms of spans, write the map-with-finite-total-space $p:X\to S$
as a finite span $1 \leftarrow X \stackrel p \to S$, and write the presheaf 
$f:S\to \grpd$ as the finite span $S \stackrel f \leftarrow F \to 1$, where $F$ 
is
the total space of the Grothendieck construction of $f$.
(In other words, the functor $F$ on $S$ corresponds to a linear functor
on $\grpd_{/S}$; write the representing span.)
Then the evaluation is given by composing these two spans, and hence
amounts just to taking the pullback of $p$ and $f$.

The statements mean: for each $p:X \to S$ in $\grpd_{/S}$, the map
\begin{eqnarray*}
  \grpd^S & \longrightarrow & \grpd  \\
  f & \longmapsto & f(p)
\end{eqnarray*}
is prolinear.
The resulting functor
\begin{eqnarray*}
  \grpd_{/S} & \longrightarrow & \un{\pro\Lin}(\grpd^S,\grpd)  \\
  p & \longmapsto & \qquad (f \mapsto f(p))
\end{eqnarray*}
is an equivalence of $\infty$-categories (cf.~Proposition~\ref{linearduals}).

Conversely, for each $f: S \to \grpd$ in $\grpd^S$, the map
\begin{eqnarray*}
  \grpd_{/S} & \longrightarrow & \grpd  \\
  p & \longmapsto & f(p)
\end{eqnarray*}
is linear.  The resulting functor
\begin{eqnarray*}
  \grpd^S & \longrightarrow & \un{\ind\Lin}(\grpd_{/S},\grpd)  \\
  f & \longmapsto &  \qquad (p \mapsto f(p))
\end{eqnarray*}
is an equivalence of $\infty$-categories (cf.~Proposition~\ref{linearduals}).
\end{blanko}

\begin{blanko}{Bases.}
  Both $\grpd_{/S}$ and $\grpd^S$ feature a canonical basis, actually an
essentially unique basis.  The basis elements in $\grpd_{/S}$ are the names
$\name s : 1 \to S$: every object $p:X \to S$ in $\grpd_{/S}$ can be written as a finite 
homotopy linear combination
$$
p = \int^{s\in S} \norm{X_s} \ \name s .
$$
Similarly, in $\grpd^S$, the representables $h^t:= \Map(t, - )$ form a basis:
every presheaf on $S$ is a colimit, and in fact a homotopy sum, of such
representables.
These bases are dual to each other, except for a normalisation: if $p=\name s$
and $f=h^t = \Map(t, - )$, then they pair to
$$
\Map(t,s) \simeq
\begin{cases}
    \Omega(S,s) & \text{ if } t\simeq s \\
    0 & \text{ else. }
\end{cases}
$$
The fact that we obtain the loop space $\Omega(S,s)$ instead of $1$ is actually a feature:
we shall see below that on taking cardinality we obtain
the canonical pairing
\begin{eqnarray*}
  \Q_S \times \Q^S & \longrightarrow & \Q \\
  (\delta_i,\delta^j) & \longmapsto & \begin{cases} 1 & \text{if } i=j \\
  0 & \text{else.}\end{cases}
\end{eqnarray*}
\end{blanko}

\subsection{Cardinality as a functor}

The goal is that each slice $\grpd_{/S}$ and each finite-presheaf category
$\grpd^S$ should have a notion of homotopy cardinality, with values in the
vector space $\Q_{\pi_0 S}$ and the profinite-dimensional vector space
$\Q^{\pi_0 S}$, respectively.  The idea of Baez, Hoffnung and
Walker~\cite{Baez-Hoffnung-Walker:0908.4305} is to achieve this by a global
assignment, which specialises to every slice $\infty$-category to define a
relative cardinality, a cardinality of families, by the observation that
families are special cases of spans, just as vectors can be identified with
linear maps from the ground field.  In our setting this amounts to functors
$\ind\lin \to \Vect$ and $\pro\lin\to \pro\vect$.

\begin{blanko}{Definition of cardinality.}\label{metacard}
  We define {\em meta cardinality} by the assignment
  $$
  \|\quad\|:\ind\lin\to\Vect
  $$
  defined on objects by
  $$
  \|\grpd_{/T}\|:=\Q_{\pi_0 T},
  $$
  and on morphisms by taking a finite-type span 
  $S \stackrel p \leftarrow M\stackrel q\to T$
  to the linear map
  \begin{eqnarray*}
    \Q_{\pi_0 S} & \longrightarrow & \Q_{\pi_0 T}  \\
    \delta_s & \longmapsto & \int^t \norm{M_{s,t}} \delta_t = \sum_t 
    \norm{T_{[t]}} \norm{M_{s,t}} \delta_t .
  \end{eqnarray*}
  That is, to the span $M$ we assign the matrix $A_{t,s} :=
  \norm{T_{[t]}} \norm{M_{s,t}}$,
  which is column finite since $M$ is of finite type.
  
  Dually we define a meta cardinality
  $$
  \|\quad\|:\pro\lin\to\pro\vect
  $$
  defined on objects by
  $$
  \|\grpd^S\|:=\Q^{\pi_0S},
  $$
  and on morphisms by the {\em same} assignment of a matrix to a finite span as 
  before.
\end{blanko}

\begin{prop}\label{cardfunctor}
  The
  meta cardinality assignments just defined
$$
\|\quad\|:\ind\lin\to\Vect,\qquad\qquad \|\quad\|:\pro\lin\to\pro\vect
$$
  are functorial.
\end{prop}
\begin{proof}
  First observe that the functor is well defined on morphisms.  Given a
  finite-type span $S \stackrel p \leftarrow M\stackrel q\to T$ defining linear
  functors $L: \grpd_{/S} \to \grpd_{/T}$ (as well as $L^\vee : \grpd^T \to
  \grpd^S$), the linear maps
  $$
  \|L\|:\Q_{\pi_0S}\longrightarrow\Q_{\pi_0T}
,\qquad\qquad 
\|L^\vee \|:\Q^{\pi_0T}\longrightarrow\Q^{\pi_0S}
$$
are defined with respect to the given (pro-)bases by the matrix $\|L\|_{t,s}=
                    \norm{M_{s,t}}      \norm{T_{[t]}} 
$. That is: 
$$
\|L\|\,\biggl(\sum_{s\in\pi_0S} c_s\,\delta_s\biggr)
\;=\;
\sum_{s,t} c_s\,
                     \norm{M_{s,t}}      \norm{T_{[t]}} 
\,\delta_t 
\;=\;
\sum_{s\in\pi_0S} c_s\!\int^{t\in T} \!\norm{M_{s,t}} \,\delta_t
\,,
$$
and
$$
\|L^\vee\| 
\,\biggl(\sum_{t\in\pi_0T} c_t\,\delta^t\biggr)
\;=\;
\sum_{s,t} c_t\,
                     \norm{M_{s,t}}      \norm{T_{[t]}} 
\,\delta_s 
$$

In particular, we note
$$
\|L^\vee\| 
\,(\delta^t)
\;=\;
\sum_s \norm{M_{s,t}}\norm{T_{[t]}} 
\,\delta_s 
  \,.
  $$

The matrix $\norm{M_{s,t}}\norm{T_{[t]}}$
  has finite entries and is column-finite: for each $s\in\pi_0S$ the fibre $M_s$
  is finite so the map $M_{s}\to T$ is finite by Lemma \ref{locfinbase}, and the
  fibres $M_{s,t}$ are non-empty for only finitely many $t\in\pi_0T$.  It is
  clear that equivalent spans define the same matrix, and the identity span
  $L=(S\leftarrow S\to S)$ gives the identity matrix: $\|L\|_{s_1,s_2}=0$ if
  $s_1,s_2$ are in different components, and
  $\|L\|_{s,s}=\norm{\Omega(S,s)}\norm{S_{[s]}}=1$ by Lemma~\ref{loop-finite}.
  It remains to show that composition of
  spans corresponds to matrix product: for $L=(S\leftarrow M\rightarrow T)$,
  $L'=(T\leftarrow N\rightarrow U)$ we have
  $$
  \norm{(M\times_TN)_{s,u}}=\int^{t\in T}\norm{M_{s,t}\times N_{t,u}}
=
\sum_{t\in\pi_0T}\norm{M_{s,t}} \norm{T_{[t]}}\norm{N_{t,u}}
$$
and so $\displaystyle
\|L'L\|_{u,s}=\sum_{t\in\pi_0T}
\norm{M_{s,t}}\norm{T_{[t]}}\norm{N_{t,u}}\norm{U_{[u]}}
=\sum_{t\in\pi_0T}\|L'\|_{u,t}\|L\|_{t,s}
$.
\end{proof}


\begin{blanko}{Cardinality of families.}
As a consequence of this proposition 
we obtain for each locally finite $\infty$-groupoid $T$ a notion of cardinality of 
$T$-indexed families.
Let  $T$ be a locally finite $\infty$-groupoid and define the functor  
  $$
  \norm{ \ \ } : \grpd_{/T}  \longrightarrow \normnorm{\grpd_{/T}}= 
  \Q_{\pi_0 T}, \qquad \norm{x}\;:=\;\|L_x\|\,\left(\delta_1\right).
  $$
Here $x:X\to T$ is an object of $\grpd_{/T}$ and  $\|L_x\|:\Q_{\pi_01}\to\Q_{\pi_0T}$ is induced by the linear functor $L_x$ defined by the finite  span $1\leftarrow X\xrightarrow{\,x\,}T$.
By the definition of $\|L\|$ in Proposition \ref{cardfunctor}, we can write
$$
\norm x
\;\;=\;\;
\sum_{t\in \pi_0T}
|X_t|\;|T_{[t]}|\;\delta_t\;\;=\;\;\int^{t\in T}\!|X_t|\;\delta_t
$$
\end{blanko}

\begin{lemma} Let $T$ be a locally finite $\infty$-groupoid.
\begin{enumerate}
\item
  If $T$ is connected, with $t\in T$, and $x:X \to T$ in $\grpd_{/T}$, then 
  $$
  \norm x \;=\; \norm X \, \delta_t \;\; \in\;\;\Q_{\pi_0 T}.
  $$
\item  The cardinality of $\name t:1\to T$ in $\grpd_{/T}$
  is the basis vector $\delta_t$.
\end{enumerate}
\end{lemma}
\begin{proof}
(1)  By definition, $\norm x\;=\;\norm{X_t}\,\norm T\, \delta_t$, and by 
Lemma~\ref{FEB}, this is $\norm X \,\delta_t$ 
\par\noindent
(2)  The fibre of $\name t$ over $t'$ is empty except when $t,t'$ are in the same component, so we reduce to the case of connected $T$ and apply (1).
\end{proof}

%
%
%

\begin{blanko}{Cardinality of presheaves.}
  We also obtain a notion of cardinality of presheaves: for each $S$, define
    $$
  \norm{ \ \ } : \grpd^S  \longrightarrow \normnorm{\grpd^S}= \Q^{\pi_0 S},
  \qquad \norm{f}\;:=\;\|L_f\|.
  $$
Here $f: S \to \grpd$ is a presheaf, and
$L_f : \grpd_{/S} \to \grpd$ its extension by inearity;
$L_f$ is given by the span $S \leftarrow 
F \to 1$, where $F\to S$ is the Grothendieck construction of $f$.
The meta cardinality of this span is then a linear map $\Q_{\pi_0 S} 
\to \Q_1$, or equivalently a pro-linear map $\Q^1 \to \Q^{\pi_0 S}$ --- in 
either way interpreted as an element in $\Q^{\pi_0 S}$.
In the first viewpoint, the linear map is
\begin{eqnarray*}
  \Q_{\pi_0 S} & \longrightarrow & \Q_1  \\
  \delta_s & \longmapsto & \int^1 \norm{F_s} \delta_1 = \norm{F_s} \delta_1
\end{eqnarray*}
which is precisely the function
\begin{eqnarray*}
  \pi_0 S & \longrightarrow & \Q  \\
  s & \longmapsto & \norm{f(s)} .
\end{eqnarray*}
In the second viewpoint, it is the prolinear map
\begin{eqnarray*}
  \Q^1 & \longrightarrow & \Q^{\pi_0 S}  \\
  \delta_1 & \longmapsto & \sum_s \norm{F_s} \delta^s
\end{eqnarray*}
which of course also is the function $s \mapsto \norm{f(s)}$.

In conclusion:
\end{blanko}

\begin{prop}\label{prop:card-prsh-ptwise}
  The cardinality of a presheaf $f: S \to \grpd$ is computed pointwise:
  $\norm f$ is the function
  \begin{eqnarray*}
    \pi_0 S & \longrightarrow & \Q  \\
    s & \longmapsto & \norm{f(s)} .
  \end{eqnarray*}
  In other words, it is obtained by postcomposing with the basic
  homotopy cardinality.
\end{prop}

\begin{eks}
  The cardinality of the terminal presheaf is the constant function $1$.
  In incidence algebras, this says that the cardinality of the zeta 
  functor~\ref{zeta} is the zeta function.
\end{eks}

\begin{eks}
  The cardinality of the representable functor $h^t : S \to \grpd$ is
  \begin{eqnarray*}
    \pi_0 S & \longrightarrow & \Q  \\
    s & \longmapsto & \norm{\Map(t,s)} = \begin{cases}
      \norm{\Omega(S,s)} & \text{ if } t\simeq s \\
      0 & \text{ else.}
    \end{cases}
  \end{eqnarray*}
\end{eks}

\begin{blanko}{Remark.}
    Note that under the Grothendieck-construction duality,
$\grpd^S \simeq \Grpd_{/S}^{relfin}$, the representable
presheaf $h^s$ corresponds to $\name s$, the name of $s$,
which happens to belong also the subcategory $\grpd_{/S} \subset
\Grpd_{/S}^{relfin}$, but that the cardinality of $h^{s}\in \grpd^S$ 
is {\em not} the same as the cardinality of $\name s \in \grpd_{/S}$.
This may seem confusing at first, but it is forced upon us
by the choice of normalisation of the functor
$$
\normnorm{ \ \ } : \ind\lin \to \Vect
$$
which in turn looks very natural since the extra factor $\norm{ 
T_{[t]}}$ comes from an integral.  A further feature of this
apparent discrepancy is the following.
\end{blanko}

\begin{prop}
  Cardinality of the canonical perfect pairing at the $\infty$-groupoid level yields
  precisely the perfect pairing on the vector-space level.
\end{prop}

\begin{proof}
  We take cardinality of the perfect pairing
\begin{eqnarray*}
    \grpd_{/S} \times \grpd^S  &\longrightarrow& \grpd \\
    (p , f) &\longmapsto& f(p) \\
    (\name s, h^t) &\longmapsto& \begin{cases}
    \Omega(S,s) & \text{ if } t\simeq s \\
    0 & \text{ else }
\end{cases}
\end{eqnarray*}
Since the cardinality of $\name s$ is $\delta_s$, while the 
cardinality of $h^t$ is $\norm{\Omega(S,t)}\delta^t$, 
the cardinality of the pairing becomes
$$
(\delta_s, \norm{ \Omega(S,t)}\delta^t) \longmapsto 
\begin{cases}
    \norm{\Omega(S,t)} & \text{ if } t \simeq s \\
    0 & \text{ else },
\end{cases}
$$
or equivalently:
$$
(\delta_s, \delta^t) \longmapsto 
\begin{cases}
    1 & \text{ if } t \simeq s \\
    0 & \text{ else },
\end{cases}
$$
as required.
\end{proof}

\begin{blanko}{Remarks.}
  The definition of meta cardinality involves a convention, namely to include
  the factor $\norm{T_{[t]}}$.  In fact, as observed by
  Baez--Hoffnung--Walker~\cite{Baez-Hoffnung-Walker:0908.4305}, other
  conventions are possible: for any exponents $\alpha_1$ and $\alpha_2$
  with $\alpha_1+\alpha_2=1$, it is possible to use the factor
  $$
  \norm{S_{[s]}}^{\alpha_1} \norm{T_{[t]}}^{\alpha_2} .
  $$
  They use $0+1$ in some cases and $1+0$ in 
  other cases, according to what seems more practical.  We think
  that these choices can be explained by which side of duality the
  constructions take place.


Our convention with the $\norm{T_{[t]}}$ normalisation yields the
`correct' numbers in all the applications of the theory that motivated
us, as exemplified below.
\end{blanko}

\begin{blanko}{Incidence coalgebras and incidence algebras of decomposition spaces.}
For $X$ a decomposition space with $X_0 \stackrel{s_0}\longrightarrow X_1
\stackrel{d_1}\longleftarrow X_2$ both finite maps, the dual space of $\grpd_{/X_1}$
is $\grpd^{X_1}$, underlying the incidence {\em algebra}.  Its multiplication
is given by a convolution formula.
In here there is a canonical element, the
constant linear functor given by the span $X_1 \leftarrow X_1 \to 1$
(corresponding to the terminal presheaf), which is called the {\em zeta functor}
\cite{GKT:DSIAMI-2}.  By \ref{prop:card-prsh-ptwise}, the cardinality of the
terminal presheaf is the constant function $1$.  Hence the cardinality of the
zeta  functor is the classical zeta function in incidence algebras.

The zeta function is the `sum of everything', with no symmetry factors.  A `sum
of everything', but {\em with} symmetry factors, appeared in our work
\cite{GalvezCarrillo-Kock-Tonks:1207.6404} on the \fdb and Connes--Kreimer
bialgebras, namely in the form of combinatorial Green functions (see also
\cite{Kock:1512.03027}).

The coalgebra in question is then the completion of the finite incidence algebra
$\Grpd_{/X_1}^{rel.fin.}$, where $X_1$ is the groupoid of forests (or
more precisely, $P$-forests for $P$ a polynomial functor).
Of course we know that $\Grpd_{/X_1}^{rel.fin.}$ is canonically
equivalent to $\grpd^{X_1}$, but it is important here to keep track of
which side of duailty we are on.  The Green function
lives on the coalgebra side, and more precisely in the completion.
(The fact that the comultiplication extends to the completion
is due to the fact that not only $d_1: X_2 \to X_1$ is finite,
but that also $X_2 \to X_1\times X_1$ is finite (a feature common
to all Segal $1$-groupoids).)  

Our Green function, shown to satisfy the
Fa\`a de Bruno formula in $\Grpd_{/X_1}^{rel.fin.}$ is
$T \to X_1$, the inclusion of the groupoid of $P$-trees $T$
into the groupoid of $P$-forests, hence essentially an identity functor.
Upon taking cardinality, with the present conventions, we obtain
precisely the series
$$
G = \sum_{t\in \pi_0 T} \frac{t}{\norm{\Aut(t)}}
$$
the sum of all tree weighted by symetry factors,
which is the usual combinatorial Green function in Quantum Field Theory.

The important symmetry factors appear correctly because we are on the coalgebra
side of the duality.

\end{blanko}